\documentclass[11pt,twoside,a4paper]{book}
\usepackage{amsmath,amsthm,amsfonts,amssymb,amscd,mathtools}
\usepackage[mathcal]{eucal}
\usepackage{latexsym}	
\usepackage{bm}
\usepackage[all,cmtip]{xy}	

\newcommand*{\jfrac}[2]{\genfrac{}{}{0pt}{}{#1}{#2}}
\newcommand{\xrar}{\xrightarrow}
\usepackage[utf8]{inputenc}
\usepackage[brazil, main=english]{babel} 
\usepackage{verbatim}
\usepackage[LGR,T1]{fontenc}
\usepackage{Alegreya} 

\usepackage[pdftex]{graphicx}
\usepackage{bm}
\usepackage{calligra}
\usepackage{wedn}
\usepackage[pdftex]{graphicx} 
\usepackage{faktor}
\usepackage{xfrac}          
\usepackage{setspace}                   
\usepackage{indentfirst} 
\usepackage{csquotes}               
\usepackage{makeidx}                    
\usepackage[nottoc]{tocbibind}          
\usepackage{courier}                    
\usepackage{type1cm}                    
\usepackage{listings}                   
\usepackage{titletoc}
 \usepackage[fixlanguage]{babelbib}
\usepackage[font=small,format=plain,labelfont=bf,up,textfont=it,up]{caption}
\usepackage[usenames,svgnames,dvipsnames]{xcolor}
\usepackage[a4paper,top=2.54cm,bottom=2.0cm,left=2.0cm,right=2.54cm]{geometry} 
\usepackage[pdftex,plainpages=false,pdfpagelabels,pagebackref,colorlinks=true,citecolor=DarkGreen,linkcolor=NavyBlue,urlcolor=DarkRed,filecolor=green,bookmarksopen=true]{hyperref} 
\usepackage[all]{hypcap}                
\usepackage[square,sort,nonamebreak,comma]{natbib}  
\fontsize{60}{62}\usefont{OT1}{cmr}{m}{n}{\selectfont}
\usepackage{pdfpages}
\RequirePackage{doi}
\usepackage{hyperref}

\usepackage{fancyhdr}
\renewcommand{\chaptermark}[1]{\markboth{\MakeUppercase{#1}}{}}

\renewcommand{\headrulewidth}{1.3pt}
\renewcommand{\headrule}{\hbox to\headwidth{\color{uibred}\leaders\hrule height \headrulewidth\hfill}}
\usepackage{subfig}
\usepackage{float}
\usepackage{dirtytalk}
\usepackage{tikz}
\usepackage{tikz-cd}
\usepackage{ragged2e}
\newtheorem{defn}{Definition}
\newtheorem{lem}{Lemma}
\newtheorem{mlem}{Mini-Lemma}
\newtheorem{thm}{Theorem}
\newtheorem{mthm}{Mini-Theorem}

\newtheorem*{pthm*}{Teorema}
\newtheorem{ex}{Example}
\newtheorem{prop}{Proposition}

 \newtheorem*{prob*}{Problema}

\newtheorem{cor}{Corollary}
\newtheorem{claim}{Claim}
\newtheorem{mclaim}{Mini-Claim}
\newtheorem{quest}{Q}
\newtheorem{qsconj}{$\boldsymbol{(\star)}$}

\newtheorem{rem}{Remark}
\newcommand{\C}{\mathbb{C}}
\newcommand{\CC}{\widehat{\mathbb{C}}}
\newcommand{\R}{\mathbb{R}}
\newcommand{\N}{\mathbb{N}}
\newcommand{\Z}{\mathbb{Z}}
\newcommand{\Q}{\mathbb{Q}}

\newcommand{\s}{\mathbb{S}}
\newcommand{\rr}{\overline{\mathbb{R}}}

\newcommand{\bg}{\mathcal{B}\mathcal{G}}
\newcommand{\bgg}{\boldsymbol{\mathcal{B}}\boldsymbol{\mathcal{G}}}
\newcommand{\p}{\mathbb{P}}

\newcommand{\cc}{\overline{\mathbb{C}}}

\newcommand{\CP}{\mathbb{C}\mathbb{P}}
\newcommand{\CPu}{\mathbb{C}\mathbb{P}^1}
\newcommand{\RPu}{\mathbb{R}\mathbb{P}^1}

\newcommand*{\myov}[1]{\overbracket[1.1pt][0pt]{#1}}

\everymath{\displaystyle}
\definecolor{deeppink}{rgb}{1.0, 0.08, 0.58}
\definecolor{uibred}  {HTML}{db3f3d}
\definecolor{uibblue} {HTML}{4ea0b7}
\definecolor{uibgreen}{HTML}{789a5b}
\definecolor{uibgray} {HTML}{d0cac2}
\definecolor{uiblink} {HTML}{00769E}
\makeatletter
\let\old@rule\@rule
\def\@rule[#1]#2#3{\textcolor{rulecolor}{\old@rule[#1]{#2}{#3}}}
\makeatother
\definecolor{rulecolor}{named}{uibred}
\usepackage{tabularx}
\usepackage{stackengine}

\graphicspath{{./figuras/}}             
\makeindex                              
\fontsize{60}{62}\usefont{OT1}{cmr}{m}{n}{\selectfont}
\cleardoublepage
\normalsize

\definecolor{laranja}{rgb}{1,0.45,0.05}

\begin{document}
\frontmatter 

\fancyhead[RO]{{\footnotesize\rightmark}\hspace{2em}\thepage}
\setcounter{tocdepth}{2}
\fancyhead[LE]{\thepage\hspace{2em}\footnotesize{\leftmark}}
\fancyhead[RE,LO]{}
\fancyhead[RO]{{\footnotesize\rightmark}\hspace{2em}\thepage}
\onehalfspacing  

\thispagestyle{empty}
\begin{center}
    \vspace*{2.3cm}
    \textbf{\Large{Branched coverings of the $2$-sphere.}}\\
    
    \vspace*{1.2cm}
    \Large{Arcelino Bruno Lobato do Nascimento}
    
    \vskip 2cm
    \textsc{
    Tese apresentada\\[-0.25cm] 
    ao\\[-0.25cm]
    Instituto de Matemática e Estatística\\[-0.25cm]
    da\\[-0.25cm]
    Universidade de São Paulo\\[-0.25cm]
    para\\[-0.25cm]
    obtenção do título\\[-0.25cm]
    de\\[-0.25cm]
    Doutor em Ciências}
    
    \vskip 1.5cm
    Programa: Matemática\\
    Orientador: Sylvain Bonnot\\

   	\vskip 1cm
    \normalsize{Durante o desenvolvimento deste trabalho o autor recebeu auxílio financeiro da CAPES}
    
    \vskip 0.5cm
    \normalsize{Povoado São Joaquim, Penalva-MA, 2021}
\end{center}

%
%
%
\newpage

%
%
%
%
\newpage
\thispagestyle{empty}
    \begin{center}
        \vspace*{2.3 cm}
        \textbf{\Large{Branched coverings of the $2$-sphere.}}\\
        \vspace*{2 cm}
    \end{center}

    \vskip 2cm

    \begin{flushright}
    This is the original version of the thesis,\\ as submitted to the thesis committee,\\ written by Arcelino Bruno Lobato do Nascimento.


   \vskip 2cm

    \end{flushright}

\pagebreak

\pagenumbering{roman}     

\chapter*{Abstract}
\noindent Nascimento, A. B. L. do \textbf{Branched coverings of the 2-sphere}. 
2021. xxx f. Tese (Doutorado) - Instituto de Matemática e Estatística,
Universidade de São Paulo, São Paulo, 2021.
\\

\justifying \emph{Thurston} obtained a combinatorial characterization for generic branched self-coverings that preserve the orientation of the oriented 2-sphere by associating a planar graph to them \cite{STL:15}. In this work, the Thurston result is generalized to any branched covering of the oriented 2-sphere. To achieve that the notion of local balance introduced by Thurston is generalized. As an application, a new proof for a Theorem of \emph{Eremenko-Gabrielov-Mukhin-Tarasov-Varchenko} \cite{MR1888795}, \cite{MR2552110} is obtained. This theorem corresponded to a special case of the B. \& M. Shapiro conjecture. In this case, it refers to generic rational functions stating that a generic rational function $R:\CPu \rightarrow \CPu$ with only real critical points can be transformed by post-composition with an automorphism of $\CPu$ into a quotient of polynomials with real coefficients. Operations against balanced graphs are introduced.
\\

\noindent \textbf{Keywords:}  branched coverings, cell graphs, geometric topology, combinatorics, balanced graphs

\tableofcontents    

\listoffigures            

\newpage

\section{INTRODUCTION}
The present work began with the task given by \emph{Sylvain Bonnot} of developing a computer program in the software \href{https://www.wolfram.com/mathematica/}{Mathematica} that would draw the preimage of the real line $\rr\subset\cc$ by a cubic rational function with real coefficients 
of the form $\phi_a:z\mapsto \dfrac{az^3 +(1-2a)z^2}{(2-a)z-1}$. The critical points of $\phi_a$ are all real points, namely, $0$, $1$, $\infty$ and $c(a)=\dfrac{2a-1}{a(2-a)}$ for $a\in \rr-\{-1,0,\frac{1}{2}, 1,2\}$. For each of these functions the inverse image of the real line yields a cellularly embedded graph into $\cc$, that is, the $1$-skeleton of a cellular decomposition of $\cc$. 
\emph{Sylvain Bonnot}'s interest was to describe how these graphs vary as we vary the critical point $c(a)$. This was done and is presented in Chapter $\ref{cap-04}.$

The central purpose of the research presented in this thesis is to determine combinatorial objects that can characterize rational functions considering their critical configuration. Consisting, therefore, in a certain sense, in a dual theory to the one initiated by Hurwitz that studies the branched coverings of the two-dimensional $\s^2$ sphere taking into account their critical values.

The family of functions $\phi_a$ was presented to me by \emph{Sylvain Bonnot} through a post by 
\href{http://www.math.hawaii.edu/~xander/index.html}{Xander Faber} on the \emph{Mathematics} question \& answer site, 
\href{https://mathoverflow.net/}{Mathoverflow}. As the title of the post presumes, 
\href{https://mathoverflow.net/questions/102506/determining-rational-functions-by-their-critical-points?rq=1}{Determining rational functions by their critical points}, \emph{Xander Faber} draws attention to the problem of determining rational functions from its critical configuration.

Fulfilling the design stated above we propose a combinatorial description of orientation-preserving branched coverings of the two-dimensional sphere via a cellular graph that captures their critical configuration. 

The most distant ancestor to this idea of to capture the essence of a mapping by restricting it to a  graph is the combination 
of the \emph{Alexander (trick) lemma} \cite{MR3203728} and the \emph{Schöenflies theorem} \cite{MR728227} that allows us to distinguishes homeomorphisms of a closed 2-cell, up to isotopy, by its restriction to the boundary circle.

\justifying A branched covering of genus $g$ 
 of the sphere $\s^2$ is a continuous surjective map $f:S_g\rightarrow\s^2$ from a genus $g$ surface $S_g$ to the 2-sphere 
that, around each point $p \in S_g$, it is given in local topological coordinates by $z \mapsto z^{e}$ around $0\in\C$ with $e:=e(p) \geq 1$ an integer and such that $|\{p\in S_g ; e(p)>1 \}|<\infty$. Each point $p_0 \in \{p\in S_g ; e(p)>1 \}$ is called \emph{critical point} of $f$ and its image $f(p_0)$ we call \emph{critical value}. The integer $e(p)\geq 1$ is the local degree (or, ramification index) of $f$ at $p$. The degree of a branched covering is the cardinality of the set $\{p \in S_g ; f(p)=q\}$ for some $q \in \s^2 - \{p\in S_g ; e(p)>1 \}$.

For a branched covering, the data consisting of its critical points, their multiplicities and their clusterings according to their image by that map is called \emph{critical configuration}. This information is given through a list of integer partitions of the degree of the branched covering, one for each critical value, called the \emph{passport} of the map, together with the sequence of critical points in $\s^2$.

\justifying The notion of equivalence suitable for the classification of branched coverings according to their critical configuration is the one that identifies them via post-composition with homeomorphisms of $\s^2$  that preserve the orientation. 
Two equivalent branched coverings, according to that notion given above, have the same critical configuration.

\justifying In the strict context of rational functions of the \emph{Riemann sphere}, $\CPu$ the equivalence considered specializes to identify rational functions by post - composition with automorphisms of $\CPu$, that is, by post - composition with \emph{Möebius transformations}.  

A natural problem 
is the counting of the equivalence classes for a prescribed configuration. Some results for this problem are known, as described below.

The problem of establishing combinatorially the equivalence class count for a given critical configuration is the pivotal guiding point for the current research.


To this end, we will associate a combinatorial object to a branched covering of $\s^2$ by a closed oriented surface $S_g$, $f: S_g\rightarrow\s^2$. This combinatorial object is a cellularly embedded graph on $S_g$, i.e., the $1$-skeleton of a cellular decomposition of $S_g$, just like the planar graphs that appear as an inverse image of $\rr\subset\cc$ by the rational functions $\phi_a$ presented above. Although, as mentioned earlier, the present research takes as its starting point questions about rational functions $f:\cc\rightarrow\cc$,  for which is already presupposed an analytic structure, we will consider branched coverings of the sphere by closed surfaces of any genus and these will be considered prior as topological objects.

\begin{figure}[h!]
\centering
\subfloat{\tikz[remember
picture]{\node(1AL){\includegraphics[width=4cm]{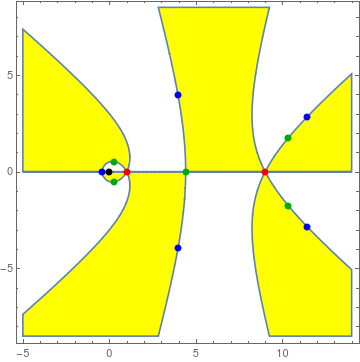}};}}%
\hspace*{1.5cm}%
\subfloat{\tikz[remember picture]{\node(1AR){\includegraphics[width=4cm]{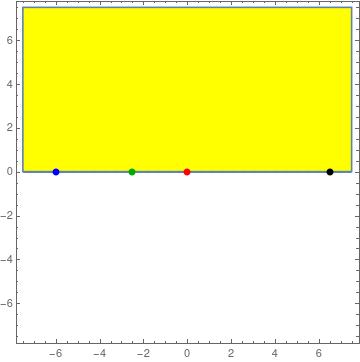}};}}
\caption{}
\end{figure}
\tikz[overlay,remember picture]{\draw[-latex,thick] (1AL) -- (1AL-|1AR.west)
node[midway,below,text width=1.5cm]{\hspace{.65cm}$f$};}

In the article \cite{STL:15}, \textrm{Sarah Koch \& Tan Lei} present the ideas and a result obtained by \textrm{William Thurston} in an email group discussion whose central goal 
 was the \textrm{determination of the form of a rational function of the complex projective line $\CPu$}. \emph{Thurston}, then introduced a class of planar graphs, named \emph{balanced graphs}, capable of combinatorially representing a generic branched selfcovering of the sphere. The graphs mentioned in the previous paragraph are a generalization of the \textrm{balanced graphs} defined by \textrm{Thurston}, as Thurston considered only regular planar graphs of degree $4$ with $2d-2$ vertices.

\justify{However, this was a later discovery in the course of the research presented here,  since the strategy of considering such graphs to represent rational functions (the starting point, and arrival/return point as well) has as inspiration the \emph{Dessins d'Enfants} (children's drawings) intruded by \emph{A. Grothendieck} to study the \emph{Absolute Galois} group  $\text{Gal}(\overline{\Q}, \Q)$ and \emph{Arithmetic Riemann Surfaces} (a Riemann surface is said to be \textrm{Arithmetic} if it admits an algebraic model defined over the \textrm{the field of Algebraic Numbers} $\overline{\Q}$) \cite{MR1483107},\cite{MR1305390},\cite{MR2895884},\cite{lando:03}). This strategy also naturally stemmed from casual conversations with \textrm{Sylvain Bonnot} about some mathematical curiosities, in particular, about degenerations of hyperbolic structures in manifolds of dimension $2$ and $3$ (\cite{MR769158},\cite{MR3053012} \cite{MR1855976}) as well as on the combinatorial structure of \emph{Moduli spaces} of Riemann surfaces via combinatorial representations of the geometric structures of these surfaces \cite{MPcom2}, \cite{MPcom1},\cite{MR1734132}, \cite{MR2497787}, those are theories in which graphs embedded in surfaces play a important role}.

\justifying The problem of counting equivalence classes of rational functions of $\CPu$  was considered previously by \emph{Eisenbud} \& \emph{Harris} in \cite{EiH:83} and by \emph{Lisa Goldberg} in \cite{Gold:91}. For the adjacent Schubert problem, Eisenbud and Harris established the necessary transversality for the intersections of Schubert varieties involved. 
The tranversality ensure zero dimensionality of the intersection and the number of points in it are computed by the \emph{Pieri formula} \cite[Theorem 9.1]{EiH:83}, \cite{Oss:06}, \cite{SF:11}, \cite{Fult:84}. 

\emph{Goldberg} established a combinatorial formula for the count of rational functions with generic critical configuration, assuring that \emph{by fixing the degree $d\geq 2$ and imposing the rational functions to have $2d-2$ critical points all with multiplicity $2$, there are}

\[\rho(d):=\frac{1}{d}{2d-2\choose d-1}\] 
\emph{equivalence classes of rational functions for each list of $2d-2$ points in general position in $\CPu$ prescribed as the critical points.}

This result was established using \emph{Algebraic Geometry}, more precisely, by translating it into a \emph{Schubert's problem}.

The \emph{Enumerative Geometric} problem to which Goldberg reduced the counting problem 
is:

\begin{prob*}\label{prob1-eg}
Given $2d - 2$ lines in general position in the projective space $\CP^d$, how many projective subspaces of codimension $2$ intersect all those lines?
\end{prob*}

The integer $\rho (d)$ is called the $d$-\emph{Catalan number}. These numbers are highly present and recurrent in \emph{Discrete Mathematics}, having a huge number of combinatorial interpretations  (see \cite{Stancat:15}). Moreover, \emph{Catalan numbers} often manifest themselves in several areas beyond  \emph{Discrete Mathematics} without there being an obvious combinatorial reason for such an appearance. For example, in the problem of determining the degree of applications or dimension of algebraic varieties in \emph{intersection theory} \cite{GH:94}\cite{Fult:84} and \emph{Schubert calculus} (\cite{GH:94}\cite{KL:72}) for Grassmanians, which a priori are problems involving much more sophisticated structures apart from the discrete mathematics.

A complete solution with obtaining a generic combinatorial formula for this problem was given by I.  Scherbak in \cite{Scherbak:02}. Such a result was established by combining \emph{Schubert's Calculus, Representation Theory, Fuchsian Differential Equations and KZ Equation Theory}.

Making use of the \emph{Limit Linear Series Theory} developed by \emph{Eisenbud} \& \emph{Harris} in \cite{EiH:86}, B. Osserman in \cite{Oss:03} established the count for the more general case of branched coverings of the sphere (including that one of positive genus over $\CPu$), i.e., He counts the rational functions with positive genus domain. Actually, the results obtained by  \emph{Osserman} are more general, they are for \emph{Linear Series} of dimension greater than $1$ as well.

Eremenko and Gabrielov in \cite{MR1888795} by proving the simplest case of the conjecture of \emph{B} \& \emph{M Shapiro} showed that the number of equivalence classes of rational $\CPu$ functions of degree $d$ with $2d-2$ prescribed critical points contained in $\RPu$ is at least $\rho(d).$ Then, this means that the genericity constraint on the prescription of the $2d-2$ critical points imposed in Goldberg's result can be taken off. This a kind of phenomenon/problem is referred to as the \emph{reality of the Schubert Calculus} \cite[and references therein]{SF:11} in \emph{Enumerative Geometry}. 

The conjecture of \emph{B} \& \emph{M Shapiro}, now a theorem due to mathematicians E. Mukhin, V. Tarasov and A. Varchenko \cite{MR2552110}, states that if the \emph{Wronskian Determinant} of a list of polynomials of degree $d$ with complex coefficients, $f_1 (z), f_2 (z), \cdots , f_m (z)\in\C[z]$, has only real zeros, then the vector subspace $\langle f_1(z), f_2(z), \\ 
\cdots , f_d(z)\rangle_{\C}\subset \C [z]$ has a basis in $\R [z]$.

The \emph{Wronskian Determinant} is the polynomial  
\[\textrm{W}(f_1, f_2, \cdots , f_m) := \det\left( \left( \frac{d}{dt}\right)^{i-1} f_j(t)\right)_{i,j=1,\cdots ,m}\]
The degree of $\textrm{W}(f_1, f_2, \cdots , f_m)$ is at most $m(d-1)$.

In \cite{MR1888795}, Eremenko and Gabrielov introduced a cellular decomposition of the \emph{Riemann sphere} $\CPu$, which they called ``net''. They use this cellular decomposition of the sphere to construct the expected number of classes of real rational functions. These cellular graphs are particular examples of the \emph{balanced graphs} introduced by \emph{Thurston} \cite{STL:15}. 

Thurston \cite{STL:15} has established a complete combinatorial characterization of generic branched selfcoverings of the two dimensional sphere $\s^2$. A branched covering of degree $d$, $\s^2\rightarrow\s^2$, is said to be generic when it has the maximum number of critical points, $2d-2$
(or equivalently, when all its critical points have ramification index $2$).

\begin{pthm*}[\cite{STL:15}]\label{tthm00}
A $4$-regular planar oriented graph $\Gamma$ with $2d-2$ vertices is equal to $f^{-1}(\Sigma)$ for some branched covering of degree $d$, $f:\s^2 \rightarrow\s^2$ and some \emph{Jordan curve}, $\Sigma\subset \s^2$, containing the critical values of $f$ if and only if :
\begin{itemize}
\item[$\boldsymbol{1. }$]{\textbf{global balancing:} for any alternating \textcolor{deeppink}{A}-\textcolor{blue}{B} coloration of the faces of $\Gamma$, there are $d$ faces of the \textcolor{deeppink}{A} and $d$ faces of the \textcolor{blue}{B}, and}
\item[$\boldsymbol{2. }$]{\textbf{local balancing:} any cycle oriented in $\Gamma$, which is incident to only faces of color \textcolor{deeppink}{A} on its left, 
 contains, in its interior, more faces of the color \textcolor{deeppink}{A} than faces of the color \textcolor{blue}{B}.}
\end{itemize}
\end{pthm*}
See definitions in Chapter $\ref{cap-03}$.

The general version encompassing branched coverings of $\s^2$ by closed surfaces of any genus and with any admissible critical configuration is given(see Section $\ref{cap-03}$). To this end, the definition of \emph{local balance} is extended so that it can properly capture the base topology.  

We introduce two classes of {cellular graphs} called \emph{Pullback graphs} [$\ref{pullbackg}$] and \emph{Admissible graphs} [$\ref{adm-g}$]. We show that an admissible graph actually encodes a recipe for constructing a branched covering of $\s^2$ [see $\ref{subsec:bcconstruction}$]. Thus the less obvious direction of the \emph{Thurston's Theorem} (generalized) [$\ref{THURSTHG1}$], consisting of to show that balanced graphs are preimage by branched coverings of special curves, transmutes into the task of ensuring that a balanced graph, i.e., a cellular graph that satisfies the global and local balance conditions, can be promoted to a \emph{admissible graph}. Half of the proof of that comprises of 
an underlying problem in (abstrac) \emph{graph theory} [$\ref{vertex-capacity-fn}$], it suffices to ensure that the enriched balanced graph admits a good vertex labelling turning it on an admissible graph. In the generic planar case \emph{Thurston} ritchs this by resorting to \emph{Cohomology}.

A genelarization for any branched selfcovering of $\s^2$ of the \emph{Thurston}'s theorem stated above was also obtained by \emph{J. Tomasini} \cite{Tomtese} in his doctoral thesis. 
He did not follow the approach introduced by \emph{Thurston}.  
Guided by a usual approach in \emph{Hurwitz's theory}, \emph{Tomasini} had considered a star map consisting into a collection of Jordan arcs connecting a chosen regular point of the branched covering, say $f$, to each critical value of $f$. Then, he consider the preimage of this cellular graph in order to get a combinatorial object associated to the branched covering as \emph{Thurnston} 
had proposed. He translated the balance condition of Thurston to a class of cellular bipartite planar graphs and then proved a complete planar version of the Thurston theorem. \emph{Tomasini} also established some results concerning the decomposition of its balanced graphs following the decompositions operations introduced by \emph{Thurston} in \cite{STL:15}. 



We count the globally balanced real graphs (these have as underlying graph those planar graphs considered by Eremenko \& Gabrielov in 
\cite{MR1888795}). For every $2d-2$ points in $\overline{\R}$ there exist $\rho(d)$ real globally balanced graphs with these points as vertices [$\ref{2.5-sect}$]. We also show that globally balanced real graphs are always locally balanced [$\ref{realgb-is-lg}$] and in this way it is established that there exist at least $\rho(d)$ equivalence classes of generic real functions with their $2d-2$ prescribed critical points. This, combined with \emph{Goldberg}'s result on counting equivalence classes of generic rational functions with pre-fixed critical points \cite{Gold:91}, culminates into a new proof [$\ref{sha-conj}$] for the \emph{Theorem of Eremenko-Gabrielov-Mukhin-Tarasov-Varchenko} (\cite{MR1888795},\cite{MR2552110}), which previously corresponded to a case of Shapiro's \emph{conjecture}.

A bunch of operations against \emph{balanced graphs} are introduced [$\ref{oper-bg}$]. These operations are interesting due 
the fact that they allow us to understand 
the structures of these objects and also allow us to produce more complex specimens of them from simpler ones. Some of these operations were formalized from computational observation of how the graph changes in parametric families (see Chapter $\ref{cap-04}$), thus representing degenerations, that is, the changes of the critical configuration. Some of them embody the changes of the isotopy class of the post-critical curve for a fixed branched covering. So, in this way, we could be able to, probably, combinatorially encode the structure of the space of branched coverings. These operations defines over the class of balanced graphs the structure of a groupoid. 

Other guiding reasons for the consideration of the operations based on balanced graphs are:

\begin{itemize}
\item{ to stablish a \emph{Reconstrution Principle}, that is, the possibility of to ensure the validity of the conjectural fact that any balanced graph can be obtained by the more simplest ones in genus $0$ and $1$ through concatenation of operations.}


\item{Is expected that, as in \cite{MR3436154}, \cite{MR3366120}, \cite{MR1396978}, \cite{MR3723168}, \cite{eynard2010laplace}, \cite{eynard:16}, this operation may produce relations into the collection of the generating series for the counting of the balanced graphs and in the topological recursion for them;}
\item{to use these approach to achieve the combinatorial proof asked by Lisa Goldberg in \cite[{PROBLEM}, at page $132$]{Gold:91} to the counting problem of equivalence classes of generic rational functions of $\CPu$ for prescribed critical points in general position.}
\end{itemize}

\subsection*{These text is organized as follows}

Chapter $\ref{cap-02}$ introduces the basic elements that support the research. The foundational results therein are conveniently presented in accordance with the taste and general point of view of the research. References for proofs are given. This chapter also includes some simple new results of technical character, namely, the proposition $\ref{vertex-capacity-fn}$ into \emph{Graph Theory}, this result is fundamental to the proof of the generalization of the \emph{Thurston} Theorem $\ref{tthm00}$, Theorem $\ref{THURSTHG1}$; and the Proposition $\ref{simultiso}$ about isotopy of collection of \emph{Jordan} arcs into surfaces. 

Chapter $\ref{cap-03}$ contains the main contributions of this thesis. There we develop the theory of combinatorial representation of a branched covering through cellular maps. We explain Thurston's proposal to capture the essence of a generic branched selfcovering by a planar graph, and then this idea is extended to any branching covering of the sphere. To this end we introduce the \emph{local balance condition} $\ref{localbalance}$ for positive genus cellular graphs and this definition recovers the one introduced by \emph{Thurston} in the generic planar case. Thurston's theorem $\ref{tthm00}$ is completely generalized, $\ref{THURSTHG1}$.

The class of \emph{Pullback Graphs} [$\ref{pullbackg}$] and \emph{Admissible Graphs} [$\ref{adm-g}$] are introduced. The \emph{Pullback Graph} is the combinatorial object rised from a branched covering whereas the \emph{Admissible Graph} is essentially a diagrammatic recipe for construction of a branched covering. Theorem $\ref{bcfromg2}$ says that this classes are essentially the same assuring that any \emph{Admissible Graphs} is realized as a \emph{Pullback Graphs}. This Chapter also presents a range of operations against balanced graphs [$\ref{oper-bg}$]. There are several reasons for introducing these operations. Some of these operations were formalized from computational observation of how the graph changes in parametric families (see Chapter $\ref{cap-04}$), thus representing degenerations, that is, the changes of the critical configuration or they embodies the changes of the isotopy class of the post-critical curve for a fixed branched covering. So, in this way, we could be able to, possibly, combinatorially encode the structure of the space of branched coverings. 
In this chapter, as an aplication of the \emph{Thurston Theorem} the simplest case of \emph{Shapiro conjecture} is proven.

The Chapter $\ref{cap-04}$ consists of a brief study of the generic cubic rational functions. For those real generic cubic rational function is showed 
that the \emph{Pullback Graph} relative to the post-critical curve $\rr$ distinguishes the two equivalence classes. Unfortunately this does not happens for complex, non real, generic cubic rational functions. Examples are given. Some results on the equivalence relation and the isotopy type of pullback graph are given.

\subsection*{$\bullet$ In a nutshell, this thesis contains:}
\begin{itemize}
\item{\textcolor{DarkGreen}{definition of adimissible graphs;}}
\item{\textcolor{DarkGreen}{construction of branched coverings from admissible graphs;}
\begin{itemize}
\item{\textcolor{DarkGreen}{in particular, construction of real rational functions;}}
\end{itemize}}
\item{\textcolor{DarkGreen}{definition of balanced graphs with positive genus;}}
\item{\textcolor{DarkGreen}{generalization of a theorem of \emph{Thurston};}}
\item{\textcolor{DarkGreen}{definition of operations on balanced graphs;}}
\item{\textcolor{DarkGreen}{demonstration that globally balanced real graphs are locally balanced;}}
\item{\textcolor{DarkGreen}{proof of a case of the conjecture of B \& M Shapiro;}}
\item{\textcolor{DarkGreen}{slight study of generic cubic rational functions.}}
\end{itemize}

\mainmatter


\fancyhead[RE,LO]{\thesection}

\singlespacing              

\chapter[Foundational results]{\rule[0ex]{16.5cm}{0.2cm}\vspace{-23pt}
\rule[0ex]{16.5cm}{0.05cm}\\Foundational results}\label{cap-02}

This chapter is not intended to present a detailed study or to develop in-depth the areas and results that underpin this work. Thus, it is meant to be a brief review and a base point for references.

\section{Topology, Coverings and Branched Coverings Spaces}

The main objects we shall work with are manifolds and maps between them. So, let's recall them.

Manifold are topological spaces that lookfs locally like a \emph{Euclidean space}.

\begin{defn}[(\textbf{Top}ological) $n$-manifold (with boundary)]\label{n-manif}
A topological $n$-manifold is a \emph{second countable Hausdorff connected topological space} $M$ for which there exists a family of pairs $\{(M_{\lambda},c_{\lambda})\}_{\Lambda}$, called \emph{atlas}, with the following properties:
\begin{itemize}
\item[$\boldsymbol{(1)}$]{for each $\lambda \in \Lambda$, $M_{\lambda}\subset M$ is an open subset of $M$ and $\bigcup_{\lambda\in\Lambda} M_{\lambda}=M$;}
\item[$\boldsymbol{(2)}$]{for each $\lambda \in \Lambda$, $c_{\lambda}:M_{\lambda}\xrar{\quad}\boldsymbol{\mathbb{H}^n}:=\{(x_1 ,x_2 ,\cdots ,x_n);x_1\geq 0\}$ is a homeomorphism for $\boldsymbol{\mathbb{H}^n}$ with the induced topology from $\R^n$.}
\end{itemize}
We call $(M_{\lambda},c_{\lambda})$ by a \emph{chart} of $M$, and if a pont $p\in M_{\lambda}$ singled out we say that $(M_{\lambda},c_{\lambda})$ is a \emph{chart} of $M$ around $p$.

The set of all points in $M$ that have a neighborhood homeomorphic to $\mathcal{H}^n$ but no neighborhood homeomorphic to $\R^n$ is the boundary of $M$ and is denoted by $\partial{M}$ and a point $p\in\partial{M}$ is a boundary point of $M$. $M-\partial{M}$ is the interior of $M$ and a point $p\in M-\partial{M}$ is a interior point of $M$.

A topological $n$-manifold $M$ is said to be compact if the underline topological space $M$ is compact.
\end{defn}

\begin{prop}[''boundary manfold´´]\label{top-b}
If $M$ is a topological $n$-manifold with boundary, then $\partial{M}$ is a
topological $(n-1)$-manifold without boundary. If $M$ compact, then $\partial{M}$ is too.
\end{prop}

\begin{defn}[compact manifold]\label{submani}
A topological $n$-manifold $M$ is said to be compact if the underline topological space $M$ is compact.
\end{defn} 

\begin{defn}[$p$-submanifold]\label{submani}
Let $M$ be an topological $n$-manifold with boundary. A $p$-dimensional submanifold of $M$ is a closed subset $L$ of $M$ for which there is an atlas $\{(M_{\lambda},c_{\lambda})\}_{\Lambda}$ of $M$ and $p\in\{0,\cdots , n\}$ such that for all $x\in L$ in the interior of $M$ there is a chart $(M_{\iota},c_{\iota})\in \{(M_{\lambda},c_{\lambda})\}_{\Lambda}$ such that $x\in M_{\lambda}$ and
\[c_{\iota}(L\cap M_{\iota})=\{(0,\cdots ,0)\}\times\R^{p}\subset\R^n\]
 for all $x\in L$ in the boundary of $M$ there is a chart $(M_{\beta},c_{\beta})\in \{(M_{\lambda},c_{\lambda})\}_{\Lambda}$ such that $x\in M_{\lambda}$ and
\[c_{\beta}(L\cap M_{\beta})=\{(0,\cdots ,0)\}\times\boldsymbol{\mathbb{H}^p}\subset\boldsymbol{\mathbb{H}^n}\]
and such that
\[c_{\beta}(x)\in\{(0,\cdots ,0)\}\times\partial{\boldsymbol{\mathbb{H}^p}}\subset \partial{\boldsymbol{\mathbb{H}^n}}\]
\end{defn}

\begin{defn}[closed manifold]\label{c-manif}
A topological $n$-manifold compact with empty boundary is said to be a closed $n$-manifold.
\end{defn}

\begin{defn}[embedding]\label{emb-manif}
Let $L, M$ be manifolds. A map $f : L \xrar{\quad} M$ is an
embedding if it is a homeomorphism onto its image $f(L)$ and $f(L)$ is a
submanifold of $M$.
\end{defn}

Now we highlite a notion of relation between manifolds.

\begin{defn}[homotopy]
Two \emph{continuous maps} $f, g : M \xrar{\quad}N$ are homotopic
if there is a continuous map $H : M \times [0, 1]\xrar{\quad}N$ such that $H(x, 0) = f$ and $H(x, 1) = g(x)$ for all $x \in M$. The map $H$ is called a homotopy between $f$ and $g$.
\end{defn}
Through the study of topology and geometry of manifolds a more stric type of homotopy is very often considered. 

\begin{defn}[isotopy]
Two \emph{embeddings} $f, g : M \xrar{\quad}N$ are isotopic if there
is a continuous map $H : M \times [0, 1]\xrar{\quad}N$ such that $H(x, 0) = f$ and $H(x, 1) = g(x)$ for all $x \in M$ and such that for all $t\in [0, 1]$, the map $f_t$ defined by $H( \cdot{}, t)$ is an embedding. The map $H$ is called an isotopy between $f$ and $g$.
Two submanifolds $N_1$, $N_2$ of $M$ are isotopic if their inclusion maps are isotopic.
\end{defn}

\begin{defn}[relative homotopy/isotopy]
A homotopy (or isotopy) $H : M \times [0, 1]\xrar{\quad}N$ between maps $f, g : M \xrar{\quad}N$ is said to be relative to a subset $A\subset M$ if the points in $A$ stay fixed throughout the homotopy(isotopy),i.e, for every $t\in[0,1]$, $f(a)=H(t,a)=g(a)$ for all $a\in A$.
\end{defn}

\begin{defn}[(\textbf{Top}ological) surface]\label{top-surf}
A topological surface is a topological $2$-manifold.
\end{defn}
\begin{ex}[basic examples]\label{ex.1}\hspace{15cm}

\begin{itemize}
\item[\textbf{$\jfrac{\text{the}}{\text{plane}}$:}]{Obviously, $\R^2$ itself is a surface. Its open subset are also  immediate examples of a surface.}\\

\item[\textbf{$\jfrac{\text{the}}{\text{disk}}$:}]{another simple example of surface is that one called disk. It is $\mathbb{D} = \{(x,y)\in\R^2 ; x^2 +y^2 \leq 1\}$. $\mathbb{D}$ is an suface with boundary, $\partial{}\mathbb{D}=\s^1:=\{(x,y)\in\R^2;{x^2 + y^2}=1\}$.

\begin{figure}[H]{
    \begin{center}
    {{\includegraphics[width=3cm]{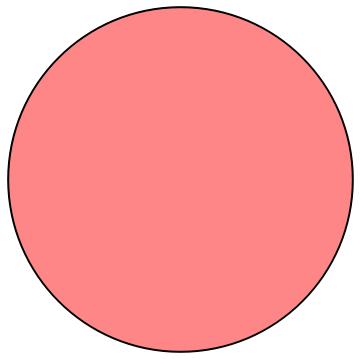}}}  \caption{disk}
    \label{disk-surf}
    \end{center}}
    \end{figure}}
    
\item[\textbf{$\jfrac{\text{the}}{\text{2-sphere}}$:}]{ The set $\s^2 := \{(x,y,z) \in\R^{2+1} ; x^2+y^2+z^2 = 1\}$ is an 2-dimensional manifold called the $2$-sphere. The Stereographic projection provides a homeomorphism $h : \s^2-{(0, 0, 1)}\xrar{\quad}\R^n$. Thus any point $x\in\s^2$ such that $x \neq (0, 0, 1)$ has the neighborhood $\s^2 - \{(0, 0, 1)\}$ that is homeomorphic to $\R^2$. To exhibit a neighborhood of $(0, 0, 1)$ that is homeomorphic to $\R^2$ we compose the reflection in $\R^2\times \{0\}$ with $h$ to obtain
$h': \s^2-\{(0, 0, -1)\} \xrar{\quad} \R^2$. Thus $\s^2-\{(0, 0, -1)\}$ is a neighborhood of $(0, 0, 1)$ homeomorphic to $\R^2$.
\begin{figure}[H]{
    \begin{center}
    {{\includegraphics[width=6.15cm]{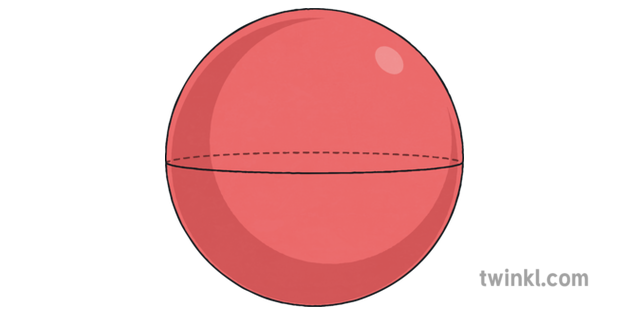}}}\quad  
    \caption{2-sphere}
    \label{b-surf}
    \end{center}}
    \end{figure}}

   \item[\textbf{$\jfrac{\text{the}}{\text{2-torus}}$:}] {
 The $2$-torus is a quotient space obtained as follows: Consider the subgroup of translations $G$ in $\R^2$ generated by the maps $(x,y)\mapsto{T_{1}} (x+1,y)$ and $(x,y)\mapsto{T^{1}} (x,y+1)$ acting in $\R^2$. Two points $x, y \in \R^2$ are identified if and only if there is a $g\in G$ such that $g(x) = y$. Let $p: \R^2\xrar{\quad}\mathbb{T}^2:=\faktor{\R^2}{G}$. For a point $[p]\in\mathbb{T}$ take the open disk centered at $p$ of radio  $1/16$,$B(x,1/16)$, then we can see that for all $g\in G$, $g(B(x,1/16))\cap B(x,1/16)=\emptyset$. In this way, $p^{-1}:p(B(x,1/16))\xrar{\quad} B(x,1/16)$ is a homeomorphism. Therefore, $\mathbb{T}^2$ is everywhere locally like $\R^2$, then is a topological 2-manifold.
 
 \begin{figure}[H]
    \begin{center}
    {{\includegraphics[width=6.15cm]{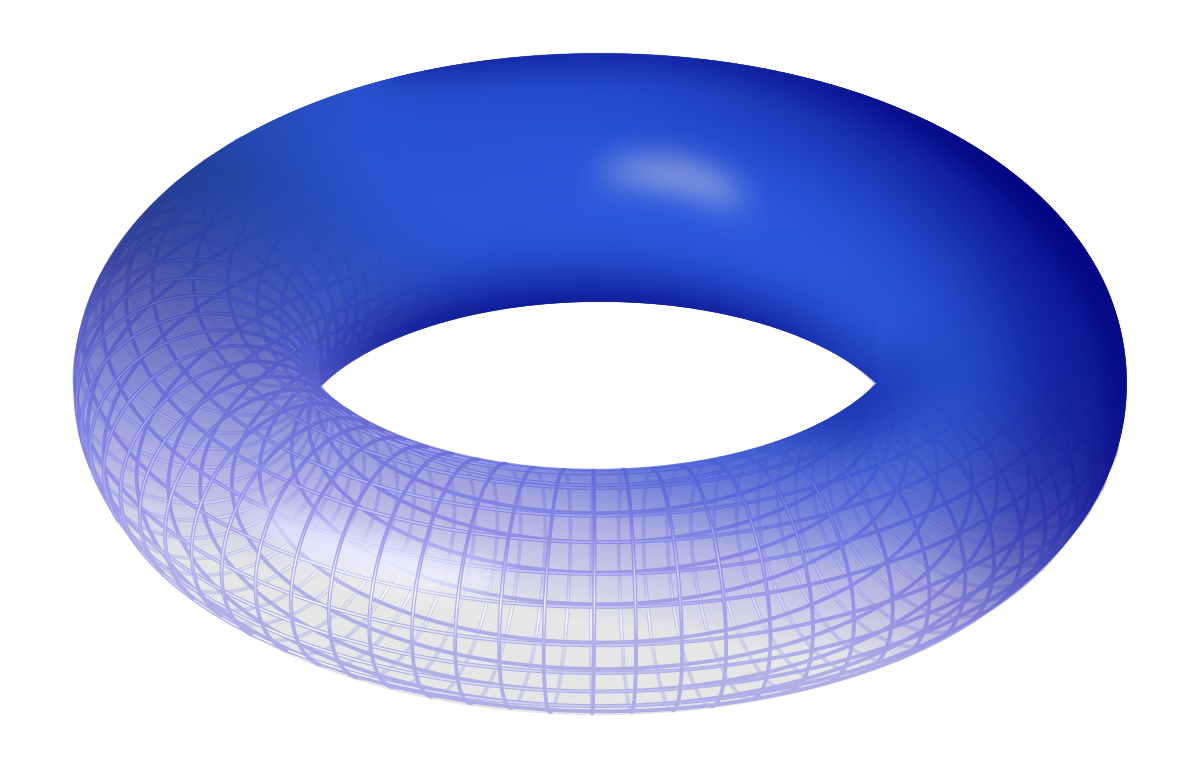}}}\quad  
    \caption{2-torus}
    \label{b-surf}
    \end{center}
    \end{figure}
    }
\item[\textbf{$\jfrac{\text{the projec-}}{\text{-tive plane}}$:}]{ Two points $(x_1,y_1)$ and $(x_2,y_2)$ on the $2$-sphere are said to be atipodal if $x_2=-x_1$ and $y_2 =y_1$. The quotiente space produced by the identification of antipodal points on the $2$-sphere is an $2$-manifold. It is called $2$-dimensional real projective space and is denoted by $\R P^2$. Let $[p]$ be a point in $\R P^2$. 
Takes a chart $(M_p, \phi_p )$ of $\s^2$ around $p$. For an open set $U\subset\s^2$ the set $-U:=\{(-x,-y)\in\s^2, (x,y)\in\s^2\}$ is also an open set, then the map $a:\s^2\xrar{\quad}\s^2$ given by $a(x,y)=(-x,-y)$ is a homeomorphism. Therefore $(-M_p, \phi_p \circ{}a)$ is a chart of $\s^2$ around $-p\in\s^2$. Changing enough $(M_p, \phi_p)$ in order to have $a(M_p)\cap M_p=\emptyset$
we obtain a homeomorphism $[\phi_p]:[M_p]\xrar{\quad}\R^2$. All those pairs $([M_p],[\phi_p])$ gives to $\R P^2$ the structure of a topological 2-manifold.}
\end{itemize}
\end{ex}
A topological  $2$-manifold (with boundary) $S$ will be referred hereafter simply as a \emph{surface} or \emph{a $2$-manifold}.

For an accurate definition of the notion of orientability of manifolds into the strict topological level (ca-\\-tegory).

But we will present here a \emph{intuitive} definition in dimension $2$ of this notion. So intuitively, an orientation on a surface is a globally consistent choice of sense to turning around each point of the surface. Our experience as conscious beings immersed in (locally) \emph{three-dimensional Euclidean space} has single out two possible senses to turn around some referential (point) in the shell of something for which here we set/define they as \emph{clockwise} and \emph{counterclockwise} meaning this exactly what it means by our collective sense of the reality. By convention, the \emph{counterclockwise} is the \emph{positive sense}. 

Thus an oriented surface is a $2$-manifold with a atlas coherent with the sense of turning around points in $\R^2$. With \emph{coherence} we mean that for the overlapping charts $(M_{\beta},c_{\beta})$ and $(M_{\iota},c_{\iota})$ the homeomorphism $c_{\iota} \circ c_{\beta}^{-1}|_{M_{\beta}\cap M_{\iota}}:c_{\beta}(M_{\beta}\cap M_{\iota})\xrar{\quad}c_{\iota}(M_{\beta}\cap M_{\iota})$ preserve a pre-chosen sence (positive or negative) to turn around points. Otherwise, if a surface does not admite a atlas enjoing the above condition, it is said to be non-orientable. 

\begin{ex}[an oriented $\&$ a non-oriented compact surfaces]
\end{ex}
\begin{figure}[H]{
    \centering
       \subfloat[\emph{2-sphere}]
    {{\includegraphics[width=5cm]{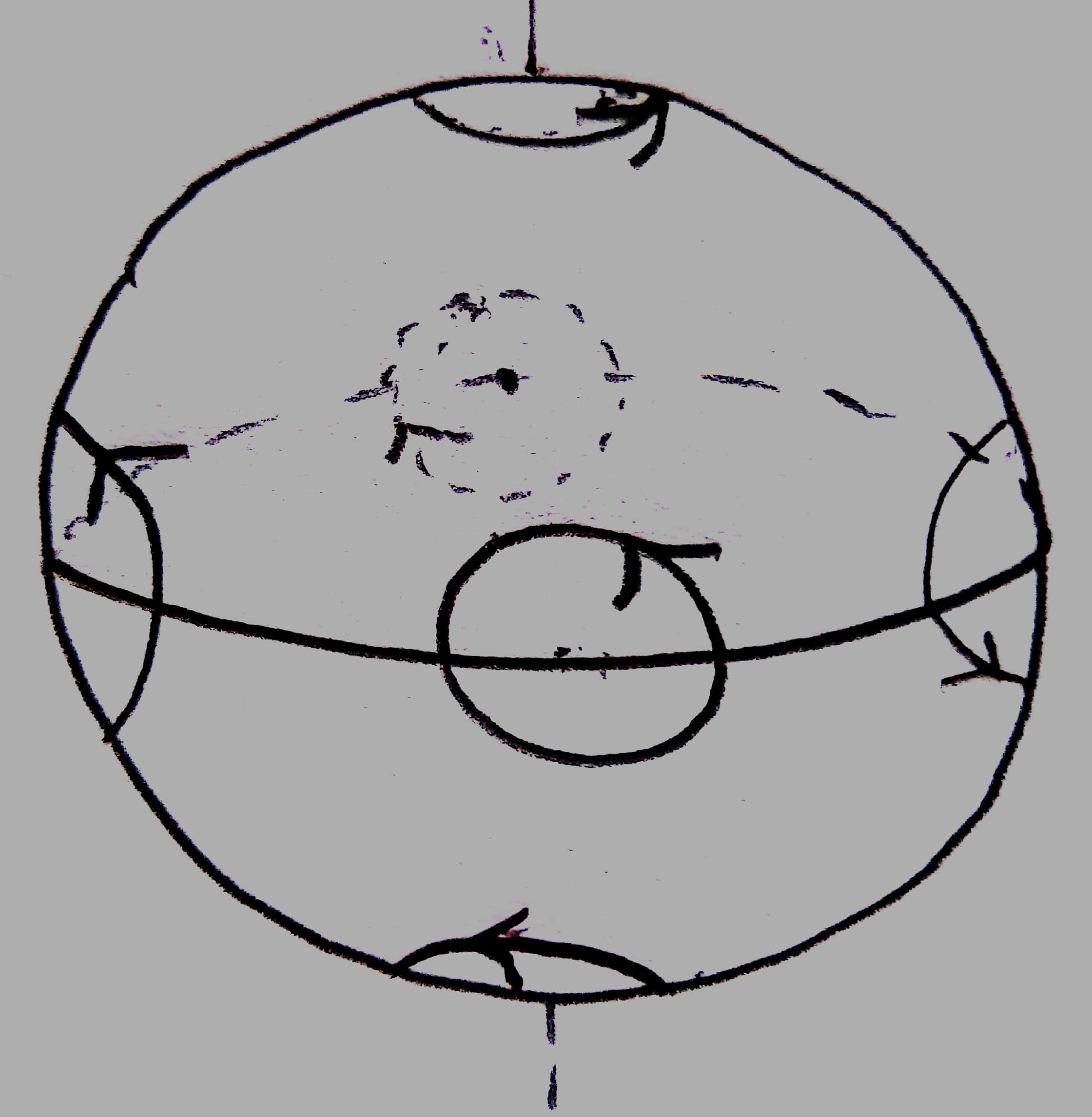} }}
        \qquad
    \subfloat[\emph{Möbius band} (non-orientable)]
    {{\includegraphics[width=4.8cm]{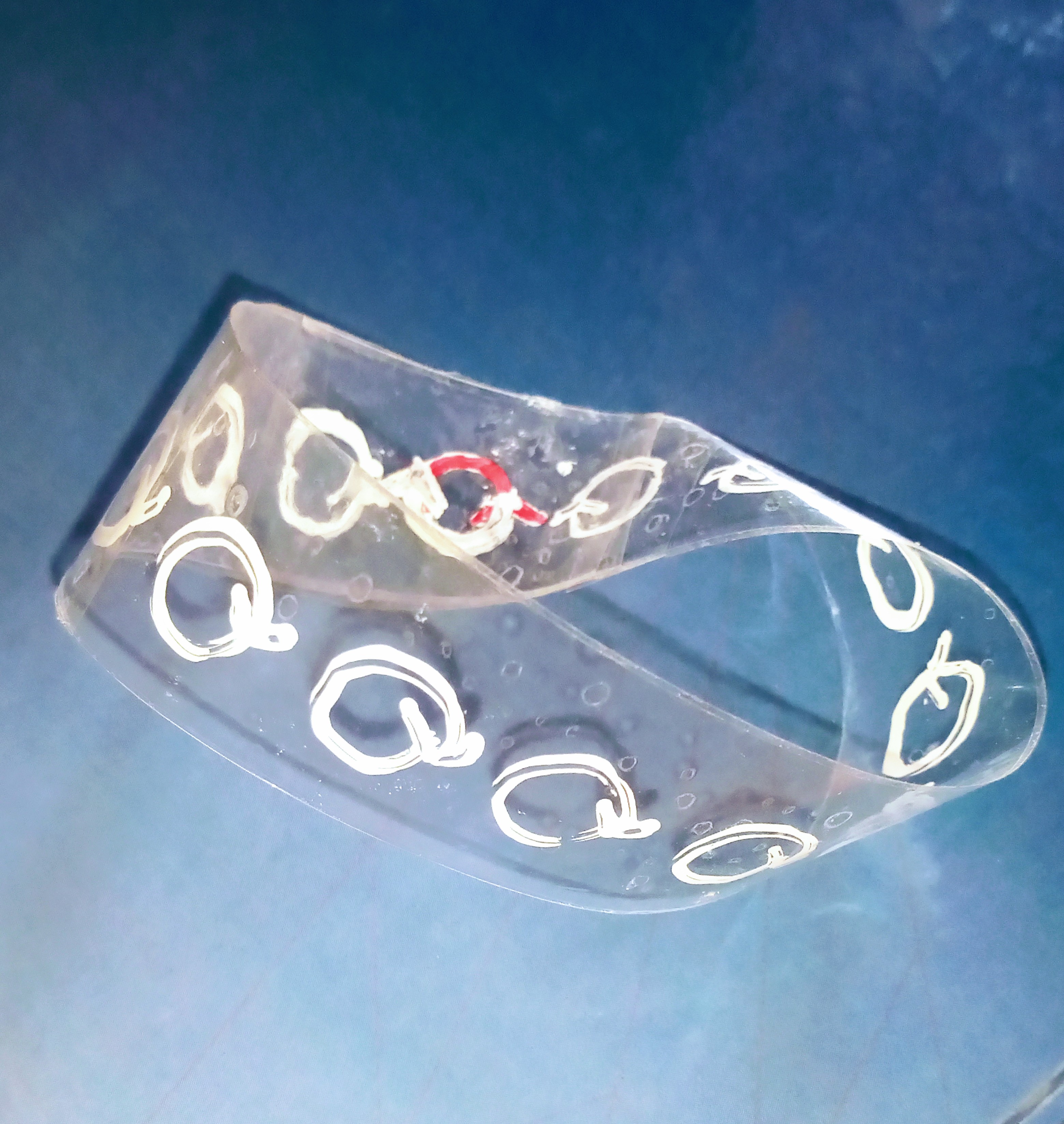} }}
       \caption{compact surfaces}
    \label{(n)-orient-surf}
    }
\end{figure}

This can be formalized resorting the topological degree theory of maps and/or to \emph{(Co)Homology theory} \cite{HAalt},\cite{MR787801},\cite{Mass},\cite{Span66},\cite{MR0200928},\cite{MR2766102}.

The surface in figure $\ref{(n)-orient-surf}$-$(b)$ is known as the \emph{Möbius strip}. It is constructed from a rectangle (a closed disk) identifying a pair of opposite sides reversely with respect to the orientation of the boundary
.

\begin{figure}[H]{
    \centering
       \subfloat[]
    {{\includegraphics[width=5.55cm]{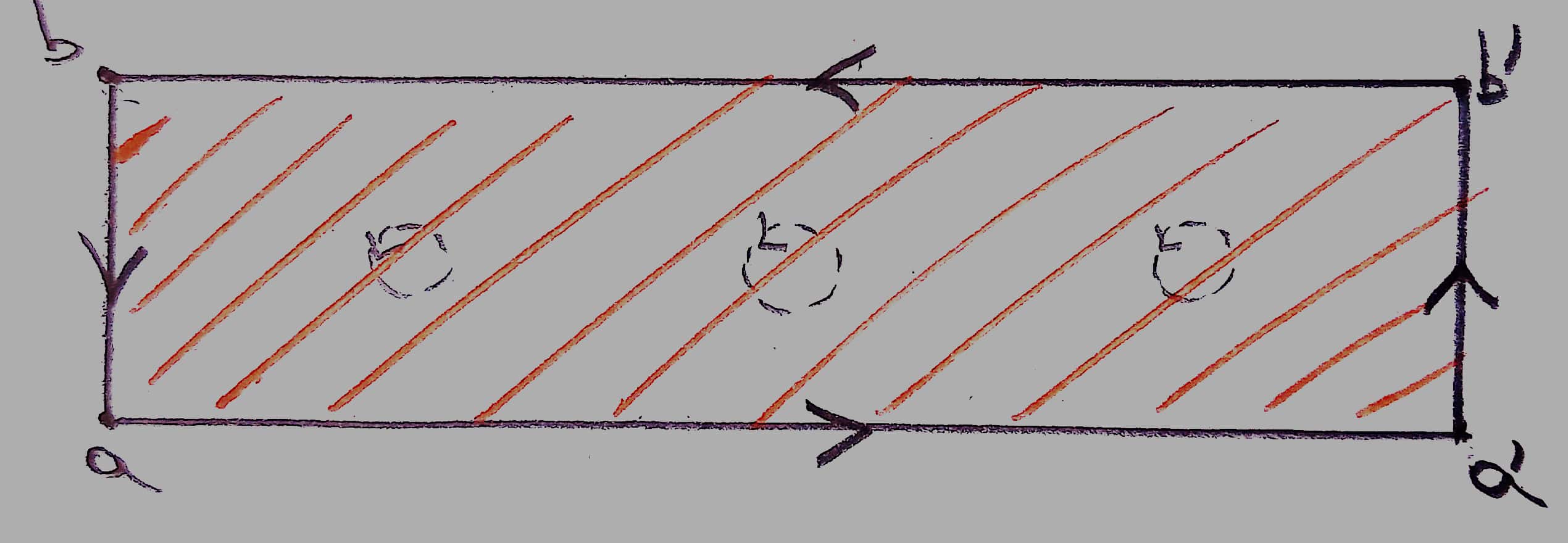}}}\quad
    \subfloat[]
    {{\includegraphics[width=5cm]{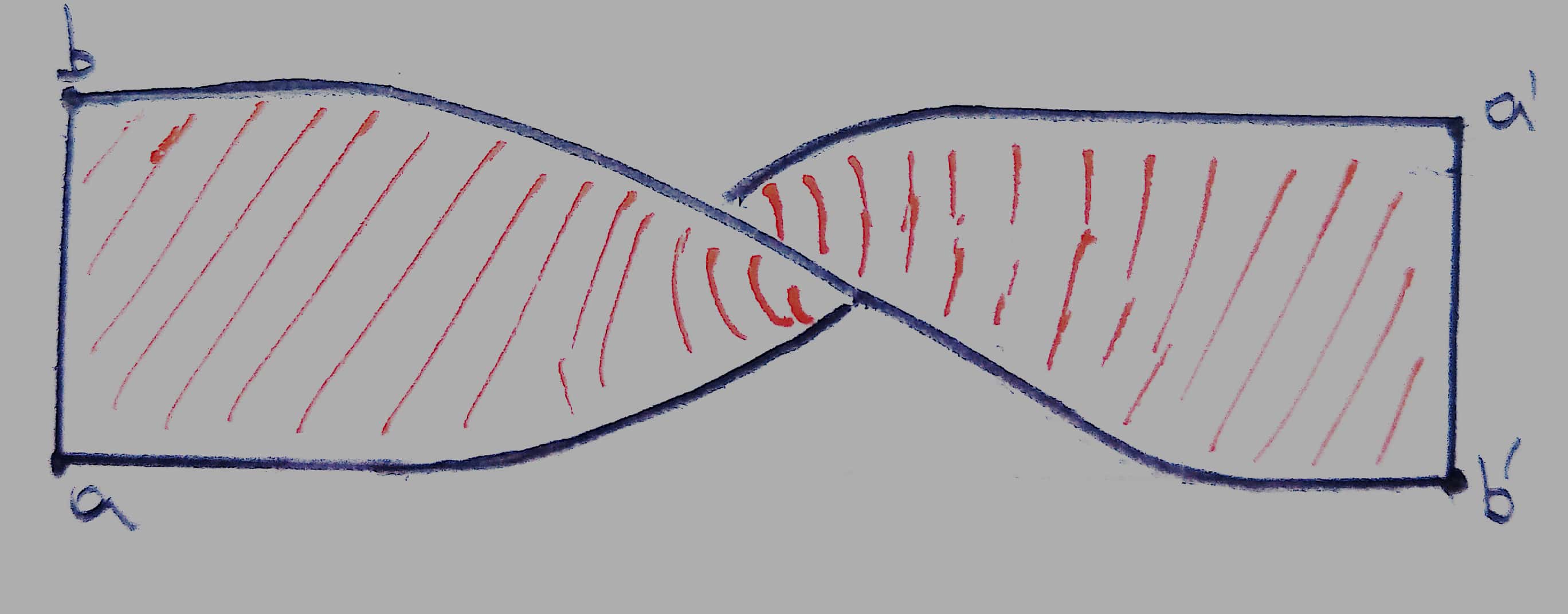}}}   \\
 \subfloat[]
    {{\includegraphics[width=4.6cm]{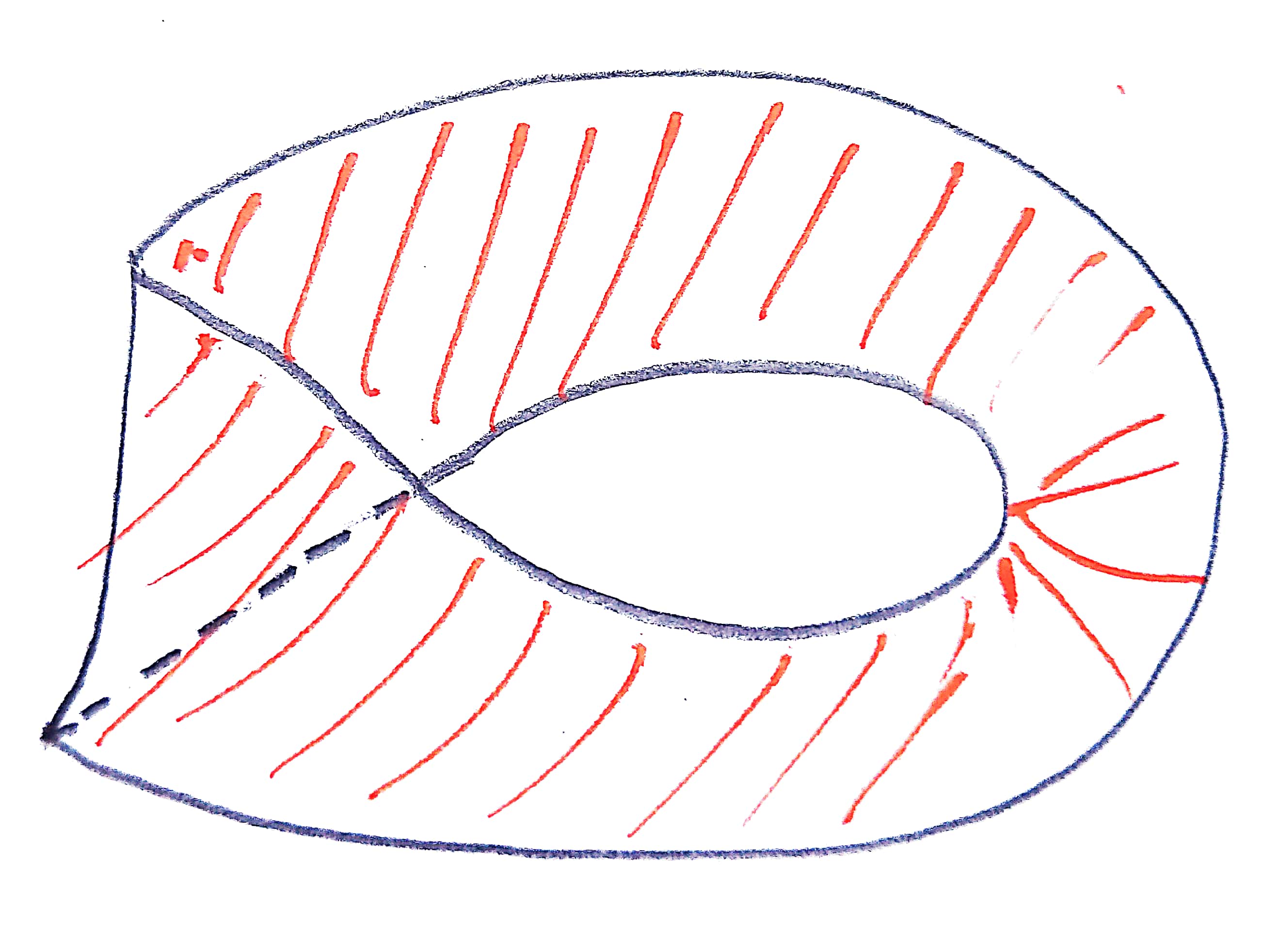}}}
    \caption{constructing a Möbius strip}
    \label{ms-surf}
    }
\end{figure}


\begin{defn}\label{o-p-homeo}
A local homeomorphism between two oriented surfaces, say $h: M\xrar{\quad}N$ is orientation-preserving if at each related through $h$ pair of points $(p,h(p))\in M\times N$ for any charts $(M_{\beta},c_{\beta})$ and $(N_{\iota},c_{\iota})$ around $p$ and $h(p)$ the homeomorphism $c_{\iota} \circ h\circ c_{\beta}^{-1}|_{M_{\beta}}:c_{\beta}(M_{\beta})\xrar{\quad}c_{\iota}(N_{\iota})$ preserves a pre-chosen sence to turn around points in $\R^2$.
\end{defn} 

There is only one, up to homeomorphism, closed $1$-manifold that is the circle \[\s^{1} :=\{(x,y)\in\R^2 ; {x^2 + y^2}=1\}\](the locally Euclidean quotient space $\faktor{\R}{\Z}$). 

Thus, from proposition $\ref{top-b}$ the boundary components of a compact surface are (topological) circles. 

\begin{defn}[curves into surfaces]\label{c-surf}
Let $S$ be a surface with marked points $\mathcal{P}\subset S$.
A \emph{arc} into $S$ is a continuous map $\alpha : [0,1]\xrar{\quad} S$.
A \emph{arc} $\alpha : [0,1]\xrar{\quad} S$ is :
\begin{itemize}
\item{\emph{simple} if it is a embedding of $(0,1)$;}
\item{\emph{proper} if $\alpha^{-1}(\mathcal{P}\cup\partial{S})=\{0,1\}$;}
\item{\emph{essential} if it is neither homotopic into a boundary component nor to a marked point of S;}
\end{itemize}

A \emph{closed curve} into $S$ is a continuous map $\gamma:\s^1\xrar{\quad}S$.
A \emph{closed curve} $\gamma : \s^1\xrar{\quad} S$ is :
\begin{itemize}
\item{\emph{simple} if it is a embedding;}
\item{\emph{essential} if it is not homotopic to a point, a puncture (or maked point), or a boundary component.}
\end{itemize}
We will usually identify a \emph{arc} or a \emph{closed curve} with its image in $S$, and see a \emph{simple arc} into the surface $S$ as a compact connected 1-dimensional submanifold of $S$ with non-empty boundary 
and a \emph{simple closed curve} into the surface $S$ as a compact connected 1-dimensional submanifold of $S$  without boundary. 
\end{defn}

\begin{thm}[Jordan curve theorem]\label{jcurve}
Let  $\gamma : \s^1\xrar{\quad} S$ be a simple closed curve into the plane, $\R^2$. Then $\R^2-\gamma$ is the disjoint union of two open sets, say $A$ and $B$ so that each one is path connected and have $\gamma$ as its boundary. Moreover, one of these sets is bounded and the other is unbounded.
If $\gamma$ is a simple closed curve in $\s^2$, then $\s^2 - \gamma$ consists of two open path connected sets sharing $\gamma$ as its (topologycal) boundary.
\end{thm}

\begin{defn}[Jordan domains(curves)]\label{Jd}
A \emph{Jordan curve} is simple closed curve into $\R^2$ (or $\s^2$). And a \emph{Jordan Domain} is a open set of $\R^2$ (or $\s^2$) with the topological boundary being a \emph{Jordan curve}.
\end{defn}

The following theorem asserts that those components are actually what our intuition says that they are. But in higher dimension this history changes[consult:\cite{MR117695},\cite{MR3203728},\href{https://lamington.wordpress.com/2013/10/18/scharlemann-on-schoenflies/}{blog post}].

\begin{thm}[Schoenflies Theorem\cite{MR2190924},\cite{MR1144352}]\label{sch-thm}
Let $B$ be the topological closure of a Jordan domain in $\s^2$
with boundary the Jordan curve $C$. Then there exists a homeomorphism $H:B \rightarrow \mathbb{B}^2$ sending $C$ onto $\s^1$.
\end{thm}

\begin{thm}[Baer-Epstein-\cite{FM:12}]
Let $\alpha$ and $\beta$ be two essential simple closed curves (or two essential proper arcs) in a surface $S$. Then $\alpha$ is isotopic to $\beta$ if and only if $\alpha$ is homotopic to $\beta$.
\end{thm}

Now we are going to introduce (recall) a procedure of to build a new surface from old ones.

\begin{defn}[gluings]
Let $X$ and $Y$ be compact surfaces with boundary. Let $h:A\xrar{\quad}B$ be a homeomorphis between one boundary components of $A\in\partial{}X$ and $B\in\partial{}Y$. The glue relation induced by $h$ is the equivalence relation defined by:\[a\thicksim_{h}b \text{\quad{if}\quad} \left\{\begin{array}{lr}
        a=b, & \text{for } a\in X-A \text{ or } a\in Y-B; \text{or } \\
        a=h^{-1}(b), & \text{for } b\in B; \text{or } \\
        b=h(a), & \text{for } a\in A
        \end{array}\right.
\]

$\faktor{X\sqcup{Y}}{\thicksim_{h}}$ is a topological space with the quotient topology and actualy it inherits the structure of surface from $X$ and $Y$, we denote it by $X\sqcup_{h}{Y}:=\faktor{X\sqcup{Y}}{\thicksim_{h}}$ and say that it is the \emph{gluing of $X$ and $Y$ along $h$ (or along $A$ and $B$)}.
\end{defn}

\begin{defn}[connected sum]
The connected sum of two surfaces, say $X$ and $Y$ consists on the procedure of to remove an open disk from each one of those and then glue them together along an homeomorphism of the circles boundaries of the cutting off open diks. The resulting space is a surface and is denoted by $X \# Y$. When $X$ and $Y$ are oriented is constrained to be along an orientation-reversing homeomorphism of the circles boundaries of the cutting off open diks with the induced orientation of $X$ and $Y$.
In ths cse, $X \# Y$ is a oriented surface.
\end{defn}

\begin{thm}[classification of compact surfaces]\label{class-surf}
Every compact surface is homeomorphic to either:
\begin{itemize}
\item[$\boldsymbol{(1)}$]{The sepher with $ n \geq 0$ boundaries components, which is obtained by removing $n$ open disks with disjoint closures;}

\item[$\boldsymbol{(2)}$]{The orientable surface of genus $g\geq 0$ with $ n \geq 0$ boundaries components, which is obtained by a connected sum of tori, and and removing $n$ open disks with disjoint closure;}

\item[$\boldsymbol{(3)}$]{The non-orientable surface of genus $g \geq 0$ with $n\geq 0$ boundaries components, which is obtained by the connected sum of $g$ projective planes, and removing $n$ open disks with disjoint closure.}
\end{itemize}
\end{thm}

\begin{defn}[genus \& type]\label{g-surf}
The integer number $g\geq 0$ in the above theorem $\ref{class-surf}$ is called \emph{genus}. 
The pair of integer number $(g, n)$ for $g$ and $n$ as in Theorem $\ref{class-surf}$ associated to a compact surface is the \emph{type} of the surface.
\end{defn}

Intuitively, $g$ corresponds to the number of holes of a surface.
\begin{ex}
\end{ex}

\begin{figure}[H]{
    \centering
    {{\includegraphics[width=5cm]{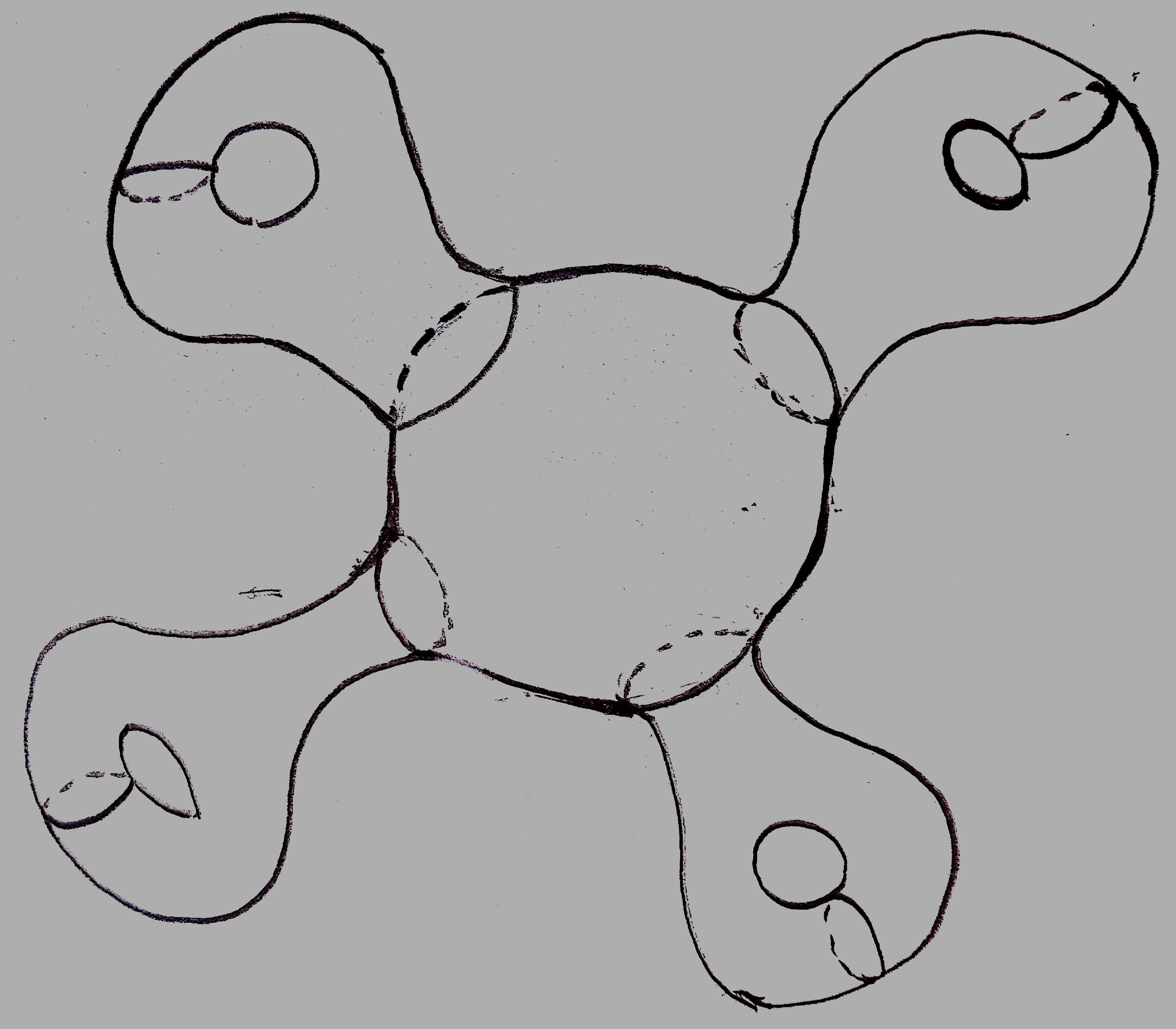} }}
  \caption{genus $4$ closed surfaces}
    \label{(n)-orient-surf}}
\end{figure}

\begin{defn}[cellular sets and cellular decompositions]
A $n$-cell into a Hausdorff space $M$ is a subset $X^n\subset M$ that is a homeomorphic to a (Euclidean) open ball of dimension $n$ under the condition that the homeomorphism extends to a continuous map from the close $n$-ball into $M$. That extended countinuous map is the $n$-cell map.
The $n$-cell for $n=0,1,2$ have distinguished name. The $0$-cell, the $1$-cell and the $2$-cell are called by \emph{vertex, edge and face}, respectively.

A \emph{cell-decomposition} of a Hausdorff space $M$ is a partition such space into cells in such a way that the boundary of each $n$-cell of the partition is contained into the union of all k-cells for $0\leq k <n.$  
\end{defn}

\begin{defn}[Euler characteristic]
For a cellular decomposition of a compact $n$-manifold $M$ the \emph{Euler chacteristic} $\chi(M)$ is the sum of the number of cells of even dimension minus the sum of cells of odd dimension. This number, actually, does not depends of the chosen cellular decomposition. Then it is associated to the topological essence of the manifold.

In particular, for a compact surface $S$ the \emph{Euler chacteristic} $\chi(S)=V-E+F$ here $V,E$ and $F$ is the number of vertices, edges and faces of a cellular decomposition of $S$. We refer to the formula
\[\chi(S)=V-E+F\]
as the \emph{Euler formula}. (\emph{Leonard Euler} was who firstly provide this formula. He proved it for \emph{polyhedral surfaces}.)
\end{defn}

\begin{thm}[\cite{MR787801}]
The \emph{Euler characteristic} is a topological invariant (i.e., homeomorphic manifolds have equal \emph{Euler chacteristic}). And, in the two dimensional case, we have the following relation with the type of a orientable surface:
\[\chi (S_{g,n}) = 2-2g-n\]
\end{thm}

\begin{defn}[coverings]
Let $X$ and $Y$ be two topological surfaces.
A continuous map $\varphi:X \xrightarrow{\quad}{Y}$ is a degree $d$ covering  (of $Y$ by $X$) if it is subject to the following condition:
\begin{itemize}
\item{ for any open set $U\subset Y$, $\rho^{-1}(U)$ is a disjoint union of $d$ open sets of $X$, $\{V_{n}\}_{n=1}^{d}$, such that \[\rho|_{V_n}:V_n\longrightarrow U\] is a homeomorphism. } 

\end{itemize}

$\rho$ is called \emph{covering map} and we also refer to the triad $(X, Y, \rho)$ as a covering.
\end{defn}

\begin{defn}[lifting of a map]\label{lift-c}
Let $\rho: Y\xrar{\quad} X$ and $f: Z\xrar{\quad} X$ be continuous maps between topological spaces. A lifting of $f$ 
by $\rho$ is a continuous mapping $\tilde{f}:Z \xrar{\quad} Y$ such that $f = p\circ \tilde{f}$, i.e., such that the following diagram commutes:

\begin{center}
\begin{tikzcd}[row sep=huge, column sep = 5.2em]
               & Y \arrow{d}{\rho} \\
Z \arrow[dashrightarrow,uibred]{ur}{\tilde{f} }\arrow{r}{f} & X
\end{tikzcd}\end{center}

\end{defn}

\begin{thm}[existence and uniqueness of liftings for covering map-\cite{MR1185074}]
Suppose $X$ and $\tilde{X}$ are Hausdorff spaces and $\rho: Y\rightarrow{X}$ is a covering map. Further, suppose $Z$ is a simply connected, pathwise connected and locally pathwise connected topological space and $f: Z\rightarrow{X}$ is a continuous mapping. Then for every choice of points $z_0\in Z$ and $y_0\in Y$ with $f(z_0) =\rho(y_0)$ there exist only one lifting $\tilde{f}:Z\rightarrow{Y}$ such that $\tilde{f}(z_0) = y_0$.
\end{thm}

\begin{defn}[Riemann surface-\cite{MR1185074},\cite{DonSrs}]\label{r-surf}
A \emph{Riemann surface} $S$ is surface $S$ with a atlas $\{(S_{\alpha}, c_{\alpha})\}_{{\alpha}\in A}$ such that for each pair of overlapping charts $(S_a,c_a)$ and $(S_b,c_b)$,
\[c_{b} \circ c_{a}^{-1}|_{S_{a}\cap S_{b}}:c_{a}({S_{a}\cap S_{b}})\xrar{\quad}c_{b}({S_{a}\cap S_{b}})\]
is a holomorphic map (identifying $\R^2$ with $\C$).
In this case, the atlas $\{(S_{\alpha}, c_{\alpha})\}_{{\alpha}\in A}$ is a \emph{Complex atlas} on $S$as \emph{} and a chart is called as \emph{complex chart}.
\end{defn}

\begin{defn}[Complex structure]\label{c-struct}
Two complex atlases on a \emph{Riemann surface} $S$ are equivalent if their union is also a complex atlas. 

A equivalence class $\mathfrak{S}$ of complex atlases on $S$ is a \emph{Complex Structure} on $S$.
\end{defn}

\begin{defn}[holomorphic maps]\label{holo-m}
A continuous map $f:S\xrar{\quad}R$ between \emph{Riemann surfaces} is said to be \emph{holomorphic} if for each pair of complex charts $(A,c_a)$ and $(B,c_b)$, such that $f(A)\subset B$, then the complex function
\[c_b \circ{}f\circ{}c_{a}^{-1}: c_a (A)\xrar{\quad}c_b(B)\]
is holomorphic.
If $f:S\xrar{\quad}R$ is bijective and its inverse $f^{-1}:R\xrar{\quad}S$ is holomorphic in the above given sence, it is said to be a \emph{biholomorphism}, and $S$ and $R$ are said \emph{isomorphic} (or even \emph{biholomorphic}) 
\end{defn}

The most importante theorem in the \emph{theory of Riemann surfaces} is a result descovered and almost completely proved by \emph{Riemann}. It guarantees that the universal covering of an arbitrary Riemann surface is always isomorphic to one of three normal (geometric) models: the Riemann sphere, the complex plane or the unit disk.

\begin{thm}[Uniformization Theorem/Riemann mapping theorem]\label{r-m-thm}
Every simply connected Riemann surface is isomorphic to $\mathbb{D}:=\{z\in\C;|z|<1\}$, $\C$ or $\cc$.
\end{thm}

Combining these result with the topological theory of covering surfaces, follows:
\begin{thm}[Uniformization of compact Riemann surfaces]
According to their universal coverings, compact Riemann
surfaces can be classified as follows:
\begin{itemize}
\item[$\boldsymbol{(1)}$]{ $\cc$ is the only compact Riemann surface of genus $0$;}
\item[$\boldsymbol{(2)}$]{Every compact Riemann surface of genus $1$ can be described in the form $\C/\Lambda$, where $\Lambda$ is a lattice, that is $\Lambda = w_1\Z \oplus w_2\Z$ for two complex numbers $w_1,w_2$ such that $w_1/w_2 \notin \R$ acting on $\C$ as a group of translations;}

\item[$\boldsymbol{(3)}$]{Every compact Riemann surface of genus greater than $1$ is isomorphic to a quotient $\mathbb{H}/K$, where $K\subset PSL(2,\R)$ acts freely and properly discontinuously.}
\end{itemize}
\end{thm}

\begin{thm}[lifting complex structure]\label{lift-c-s}
Suppose $S$ is a \emph{Riemann surface}, $R$ is a Hausdorff topological
space and $\rho: R \xrar{\quad}S$ is a local homeomorphism. Then there is a unique complex structure on $R$ such that $\rho$ is holomorphic.
\end{thm}

\begin{defn}[branched coverings]
Let $X$ and $Y$ be two topological surfaces.
A surjective continuous map $\rho:X\xrar{\quad}{Y}$ is a degree $d$ branched covering (of $Y$ by $X$) if:
\begin{itemize}
\item[$\boldsymbol{(1)}$]{there exists a discrete subset $B\subset Y$, such that:
{\[\rho|_{X-\rho^{-1}(B)}:X-\rho^{-1}(B)\longrightarrow Y-B\] is a degree $d$ covering;}}
\item[$\boldsymbol{(2)}$]{for each $b\in \rho^{-1}(B)$, $\varphi$ is topologically the map $z\mapsto z^k$ with $k(p):=k\geq 1$ a integer,i.e., there exists pairs of charts $(c_1, U\ni b)$ and $(c_2, V\ni \rho(b))$ of $X$ and $Y$ around $b$ and $\rho(b)$, respectively, such that\[c_2 \circ{}\rho\circ{}c_{1}^{-1}(z)=z^k : c_1 (U)\xrar{\quad}c_2 (V)\]
$k(p)$ is the multiplicity of $p$ for $\rho$;}
\item[$\boldsymbol{(3)}$]{and for each $p\in B$, \[\sum_{b\in\rho^{-1}(p)}k(b)=d\] and $k(b)>1$ for at least one $b\in\rho^{-1}(p)$.}
\end{itemize}
$\rho$ is called \emph{branched covering map} and we also refer to the triad $(X, Y, \rho)$ as a branched covering.

A point $b\in \rho^{-1}(B)$ with $k(b)>1$ is called by \emph{critical point} or \emph{ramification point} of $\rho$ and each point $q\in B$ is called by \emph{critical value} or \emph{brach point}.

\end{defn}

\begin{defn}[orientation-preserving (branched) covering]\label{o-p-b-c}
A (branched) covering $\rho:X\xrar{\quad}Y$ between two oriented surface $X$ and $Y$ is said to be \emph{orienttion-preserving} (o even, \emph{that preserves the orientations}) if the underline local homeomorphism $\rho|_{X-\rho^{-1}(B)}:X-\rho^{-1}(B)\longrightarrow Y-B$ is orientation-preserving. 
\end{defn}

\begin{defn}\label{b-c-sets}
Denote by $\mathfrak{R}_g$ the set of all orientation-preserving branched coverings of $\s^2$ by the oriented closed surfaces of genus $g$. Then for $f\in \mathfrak{R}_g$ we denote the set of critical points and critical values of $f$ by $C_f$ and $R_f :=f(C_f)$, respectively.
\end{defn}

\begin{defn}[passport of a branched covering of $\s^2$]\label{passp}
Let $f\in{\mathfrak{R}_g}$ be a branched covering of $\s^2$ of degree $d$ with critical value set $\{w_1,w_2 , ..., w_m\}$.
The passport $\pi=\pi(f)$ of $f\in{\mathfrak{R}_g}$ is the following list of $m$ non-trivial integer partition of $d$, $\pi(f)=[\pi_{1}, \pi_{2},..., \pi_{m}]$, that is, $\pi_{j}=[d_{(j,1)}, d_{(j,2)}\cdots, d_{(j,l_j)}]$ is a list of positive interger satisfying: $d_{(j,k)}\in\{1,2 ..., d\}$, $d=\sum_{k=1}^{l_j}d_{(j,k)}$
and for at least one $k\in\{1, 2, ..., l_j\}$, $d_{(j,k)}\neq 1$; such that the numbers $d_{(j,k)}$ are the multiplicities of the  critical points of $f$ that are in the fibre of $f$ above the critical value $w_j \in R_f$.

It is also convinient 
to consider the following notation for those integer partitions $\pi_j$'s:
\begin{eqnarray}
\pi_j^{\ast} &=&(1^{p_{1}^{j}}, 2^{p_{2}^{j}}, 3^{p_{3}^{j}}, ..., (d-1)^{p_{d-1}^{j}}, d^{p_{d}^{j}})
\end{eqnarray}
Such that
\begin{eqnarray}
\sum_{n=1}^{d}n\cdot{}{p_{n}^{j}}=d
\end{eqnarray}
i.e, $p_{n}^{j}$ is the number of times that the integer $n\in\{1, 2, ..., d\}$ appears as a sommand in the partition $\pi_j$ of $d\in \Z$. We say that $n\in\{1, 2, ..., d\}$ is in the support of $\pi_j$ if $p_{n}^{j}\neq 0$. 
\end{defn}

\begin{defn}[genus of a $d$-passport]
The genus of a $d$-passport
 $\pi=[\pi_{1}, \pi_{2},..., \pi_{m}]$, with\\
  $\pi_{j}=[d_{(j,1)}, d_{(j,2)}\cdots, d_{(j,l_j)}]$, is the number
\[g=g(\pi):=1-d+\frac{1}{2}\displaystyle{\sum_{j=1}^{m}\sum_{k=1}^{l_j}(d_{(j,k)} -1)}\]
\end{defn}

\begin{defn}[admisible passport]\label{adm-passp}
An admissible passport of degree $d$ and genus $g$ is a finite list of integer partitions of the integer $d>0$ that satisfies the \emph{Riemann-Hurwitz condition}.
\end{defn}


\subsection{Liftings by branched covers}

Let $f:\tilde{X}\rightarrow{X}$ be an branched cover and $C\subset \tilde{X}$ the set of branchig points of $f$.

\begin{defn}[landing paths for branched covers]
Let $f:\tilde{X}\rightarrow{X}$ be an branched cover.
A path $\gamma:[0,1]\rightarrow{X}$ with start point $\gamma(0)\in X$ being a regular point and end point $c:=\gamma(1)\in X$ being a critical value for $f$ is said to be a landing path for $f$. We also say that $\gamma$ is a path landing on $c$. And, in general,
we will say that a path ending at a point $p$ is a path landing at $p$. 
\end{defn}

\begin{lem}\label{landing}
Any landing path $\gamma$ for $f$ have a unique lift to $\tilde{X}$ through $f$ for each $x\in f^{-1}(\gamma(0))$ that lands at points on the fibre of $f$ over $\gamma(1)$. Furthermore, if a point $a\in f^{-1}(c)$ has local degree $k=\deg_{loc}(f, a)$ then there is $k$ start points over $\gamma(0)$ for liftings of $\gamma$ landing at $a$.
\end{lem}

\begin{thm}[lifting landing path isotopies]\label{lisobc}
Let $f:\tilde{X}\rightarrow{X}$ be an branched cover and $F:[0,1]\times [0,1] \rightarrow X$ a isotopy between the landing paths $\gamma_0:[0,1] \rightarrow X$ and $\gamma_1:[0,1] \rightarrow X$ for $f$ with fixed extremal points $p_0 = \gamma_0(0) = \gamma_1(0)$ and $p_1 = \gamma_0(1) =\gamma_1(1)$. Let $\tilde{p} \in f^{-1}(p_0)$. Then $F$ can be lifted to a isotopy $\tilde{F}: [0,1] \times [0,1]\rightarrow X$ with initial
point $\tilde{p}$. In particular,
the lifted paths $\tilde{\gamma_0}$ and $\tilde{\gamma_1}$ with start poit $\tilde{p}$ have the same landing point $\tilde{p_1}\in f^{-1}(p_1)$, and are isotopic.
\end{thm}

\begin{defn}[saddle-connection]
A saddle-connection for a branched covering $f:\tilde{X}\rightarrow{X}$ is a path $\tilde{\gamma}:[0,1]\rightarrow \tilde{X}$ in $\tilde{X}$ with distinct extremal points in $C$ and with interior $\tilde{\gamma}((0,1))$ disjoint from $C$.

\end{defn}

That is, a saddle-connection for $f:\tilde{X}\rightarrow{X}$, a branched covering, is a path into $\tilde{X}$ connecting only two different points in  $C\subset{\tilde{X}}$.

\begin{defn}[postcritical arc]
A postcritical arc for a branched covering $f:\tilde{X}\rightarrow{X}$ is a simple arc ${\gamma}:[0,1]\rightarrow X$  in $X$ with distinct extremal points in $f(C)$ and with interior $\gamma((0,1))$ disjoint from  $f(C)$.
\end{defn}

That is, a postcritical arc for $f:\tilde{X}\rightarrow{X}$, a branched covering, is a path into $X$ connecting only two different points in $f(C)\subset{X}$.

\begin{cor}[liftings of postcritical arc]
Let ${\gamma}:[0,1]\rightarrow X$ be a postcritical arc for a branched covering $f:\tilde{X}\rightarrow{X}$ with a marked point $x$.  Then, for each $p\in f^{-1}(x)$, ${\gamma}$ have a unique lift to $\tilde{X}$ through $f$. 
\end{cor}
\begin{proof}
That is a immediate consequence of Lemma $\ref{landing}$.
\end{proof}

\begin{defn}
We say that a par $(\Gamma,\Gamma^{\prime})$ of finite collections $\Gamma:=\{\gamma_1,... ,\gamma_n\}$ and $\Gamma^{\prime}:=\{\gamma^{\prime}_1,... ,\gamma^{\prime}_n\}$ of proper arcs on a surface $X$ has the property $(P)$ if it satifies:
\begin{itemize}
\item[(1)]{The arcs in $\Gamma$ and $\Gamma^{\prime}$ are pairwise in minimal position;}
\item[(2)]{The arcs in $\Gamma$ and $\Gamma^{\prime}$ are pairwise nonisotopic;}
\item[(3)]{each bigon between $\gamma_k$ and $\gamma^{\prime}_k$ does not contains intersection of arcs from $\{\gamma_1,... ,\gamma_n\}$ or from $\{\gamma^{\prime}_1,... ,\gamma^{\prime}_n\}$.}
\end{itemize}

\end{defn}

\begin{lem}[adjoining isotopy (simultaneous isotopy)]\label{simultiso}
Let $X$ be a compact surface, possibly with marked points, and
let $(\Gamma,\Gamma^{\prime})$ one pair of finite collections of proper arcs on $X$ with the property $(P)$. If $\gamma^{\prime}_i$ is isotopic to $\gamma_i$ relative to $\partial{X}\cup C$ for each $i$. Then there is an isotopy of $X$ relative to $\partial{X}\cup C$ that takes $\gamma^{\prime}_i$ to $\gamma_i$ for all $i$ simultaneously and hence takes $\cup_{i}\gamma_i\prime$ to $\cup_{i}\gamma_i$.
\end{lem}
\begin{proof}
Compare with \cite[Lemma 2.9]{FM:12} 
\end{proof}

\begin{thm}[lifting isotopies]\label{isotpcbc}
Let $f:\tilde{X}\rightarrow{X}$ be a branched cover and $\Sigma_f$ be a \emph{Jordan} curve running through the critical values of $f$, $f(C)$. Then, for every \emph{Jordan} curve $\Sigma$ isotopic to $\Sigma_f$ relative to $f(C)$, the pullback graph $\Gamma_f (\Sigma):=f^{-1}(\Sigma)$ is isotopic to $\Gamma_f (\Sigma_f)$.
\end{thm}
\begin{proof}
$\Sigma$ and $\Sigma_f$ determines, each one respectivelly, two collections $\sigma:=\{\sigma_1,... ,\sigma_m\}$ and $\sigma^{\prime}:=\{\sigma^{\prime}_1, ...,\sigma^{\prime}_m\}$ of post saddle-connections for $f$ where $m:=f(C)$. By hypotesis follwos that $(\sigma,\sigma^{\prime})$ is a pair of collections of proper arcs in $X$ with the property $(P)$. Then, applying Lemma $\ref{simultiso}$ follows the result expected.
\end{proof}

\begin{thm}[Riemann-Hurwtiz formula]

For any branched covering between compact surfaces of degree $d$, $\rho:S\rightarrow R$ it holds:
\begin{eqnarray}
\chi(S)&=&d \cdot \chi(R)-\displaystyle{\sum_{j=1}^{m}\sum_{k=1}^{l_j}(d_{(j,k)} -1)}
\end{eqnarray}
with notation in accordance with $\ref{passp}$.
\end{thm}

\section{Terminologies and some results from graph theory}
\label{ape:gtheory}

\begin{defn}[(abstract) graph]\label{def-g}
A graph $G$ is an ordered pair $(V (G),E(G))$ consisting of a set $V(G)$ whose elements are called vertices and a set $E(G)$, disjoint from $V(G)$, whose elements are called edges, together with an incidence function
$\psi_G: E(G)\rightarrow \sfrac{(V(G)\times V(G))}{\mathcal{S}_2}$ that associates to each edge of $G$ an unordered pair of (not necessarily distinct) vertices of $G$. If $\psi_G (e)=\{u,v\}\in \sfrac{(V(G)\times V(G))}{\mathcal{S}_2}$, we write $e=\{u,v\}$. The vertices $u$ and $v$ are called endpoints (or extremal points) of the edge $e$ and we say that those vertices are connected (or joined) by the edge $e$. We also say that a edge $e=\{u,v\}$ is incidente to the vertices $u$ and $v$, and that the vertices  
$u$ and $v$ are incident to the edge $e=\{u,v\}$. Two vertices (edge) which are incident with a common edge (vertex) are said to be adjacent.
\end{defn}

\begin{defn}[degree of a vertex]\label{deg-v-g}
The degree(valence) of a vertex $v\in V(G)$ of a graph $G$ is the number of edges that are incident to $v$, and it is denoted by $deg(v)\in\N$. A vertex of degree (valence) $k$ is a vertex of degree (valence) $k$ or a $k$-valent vertex.
\end{defn}

\begin{defn}[comparing graphs]\label{isom-g}
Two graphs $G$ and $H$ are isomorphic, if there are
bijections $\theta : V(G) \rightarrow V(H)$ and $\varphi : E(G) \rightarrow E(H)$ such that the following diagram commutes

\begin{center}
\begin{tikzcd}[row sep=large, column sep = large]
\arrow{d}{\psi_G} E(G) \arrow{r}{\varphi} & \mathcal{E}(S)  \arrow{d}{\mathfrak{p}}  \\
 \faktor{(V(G)\times V(G))}{\mathcal{S}_2} \arrow{r}{\theta\times\theta} & \faktor{(V(H)\times V(H))}{\mathcal{S}_2}
\end{tikzcd}\end{center}

That is, such that $\psi_G (e) = uv$ if and only if $\psi_H (\varphi(e)) = \theta(u)\theta(v)$.

Such a pair of mappings $(\theta, \varphi)$ is called an isomorphism between $G$ and $H$, and we indicate its existence writing, $ G \cong H $.
\end{defn}

\begin{defn}[labeled graph]\label{label-g}
A vertex-labeling of a graph $G$ by a set ${\mathfrak{U}}$ is a surjetive map $\mathfrak{l}:V(G)\rightarrow{\mathfrak{U}}$. This permits to single out vertices into subclasses in accordance with its image by that map. For a vertex $v\in V(G)$ such that $\mathfrak{l}(v)=x$ we write $v_x$.
\end{defn}

\begin{defn}[$k$-path/$k$-cycle]\label{paths}
A path into a graph $G$ is a collection of edges of $G$, $\gamma\subset{E(G)}$, whose vertices can be arranged in a linear sequence (that is, labeled from a total ordered set) in such a way that two vertices are adjacent if they are consecutive in the sequence, and are nonadjacent otherwise. A path with $k\in\N$ edges is called $k$-path and $k$ is the length of that path. If $v\in V(G)$ and $u\in V(G)$ are the initial vertex and the terminal vertex of the first and the last edges respectively on the linear sequence of a path $\gamma$ into $G$. The edges incident to this vertices inherted that nomenclature. We say that $\gamma$ join (connects) $v$ to $u$ and also that $v$ and $u$ are joined (connected) by $\gamma$.  

Likewise, a cycle into a graph $G$ is a collection of edges of $G$, $\gamma\subset{E(G)}$, whose vertices can be arranged in a cyclic sequence in such a way that two vertices are adjacent if they are consecutive in the sequence, and are nonadjacent otherwise. A cycle that contains $k\in\N$ edges is called $k$-cycle and $k$ is the length of that path. 
\end{defn}

\begin{defn}[connected graph]\label{conn-g}
A graph is connected if, for every partition of its vertex set into two nonempty sets $X$ and $Y$, there is an edge with one endpoint in $X$ and one endpoint in $Y$ ; otherwise the graph is disconnected. This is equivalent to set that a graph is connected if any pair of its vertices are joined by a path. A maximal connected subgraph of a graph $G$ is called connected component of $G$.
\end{defn}

\begin{defn}[bipartite graph]\label{bip-g}
A graph $G$ is a \emph{bipartite} if its vertex set $V(G)$ is partitioned into two sets, say $X$ and $Y$ such that any edge in $E(G)$ has one endpoint in $X$ and the another one in $Y$. The partition $V(G)=X\sqcup Y$ is called a bipartition and the subsets $X$ and $Y$ are called parts. We denote such a bipartite graph by $G=[X,Y]$.
\end{defn}

\begin{defn}[direct graph (digraph)]\label{dig}
A direct graph (or simply, digraph) $G$ is an ordered pair $(V(G), E(G))$ consisting of a set $V(G)$ whose elements are called vertices and a set $E(G)$, disjoint from $V(G)$, whose elements are called directed(or oriented) edges, together with an incidence function $\psi_G:E(G)\rightarrow V(G)\times V(G)$ that associates to each edge of $G$ an \textbf{ordered} pair of (not necessarily distinct) vertices of $G$. If $\psi_G (e)=(u,v)\in(V(G)\times V(G))$, we write $e=(u,v)$. The vertices $u$ and $v$ are called endpoints (or extremal points) of the edge $e$ and we say that those vertices are connected  (or joined) by the edge $e$. For a directed edge $(u,v)$ of $G$ we say that $u$ dominates $v$.
For a vertex $v\in V(G)$ a edge of $G$ of the form $(s,v)\in E(G)$ is called incomig-edge at $v$ and those ones of the form $(u,s)$ are called outgoing-edge at $v$.
\end{defn}

\begin{defn}[matching]\label{match}
A \emph{matching} on a graph $G$ is a subset of edges $\mathcal{M}\subset E(G)$ that do not have vertices in commom.
\end{defn}

We refer to the problem of find out a matching on a bipartite graph as the \emph{Matching Problem}.

\begin{defn}[perfect matching]\label{perf-match}
A matching $\mathcal{M}\subset{E(G)}$ in a graph $G$ which covers all vertices of $G$ is called \emph{perfect matching.} 
\end{defn}

\begin{defn}[pontential mates]\label{pot-m}
Let $G$ be a graph and $S\subset V(G)$ be a collection of vertices of $G$.
The neighbors set of $S$ in $G$ is 
\[N_{G}(S):=\{x\in V(G); \exists\; e\in E(G), \psi_G (e)=\{x,v\}\}.\]

When we are considering the matching problem on a bipartite graph $G=G[X, Y]$ we commonly refer to the set $N_{G}(S)$ as the set of the potential mates for the subset $S\subset V(G)$.
\end{defn}

\begin{thm}[Hall's Merriage Theorem-\cite{BondyG},\cite{Ha94}]\label{ksamento}
A bipartite graph $G := G[X, Y]$ has a matching which covers every vertex in $X$ if and only if
\[ |N_G(S)| \geq |S|\]
for all $S \subset X$.
\end{thm}
\begin{cor}[(perfect) Matching Theorem]\label{cor-ksamento}
A bipartite graph $G := G[X, Y]$ has a perfect matching if and only if
$|X| = |Y|$ and $|N_G (S)| \geq |S|$ for all $S\subset X$.
\end{cor}

\begin{defn}[multi-extremal chargeable graph]\label{ch-g}
A multi-extremal chargeable graph $C:=C[I,O]:=C[I,X$ $;O,Y]$ is a bipartite graph $G[X,Y]$ (the underlying graph of $C$) with two distinguished set of vertices, an input set $I\subset X$ and an output set $O\subset Y$, together with a nonnegative real-valued function $c:V(G)-(I\sqcup O)\rightarrow \R_{>0}$. $c$ is the vertex-capacity function of $C$ and its value on an vertex $v$ the capacity of $v$. When is necessary to emphasize the capacity function we say that $C$ is a multi-extremal chargeable graph with capacity $c$.

The vertices in 
$V(G)-I\sqcup O$ are called interior vertices. 
We denote by $int.V(G):=V(G)-(I\sqcup O)$ the subset of interior vertices. The edges with endpoints in $int.V(G)$ 
is called interior edges.  
\end{defn}

\begin{defn}[edge-weighting on a graph]\label{e-w-g}
A edge-weighting on graph $G$ is a real function $w:E(G)\rightarrow \R$. A graph with a edge-weighting is a weighted graph.
\end{defn}

\begin{defn}[feasible weighting]\label{f-weight}
A edge-weighting $w$ on a multi-extremal chargeable graph $G$  with capacity $c$ is feasible if it satisfies the following additional constraints:
\begin{itemize}
\item[$\boldsymbol(1)$]{$w$ is a real estrictly positive function,i.e., $w(E(G))\subset \R_{>0}$;}
\item[$\boldsymbol(2)$]{$\sum_{x\in N_G (v)} w(\{x,v\})=c(v)$ for each interior vertex $v\in int.V(G)$.}
\end{itemize}
 
The sums \[|w|^{in}:=\sum_{x\in N_G (I)}\sum_{\jfrac{v\in I;}{\{v,x\}\in E(G)}} w(\{v,x\})\]
and
\[|w|^{out}:=\sum_{x\in N_G (O)}\sum_{\jfrac{v\in O;}{\{v,x\}\in E(G)}} w(\{v,x\})\]
are respectively the input value and output value of $w$.

A multi-extremal chargeable graph $G$ with a feasible weighting is called  multi-extremal weighted graph.
\end{defn}

\begin{prop}[charge conservation]\label{vertex-capacity-fn}
Let $N=N[I,X;O,Y]$ be a multi-extremal chargeable graph with constant capacity $M\in\R_{>0}$. Then for any 
feasible weighting $w$ on $N$ the input and output values are equal.
\end{prop}

\begin{proof}
The proof we are going to give will be by induction on the number of interior edges of the multi-extremal chargeable graph.

Let's start verifying the base case.

Let $N$ be a multi-extremal chargeable graph with constant capacity $M$ with only one interior edge $e\in E(N)$, $n$ initial edges and $m$ terminal edges. And let $w:E(N)\rightarrow \R_{>0}$ be a feasible weighting on $N$ assigning the weight $k>0$ to $e\in E(N)$.
Then,
\[|w|^{in} +k=M=k+ |w|^{out}\]
Thus, we have
\[|w|^{in}=|w|^{out}\]

\begin{figure}[H]
\begin{center}
\includegraphics[width=4cm]{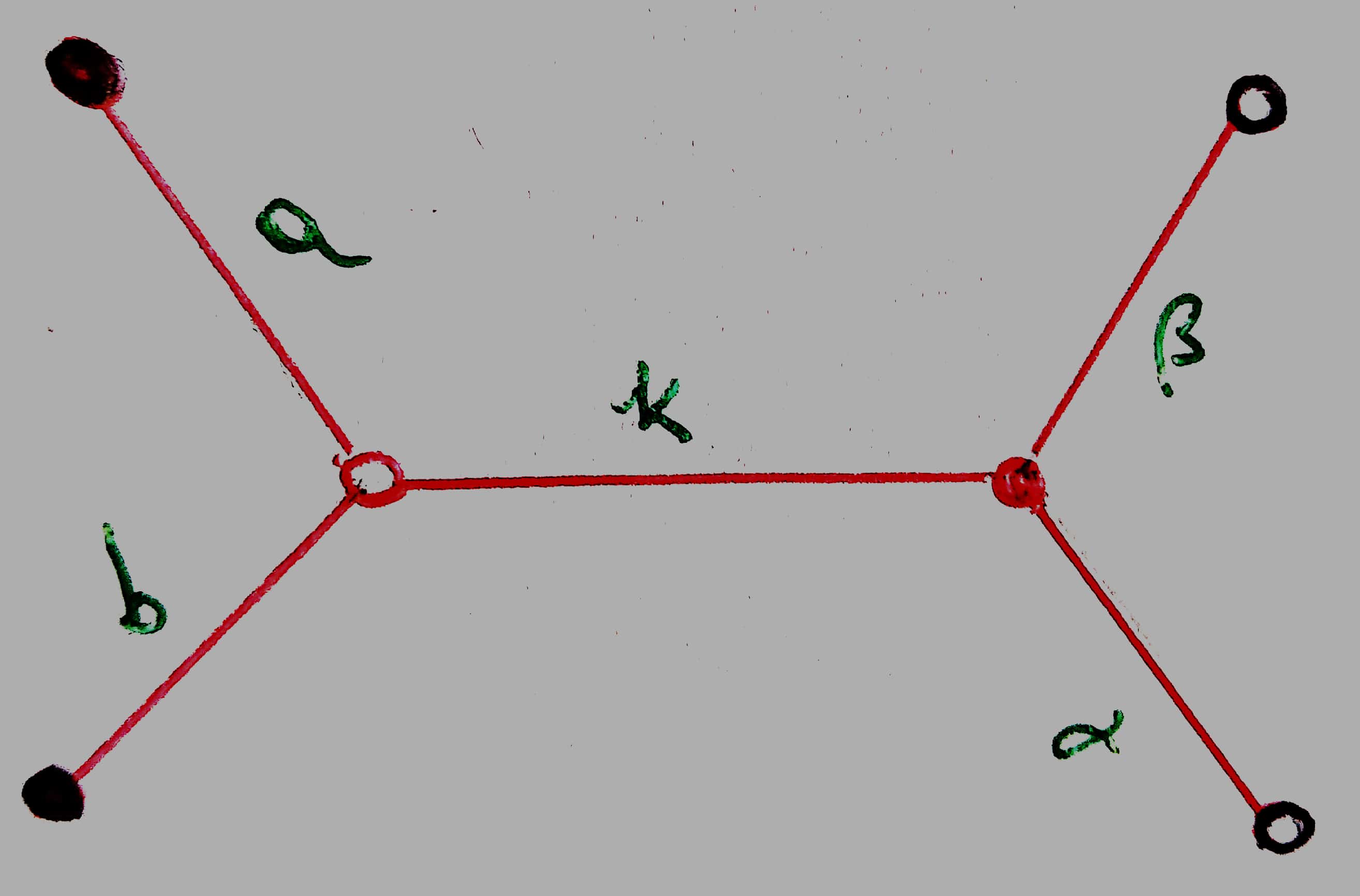}
\caption[chargeable graph]{$a+b+k=k+\alpha +\beta$}
\label{deg7pbg}
\end{center}
\end{figure}

Given $k>1$, we assume that for an arbitrary multi-extremal chargeable graph with constant capacity $M$ with $1<l\leq k$ interior edges, it is true that
\[|w|^{in}=|w|^{out}\]
for any feasible weighting $w:E(N)\rightarrow \R_{>0}$ on it.

Now, let $N'$ be a bipartite multi-extremal chargeable graph with constant capacity $M$ with $k+1$ interior edges and with a feasible weighting $w':E(N')\rightarrow \R_{>0}$ on $N'$. 

Let $E_1$ be a interior edge of $N'$ adjacent to at least one terminal edge of $N'$ and let $\epsilon_1=w'(E_1)>0$.

Let $a_1:=w'(A_1),a_2:=w'(A_2), \cdots ,a_p:=w'(A_p) $ be the list of the weights assigned by $w'$ to each terminal edge $A_h$ adjacent to $E_1$ with $h\in\{1,2,\cdots ,p\}$. 

There may exist more than one internal edge of $N'$ that is incident to the set of terminal edges $\{A_1, A_2,$ 
\\ $\cdots , A_p\}$. So, let $E_1, E_2, \cdots, E_u$ be those, possibly existing, edges with $\epsilon_l=w'(E_l)>0$ and let $\{b_{ij}\}_{j=1}^{u_i}$ the list of weights assigned by $w'$ to the interior edges of $N'$ that are incident to $E_{i}$.  And, let  $d_{j}>0$ for $j\in{1, 2, \cdots , r}$ be the list of weights assigned by $w'$ to all terminal edges of $N'$ different from those $E_i$ already considered.

Furthermore, let $F$ be the subgraph of $N'$ formed by the edges $\{A_1, A_2,  \cdots , A_p\}\cup\{E_1, E_2, \cdots , E_u\}$.

Then, the graph $N:=N'-F$ is a bipartite multi-extremal chargeable graph with constant capacity $M$ with $k+1-u\leq k$ interior edges and with a feasible weighting $w:w'|_{E(N)}:E(N)\rightarrow \R_{>0}$ on $N$. Thus, by the induction hypotesis,

 \begin{eqnarray}
 |w'|^{in}=|w|^{in}&=&|w|^{out}\\\nonumber
 &=& \left(\sum_{j=1}^{r}d_j\right) + \left(\sum_{i=1,j=1}^{u,u_i}b_{ij}\right)
\end{eqnarray}

But we also have
 \begin{eqnarray}
\left(\sum_{l=1}^{u}\epsilon_l\right) + \left(\sum_{i=1,j=1}^{u,u_i}b_{ij}\right) = M = \left(\sum_{l=1}^{u}\epsilon_l\right) +\left(\sum_{l=1}^{p}a_l\right)
\end{eqnarray}
Hence,
\[\sum_{i=1,j=1}^{u,u_i}b_{ij}=\sum_{l=1}^{p} a_l\]
Therefore,
\begin{eqnarray}
|w'|^{out} &=& \left(\sum_{j=1}^{r}d_j\right) + \left(\sum_{l=1}^{p}a_{l}\right)\\\nonumber
 &=&\left(\sum_{j=1}^{r}d_j\right) + \left(\sum_{i=1,j=1}^{u,u_i}b_{ij}\right)\\\nonumber
 &=& |w'|^{in}
\end{eqnarray}

\begin{figure}[H]
\begin{center}
\includegraphics[width=10cm]{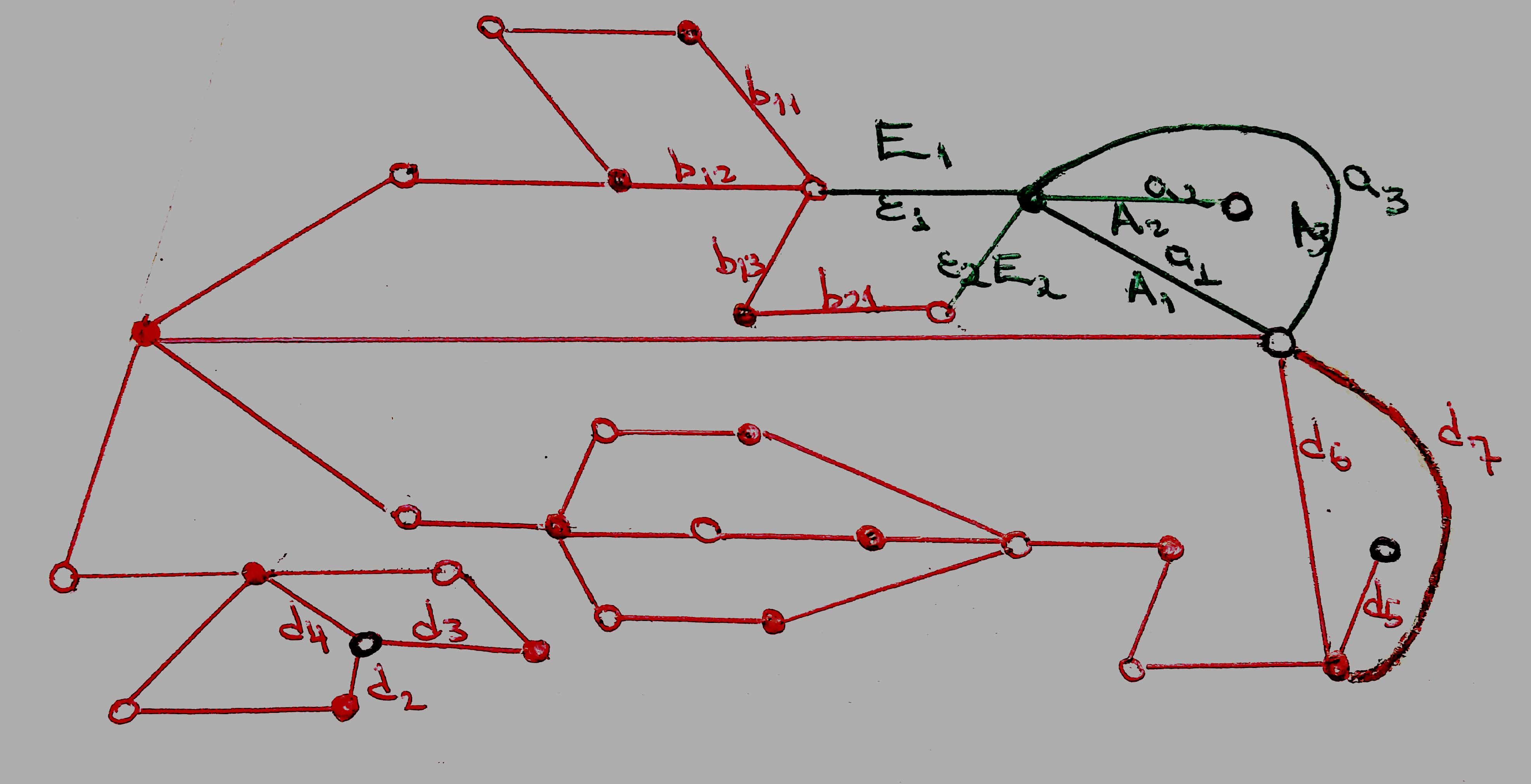}
\caption[weighted graph]{a weighted graph(with only the weights mentioned in the proof being visible)}
\label{deg7pbg}
\end{center}
\end{figure}
\end{proof}

\subsection{Cellularly Embedded Graphs}

\begin{defn}\label{comp(s)}
Let $S$ be a topological surface (possibly with boundary). 
$\mathcal{E}(S)$ is, by definition, the set of all Jordan arcs on $S$. 
And $\mathfrak{p}:\mathcal{E}(S)\rightarrow  \faktor{(S\times S)}{\mathcal{S}_2}$ the map that takes the endpoints of a Jordan arc on $S.$
\end{defn}

\begin{defn}\label{embed-g}
A cellular embedded graph $\Gamma=(G,S,R_V,R_E)$ is the data of a graph $G$, a topological connected oriented closed surface $S$ and a pair of injective map $R_V: V(G)\rightarrow S$ and $R_E: E(G)\rightarrow CP(S)$, such that:
\begin{itemize}
\item[(a)]{$R_V(v)$ is a point of $S$ for each $v\in V(E)$;}
\item[(b)]{ $R_E(e)$ is a Jordan arc on $S$ for each $v\in V(E)$;}
\item[(c)]{\hspace{9cm}

\begin{center}
\begin{tikzcd}[row sep=large, column sep = large]
E(G) \arrow{r}{R_E}\arrow{d}{\psi_G} & \mathcal{E}(S)  \arrow{d}{\mathfrak{p}}  \\
 \faktor{(V(G)\times V(G))}{\mathcal{S}_2} \arrow{r}{R_V\times R_V} & \faktor{(S\times S)}{\mathcal{S}_2}
\end{tikzcd}\end{center}

is a commutative diagram;}
\item[(c)]{$R_E (e_1)\cap R_E (e_2)=\emptyset$ if $e_1\neq e_2 \in E(G)$;}
\item[(e)]{$S- R_E(E(G))$ is finite union of simply connected open subsets of $S$.}
\end{itemize}
We say that $\Gamma$ is a cellularly embedded graph in $S$ and $G$ is the \emph{graph model}.
\end{defn}

\begin{defn}[planar graphs]\label{plan-g}
A \emph{planar graph} is a cellular embedded graph in $\s^2$
\end{defn}

\begin{defn}\label{equivg}
Two embedded graphs $(G,S,R_V,R_E)$ and $(G',S',R'_V,R'_E)$ are isomorphic if there exist an orientation preserving homeomorphism $h:S\rightarrow S'$ that induces a isomorphism of (abstract) graphs. That is, such that the following diagram commutes

\begin{center}
\begin{tikzcd}[row sep=huge, column sep = 5.2em]
E(G) \arrow{r}{R_E}\arrow{d}{\psi_G} & \mathcal{E}(G)\arrow{d}{\mathfrak{p}_S}  \arrow[dashrightarrow,uibred]{r}{h} &\mathcal{E}(G')\arrow{r}{(R')_{E}^{-1}} \arrow{d}{\mathfrak{p}_{S'}} &E(G') \arrow{d}{\psi_{G'}}  \\
\sfrac{V(G)\times V(G)}{\mathcal{S}_2} \arrow{r}{R_{V}\times R_{V}} & \sfrac{(S\times S)}{\mathcal{S}_2}\arrow[dashrightarrow,uibred]{r}{h\times h}&\sfrac{(S'\times S')}{\mathcal{S}_2}\arrow{r}{(R')_{V}^{-1}\times (R')_{V}^{-1}}&\sfrac{V(G')\times V(G')}{\mathcal{S}_2}
\end{tikzcd}\end{center}

The graph isomorphism determined by $h$ is $(\theta, \varphi)$ with $\varphi:=(R')_{E}^{-1}\circ{}h\circ{}R_{E}$ and $\theta:=((R')_{V}^{-1}\circ{}h\circ{}R_{V}) \times ((R')_{V}^{-1}\circ{}h\circ{}R_{V})$. 
\end{defn}

\begin{defn}[faces]
For a embedded graphs $\Gamma:=(G,S,R_V,R_E)$ each component of $S-R_E(E(G))$ is called \emph{face} and its closure is a \emph{closed face}. $F(\Gamma)$ is the set of faces of $\Gamma$. 

The edges and vertex in the boundary of a face is said to be incident to that face and vice-versa. 
\end{defn}
We also resort to the word \emph{adjacent} to announce that relation between vertices, edges, and faces of a cellularly embedded graph.

\begin{defn}[corner]\label{corner}
A \emph{corner} of a embedded graph $\Gamma$ is a vertex of degree greater or equal to $3$.
\end{defn}

\begin{defn}[parity]\label{par}
A (\emph{odd}) \emph{even} graph is a graph whose all of its vertices have a (odd) even degree. The same words are atributed to embedded graph in accordance with its abstract model graph.  
\end{defn}

When the graphs are endowed along with an additional structure, for instane with a labeling of its vertices, the horizontal morphisms are required to respect this structure.

\begin{defn}[dual graph]
The dual graph of a cellularly embedded graph $\Gamma := ((H,S,R_V,R_E))$ is a cellularly embedded graph $\Gamma^{\ast}$ in $S$ with graph model $G^{\ast}$ such that:
\begin{itemize}
\item[(a)]{for each $f\in F(\Gamma)\,\exists !\, f^{\ast}\in V(G)$ such that $R_{V}^{\ast}(f^{\ast})\in f$;}
\item[(b)]{for each $e\in E(G)\,\exists !\,e^{\ast}\in E(G^{\ast})$ such that:
\begin{itemize}
\item[(b.1)]{ $|R_E (e)\cap R_{E}^{\ast}(e')|=1$;}
\item[(b.2)]{ if $f_1 , f_2 \in F(\Gamma)$ are the two faces of $\Gamma$ adjacents to $\epsilon:=R_E(e)$ then $\mathfrak{p}_S (R_{E}^{\ast} (e^{\ast}))=\{f_{1}^{\ast}, f_{2}^{\ast}\}$.}
\end{itemize}}
\end{itemize} 
\end{defn}

That is, the dual graph is the embedded graph in $S$ constructed in the following way:
\begin{itemize}
\item[\textbf{1st}]{choosing a unique point into each face of $\Gamma$;}
\item[\textbf{2nd}]{and then, for each pair of those points that are in adjacent faces of $\Gamma$ we connect they by a Jordan arc. Being one Jordan arc for  each edge of $\Gamma$ that those two faces share, with the constraint that they intersect once.}
\end{itemize}

\begin{defn}[face coloring]\label{fc-g}
A face coloring of a cellular embedded graph $\Gamma$ is a surjective function $\mathfrak{c}:F(\Gamma)\xrar{\quad}C$ where $C$ is a finite set. The elements of $C$ are called \emph{colors}.
\end{defn}

\begin{defn}[alternate face coloring]\label{fc-g}
An alternating face coloring of $\Gamma$ is a face coloring $\mathfrak{c}:F(\Gamma)\xrar{\quad}C$ with $|C|=2$ and such that adjacent faces have distint colors from $\mathfrak{c}$.
That is, the dual graph $\Gamma^{\ast}$ is a embedded bipartite graph.
\end{defn}


\chapter[A combinatorial presentation of branched coverings]{\rule[0ex]{16.5cm}{0.2cm}\vspace{-23pt}
\rule[0ex]{16.5cm}{0.05cm}\\A combinatorial presentation for branched covers}
\label{cap-03}

We are interested in to understand and to classify rational functions on $\cc$ through its critical datum. From the \emph{Riemann-Hurwitz formula} is known that a degree $d$ rational function has $2d-2$ critical points counted regarding a degree of coincidences, its multiplicities.

In \cite[ theorem 9.1]{EiH:83} \emph{Eisenbud and Harris} showed that there are, up to post-composition with \emph{M\"obius} transformations of $\cc$, finitely many degree $d$ rational functions with given critical points at $p_1 , p_2, \cdots ,\\ p_{2d-2}\in\cc.$


With the constraint that $p_i\neq p_j$ for $i\neq j$, 
\emph{L. Goldberg} in \cite{Gold:91} showed 
that there exist, up to post-composition with $\cc$-automorphisms, at most the $d$-\emph{Catalan} number ${\rho(d):=\frac{1}{d}{2d-2\choose d-1}}$  of 
rational functions of degree $d$ with the critical set being the given subset $R:=\{p_1 , p_2, \cdots , p_{2d-2}\}\subset\CPu$. 

Let $\mathcal{R}_{R}=\mathcal{R}_{R}^{d}$ be the set of degree $d$ rational functions that possesses $R:=\{p_1 , p_2, \cdots , p_{2d-2}\}\subset\cc$ as its critical set. For a given subset $R\subset\cc$, by changes of coordinates with \emph{Mobius transformations} on the domain and codomain space, we can modify 
those rational functions with $R$ as its critical set such that three chosen points of $R$ turn to $0, 1$, and $\infty$ and so that they are fixed points for those rational functions from $\mathcal{R}_{R}$ after that appropriate changes of coordinate. For counting purposes this procedure of normalization is allowed (see $\ref{just-norml}$), thus we can consider only as prescribed critical sets, subsets of the form \[R=\{p_1, p_2, ...,p_{2d-5}, 0, 1, \infty\}\subset\cc\] and in this case those rational functions having 
$R$ as its common critical set and keeping the set $ \{0, 1, \infty \} $ pointwise fixed can not be transformed one into another by post-composition with a \emph{Möbius} transformation (since the identity $\cc$-automorphism is the unique one that fixes $3$ points of $\cc$).

\begin{defn}[$\mathcal{C}$-equivalence]
Two rational functions of the same degree $f,g\in\C(z)$ are $\mathcal{C}$-equivalent 
if there exists an automorphism $\sigma\in Aut(\overline{\C})$ such that
\[f=\sigma\circ g\]
A class of rational function for that equivalence will be assined by $[\bullet]_\mathcal{C}$.
\end{defn}

\section{Normalizations}
Given a rational function $f\in{\C(z)_d}$, let $M:= M_{crit(f)} \in Aut (\cc)$ be the conformal automorphism that send $m, n, p$ to $1, 0, \infty$ respectively for some choice $\{m, n, p\}\subset crit(f)$. And let  $M_f \in Aut (\cc)$ be the \emph{M\"obius} mapping sending $f(m)$, $f(n)$ and $f(p)$ to  $1$,  $0$ and $\infty$ respectively.

\begin{lem}\label{just-norml}
Let $f\in\C(z)$ and $g\in\C(z)$ two rational functions. $f$ is equivalent to $g$ if and only if $M_f\circ f\circ M$ and $M_g\circ g\circ M$ are equivalent.
\end{lem}
\begin{proof}
If for some $\sigma\in Aut(\cc)$, $M_f\circ f\circ M=\sigma\circ (M_g\circ g\circ M)$ then, $f=(M^{-1}_{f}\circ \sigma\circ M_g)\circ g$. So $f$ and $g$ are equivalent. 

Now, if for some $\sigma\in Aut(\cc)$, $f=\sigma\circ g$ we also have $M_f \circ f \circ M=(M_f \circ \sigma\circ M_g^{-1})\circ M_g\circ g\circ M$. Therefore, $M_f\circ f\circ M$ is equivalent to $M_g\circ g\circ M$. 
\end{proof}
That lemma is certainly valid for any other choice of three distinct points in $ \overline {\C} $ and we can choose $ M $ and $ M_{.}$ so that the normalized rational function $ M_f \circ{} f \circ{} M$ exchanges two of those distinct points and therefore leaves the third fixed.

This lemma enable us to care about only with those rational functions that have $0$, $1$ and $\infty$ as critical and fix points among its critical points. It garantees that the number of equivalent classes of rational functions sharing the set $\{c_1, c_2, c_3, \cdots , c_k\}$ as its critical set is the same for those one sharing the set $\{0, 1, \infty, p_3 \cdots , p_k\}$ where $\{p_3,\cdots, p_k\}\subset\{c_1, c_2, c_3, \cdots , c_k\}$ as its critical set and that maintain fix $0$, $1$ and $\infty$.

Furthermore, notice that not necessarily $f$ is equivalent to $M_f \circ f\circ M$. But if $M$ is a covering map to the branched cover $f:\cc\rightarrow\cc$, i. e., if $f=f\circ M$ or if $M$ is an automorphism of $f$, i.e., if $M\circ f=f\circ M$ than we have $f$ equivalent to $M_f \circ f\circ M$.  

We can also ask if there is some good relation between the dynamical moduli space \[\mathfrak{M}_d:=\C (z)_d\left/\left(\text{conjugation in} Aut(\cc)\right)\right.\] and \[PM_d:=\C(z)_d\left/\left(f\sim \text{Möbius}\circ{}f\right)\right.\].

If $f$ and $g$ are rational functions of the same degree corresponding to the same point in $\mathfrak{M}_d$ and $PM_d$, that is $f=\alpha^{-1}\circ{}g\circ{}\alpha$ and $f=\sigma\circ{}g$ for some $\alpha, \sigma \in Aut(\cc)$, it follows that 
\begin{eqnarray}
(\alpha\circ{}\sigma)\circ{}g=g\circ{}\alpha
\end{eqnarray}

For a fixed rational map $R$, as we mentioned above, a \emph{M\"obius} map that satisfies $\alpha\circ{}R=R\circ{}\alpha$ is called an \emph{automorphism of} $R$. The group of such degree one rational maps is called the \emph{automprphism group} of $R$ and is denoted by $Aut(R)$. This group is always finite.

So, what can we say about the set $qAut(R):=\{\alpha\in Aut(\cc); \exists\sigma\in Aut(\cc)\quad\mbox{such that}\quad R\circ{}\alpha=\alpha\circ{}\sigma\circ{}R\}$?

\begin{lem}
$qAut(R)$ is a subgroup of $Aut(\cc)$. 
\end{lem}
\begin{proof}
First, is clear that $Id \in qAut(R)$.
 
Suppose that $(\alpha\circ{}\sigma)\circ{}R=R\circ{}\alpha$ and $(a\circ{}s)\circ{}R=R\circ{}a$ for some $\alpha, \sigma, a, s \in Aut (\cc)$. Then taking $\delta:=a^{-1}\circ \sigma\circ a\circ s$ we obtain $((\alpha\circ a)\circ{}\delta)\circ{}R=R\circ{}(\alpha\circ a)$. Therefore, $qAut(R)$ is closed for composition.

And choosing $s:=\alpha^{-1}\circ \sigma^{-1}\circ \alpha^{-1}$ it follows that $(\alpha^{-1}\circ s)\circ R= R\circ\alpha^{-1}$. So $qAut(R)$ is also closed for taking the inverse into $Aut(\cc)$. So, we are done. 
\end{proof}

Note that the non triviality of $qAut(R)$  is the same that the existence of a rational map $S$ that is equivalent to $R$ and also conjugated to it. 

From an observation by Thurston presented in \cite{MR3515032}-Lemma $8.6$ we have:
\begin{lem}
 $qAut(R)$ is non trivial for all generic cubic rational function $R\in\C (z)_3$.
\end{lem}

That is, for each cubic rational function there is a cubic rational function that is both conformally conjugated and equivalent to it.




\section{Embedded Graphs and Branched Coverings}\label{emb-cell-g}

Let $f:X\rightarrow \s^2$ an orientation-preserving branched
covering map of $\s^2$ of degree $d\geq 2$ with $m$ critical values. 

Let $\Sigma$ be an oriented Jordan curve passing through the critical values of $f$, and let ${\Gamma}:=f^{-1}(\Sigma)$ be the inverse image of $\Sigma$ by $f.$ $\Gamma$ is a cellularly embedded graph into $X$. 
That is the principal object of the study of the present chapter. 

Except when explicitly stated in a different way, hereinafter $f:X\rightarrow \s^2$ will be an orientation-preserving branched covering of $\s^2$ of degree $d\geq 2$ with $m$ critical values.

\begin{defn}[Post-critical curve]\label{def:pcc}
A \emph{post-critical curve} for $f$ is an isotopy class relative to $R_f$ of a \emph{Jordan curve} $\Sigma\subset\s^2$ passing through the critical values of $f$ into $\s^2.$ Such an isotopy class will be simply denoted by $\Sigma$, some representative of it.
\end{defn}

The points in the critical values set $R_f$ of $f$ will be labeled by $1, 2, \cdots , m$ according to the order that $\Sigma$ pass through them positively regarding the orientation of $\s^2$.  

\begin{defn}[Pullback graph]\label{pullbackg}
The isotopy class relative to $C_f$ of ${\Gamma}:=f^{-1}(\Sigma)$ is called the \emph{pullback graph} of $f$ with respect to $\Sigma$, or simply, $\Sigma$-\emph{pullback graph} of $f$. 

A point in $f^{-1}({R_f})$ will be called by \emph{vertex} and $f^{-1}(R_f)$ will be called the vertex set of ${\Gamma}$ and denoted by $V({{\Gamma}})$.

An arc into ${\Gamma}$ connecting two points in $f^{-1}(R_f)$ will be called by \emph{edge} and the set of edges joining the points in $V({{\Gamma}})$ will be called by \emph{edge set} of ${\Gamma}$ and denoted by $E({{\Gamma}})$. 

A connected component of $X-{\Gamma}$ will be called by \emph{face of} ${\Gamma}$ and the set of such connected components will be called by \emph{face set} of ${\Gamma}$ and denoted by $F({{\Gamma}})$.
\end{defn}

Since, by definition, a \emph{post-critical curve} of $f$ is a isotopy class, for a \emph{pullback graph} to be well definide we have to ensure that the inverse image of two representatives of a \emph{post-critical curve} are isotopic relative to the critical set of $f$. But that is guaranteed by Lemma $\ref{isotpc1}$.

\emph{Thurston} has introduced the notion of \emph{balanced planar graphs}\cite{STL:15} and then showed that they combinatorially characterizes all such ${\Gamma}=f^{-1}(\Sigma)$, where $f:\s^2\rightarrow\s^2$ has $2d-2$ distinct critical points. In other words, we can say that \emph{Thurston} characterized how oriented planar graphs into $\s^2$ with $2d-2$ vertices of valence $4$ corresponds to the \emph{inverse image} by some generic orientation-preserving branched cover $f:\s^2\rightarrow\s^2$ of an oriented \emph{Jordan curve} passing through the critical values of such $f$ and vice versa.

The general version of this characterization that we will obtain here refers to finite degree branched covers of the $2$-sphere, whose domain can be any compact oriented surface and for all possible critical configurations, in addition to the generic branched selfcoverings of $\s^2$ initially considered by $Thurston $. That consists of a full compact oriented $2$-dimensional version of the \emph{Thurston result}.

Thus, to get that, we are going to adapt the notion of \emph{balanced graphs} to the broader class of embedded graphs on surfaces of arbitrary genus. The meaningful fact about the modified \emph{balance condition} is that in the palnar case it is equivalent to the \emph{Thurston's balance condition}.

Consider $\Sigma$ as an oriented graph with vertex set $V(\Sigma) =R_f$. Hence each vertex of $\Sigma$ has valence $2$ and $V({\Gamma})=f^{-1}(R_f)$. Notice that ${\Gamma}$ is an oriented graph 
on $\s^2$.

Let  $\pi(f,\Sigma)=(\pi_{1}^{\ast}, \pi_{2}^{\ast}, ..., \pi_{m}^{\ast})$ be the passport of $f$ with each partition labeled in accordance with the labeling of the critical points determined by $\Sigma$.
 
For each vertex $w_j\in R_f$ there are $l_j$ vertices in ${\Gamma}$ corresponding to it by $f$. 
Among them there are $p_{k}^{j}$ vertices of valence $2\cdot k$ for each $k$ on the support of $\pi_j$. 

Thus $|V({{\Gamma}})|=\sum_{j=1}^{m}l_j.$ The vertices of valence strictly greater than $2$ are the critical points of $f$ and the other vertex are regular preimage of the critical value $v$ that we will call  \emph{cocritical} points for $f$.

If $v\in V({{\Gamma}})$ is such that $f(v)=w_j$ then label it by $j.$  Thus we will have $\sum_{n=1}^{l_j}p_{n}^{j}$ vertices in $V({{\Gamma}})$  labeled by $j$ for each $j\in\{1, 2, \cdots, m\}$.

Each connected component of $X-f^{-1}(\Sigma)$ is mapped by $f$ over the $2$-cell in the left or right side of $\Sigma\subset\s^2$. We can see that those faces are also topological disk, furthermore, having as boundary a finite union of \emph{Jordan archs} connecting points of $f^{-1}(R_f)$.

Color the left side of $\Sigma\subset\s^2$ \emph{pink} and \emph{blue} the right side of it.

Then, color each connected component of $X-{\Gamma}$ by the color of its $2$-cell image by $f$ in the right or left side of $\Sigma\subset\s^2$. That will give us a chessboardlike decoration to $X$, that is, a cellular decomposition of $X$ with an alternating bi-colouration of the faces.

\begin{figure}[H]
\begin{center}
\includegraphics[scale=0.55]{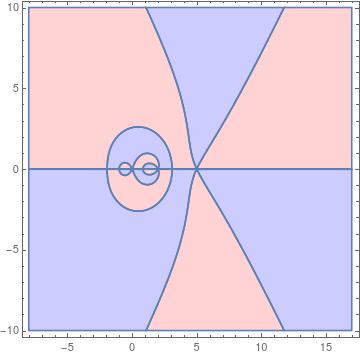}\
\caption[degree $7$ pullback graph]{$f(z)=\dfrac{500- 50z- 1215z^2 +1388.5z^3 -674.5z^4  +166.5z^5 -20.5z^6 +z^7 }{-z^3 +z^4 }$}
\label{deg7pbg}
\end{center}
\end{figure}

\begin{defn}
A vertex in $V({{\Gamma}})-C_f$ is called \emph{cocritical} vertex. A vertex in $C_f \subset V({{\Gamma}})$ will be called by a \emph{corner} and a path in ${\Gamma}$ connecting two corners will be called \emph{saddle-connection}.
\end{defn}
\begin{prop}\label{prop-shape-graph} 
Let $f:X\rightarrow \s^2$ with passport $\pi_{f}^{\ast}=(\pi_{1}^{\ast}, \pi_{2}^{\ast}, ..., \pi_{m}^{\ast})$. Then, ${\Gamma}={\Gamma}_f(\Sigma):=f^{-1}(\Sigma)$ is a connected embedded graph on $X$ with $2d$ faces and $\sum_{j=1}^{m}l_j$ vertices such that each of its faces is a \emph{Jordan domain} containing on its boundary only one vertex corresponding to each critical value of $f$ with the labelings appearing cyclically ordered around it.

Furthermore, we have $p_{n}^{j}$ vertices in $V_{\Gamma}$ of valence $2n$ corresponding to the critical value $w_j$ for each $j\in\{1, 2, ..., m\}$ and $n\in\{1, 2, ..., d\}.$
\end{prop}
\begin{proof}
Everything except the fact that the faces are Jordan domains was clarified above. Then, let's prove it.

Let $\Sigma$ be a $\emph{Jordan}$ curve passing through the critical values of $f$, then $f|_{X-\{f^{-1}(\Sigma)\}}:X-\{f^{-1}(\Sigma)\}\longrightarrow \s^{2}-\Sigma$ is a covering map. 

By \emph{Jordan-Sch\" oenflies theorem}, $\s^{2}-\Sigma$ is a disjoint union of two Jordan domains, say $A$ and $B$. Let $a\in A$ and $b\in B$. Then the fibres of $f$ above $a$ and $b$ contains, each one, $d$ distincts points. For each point $x_1 \in f^{-1}(a)$ and $y_1 \in f^{-1}(b)$ the map $Id_A:A\hookrightarrow \s^2 -\Sigma$ and the map $Id_B:B\hookrightarrow \s^2 -\Sigma$ lifts uniquely to a map $S_{1a}:A\longrightarrow X-\Gamma$ and $S_{1b}:B\longrightarrow X -\Gamma$ over the component of
$X -\Gamma$ that contains $x_1$ and $y_1$, respectively, giving therefore a section to $f$ over each face of ${\Gamma}$. Thus, being $S_{1a}$ and $S_{1b}$ homeomorphisms over its image, 
$X-\{f^{-1}(\Sigma)\}$ is a union of $2d$ open sets
that are homeomorphic to Jordan domains. Then, we are done. 
\end{proof}

Therefore, the guiding question is:

\begin{quest}\label{Q1}
\begin{center}\textcolor{DarkGreen}{What oriented embedded graphs into a compact surface $X$ can be realized as a pullback graph?}
\end{center}
\end{quest}

This question is motivated by the following 
visionary issue raised by Thurston:
 
\begin{quest}\label{Q2}
\begin{center}
\textcolor{DarkGreen}{What is the shape of a rational map?}
\end{center}
\end{quest}

The Proposition $\ref{prop-shape-graph}$ points out that the embedded graphs wondered in $\ref{Q1}$ should to be among those cellularly embedded graphs that admits an alternate $2$-coloring with the same number of faces colored by each color and for wich is possible grouping the vertices in a suitable manner compatible with an 
branched covering passport. 

This latter condition will be duly presented and examined 
in the next section.

\begin{defn}[Globally Balanced Graph]\label{fefn-bg}
A \emph{Globaly Balanced Graph} of type $(g, d, m)$
is a cellularly embedded graph on an oriented compact surface of genus $g$, $S_g$, with $2d$ faces, $m$ corners $\ref{corner}$ 
 and which admits an alternating $2$-coloring of the its faces with 
$d$ faces colored by each color. We say also that such an embedded graph satisfy the \emph{Global balance condition}.
\end{defn}

\begin{figure}[H]
\begin{center}
\includegraphics[width=9cm,height=5.5cm]{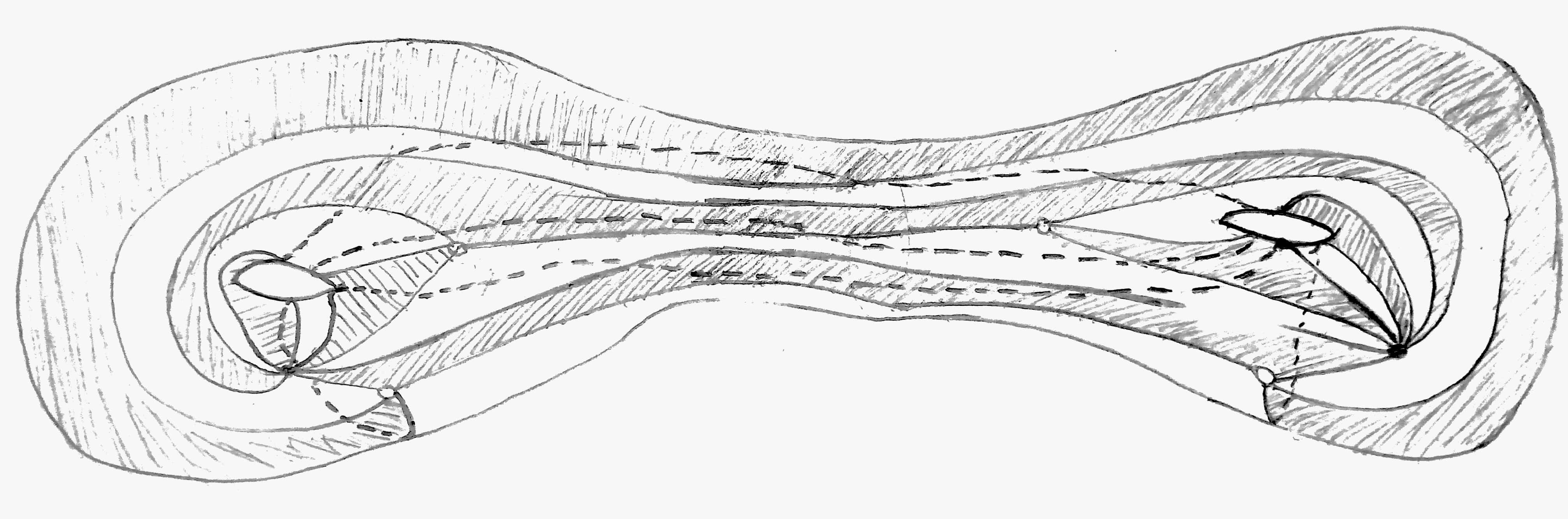}
\caption{Globally Blalanced graph of type $(2, 4, 6)$}
\label{par_n_gbal}
\end{center}
\end{figure}

Notice that any connected even planar graph admite an alternating coloring for its faces (there is only tow possible colorings).  Nonetheless, 
it doesn't always happen that these graphs are globaly balanced as we can see in the Figure $\ref{par_n_gbal}$.

 \begin{figure}[H]
\begin{center}
\includegraphics[width=4cm]{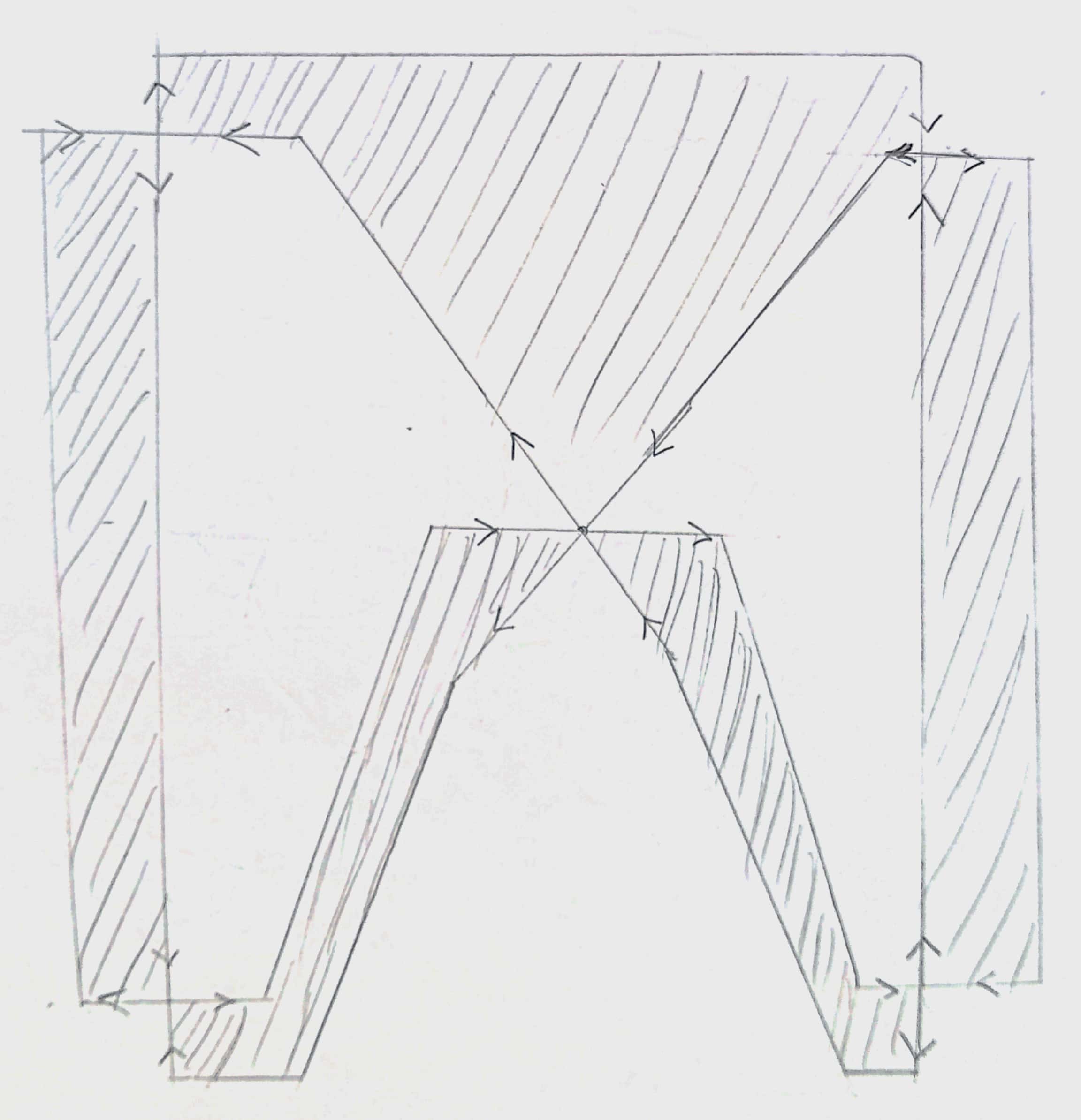}
\caption{even graph non Globally Blalanced}
\label{par_n_gbal}
\end{center}
\end{figure}

\begin{lem}\label{mcornern}
The maximal number of corners on a balanced graph of degree $d$ and genus $g$ is
\[2g+2d-2\]
\end{lem}
\begin{proof}
From the \emph{Euler formula},\[2-2g=V(\Gamma)-E(\Gamma)+2d\]. And, sice each corner has degree greater or equal to $4$, $E(\Gamma)\geq{2V(\Gamma)}=\frac{4V(\Gamma)}{2}.$

Therefore,
\[V(\Gamma)=2V(\Gamma)-V(\Gamma)\leq E(\Gamma)-V(\Gamma)=2g+2d-2\]
\end{proof}

\begin{defn}\label{deg-bg}
The degree of a  globally balanced graph is half of the number of its faces
 (i.e., is the number of faces with the same color).
\end{defn}

Now we are going to introduce a class of \emph{embedded graphics} and we will describe how to build a branched covering from them.
 
\subsection{Construction of branched coverings from diagrams}

\begin{defn}[vertex labeling]\label{def:labeling}
For a graph $G$, a surjective map $L:V_{{G}} \longrightarrow J$ from the vertex set to a finite set $J$ is called a vertex labeling of $G$ by $J$. For a vertex $v\in V_G$ such that $L(v)=j$ we write $v_j .$
\end{defn} 

\begin{defn}[admissible vertex labeling]\label{adm-v-l}
Let ${G}$ be a degree $d>0$ 
globally balanced graph 
with the same number $m\geq 2$ of vertices incident to each one of its faces (here we are also considering vertices of valence $2$).  
A vertex labeling of $G$ by the ordered set $\{1<2<\cdots <m\}$ 
is called \emph{admissible labeling} if:
\begin{itemize}
\item[(1)]{at each face of $G$ the labelings $1<2<\cdots <m$ appears cyclically ordered around it and such that every bordering cycle when is traveled in the increasing sense of the labelings it is incident to one prefered color on the left side (in consequence of the alternating hypothesis over the coloring, it is incident to the other color on the right side of the border);}
\item[(2)]{and for each label $j\in\{1<2<\cdots <m\}$ it holds
\begin{eqnarray}
\displaystyle{\sum_{k=1}^{l_j} \dfrac{deg(v_{j}^{k})}{2}=d}
\end{eqnarray}where those $v_{j}^{k}$'s are the vertices of $G$ labeled with $j\in\{1<2<\cdots <m\}$, i.e., $\{v_{j}^{1}, v_{j}^{2},\cdots ,v_{j}^{l_j}\}=L^{-1}(j)$.}
\end{itemize}
\end{defn}

\begin{defn}[Admissible Graph]\label{adm-g}
An \emph{admissible} 
graph is a globally balanced graph with an admissible labeling.
\end{defn}

Note that if $M$ is the biggest number of corners (those topologicaly not hiden vertices of the embedded graph) that are incident to a face among all faces of the graph, than necessarily it follows that $m\geq M.$

We can have the same alternating bicolored cellular decomposition corresponding to different admissible graphs.

\begin{figure}[H]{
    \centering
       \subfloat[]
    {{\includegraphics[width=6.6cm]{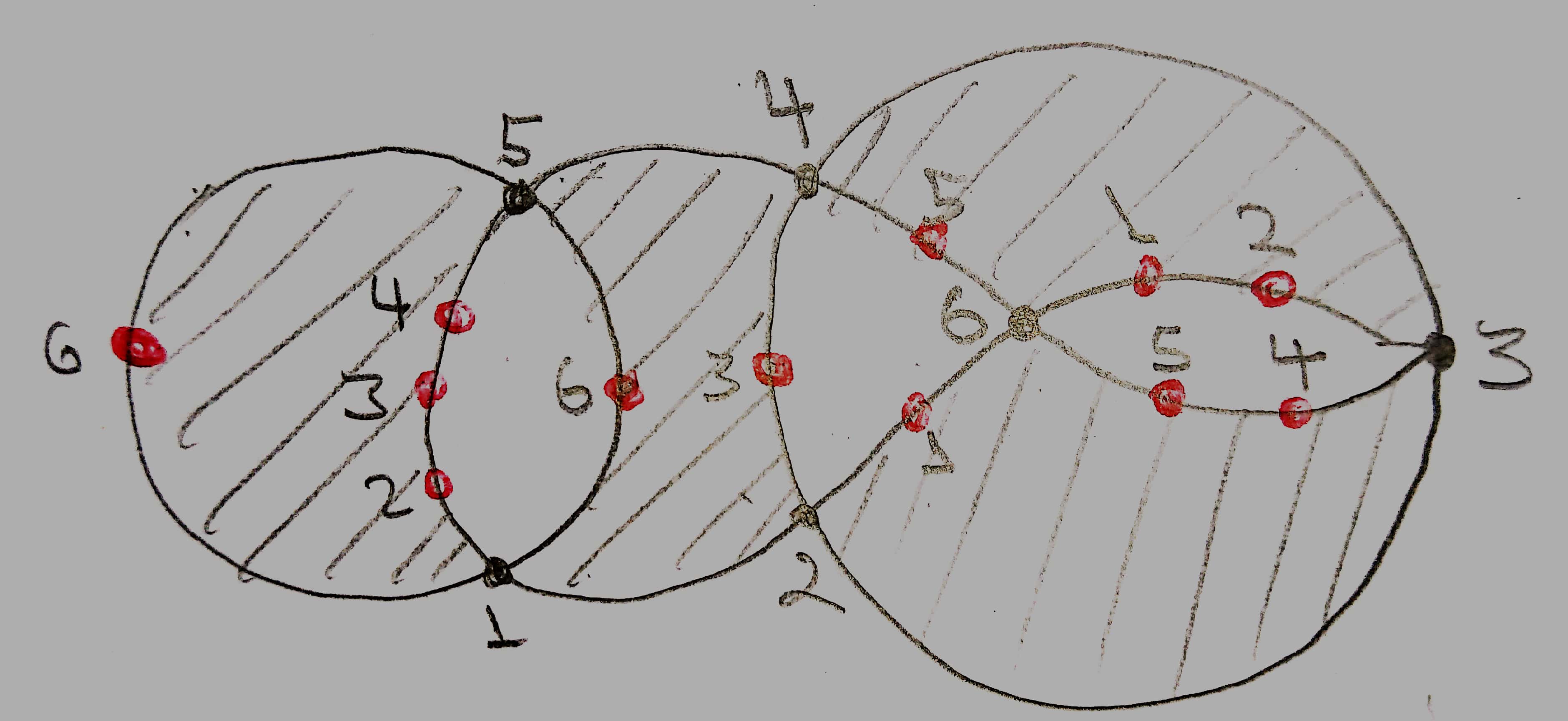} }}
        \qquad
    \subfloat[]
    {{\includegraphics[width=6.4cm]{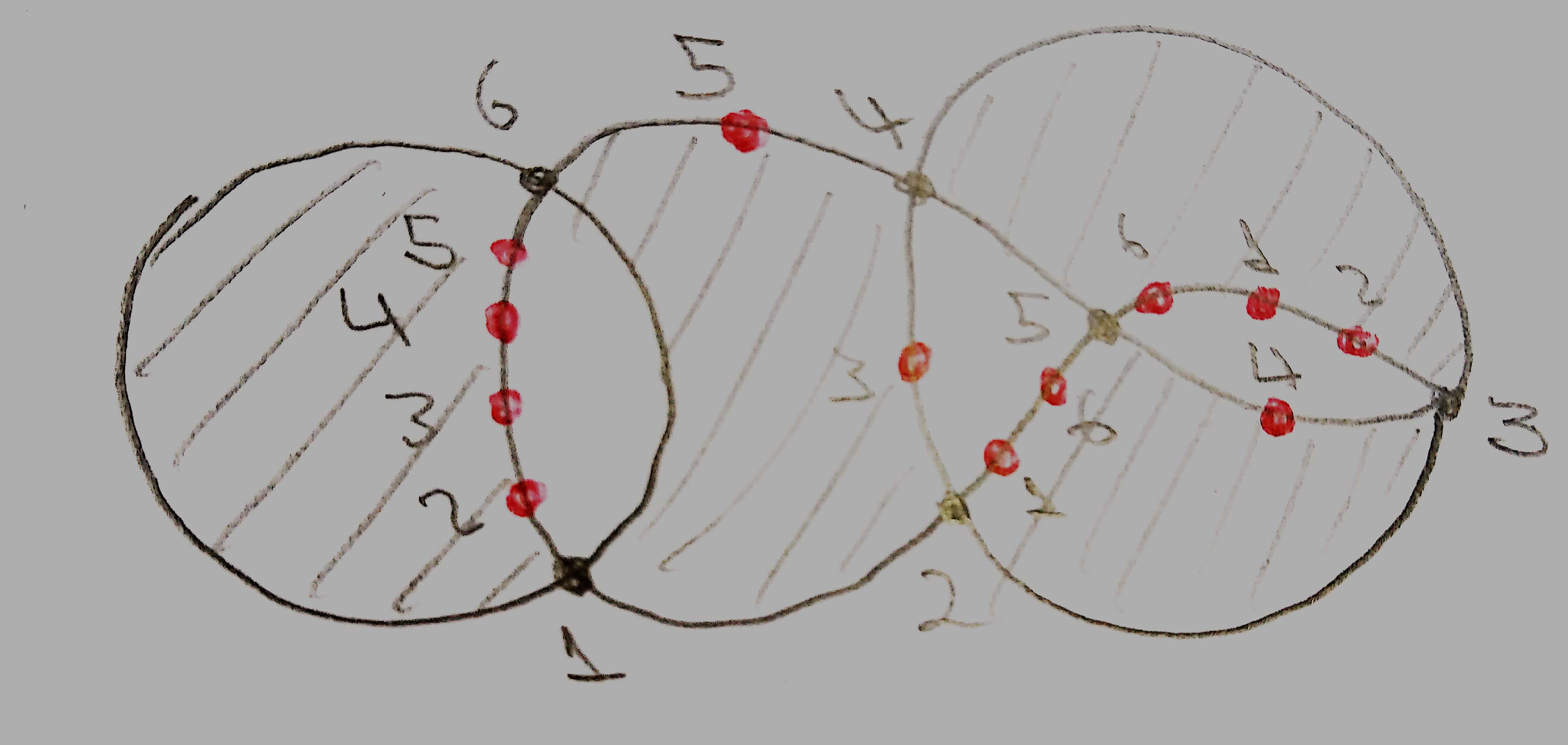} }}
    \newline
    \centering
    \subfloat[]
    {{\includegraphics[width=6cm]{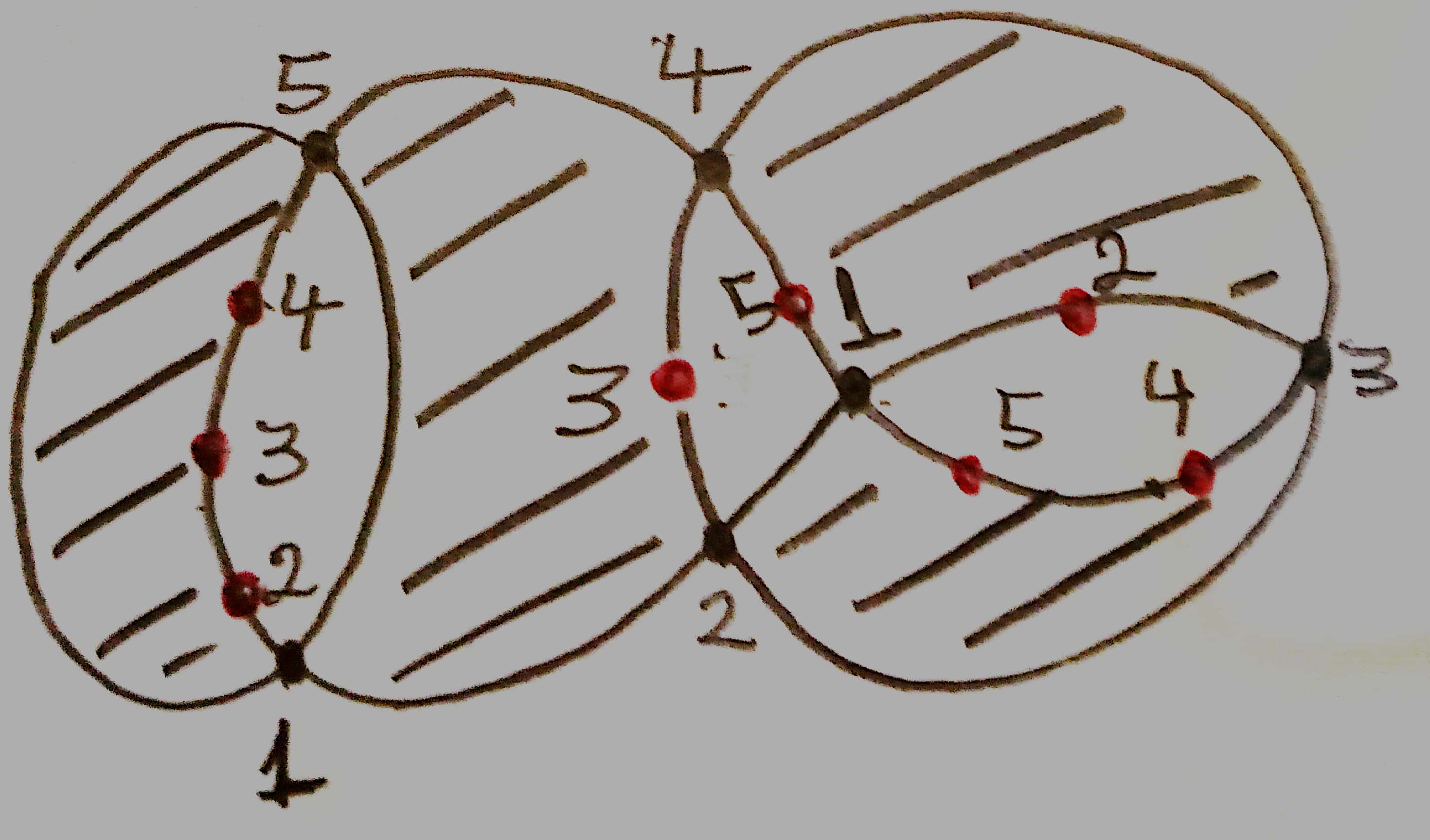} }}\qquad
       \caption{ different admissible graphs from the same cellular decomposition of $\s^2$ }%
    \label{fig:h12pullb1}%
    }
\end{figure}

\subsubsection{construction of a branched selfcovering of $\s^2$ from an admissible graph}\label{bcconstruction}

Let $G=(G, L)$ be an admissible graph. 
Choose a enumeration for the faces of $G$, $N:\{1, 2,\cdots,2d\}$ $\longrightarrow F_G$ such that $G_k=N(k)$ has a saddle-connection in common with $G_{k+1}=N(k+1)$ for each $k$ $\mod 2d$ and the face $G_1$ has at least $3$ vertex with valence strictly greater than $2$. Notice that for any  admissible graph with admissible labeling different from $L:V_{{G}} \longrightarrow\{1<2\}$ it has at least one face with at least $3$ corners incident to it.

We start distinguishing $3$ consecutive corners along $\partial{}G_1$ and appoint they by $\alpha$, $\beta$ and $\gamma$.
By \emph{Sch\"oenflies Theorem} we can embed the closure of the face $G_1$ into $\s^2$, $\iota_1:\overline{G}_{1}\hookrightarrow\s^2$ and in such a way that only those $3$ distinguished corners are sended over itself by $\iota_1$. We will refer to this choice and imposition as normalization and to admissible graphs with such corners highlighted as normalized admissible graph.

Again, using \emph{Sch\"oenflies Theorem}, we embed the closure of $G_2$ onto the closure of the complement of $A:=\iota_1(G_1)\subset\s^2$ in $\s^2$ in such a way that those embeddings agrees on the common saddle-connections and such that 
 $\partial G_2$ is sended over the image of $\partial G_1$  by $\iota_1$ with vertices with the same label 
 having the same image by the corresponding embeddings $\iota_1$ and $\iota_2$. Furthermore, such that $\iota_2(p)\neq p$ for all points into $G_2$ except for those $3$ distinguished corners on $\partial{}G_1$.

Then, repeating that procedure up to $G_{2d}$ we will have constructed a finite degree continuous map $f:\s^2\longrightarrow\s^2$ since by construction every point in $\s^2$ excepting those ones corresponding to the corners possesses exactly $d$ points above it. 

Let $\Sigma:=\iota_{1}(\partial G_1)$(notice that: $\Sigma:=\iota_{k}(\partial G_k),\;\forall\; k= 2, 3, ..., 2d$).

By construction $f:\s^2-G\longrightarrow\s^2-\Sigma$ is a local homeomorphism. Let $C\subset\s^2$ be the set of vertices of $G$ with degree strictly greater than $2$ and  $Q:=f(C)\subset\s^2$. Due to $f^{-1}(\Sigma)=G$ and the coincidence of the imbeddings over the saddle-connections, $f$ is a local homeomorphism in each point in $G-C. $

Note also that the local degree of $f$ around each point $v^{k}_{j}$ in the fiber of $f$ over the point $q_j:=f(\{v_{j}^{1},\cdots ,v_{j}^{l_j}\})$ is equal to $\dfrac{deg(v_{j}^{k})}{2}$.
 
Therefore, $f$ has the passport $\pi_j=\left(\dfrac{deg(v_{j}^{1})}{2}, \dfrac{deg(v_{j}^{2})}{2}, \cdots , \dfrac{deg(v_{j}^{l_j})}{2}\right)$ for each $j\in\{1, 2, \cdots, m\}.$ 

All that procedure described 
above is depicted into the following figure:
 
 \begin{figure}[H]
\begin{center}
\includegraphics[scale=0.08]{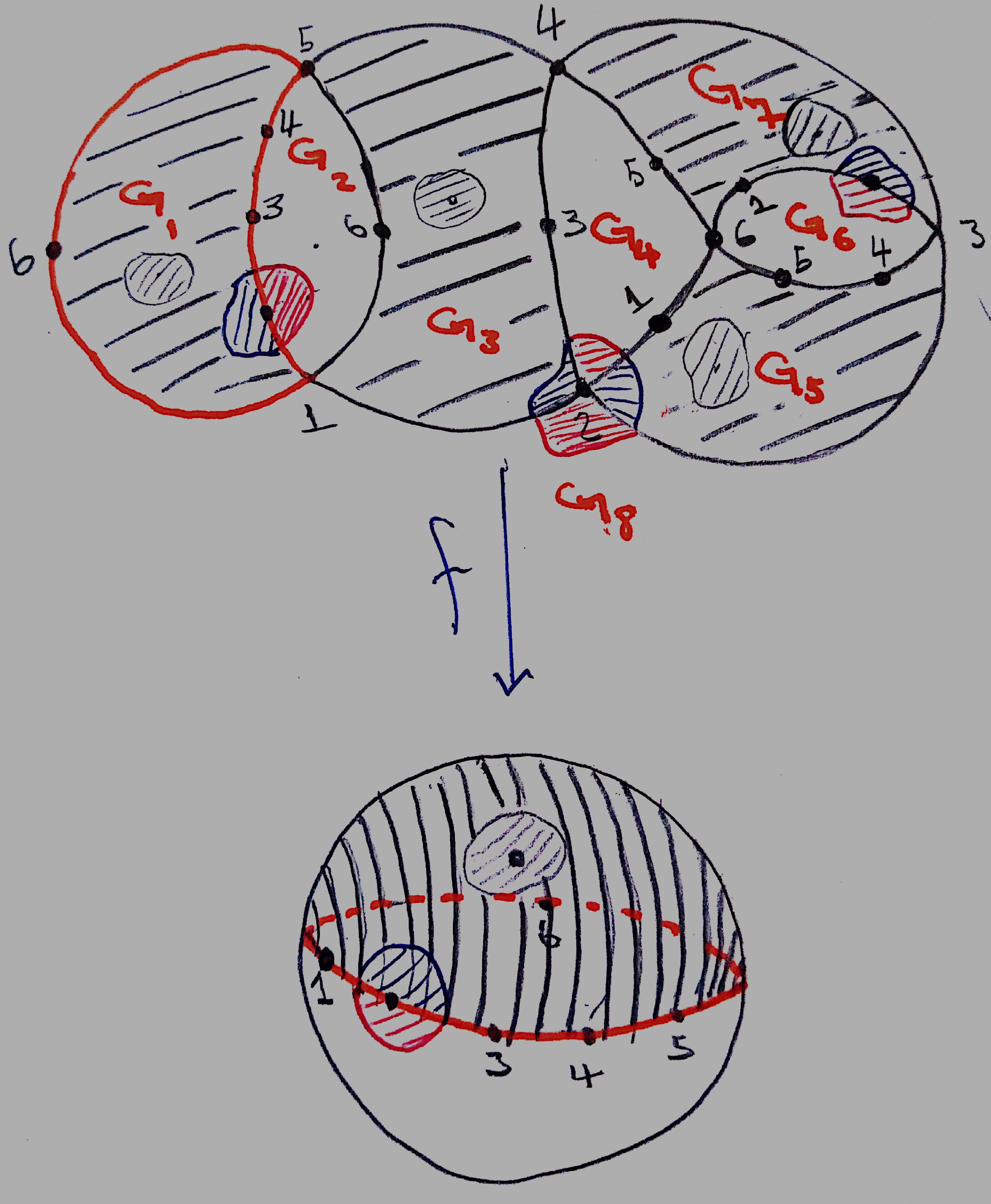}\
\caption{construction of a branched cover $\s^2\rightarrow\s^2$}
\label{cbc1}
\end{center}
\end{figure}

Now, by the \emph{uniformization theorem} there exist a unique homeomorfism $\mu:(\s^2, f(C)) \rightarrow(\overline{\C}, \{0,1,\infty,v_1,\cdots$\\
$, v_{m-3}\})$ for wich $\mu (\alpha)=1$, $\mu (\beta)=0$ and $\mu (\gamma)=\infty$ defining a complex structure over $(\s^2, f(C))$. But being $f$ a local homeomorphism over $\s^2 -C$ we can pullback the complex structure by $f$ to a new complex structure over $(\s^2 ,C)$, $\nu:(\s^2 ,C)\rightarrow(\overline{\C},\{0,1,\infty,c_1,\cdots,c_{2d-5}\})$ such that $\nu (\alpha)=1$, $\nu (\beta)=0$ and $\nu (\gamma)=\infty$. Therefore, the map \[F_{\nu\,\mu}:=\mu\circ f\circ\nu^{-1}:(\overline{\C},\{0,1,\infty,c_1,\cdots,c_{2d-5}\})\rightarrow(\overline{\C},\{0,1,\infty, v_1,\cdots, v_{m-3}\})\] is a holomorphic function, that is, is a rational function.


Then we have achieved

\begin{thm}\label{bcfromg1}
For each admissible planar graph $G$ there exist a holomorphic ramified selfcover of $\overline{\C}$ (\emph{i. e.}, a rational function), having $G$ as a pullback graph.
\end{thm}

All of the above argument also works for admissible non-planar graphics. Thus we actually have:
\begin{thm}\label{bcfromg2}
For each admissible graph $G$ into a genus $g$ compact surface $S_g$ there exist a holomorphic ramified cover $S_g\rightarrow{\cc}$ (\emph{i. e.}, a rational function), having $G$ as a pullback graph.
\end{thm}
Notice that we obtain the same branched cover if we choose a different suitable face enumeration but keeping the embeddings of the faces.

But, is there some distinction between, a priori, different rational functions obtained from the same admissible graph but constructed from different choices of those distinguished corners or from different embeddings  of the faces?

\begin{prop}\label{prop2}
Let $\mathcal{G}=(G,\iota)$ and $\mathcal{H}=(H,\eta)$ two equivalent  embedded cellular graphs. 
Let $f:(S_g,\mathcal{G})\rightarrow \s^2$ and $h:(S_g,\mathcal{H})\rightarrow \s^2$ two continuous surjective maps that restricts to homeomorphisms over the topological closure of each face and such that 
\begin{eqnarray}
f(\myov{F})=h(I(\myov{F}))
\end{eqnarray}
for $I$ as in definition $(\ref{equivg})$ and each face of $\mathcal{G}$. Then there exist a homeomorphism $\mathcal{I}:S_g\rightarrow S_g$ such that
\[f=h\circ{}\mathcal{I}\]. 
\end{prop}
\begin{proof}
Via the homeomorphism $I$ and the property $(\ref{prop2})$ define
\begin{eqnarray}
\mathfrak{I}(z):= h^{-1}|_{I(\myov{F})}\circ{}f(z)
\end{eqnarray} for each $z\in \myov{F}$ for each face $F$ of $\mathcal{G}$. And $\mathfrak{I}$ is a homeomorphism due the hypotesis that $f$ and $h$ restricts to homeomorphisms over each closed face of $\mathcal{G}$ and $\mathcal{H}$.
\end{proof}
 This proposition is essentially the Lemma 2 in \cite{MR1888795}.
 \begin{defn}
 Two embedded admissible graphs with a normalization are equivalent if they are equivalent (see definition $\ref{equivg}$) and there exist morphisms atesting that equivalence preserving the vertex labeling and the normalization.
\end{defn}
 
\begin{cor}
Given two equivalent admissible graphs with a 
{normalization}, say $\mathcal{G}$ and $\mathcal{H}$, the rational functions produced from it as in the preceding construction $\ref{bcconstruction}$ are equals if the face embeddings are isotopic relative to the critical value set.
\end{cor}
\begin{proof}
First, what we mean by saying that the face embeddings are isotopic relative the critical value set is that the two Jordan curves image of the boundary of some face (therefore, of any one) of each graph from the face  embeddings are isotopic relative to the critical value set. to the Jordan curve bounding the image of the embedding of some face of the other graph relative to the image of the image of the vertices by the embeddings.

The isotopy hypothesis guarantees the existence of a homeomorphim $\phi:\s^2\rightarrow\s^2$ compatible with the face embeddings, i.e.,
\begin{eqnarray}
h(I(\myov{F}))=\phi(g(\myov{F}))
\end{eqnarray} for each face of $\mathcal{G}$.

So, Proposition $\ref{prop2}$ gives a homeomorphism $\mathfrak{I}$ such that
\begin{eqnarray}
h\circ{}\mathfrak{I}=\phi\circ{}g
\end{eqnarray}

Let $G:=\mu_g\circ{}g\circ{}\nu_{g}^{-1}$ and $H:=\mu_h\circ{}h\circ{}\nu_{h}^{-1}$ be those two rational functions as anounced, where $\nu_g, \mu_g, \nu_h , \mu_h$ the uniformizing maps of the domain and codomain of the topological branched coverings $g$ and $h$ constructed from $\mathcal{G}$ and $\mathcal{H}$ (as in $\ref{bcconstruction}$). 

Since $\mathfrak{I}$ and $\phi$ fix the distinguished corners $\alpha$,$\beta$, and $\gamma$ (the normalization), and $\nu_g(0)=\nu_h(0)=\alpha$,$\nu_g(1)=\nu_h(1)=\beta$, $\nu_g(\infty)=\nu_h(\infty)=\gamma$, $\mu_g(\alpha)=\mu_h(\alpha)=0$, $\mu_g(\beta)=\mu_h(\beta)=1$ and $\mu_g(\gamma)=\mu_h(\gamma)=\infty$, follows that $\nu_h\circ{}\mathfrak{I}\circ{}\nu_g^{-1}=Id_{\CC}$ and $\mu_h\circ{}\phi\circ{}\mu_g^{-1}=Id_{\CC}=Id_{\CC}$ as they are conformal automorphisms of $\CC$ that fixes three points.

Therefore,
\begin{eqnarray}
H &=& \mu_{h}\circ{}h\circ{}\nu_{h}^{-1}\\ \nonumber
&=& \mu_{h}\circ{}h\circ{}\mathfrak{I}\circ{}\nu_{g}^{-1}\\ \nonumber
&=& \mu_{h}\circ{}\phi\circ{}g \circ{}\nu_{g}^{-1}\\ \nonumber
&=& \mu_{g}\circ{}g\circ \nu_{g}^{-1}\\ \nonumber
&=& G\\ \nonumber
\end{eqnarray}

\begin{center}
\begin{tikzcd}[row sep=large, column sep = large]
\cc \arrow{r}{\nu_g}\arrow{d}{Id_{\cc}} & (\s^2 ,\mathcal{G})  \arrow[dashrightarrow, xshift=-1.4ex]{d}{\mathfrak{I}}  \arrow{r}{g} \arrow[xshift=1.4ex]{d}{I} &(\s^2 ,g(C))\arrow{r}{\mu_g} \arrow{d}{\phi} &\cc \arrow{d}{Id_{\cc}}  \\
\cc \arrow{r}{\nu_h} & (\s^2 ,\mathcal{H})\arrow{r}{h}&(\s^2 ,h(C))\arrow{r}{\mu_h}&\cc
\end{tikzcd}\end{center}
\end{proof}
\subsection{A special case: Real Functions from diagrams.}\label{realcase?}

Now, we will focus on a special class of 
admissible graphs. We will consider those planar admissible graphs with an additional structure: as embedded graph into $\overline{\C}$ with vertices into $\overline{\R}$, each face incident to the real line $\overline{\R}$ and
with the set of faces being invariant by the complex conjugation, 
$\overline{{\hspace{.1em}}^{\hspace{.35em}}}:z=x+iy\mapsto\overline{z}:=x-iy$. Or more generally, we are now considering those graphs that are embeddable into $\cc$ and are (ambient) isotopic to one planar graph that enjoy the properties described above.

We will refer to these planar graphs by admissible real graphs\label{real-adm-g} and the underlined embedded graph to it will be called a globally balanced real graph or, for short, by a real \emph{GB-graph}.

\begin{figure}[!h]{
    \centering
       \subfloat[degree $3$ GB-graph]
    {{\includegraphics[width=4.6cm]{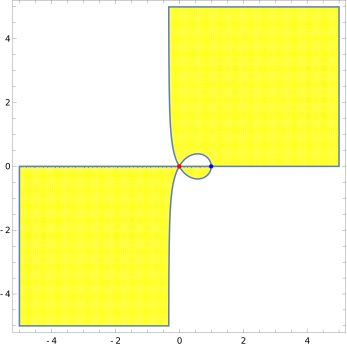} }}
        \qquad
    \subfloat[degree $3$ GB-graph]
    {{\includegraphics[width=4.6cm]{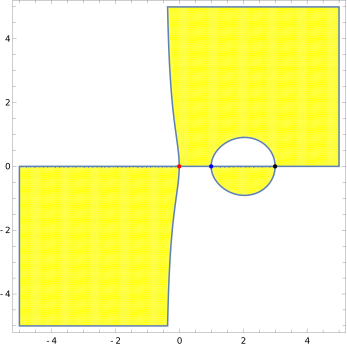} }}
    \qquad
    \subfloat[degree $4$ GB-graph]
    {{\includegraphics[width=4.6cm]{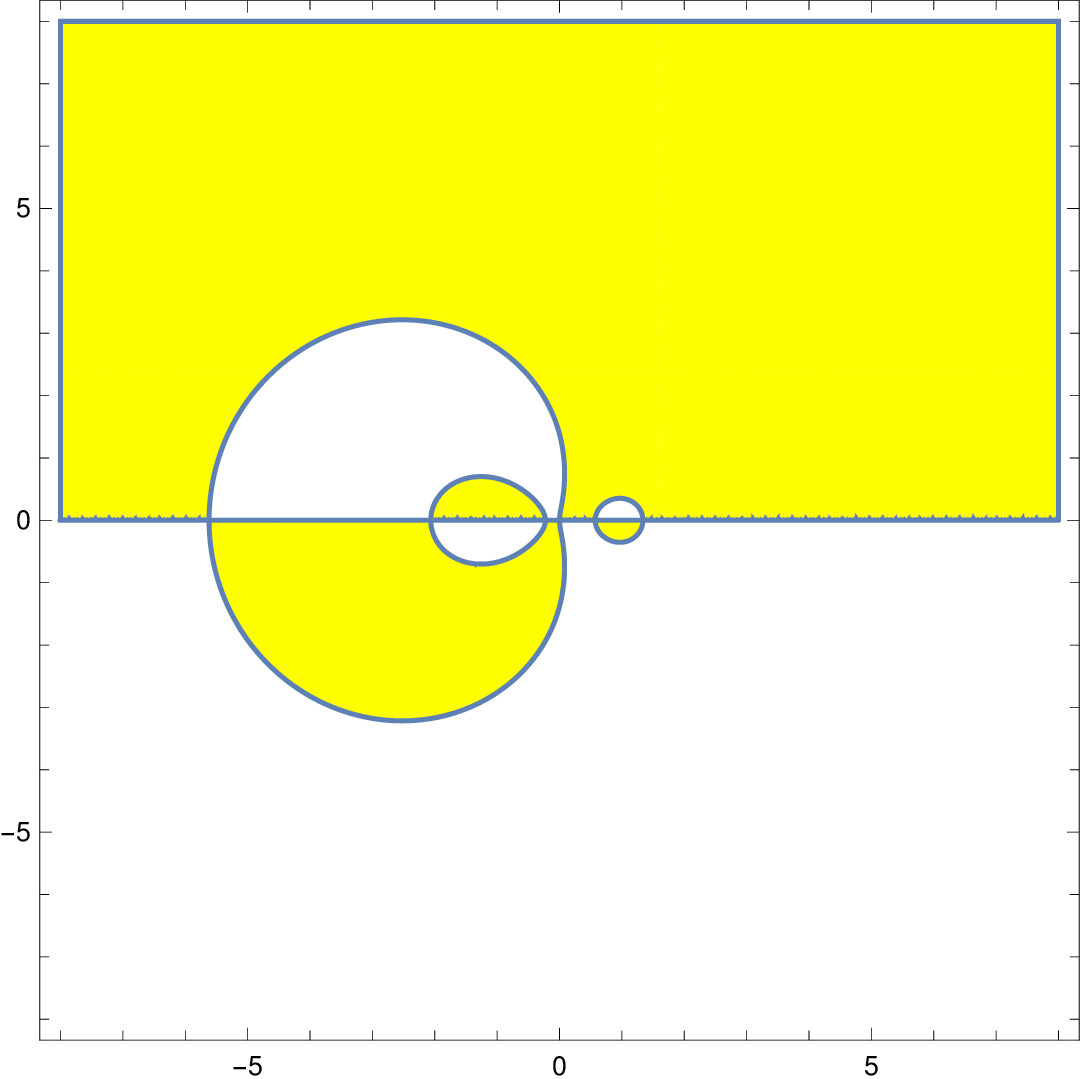} }}\qquad
    \subfloat[degree $6$ GB-graph]
    {{\includegraphics[width=4.6cm]{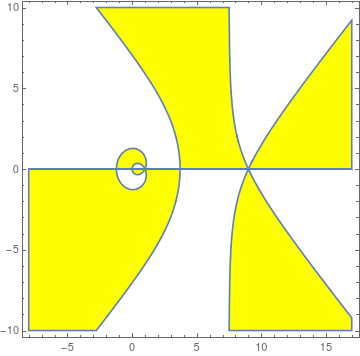} }}%
    \qquad
     \subfloat[degree $7$ GB-graph]
    {{\includegraphics[width=4.6cm]{imagem/2COLORS-PBG-DEG10H.jpeg} }}
        \qquad
    \subfloat[degree $7$ polynomial real GB-graphs]
    {{\includegraphics[width=4.6cm]{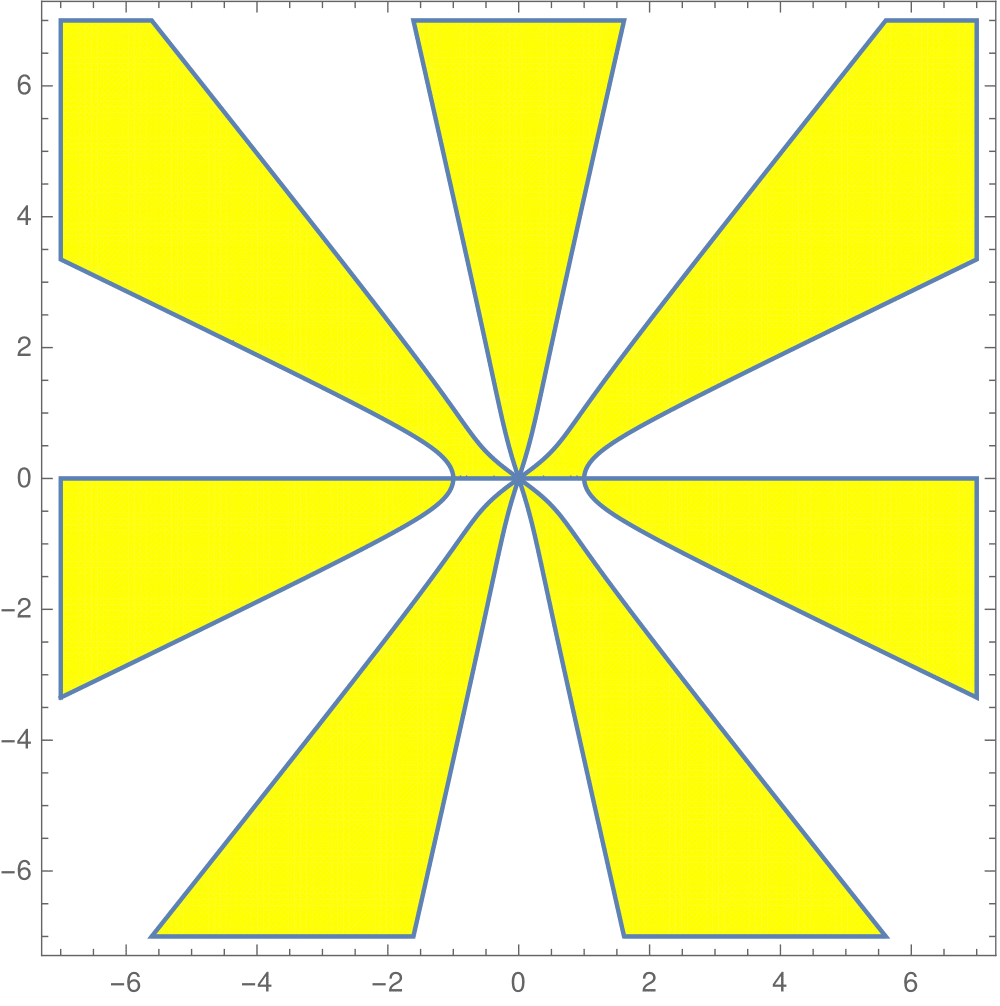} }}%
    \caption{real GB-graphs}%
    \label{fig:h12pullb1}%
    }
\end{figure}

Let $G\subset\overline{\C}$ be a degree $d$ real admissible GB-graph. Then, we can perform for such a map the procedure  described in $\ref{bcconstruction}$, 
then producing a finite degree branched cover $f:(\s^2,C)\rightarrow(\s^2 ,f(C))$. Furthermore, we can performe the embeddings $\iota_k 's$ in accordance with the symmetry of the graphs asking to $\overline{\iota_k (z)}=\iota_l (\bar{z})$ for all $z$ on the $1$-squeleton of the embedded graph, where $\iota_k$ and $\iota_l$ are embeddings of two complex conjugated closed $2$-cells of $G$ that have the point $z$ at its boundary. 

We endow the target space $(\s^2 ,f(C))$ with a complex structure $\mu:\s^2\rightarrow\overline{\C}$ that identify $\Sigma_f :=\iota_1(\partial{G_1})$ with $\overline{\R}$. Thus we pulled back that complex structure on the codomain to the domain 2-sphere getting a new complex structure $\nu:\s^2\rightarrow\overline{\C}$. Therefore, we obtain a holomorphic function $F:=\mu\circ{}f\circ \nu^{-1}:\overline{\C}\rightarrow\overline{\C}$ that satisfies the functional equation $\overline{F(z)}=F(\bar{z})$ over the $1$-squeleton of $G$,  then by the \emph{Identity Principle} (vide\,\cite{LVAhl},\cite{GamAC}) for holomorphic mappings $\overline{F(z)}=F(\bar{z})$ on $\overline{\C}$.

Hence we have

\begin{prop}\label{realbcfromg2}
For each real admissible graph $G$ there exist a holomorphic branched cover $\cc\rightarrow{\cc}$\newline\emph{(} \emph{i. e.}, a rational function\emph{)}, having $G$ as a pullback graph and satisfying the identity 
\[\overline{F_{\mu}(z)}=F_{\mu}(\bar{z})\]
for all $z\in\cc$.
\end{prop}

\begin{figure}[H]{
  \begin{center}
    \subfloat[admissible gaph]
    {{\includegraphics[width=6.5cm]{imagem/realguide-ex-01.png}}}    
    \quad\subfloat[canonical postcritical curve$=\R$]
    {{\includegraphics[width=6.5cm]{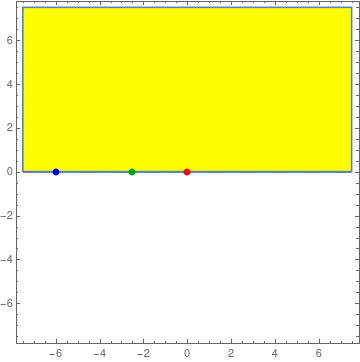} }}\end{center}
    \caption{real Admissible Graph}%
    \label{fig:h12pullb1}%
}\end{figure}

\begin{lem}\label{lem_relfunc}
A meromorphic function $F:\overline{\C}\rightarrow\overline{\C}$ satisfying for all $z\in\overline{\C}$ the identity $\overline{F_{}(z)}=F(\bar{z})$ is a quotient of two polynomials with real coefficients.
\end{lem}

\begin{proof}
First, notice that for a non-constante rational fraction $f\in\C(z)$, the new one $F(z)=\overline{f(\overline{z})}$, where $\overline{*}$ denotes the complex conjugation, is obtained by taking simply the conjugates of the coeficcients of $f$.
\begin{lem}
Given pollynomials $A, B, C, D\in \C[z]$ such that 
\[\dfrac{A}{B}=\dfrac{C}{D} \in \C(z)-\C\],
than there are $k\in \C-\{0\}$ such that 
\[A=k\cdot C\quad\text{and}\quad B=k\cdot{}D\]. 
\end{lem}
\begin{proof}[proof of the lemma]
The zeros and poles of $f(z):=\dfrac{A}{B}$ and $g(z):=\dfrac{C}{D}$ are the same and with the same multiplicite since they are the local degree of the two maps $f$ and $g$.  

Hence ${A}$ and ${B}$ as $C$ and $D$ has the same zeros with the same multiplicite, then \[A=k_1\cdot{}C\quad\text{and}\quad B=k_2\cdot D\] for some $k_1 , k_2 \in\C-\{0\}$. But, $\dfrac{A}{B}=\dfrac{C}{D}$, thus, $k_1 = k_2$.
\end{proof}

By multiplying the numerator and denominator by a suitable non-zero constant we can assume $Q(z)$ monic in the fraction $\dfrac{P}{Q}$.

Now, $\dfrac{P}{Q}=\dfrac{\overline{P}}{\overline{Q}}$ implies ${P(z)}={\overline{P(z)}}$ and ${Q(z)}={\overline{Q(z)}}$, therefore, $P, Q\in\R[z]$.

\end{proof}

Therefore,
\begin{cor}\label{r_a_gb_r_r_f}
For each admissible real planar graph $G$(\emph{i. e.}, $G$ is real planar GB-graph with an admissible vertex labeling) there exist a real rational function having $G$ as a pullback graph for the (canonical) postcritical curve $\R$.
\end{cor}
\begin{proof}
Follows straightforwardly from $\ref{realbcfromg2}$ and $\ref{lem_relfunc}$.
\end{proof}

Now, we draw attention to the fact that a given real rational function 
can have a non-real pullback graph. 
For a given rational function, the pullback graph depends on the isotopy type of 
the chosen post-critical curve. 
Here goes some examples:

\begin{ex}\label{isotopy-real-map}
Some differents post-critical curves for $f(z)=\tiny{\frac{\frac{1}{2} \left(3-\sqrt{7}\right) z^3+\left(\sqrt{7}-2\right) z^2}{\left(\frac{1}{2} \left(\sqrt{7}-3\right)+2\right) z-1}}$ and its respectives pullback graphs. The critical points of $f$ are $-2,0,1$ {and} $\infty$.

\begin{figure}[H]
    \begin{center}
       \subfloat[$\R$: post-critcal cuve for $f$]
    {{\includegraphics[width=4.3cm]{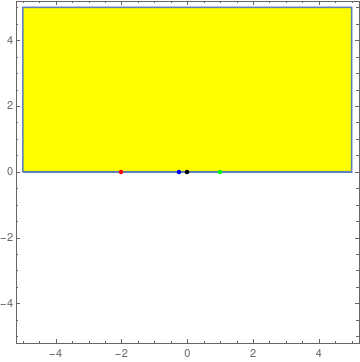} }}
        \qquad
    \subfloat[pullback graph $\Gamma(f,\R)$]
    {{\includegraphics[width=4.6cm]{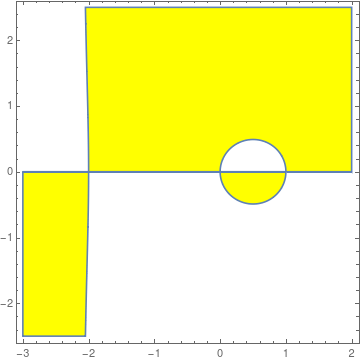} }}\end{center}
      \label{fig:h12pullb01}	
\end{figure}
    
    \begin{figure}[H]
    \begin{center}
    \subfloat[$\Gamma(f,\Sigma_1)$: post-critcal cuve for $f$]
    {{\includegraphics[width=4.6cm]{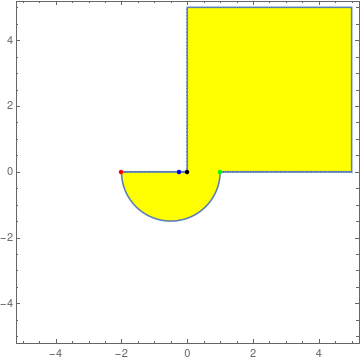} }}\qquad
    \subfloat[pullback graph $\Gamma(f,\Sigma_1)$]
    {{\includegraphics[width=4.3cm]{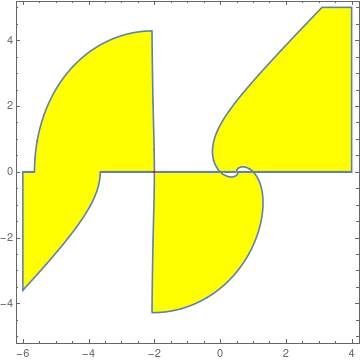} }}\end{center}%
    
    \begin{center}
     \subfloat[$\Gamma(f,\Sigma_2)$: post-critcal cuve for $f$]
    {{\includegraphics[width=4.6cm]{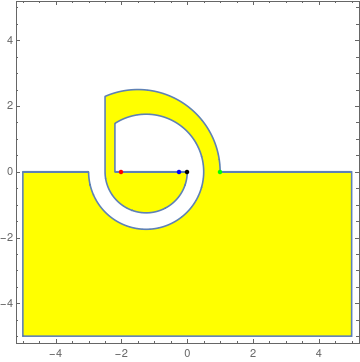} }}
        \qquad
    \subfloat[pullback graph $\Gamma(f,\Sigma_2)$]
    {{\includegraphics[width=4.6cm]{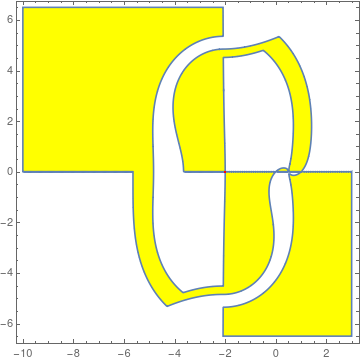} }}%
    \caption{real GB-graphs}%
    \label{fig:h12pullb1}%
    \end{center}	
\end{figure}
\end{ex}

In the next section, we will achieve a full generalization of a theorem by Thurston proved firstly for generic branched self-coverings of the $2$-sphere.
\section{General version of a theorem by Thurston}\label{sect-thurstonthm}
\begin{defn}
A \emph{simple closed curve} $\gamma$ into a surface $S$ 
 is \emph{separating} if $S-\gamma$ has two components. Otherwise, $\gamma$ is \emph{non separating}.
\end{defn}

\begin{defn}\label{set-upththm}
Let $\Gamma$ be an oriented globally balanced graph on an oriented compact surface $S$ 
that admits an alternate \textcolor{deeppink}{A}-\textcolor{blue}{B} face coloring such that the $A$ faces are kept on the left side of the edges of $\Gamma$ regarding the orientation. We say that the color \textcolor{deeppink}{A} is the \emph{preferred} one of that alternate face coloring.

Each cycle into $\Gamma$ (i.e., a concatenation of edges of $\Gamma$ that forms a simple closed curve) that keeps only \textcolor{deeppink}{A} faces on its left side is then said to be a \emph{positive cycle}.

If $\gamma$ is a positive separating cycle of $\Gamma$ we will call by the \emph{interior of} $\gamma$ the component of $S-\gamma$ that contains those \textcolor{deeppink}{A} faces incident to $\gamma$.
\end{defn}

\begin{defn}[cobordant cycles]\label{sub-surfs}
Let $\Gamma\subset S$ 
be a globally balanced graph.
We say that a collection of disjoint cycles $L:=\{\gamma_1, \cdots, \gamma_k\}$ of $\Gamma$ are cobordant if:
\begin{itemize}
\item[i.]{$S-\{\gamma_1, \cdots , \gamma_k\}$ is disconnected;}
\item[ii.]{there is a connected component $R$ of $S-\{\gamma_1, \cdots , \gamma_k\}$ such that $\partial R = \bigsqcup_{n=1}^{k}\gamma_n$.}
\end{itemize}
We will reffers to a such collection $L=\{\gamma_1, \cdots, \gamma_k\}$ as a \emph{cobordant multicycle of} $\Gamma$. If each cycle $\gamma_n \in L$ is positive, then we call $L$, \emph{positive cobordant multicycle of} $\Gamma$.

$R$ is called the \emph{interior of} $L$.
\end{defn}

\begin{defn}[local balancedness]\label{localbalance}
Let $\Gamma$ be globally balanced graph with an alternating \textcolor{deeppink}{A}-\textcolor{blue}{B} face coloring.
We say that $\Gamma$ is \emph{locally balanced} if for any \emph{positive cobordant multicycle of} $\Gamma$ the number of $\textcolor{deeppink}{A}$ faces inside it (i.e, on the interior of that multicycle) is strictly greater than the number of $\textcolor{Blue}{B}$ faces.
\end{defn}

That definition of the \emph{local balance condition} is a generalization of the former one introduced by \emph{Thurston}\\\cite{STL:15}.  Although in the planar situation Definition $\ref{localbalance}$ it seems more restrictive than the one given by Thurston, they are actualy equivalent. To show that, let us first presents the definition settled by \emph{Thurston}:

\begin{defn}[planar local balance condition from Thurston]\label{thurstonlb}
A planar globally balanced graph $\Gamma$ with an alternating \textcolor{deeppink}{A}-\textcolor{blue}{B} face coloring is \emph{locally balanced} if for every positive cycle of $\Gamma$ the number of $\textcolor{deeppink}{A}$ faces inside it, is strictly greater than the number of $\textcolor{Blue}{B}$ faces.
\end{defn}

\begin{prop}[meaningfullness of Definiton $\ref{localbalance}$]
For \emph{Planar} globally balanced graphs those two definitions of local balancedness are equivalents.
\end{prop}
\begin{proof}
Thanks to the \emph{Jordans Theorem} is immediate that Definition $\ref{localbalance}$ implies the Definition $\ref{thurstonlb}$.

So, let's prove the reverse implication. That is, we will guarantee that if a planar balanced graph that satifies the Definition $\ref{thurstonlb}$ then it also enjoys the Definition $\ref{localbalance}$. 

Let $\Gamma$ be a planar globally balanced graph with an alternating \textcolor{deeppink}{A}-\textcolor{blue}{B} face coloring and $\Lambda$ be a \emph{cobordant positive multicycle of} $\Gamma$ with interior $R$.

Let $Y$ be a connected component of $\s^2 -R$. Since $R$ is connected the boundary of $Y$ has only one component $\gamma\in L$. 

Thus $\gamma$ encloses the complement of $Y$ leaving \textcolor{deeppink}{A} faces on its left side.

Hence, from the \emph{local balance} condition we conclude that are more pink faces than blue ones outside $Y$.

Let $Y_1 , Y_2, \cdots, Y_n$ be the components of $\s^2 - R$, and $\textcolor{deeppink}{a_{k}}$ and $\textcolor{blue}{b_{k}}$ the number of \textcolor{deeppink}{A} faces into $Y_k$ and the number of \textcolor{blue}{B} faces into $Y_k$, respectively. $\textcolor{deeppink}{a_{R}}$ and $\textcolor{blue}{b_{R}}$ are the numbers of \textcolor{deeppink}{A} faces and \textcolor{blue}{B} faces into $R$.

Hence, from the above argumentation
\begin{eqnarray}
\textcolor{deeppink}{a_{k}}<\textcolor{blue}{b_{k}}
\end{eqnarray}for each $k=1,2, \cdots , n$

And, since,
\begin{eqnarray}
\textcolor{deeppink}{a_R} + \sum_{k=1}^{n}\textcolor{deeppink}{a_{k}}=\textcolor{blue}{b_R} + \sum_{k=1}^{n}\textcolor{blue}{b_{k}}
\end{eqnarray}

Then,

\begin{eqnarray}
\textcolor{deeppink}{a_R} >\textcolor{blue}{b_R}
\end{eqnarray}

\end{proof}

\begin{defn}[Balanced Graph]\label{balance}
A \emph{balanced graph} is an oriented cellularly embedded graph $\Gamma$ into an oriented compact surface that it is both, global  and locally balanced. The \emph{type} of a balanced graph is its type as a globally balanced graph. 
\end{defn}

\begin{figure}[H]
\begin{center}
\subfloat[balanced graph of type (0,4,6)]
{{\includegraphics[width=4cm,height=4cm]{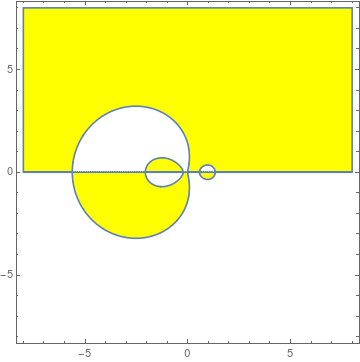}}}\quad
\subfloat[balanced graph of type (1,4,4)]
{{\includegraphics[width=6.5cm,height=4.cm]{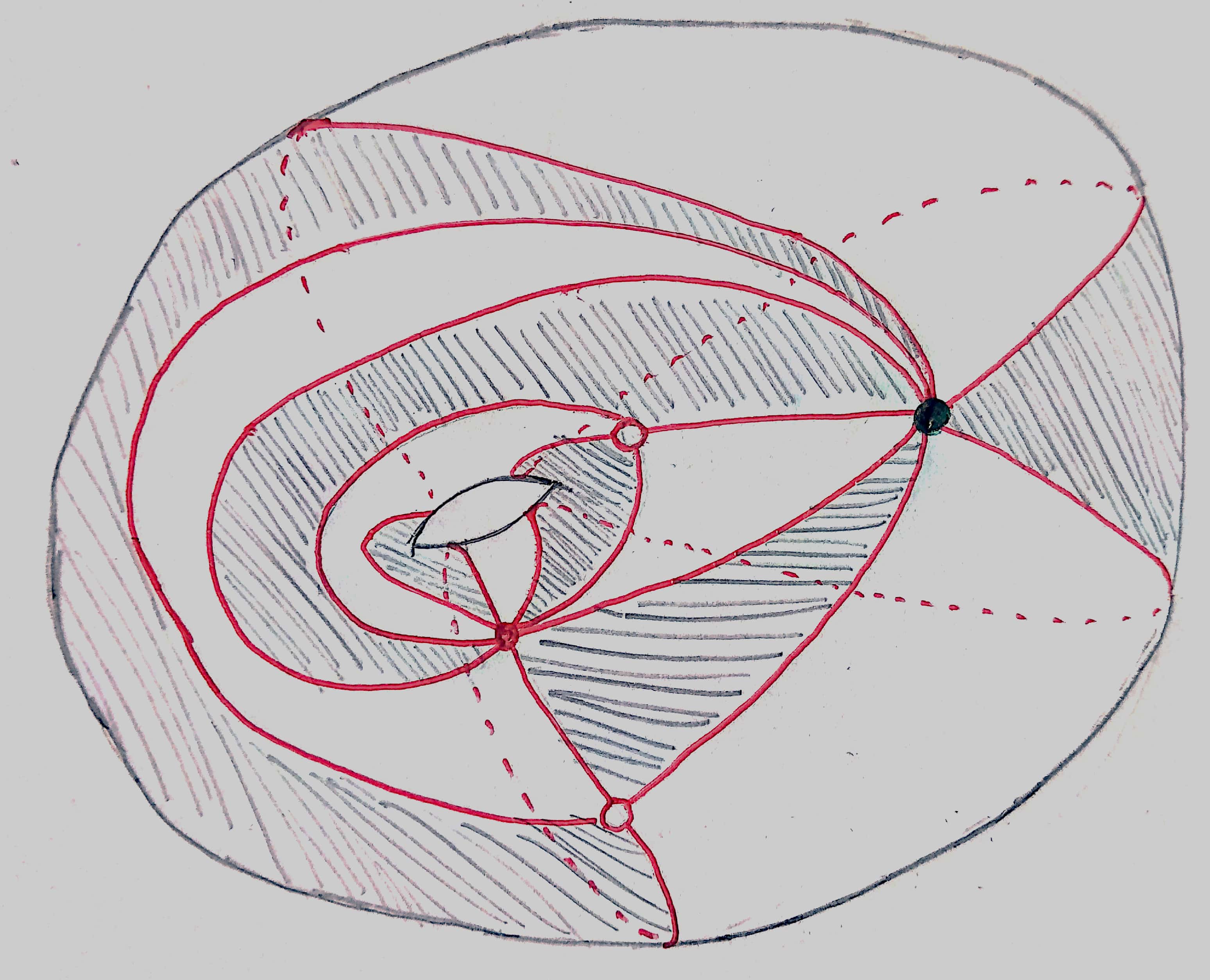}}}
\label{bg-000}
\end{center}
\end{figure}

\begin{figure}[H]
\begin{center}
\subfloat[Globally Blalanced graph of type $(2, 4, 6)$]
{{\includegraphics[width=8cm,height=4.75cm]{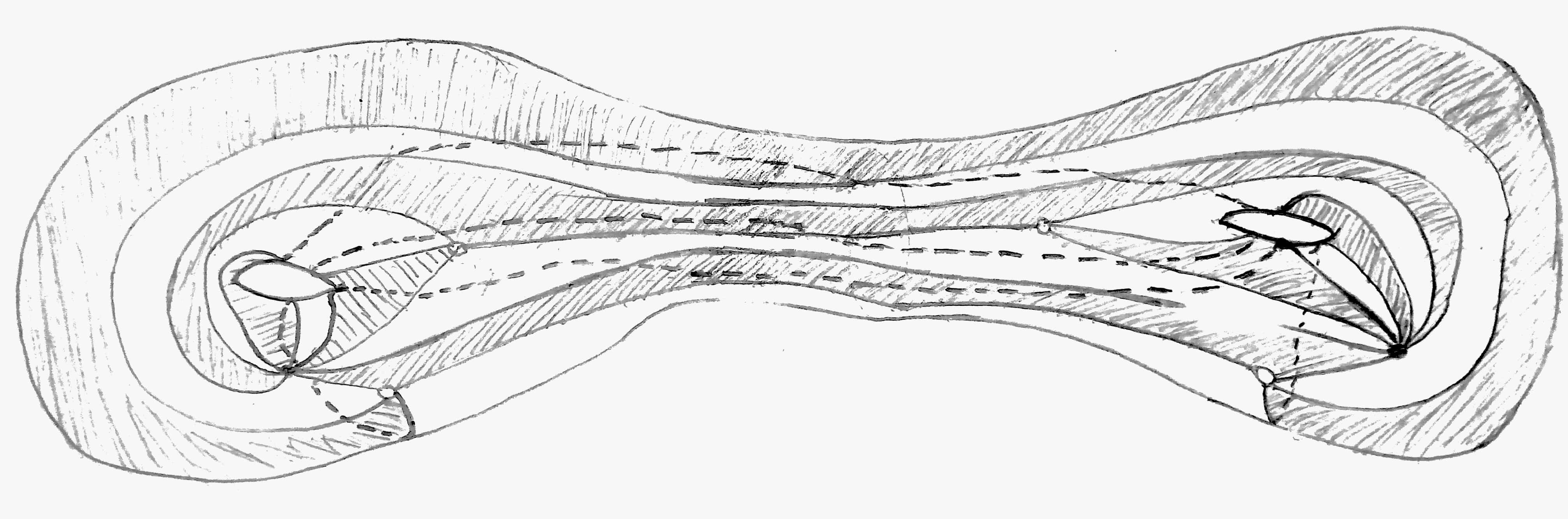} }}
\label{bg-001}
\end{center}
\end{figure}

\begin{thm}[\textcolor{DarkGreen}{General version of a theorem by Thurston}]\label{THURSTHG1}
An oriented cellularly embedded graph $\Gamma$ into a genus $g$ oriented compact surface $S_g$ is a pullback graph if and only if it is a balanced graph.
\end{thm}

\begin{proof}
We will follow closely the initial proof given by Thurston\cite{STL:15}.

The gist of the proof 
is:
\begin{itemize}
\item[i.]{to translate the realization problem into
finding a pattern of vertices, including those $2$-valents ones, so that each face of $\Gamma$ must have 
the same number of vertices on its boundary;
\begin{itemize}
\item[$\bullet$]{this pattern is the one that admissible and pullback graphs present (see $\ref{emb-cell-g}$).  }
\end{itemize}}
\item[ii.]\label{item-ii}{then to reduce the problem to a \emph{matching problem} in graph theory in the follow way:\\
Let $f:X\rightarrow\s^2$ be a degree $d$ branched cover with $m$ critical values. Consider $\Sigma$, a post-critical curve for $f$, and let ${\Gamma}$ be the corresponding pullback graph $\ref{pullbackg}$. 

For each $2$-valent vertex we mark a dot into those two face of ${\Gamma}$ incident to it. 
 Thus, each $2$-valent vertex will have two marked dots corresponding to it into each one of its two neighboring faces. 
Since the boundary of each face contains exactly $m$ vertices in its boundary, after we did that, each face of $\Gamma$ will contain $m-e$ dots, where $e$ is the number of corners (i.e., vertices of degree $k>2$). Each dot corresponds through $f$ to a different critical value.

Now, to group into pairs those dots from adjacents faces back together forming vertices for each pair then becomes a graph theoretic \emph{matching problem}.}
\end{itemize}

More precisely, what we are doing is constructing an adjacent bipartite graph $G:=G(\Gamma)$ from the given pullback graph $\Gamma$.
Let $A\subset \s^2$ and $B\subset \s^2$  be the two connected components of $\s^2 - \Sigma$.
The graph $G$ then is the pair $(D=DA\sqcup DB, E)$ where $DX$ is the set of dots from those faces whose image by $f$ is $X\in\{A, B\}$ and $E$ is the adjacency relation from $\Gamma$ in the sense that a vertex $u\in DA$ is connected to a vertex $v\in DB$ by an edge $e\in E$ only if they belong to adjacent faces of $\Gamma$.

Then, that spliting procedure of the $2$-valent vertices described at item $\textrm{ii.}$ above $\pageref{item-ii}$, provides a \emph{perfect matching} on the graph $G$.

On the other hand, if we have a balanced graph $\Gamma$ we can also construct that adjacent graph $G$ inserting $m-e_F$ dots into each face $F$ of $\Gamma$ being $m$ the number of corners of $\Gamma$ and $e_f$ the number of corners incident to $F$. The vertex set of $G$ in this case is partitioned into two subsets whit respect to the face coloring of the balanced graph $\Gamma$.

Thus, now the existence of a \emph{perfect matching} on $G$ will allow us to enrich $\Gamma$  to a new graph, that we will continue to denote by $\Gamma$,  that it has $m$ vertices incident to each face. As described in item \textrm{ii.} each new vertex of degree $2$ arouses from each pair of vertices of $G$ matched.

Thus we ask:


\begin{quest} \begin{center}Is there a perfect matching for those dots?\end{center} \end{quest}

But notice that such a \emph{matching} must also admit a nice vertex labeling as described above in $\ref{emb-cell-g}$ (see also Definition $\ref{def:labeling}$).

\paragraph{Let's prove the \underline{\emph{\texttt{if}}} part:}
Let $\Gamma\subset X$ be a pullback graph on the compact oriented surface $X$ with post-critical curve $\Sigma$.

From Proposition $\ref{prop-shape-graph}$ follows that the faces are \emph{Jordan} domain's .

Color by \textcolor{deeppink}{pink} the interior of $\Sigma \subset\s^2$ and call it by $P$ and color by \textcolor{blue}{blue} the another component of $\s^2-\Sigma$ and call it by $B$. 

Each point $p\in P$ and $b\in B$ possesses exactly $d$ distinct preimages in $X-\Gamma$, since all critical values are on $\Sigma$. Due to the continuity of $f$ a preimage $\tilde{p}\in f^{-1}(p)$ and $\tilde{b}\in f^{-1}(b)$ can not be in the same face of $\Gamma$, say $F$, for otherwise, we could connect $\tilde{p}$ and $\tilde{b}$ by a curve $\gamma$ into $F$ and in this way $f(\gamma)$ will be a connected set connecting $p\in P$ to $b\in B$ but being interelly contained into $f(F)$ that is equal to $P$ or $B$, what is certantily impossible, since  $P$ and $B$ are disjoint open set. Since $f:X-\Gamma\rightarrow \s^2-\Sigma$ is a local homeomorphism, we also can not have $\tilde{p_0},\tilde{p_1}\in f^{-1}(p)$ into the same face (recall the lifting property of local homeomorphisms). The same, for sure, works for that points over $b$. Therefore, there are $d$ faces of $\Gamma$ colored pink and $d$ faces of $\Gamma$ colored blue. This means that $\Gamma$ is \emph{globally balanced}.


Let $L=\{\gamma_1 , \cdots , \gamma_k\}$ be a cobordant positive multicycle of $\Gamma$ with interior $R$.

Let:
\begin{itemize}
\item[(1)]{$E_{n}>0$ to be the number of corners of $\Gamma$ in $\gamma_n$ that do not are incidente to blue faces inside $\gamma_n$, for each $n\in\{1, \cdots , k\}$;}
\item[(2)]{$a$ be the number of pink faces in $R$;}
\item[(3)]{$b$ be the number of blue faces in $R$;}
\item[(4)]{$D_A$ be the number of dots into those pink faces in $R$;}
\item[(5)]{$D_B$ be the number of dots into those blue faces in $R$;}
\end{itemize}

Then:

\begin{itemize}
\item[(1)]{
since the number of edges $e_j$ bordering a face $B_j$ is equal to the number of corners on its boundary, it follows that
\begin{eqnarray}\label{eqI}
D_B = (m-e_1)+(m-e_2)+\cdots+(m-e_B)=mb-(\sum_{j=1}^{b}e_j)
\end{eqnarray}}
\item[(2)]{and
\begin{eqnarray}\label{eqII} 
D_A = ma-(\sum_{j=1}^{b}e_j)-\sum_{n=1}^{k}E_{n}
\end{eqnarray}}
\end{itemize}

Suppose 
$E_{n}=0$, for each $n\in\{1,\cdots ,k\}$.
Then each connected component of $X-R$ is a \emph{simply connected} domain. This stems from the fact that $\Gamma$ to be conected and $E_n=0$ to imply that each positive cycle $\gamma_{n}$ to be incident to only one blue face outside $R$. Therefore, each component of $X-R$ is a blue face and since $\Gamma$ has so many blue as pink faces, say $d>0$, it follows:

\begin{eqnarray}
b=d-k<d=a
\end{eqnarray}

Now, suppose $E_n >0$ for at least one $n\in\{1\cdots , k\}$.
Then, by the necessary condition from the \emph{marriage theorem} $\ref{ksamento}$ we have:

\begin{align}\nonumber
mb-(\sum_{j=1}^{b}e_j)=D_B &\leq  D_A = ma-(\sum_{j=1}^{b}e_j)-\sum_{n=1}^{k}E_{n} \\ \nonumber 
\shortintertext{since $\sum_{n=1}^{k} E_{n} \geq 1$ ,}
 b&<a 
\end{align}

Thus, $\Gamma$ is \emph{locally balanced}.
\paragraph{Now, let's prove the \underline{\texttt{\emph{only if}}}\, part of our statement $\ref{THURSTHG1}$ :}\hspace{6cm}

Let $\Gamma$ be a balanced graph with $m$ corners.

Since each face $F$ of $\Gamma$ is a Jordan domain the number of saddle-connections of $\Gamma$ surrounding $F$ is equal to the number of corners on $\partial{F}\subset\Gamma$. 

Recall that each face $F$ of $\Gamma$ contains $m-e_F$ dots, where $e_F$ is the number of corners incident to $F$.

Let $S$ be an arbitrary set of dots from blue faces of $\Gamma$.

Then the task is: \emph{to show that the set of potential mates for $S$ is at least so large as $S$} (that is the sufficient condition of the \emph{Hall's marriage theorem} $\ref{ksamento}$).

Note that the potential mates for a dot into a blue face is exactly the same set of potential mates for any other dot from the same face. Therefore, we can change $S$ adding to it all the remains dots in a face that already has at least one of its dots in $S$. That change will not affect the number of potential mates and, of course, the condition is satisfied for any subset of dots from that enlarged set $S$ whether it itself satisfies the condition. Therefore, due to that, we will take $S$ as being the subset of all dots from a collection $ U $ of blue faces of $\Gamma$.

Denote by $R$ the topological closure of the collection $U$ together with its neighboring pink faces, i.e., $R$ is the union of the faces in $U$ with its neighboring pink faces and all boundaries of those faces.

Then the dots inside pink faces in $R$ are exactly those potential mates for the dots into $S$.

Note that the boundary of $R$ leaves pink faces in its left side, except at the corners.

If the interior of $R$ is not connected, then dots into blue faces of one component can only be matched with those dots inside pink faces from the same connected component of the interior of $R$. Hence we should have enough mates for the individuals of $S$ in each connected component of the interior of $R$. In this way we will have enough mates in $R$ for all individuals. Then is enough to assure the condition for each connected component what allows us to consider $R$ with the interior connected.

Let:
\begin{itemize}
\item[(1)]{$D_A$ denote the number of dots into pink faces inside $R$;}
\item[(2)]{$D_B=|S|$ denote the number of dots into blue faces inside $R$;}
\item[(3)]{$E_1$ be the number of corners on $\partial{R}$ that have only one face from $R$ neighboring it(that number was the number $E_\gamma$ when we prove the local balance condition of a pullback graph above);}
\item[(4)]{$\mu_k$ be the number of corners on $\partial{R}$ that have $k$ blue faces incident to it  from $R$ ;}
\item[(5)]{$\nu_k$ be the number of corners in the interior of $R$ with degree $k$;}
\item[(6)]{$a$ be the number of pink faces in $R$;}
\item[(7)]{$b$ be the number of blue faces in $R$.}
\end{itemize}



Thus, going back to the equations $\ref{eqI}$ and $\ref{eqII}$ we have:
\begin{eqnarray}\nonumber
D_B = mb-\sum_{j=1}^{b} e_j &=& mb-\frac{1}{2}\left(\sum_{j=2}^{d}2\mu_{j} + \sum_{j=2}^{d}{2j}\nu_{2j}\right)\\
&=& mb-\left(\sum_{j=1}^{d}\mu_{j}+\sum_{j=2}^{d}{j}\nu_{2j}\right)
\end{eqnarray}
and
\begin{eqnarray}\nonumber
D_A &=& ma-\frac{1}{2}\left(2E_{1}+\sum_{j=1}^{d}2\mu_{j} + \sum_{j=2}^{d}{2j}\nu_{2j}\right)\\
&=& ma-\left(E_{1}+ \sum_{j=1}^{d}\mu_{j} + \sum_{j=2}^{d}{j}\nu_{2j}\right)
\end{eqnarray} 

From the local balance condition we have $b \leq a-1$, and we also have $E_1 + \sum_{j=1}^{d}\mu_{j} + \sum_{j=2}^{d}\nu_{2j}\leq m$ where $m$ is the total number of corners of $\Gamma$. 

Hence
\begin{eqnarray}
D_A &=& ma-\left(E_{1}+ \sum_{j=1}^{d}\mu_{j} + \sum_{j=2}^{d}{j}\nu_{2j}\right)\\
&\geq& ma-\left(m+ \sum_{j=1}^{d}\mu_{j} +\sum_{j=2}^{d}{j}\nu_{2j}\right)\nonumber\\
&=&m(a-1)-\left(\sum_{j=1}^{d}\mu_{j} + \sum_{j=2}^{d}{j}\nu_{2j}\right)\nonumber\\
&\geq &  mb-\left(\sum_{j=1}^{d}\mu_{j}+\sum_{j=2}^{d}{j}\nu_{2j}\right)\nonumber\\
&=&D_B
\end{eqnarray} 

That is the desired inequality.

Therefore, we have proved that for an arbitrary set $S$ of dots from blue faces of $\Gamma$ the set of potential mates for those dots into $S$ is so large as $S$. Then the \emph{Hall's Marriage Theorem $\ref{ksamento}$} with the global balancedness assures the existence of a \emph{perfect matching}.

For each pair of dots matched we get a new vertex on the common side separating the faces containing those dots. These new vertices are taken distinct for each matched pair of dots from the same pair of faces.

Then, $\Gamma$ was enriched into a new graph, now with a bunch of $2$-valent vertices inserted, that we shall continue denoting by $\Gamma$.

But in addition to having $m$ vertices incident to each face, these vertices must be numbered cyclically (regarding the graph orientation) in such a way that the number at a corner given from each face labeling incident to it is the same and, furthermore, with such labeling being in accordance with an admissible passport. With `` to be in accordance with a passport '' we mean that the sum of half the degree of the vertices for a fixed label is equal to the degree $d$ of $ \Gamma $, for each label $ j\in \{1, 2, \cdots , m \}$.

Thus we have to ensure that we can always perform a vertex labeling with that especifications on such a enriched balanced graph. That is, every balanced graph is an admissible graph. Therefore, from Theorem $\ref{bcfromg2}$, we will be done!

\begin{lem}\label{lththm}
The enriched balanced graph obtained above is admissible.
\end{lem}
\begin{proof}[proof of the lemma $\ref{lththm}$]
We must display 
one admissible vertex labeling for $\Gamma$ (the enriched graph). $\Gamma$ has $m$ corners
. We can construct an admissible vertex labeling $N:V(\Gamma)\rightarrow \{1<2<\cdots<m\}$ 
inductively, 
as follows.

First, choose a pink face $F_1 \in F(\Gamma)$ with a numbering of the $m$ vertices incident to it by $1, 2, \cdots, m$ appearing in this order around the face keeping it on the left side. 

For a (labeled) corner adjacent to $F_1$, say $c_1$, we consider all the pink faces incident to it. Then we complete the labeling of the left $m-1$ vertices on each face respecting the already labeled corner $c_1$ incident to it in such a way that the increasing order of the labelings coincide with the positive sence of the orientation. Let $F_{2}$ be a face incident to $c_1$, but also incident to another corner, say $c_2 \in \partial{F_1}$. Since each vertex has to have a unique label assigned to it we must to ensure that the label assigned to the corner $c_2\in \Gamma$ when we label the vertex adjacent to $F_{2}$, as especified above, is equal to the one assigned to it from the label of it as a vertex incident to $F_{1}$. We shall see that this is the case, but for the sake of readability, we will leave the proof of that to the end, and then continuing the argumentation assuming it. 

That procedure stops at some point since we have a finite number of faces, each one with only $m$ vertex adjacent to it. In that way we have constructed a surjective map 
$N:V_{\Gamma}\rightarrow \{1, 2, \cdots, m\}$. And at each blue face the indices $1, 2, \cdots, m$ appears at this order but in reverse sense of the edges orientation (recall that the edges are oriented kepping pink faces on its left side). 

But can occur that one index $k\in\{1, 2, \cdots, m\}$, or actually more than only one, do not be attained by a corner through the map $N$, i. e., so that $N^{-1}(k)$ concists only by $2$ valent vertices of $\Gamma$.

If that was not the case, then $N$ defines an admissible vertex labeling to $\Gamma$ since by construction a label $j$ is assined to only one vertex of each pink face and we have $d$ faces, furthemore, if $e$ is the valence of a vertex with label $j$ there are exacle $\frac{e}{2}$ pink face incident to it.

On the other hand, let 
$M\subset \{1,  2, \cdots, m\}$ with $|M|=p<m$ be the subset of the labelings $k\in \{1,  2, \cdots, m\}$ such that $N^{-1}(k)$ is made up only by $2$-valent vertices.
Then we can erase from the enriched graph all the vertices with label in $M$ and in the sequel to repeat the procedure of the construction of $N$ presented above wth the label set $\{1, 2, \cdots, m-p\}$.
Thus we will get a vertex labeling that tags more than one corner of the graph with the same label, for at least one label into $\{1, 2, \cdots, m-p\}$. For the same reason given above, that labeling is admissible. 

Now, let's prove the part left about the (global) consistency of the procedure presented above to construct a vertex labeling. 

Let $\{e_{j}^{1}\}_{j=1}^{x}\subset E(\Gamma^{\ast})$ and $\{e_{j}^{2}\}_{j=1}^{y}\subset E(\Gamma^{\ast})$ be the sets of edges of the bipartite dual graph $\Gamma^{\ast}$ of $\Gamma$ made up by the edges duals to the saddle-connections adjacents to $F_1$ and $F_{2}$, respectively, that form the positive path into $\Gamma$ connecting $c_1$ to $c_2$.

Thus, we consider the subgraph $G^{\ast}\subset\Gamma^{\ast}$ formed by the collection of paths into $\Gamma^{\ast}$ that possesses the inital edge in $\{e_{j}^{1}\}_{j=1}^{x}\subset E({\Gamma^{\ast}})$ and terminal edge in $\{e_{j}^{2}\}_{j=1}^{y}\subset E({\Gamma^{\ast}})$. $G^{\ast}$ have two sets of distinguished vertices, one is the singleton $I:=\{F_{1}^{\ast}\}$ and the another one is the subset 
$O\subset V(\Gamma^{\ast})$ of the vertices of $\Gamma^{\ast}$ 
duals to those blue faces that are incident to the positive path adjacent to $F_2$ joining $c_1$ and $c$. 

Note that if the cycle $\gamma:=\prod_{j=1}^{x}e_{j}^{1}\cdot{}\prod_{j=1}^{y} e_{j}^{2}$ is a separating curve of the underline surface $X$ such that the component $\Omega\subset X$ that not contains the face $F_1$ is a disk, then the defining condition of $G^{\ast}\subset\Gamma^{\ast}$ is the same that define $G^{\ast}\subset\Gamma^{\ast}$ as the subgraph of $\Gamma^{\ast}$ consisting of its part inside $\Omega$ together the edges dual to the saddle-connections into $\gamma$.

To each edge $e\in E({G^{\ast}})$ of $G^{\ast}\subset\Gamma^{\ast}$ we assign the positive integer $n-1$ where $n$ is the number of vertices over its dual saddle-connection $e^{\ast}\in E({\Gamma})$. Therefore, for each vertex of $G^{\ast}$ not being in $I\subset V(G)$ or $O\subset V(G)$
the sum of the numbers attached to the edges incident to it equals $m$. That is $G$, endowed with the above decribed structure, is a multi-extremal weighted graph with charge $m$. 

Let $\epsilon_{j}^{1}\geq 1$ be the number assigned to the edge $e_{j}^{1}$ for $j\in\{1,2,\cdots, x\}$ and $\epsilon_{j}^{2}\geq 1$ be the number assigned to the edge $e_{k}^{2}$ for $k\in\{1,2,\cdots, y\}$.

If $\sum_{j=1}^{x}\epsilon_{j}^{1}=\sum_{j=1}^{y}\epsilon_{j}^{2}$, the positive paths $\prod_{j=1}^{x}(e_{j}^{1})^{\ast}$ and $\prod_{j=1}^{y} (e_{j}^{2})^{\ast}$ from $c_1$ to $c$ have the same number of vertex on it, therefor the labeling atributed to $c$ by the labeling of the vertices adjacents to $F_{2}$, as described previously, will agree with the one assigned by the labeling of the vertices that are incident to $F_1.$

But, Proposition $\ref{vertex-capacity-fn}$ assure the expected equality between the numbers $\sum_{j=1}^{x}\epsilon_{j}^{1}$ and $\sum_{j=1}^{y}\epsilon_{j}^{2}$, sice they are the input and output values of the multi-extremal weighted graph with constant capacity $m$, $G$.


\begin{figure}[H]
 \begin{center}
 \subfloat[]
{\includegraphics[width=5cm,height=5.2cm]{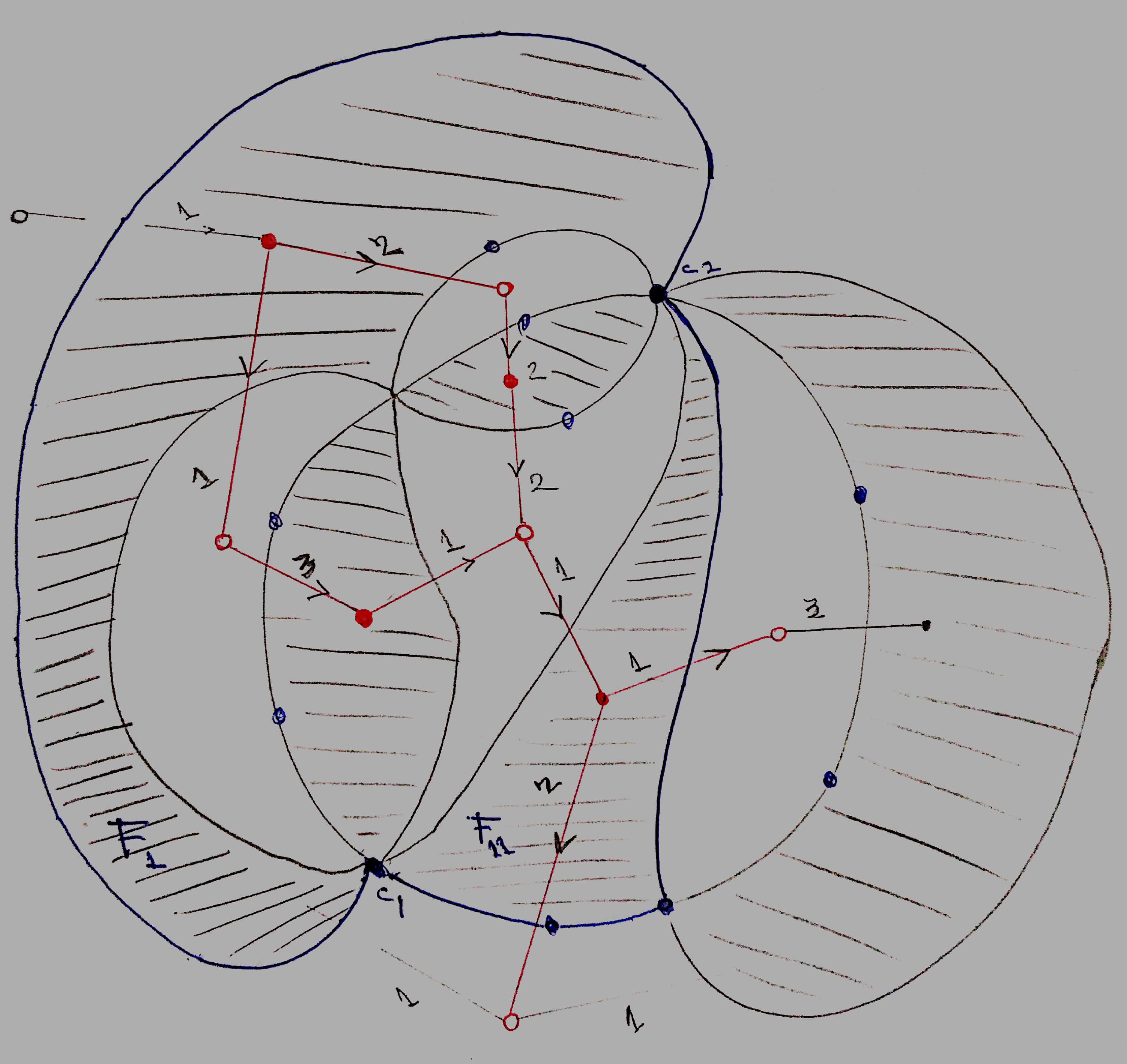}}
\quad
\subfloat[]
{\includegraphics[width=4cm]{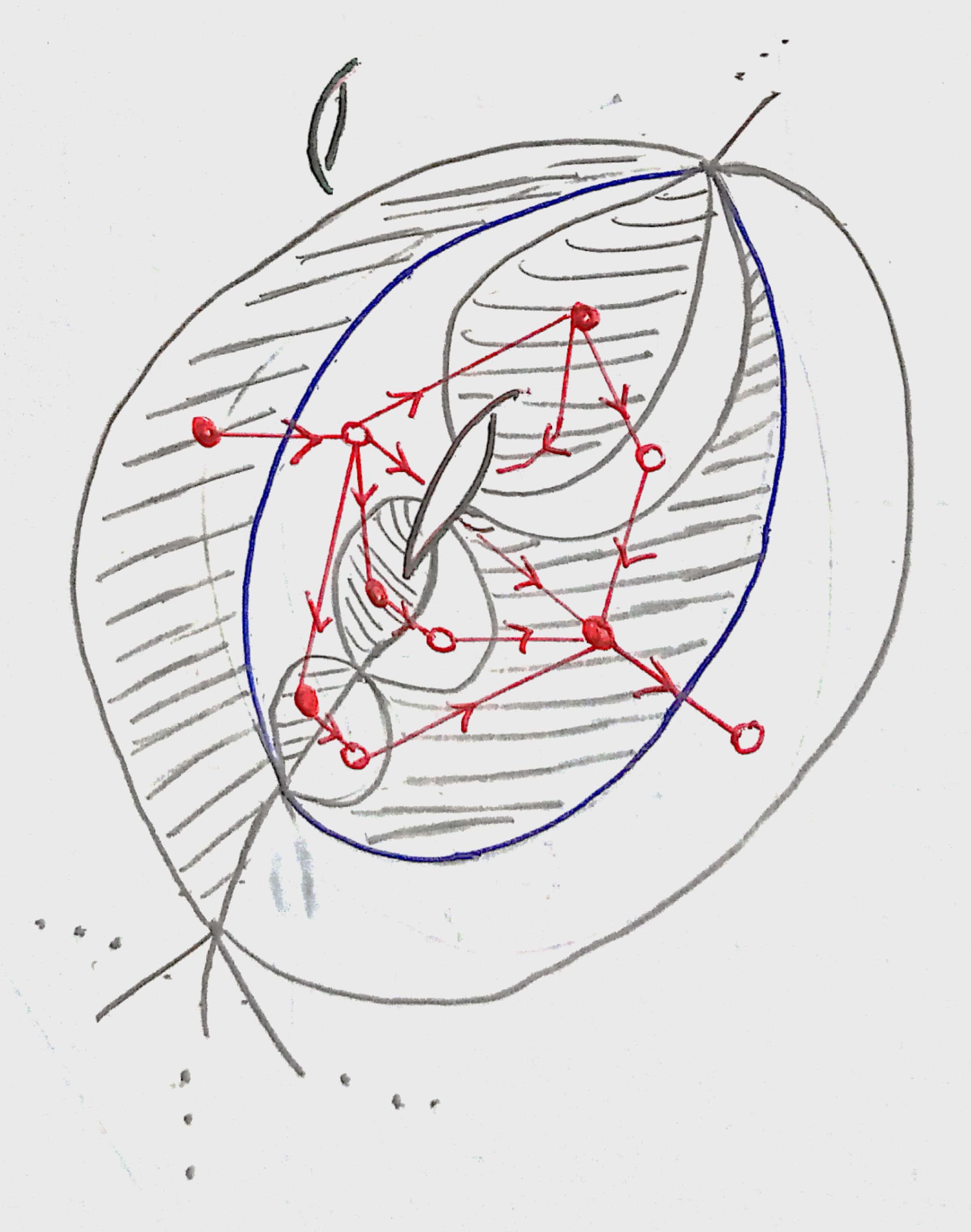}}\\
\subfloat[]
{\includegraphics[width=4cm]{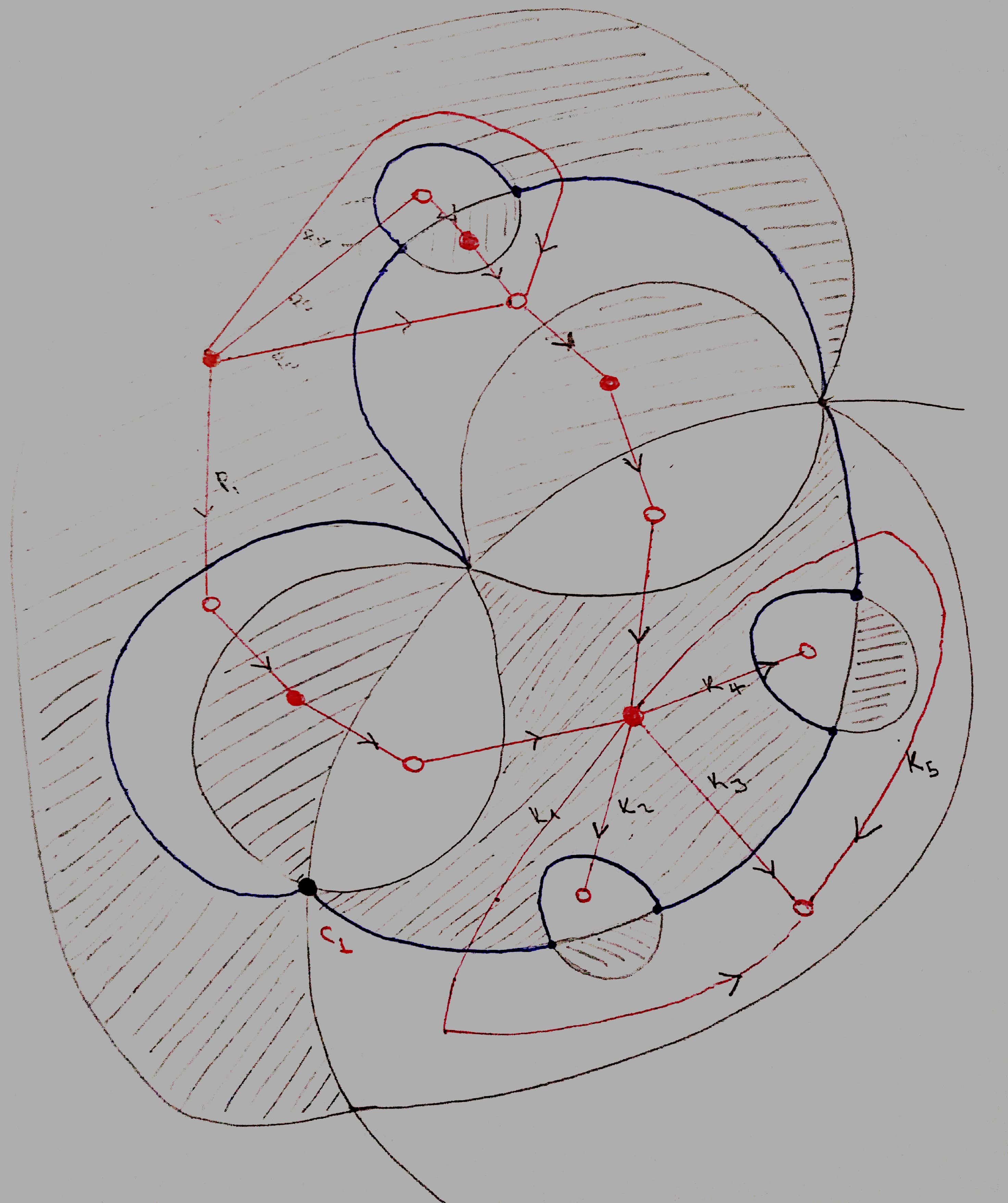}}
\quad
\subfloat[]
{\includegraphics[width=5cm,height=4.8cm]{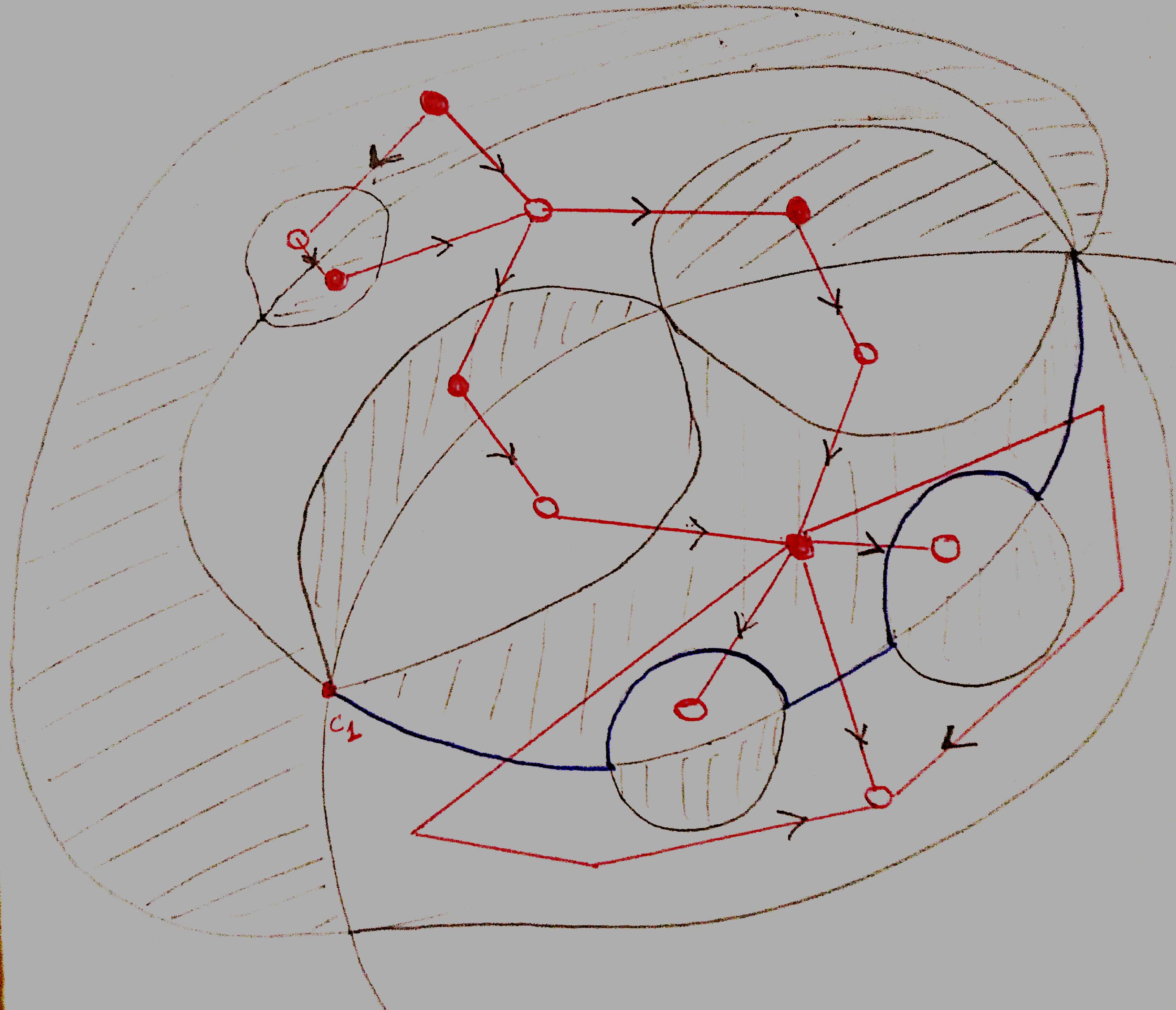}}
\caption{}\label{rp02}
\end{center}
\end{figure}

\end{proof}

Then we are done.
\end{proof}

But notice that the admissible vertex labeling depends on the matching realized to enrich the balanced graph. So a balanced graph can be, ignoring the $2$-valent vertices, the pullback graph of more than one branched cover, but all being of the same degree. See the example below:
\begin{ex}
Distincts matchings on the same balanced graph:
 \begin{figure}[H]
 \begin{center}
       \subfloat[enriching a balanced graph from a perfect matching]
    {{\includegraphics[scale=.049]{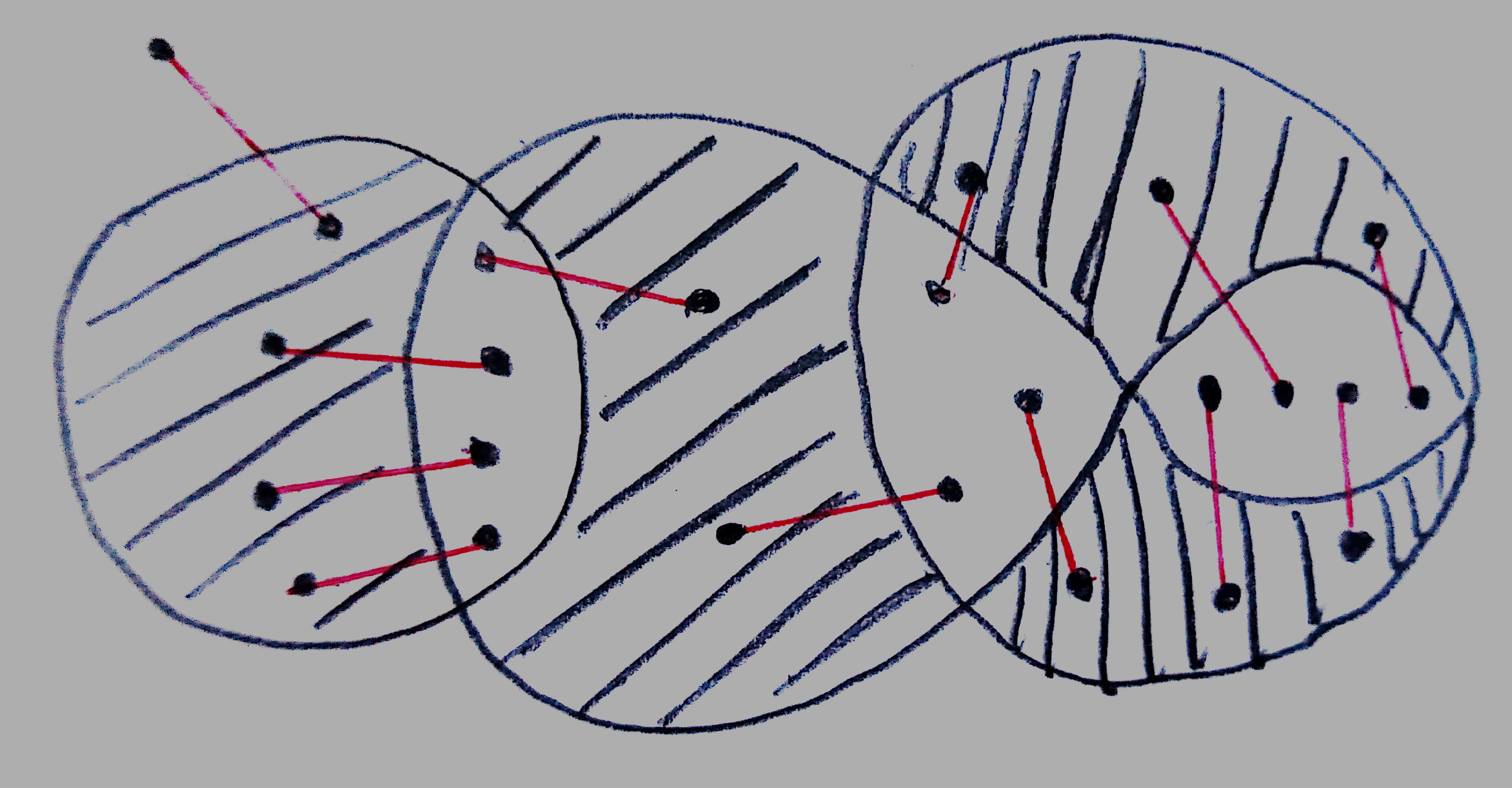}}}\hspace{0.0001mm} 
    \subfloat[admissible graph from (a)]  
    {{\includegraphics[width=7.3cm]{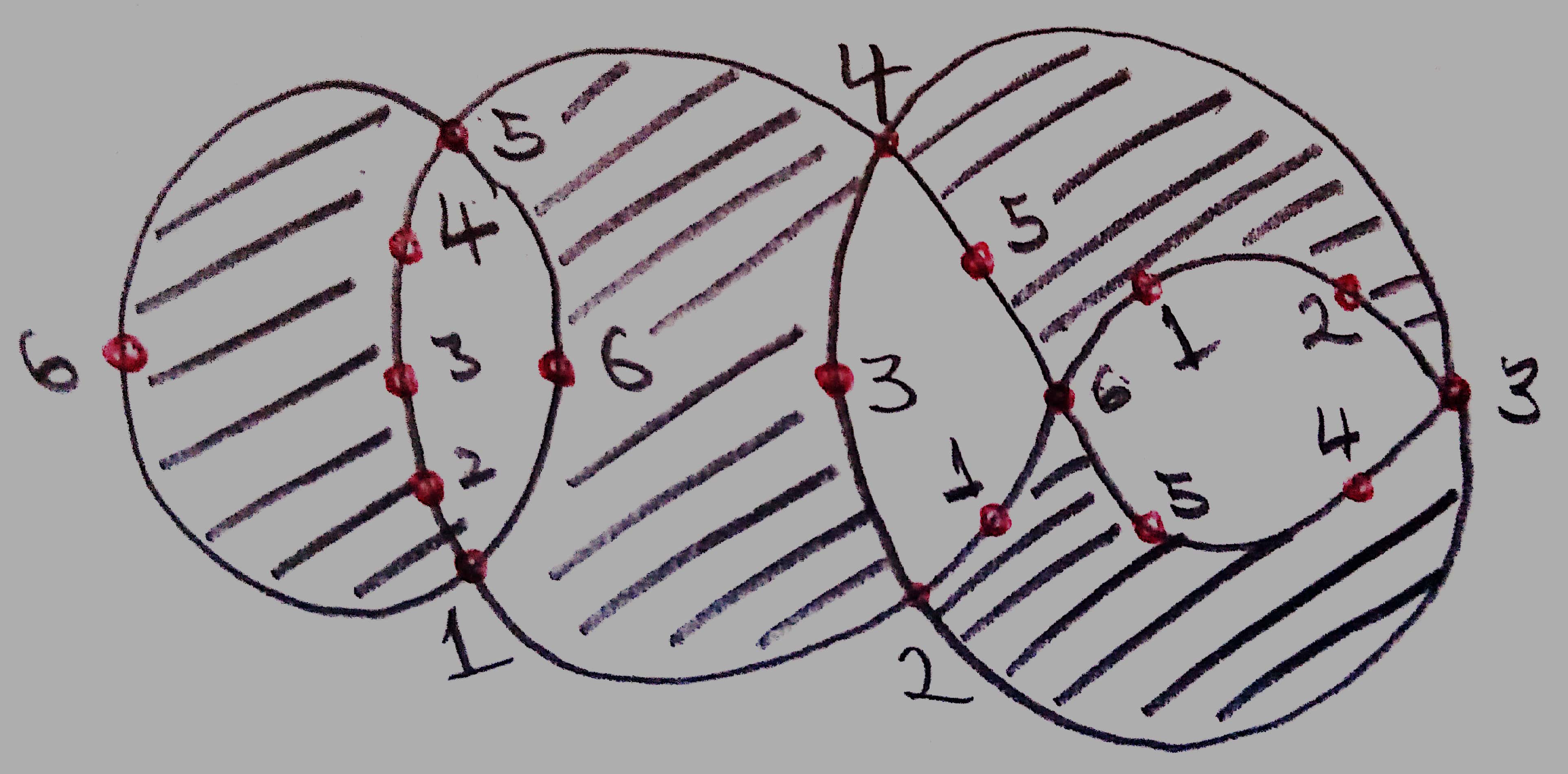} }}
     \caption[matching on a balanced graph]{}
\end{center}
\end{figure}
 \begin{figure}[H]
 \begin{center}
       \subfloat[enriching a balanced graph from a perfect matching]
    {{\includegraphics[scale=.05]{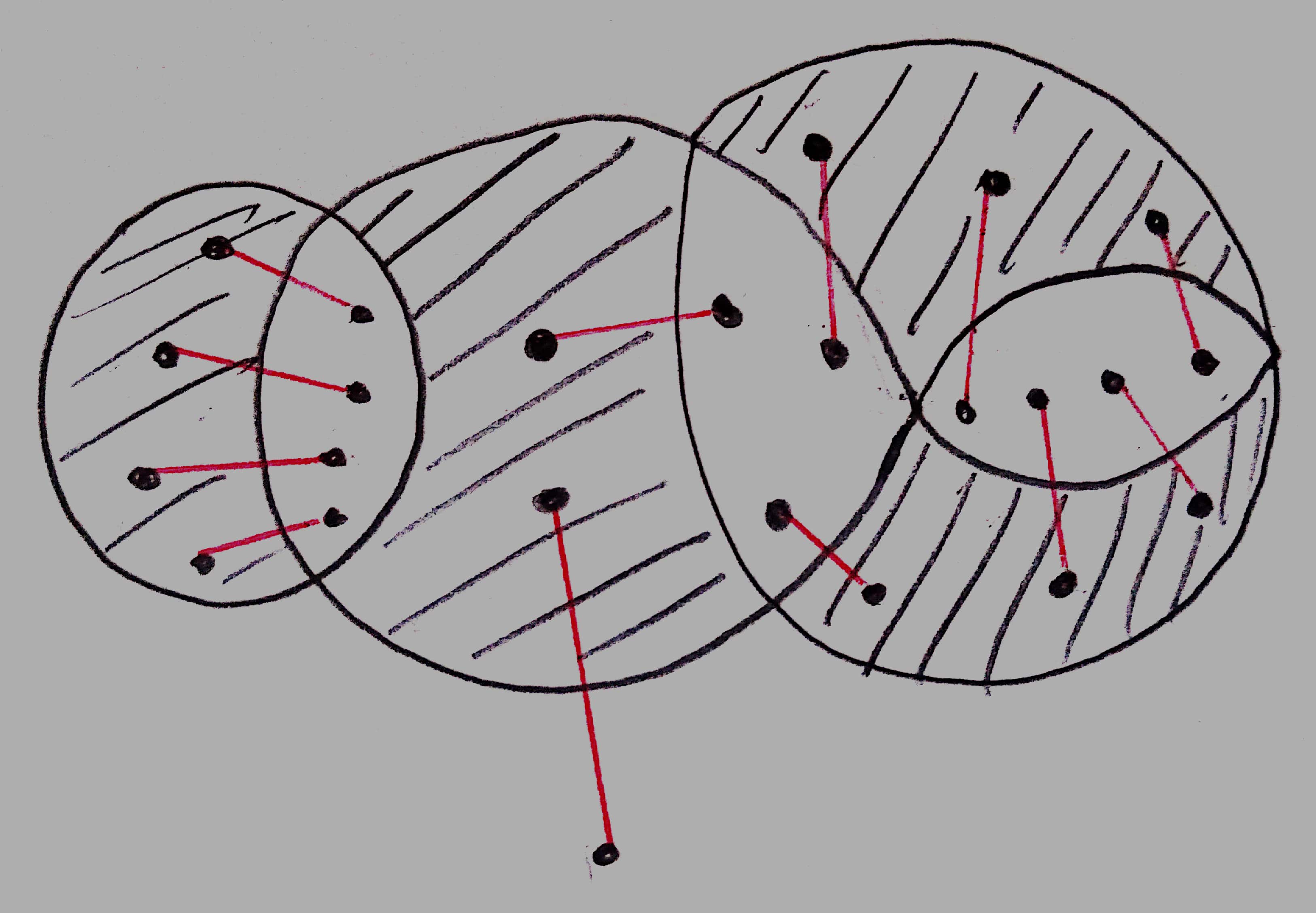} }}    
    \subfloat[``superfluous'' admissible labelling]
    {{\includegraphics[width=6.5765cm]{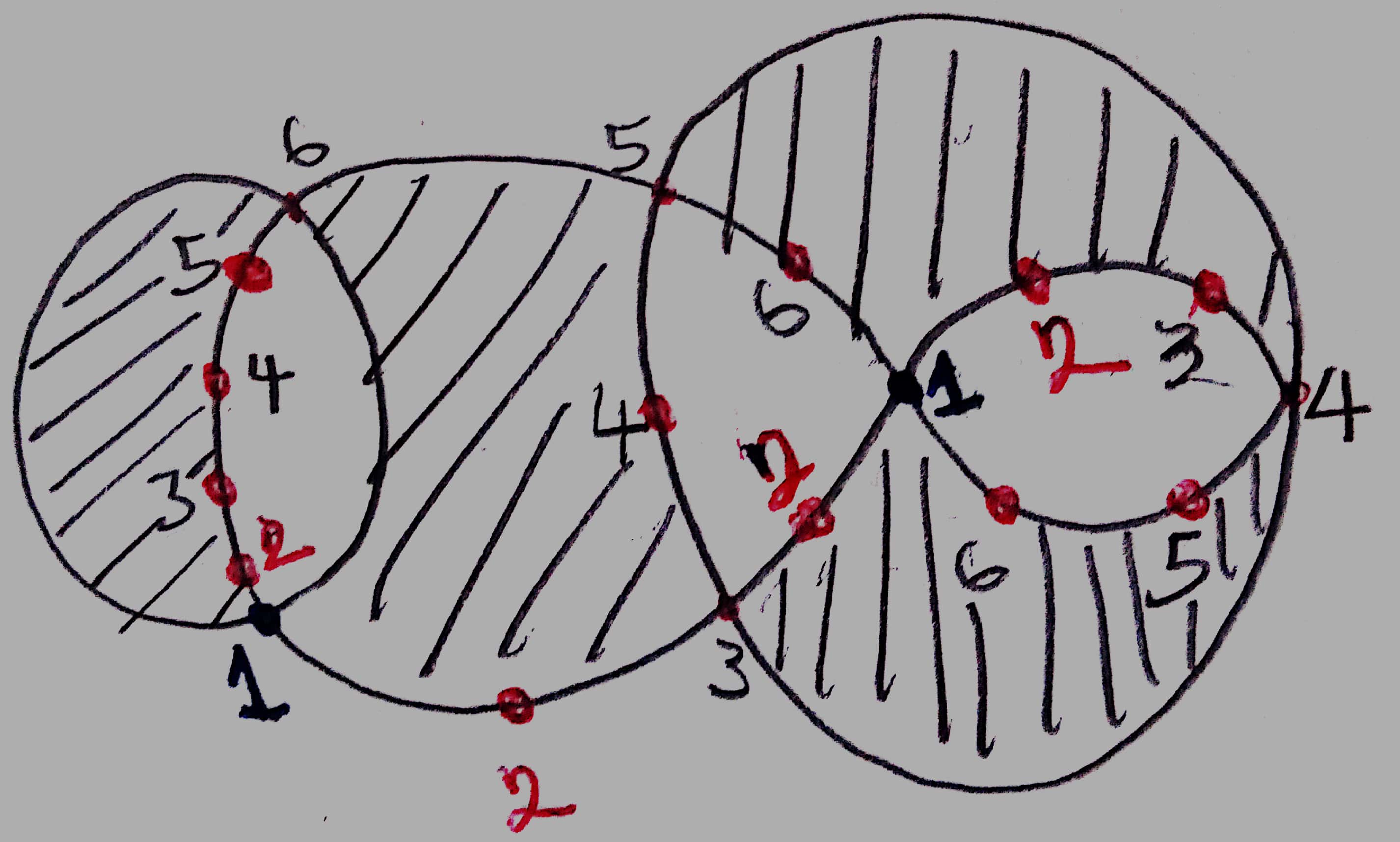} }}\\
    \subfloat[admissible graph from (b)]
    {{\includegraphics[width=6.875cm]{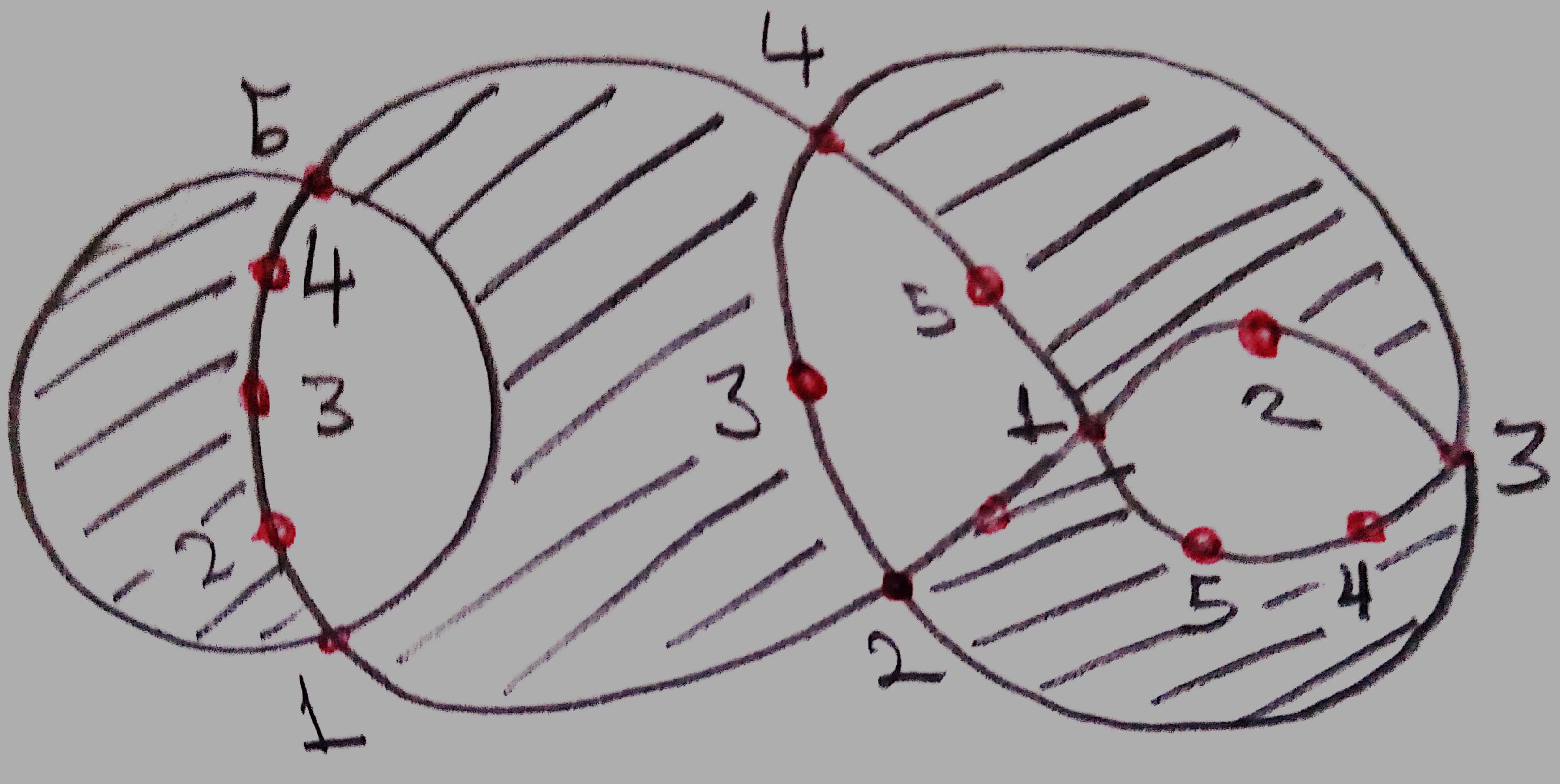} }}
    \caption[matching on a balanced graph]{}
\end{center}
\end{figure}
\end{ex}

\begin{ex}[another example]\hspace{7cm}
\end{ex}
 \begin{figure}[H]
 \begin{center}
       \subfloat[enriching a balanced graph from a perfect matching]
    {{\includegraphics[scale=.0517]{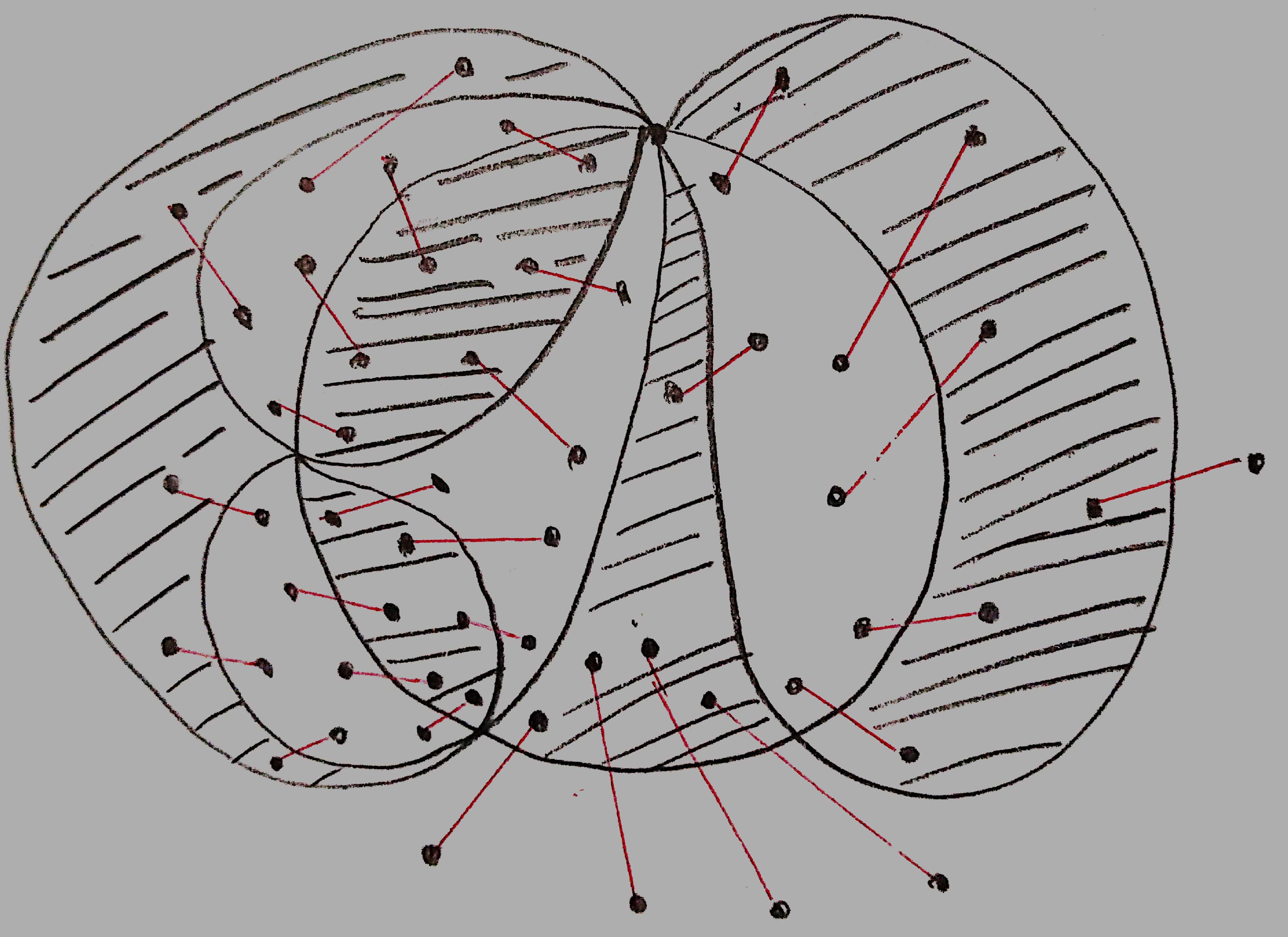} }}    
    \subfloat[``superfluous'' admissible labelling]
    {{\includegraphics[width=5.575cm]{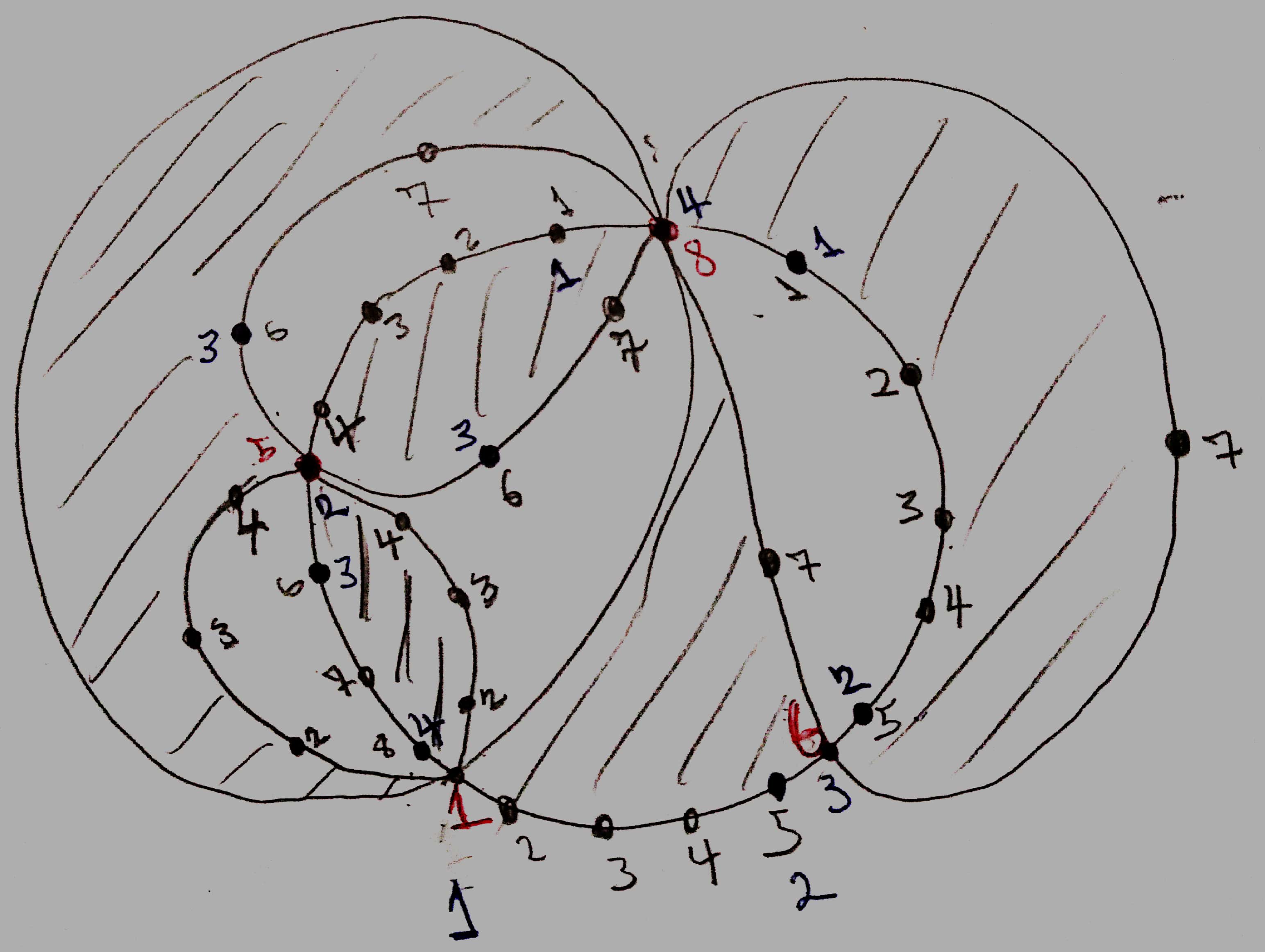} }}
    \end{center}
\end{figure}
\begin{figure}[H]
 \begin{center}
    \subfloat[admissible graph from (b)]
    {{\includegraphics[width=5.2cm]{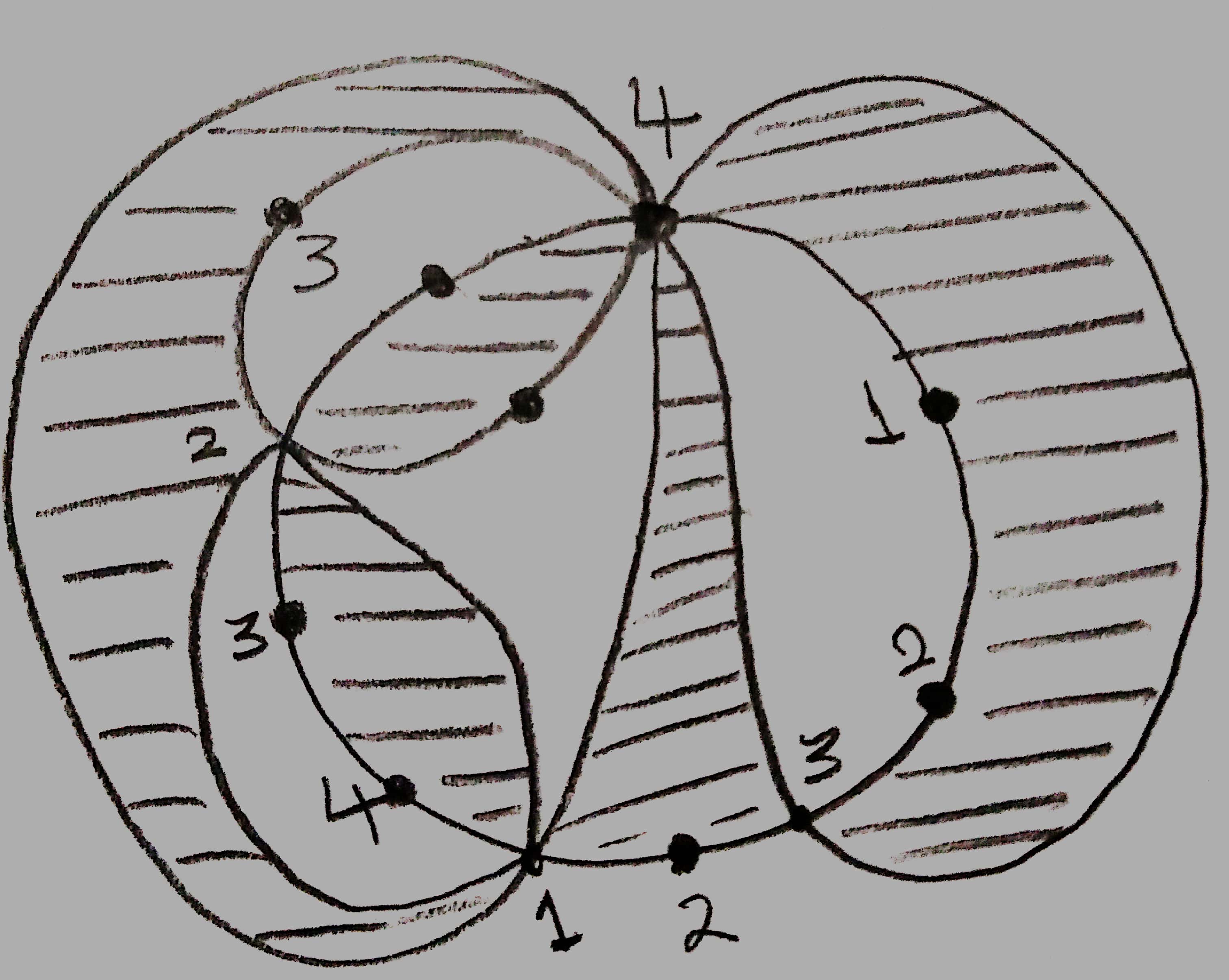} }}
    \caption[matching on a balanced graph]{}
\end{center}
\end{figure}

\section{Pullback graphs of real rational functions with real critical points}\label{2.5-sect}

For fixed integer $d\geq 3$, let $R_{\R}\subset \C(z)_{ d}$ be the set of rational function of degree $d$ with real coefficients and the set of critical points $C$ contained in $\overline{\R}$. We refers to such a map as a \emph{degree $d$ real rational function}.

That class of functions has a \emph{canonical post-critical curve}, namely the real  line $\overline{\R}\subset\overline{\C}$, since $f(\overline{\R})\subset\overline{\R}$ for all $f\in R_{\R}$. 

Each function $f\in R_{\R}$ satisfies $f(\overline{z})=f(z)$ for all $z\in f^{-1}(\R)$. Therefore, the pullback graph $\Gamma =f^{-1}(\R)$ are symmetric with respect to $\R$ for every $f\in R_{\R}$. 

\begin{ex}
Real pullback graphs of some degree $3$ rational functions: $f_1 (z)=\frac{z^2 \left(-\left(\sqrt{7}+2\right) z+2 \sqrt{7}+1\right)}{\left(\sqrt{7}-4\right) z+3}$, \\
$f_2 (z)=\frac{z^2 \left(\left(\sqrt{7}-2\right) (-z)+2 \sqrt{7}-1\right)}{\left(\sqrt{7}+4\right) z-3}$ and $f_3 (z)=\frac{z^3}{3 z-2}$, respectively.
\begin{figure}[H]
 \begin{center}
\includegraphics[scale=.35]{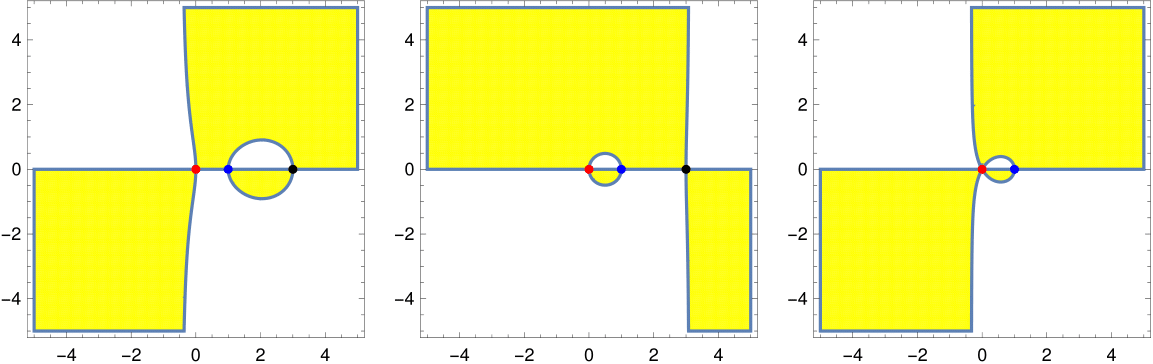}
\caption{Real Pullback Graphs}\label{rp02}
\end{center}
\end{figure}
\end{ex}

Thus, by the symmetry, each pullback graph $\Gamma=f^{-1}(\R)$ is uniquely determined by its non-real edges into the upper half-plane $\mathbb{H}^{u}$. Any two edges of $\Gamma$ do not intersect unless at their terminal real points. Notice that all of these terminal points forms the vertex set of the graph $\Gamma$. By the \emph{Thurston} theorem $\ref{THURSTHG1}$ such graphs are balanced.

The first and the second pullback graphs in the Figure $\ref{rp02}$ correspond to the unique two non-equivalent cubic generic real rational functions that maintains fixed the points $0,1$ and $\infty$ and it has critical points at $0,1,3, \infty$. 

Recall that a generic degree $d$ rational function, by definition, has $2d-2$ distinct critical points and $2d-2$ distinct critical values.
 
In another hand, each perfect matching of $2d-2$ vertices on $\overline{\R}$ in such a way that we can connect the vertices paired 
by non intersecting arcs into $\mathbb{H}^{u}$
it determines a connected graph with $2d-2$ vertices, all of them being of valence $4$ and with $2d$ faces if for such each arch into $\mathbb{H}^{u}$ connecting matched vertices we consider also its reflexion into the lower half-plane $\mathbb{H}^{d}$ with respect to the real line $\overline{\R}$. For a complete picture consult Figure $\ref{match00}$. 

\begin{figure}[H]{
 \centering
\includegraphics[width=6cm]{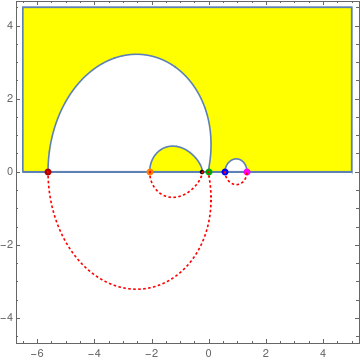}
\caption{Noncrossing matching of $6$ real points}\label{match00}
}
\end{figure}


Due to the symmetry, it is then immediate that an alternating face coloration of its faces turns it into a  globally balanced graph. Thus these graphs forms a subclass of the class of the underline graphs of the \emph{real admissible graphs} in $\ref{realcase?}$. 

In order to be able to construct, as in $\ref{realcase?}$, a rational function from some that globally balanced graphs as depicted above it should to support an admissible labeling. Or, as we saw in the proof of \emph{Thurston Theorem} $\ref{THURSTHG1}$ it should to be also \emph{locally balanced}.  Moreover, if they support an admissible vertex labeling $L:V(G)\longrightarrow\{1, 2, \cdots, 2d-2\}$ then the rational function aroused from it will be a degree $d$ generic rational function.\label{goldissue}

\begin{defn}[generic real planar GB-graph]\label{rgb-g}
A planar GB-graph as described above coming from a noncrossing matching of $2d-2$ real points will be called by \emph{generic degree $d$ real planar GB-graph}. And a generic real admissible graphs will be a generic degree $d$ real planar GB-graph with an admissible labeling 
$L:V_G\longrightarrow\{1, 2, \cdots, 2d-2\}$.
\end{defn}

Note that for non-isotopic {generic real admissible graphs} with vertex set $\{v_1, v_2 , \cdots ,v_{2d-2}\}\subset\overline{\R}$ the corresponding real rational functions from Theorem $\ref{r_a_gb_r_r_f}$, say $f$ and $g$,  are not equivalent. For if $g=\sigma\circ{}f$, since $\overline{\R}=f(\overline{\R})=g(\overline{\R})$ then $\sigma(\overline{\R})=\overline{\R}$. Therefore, $g^{-1}(\overline{\R})=f^{-1}(\sigma^{-1}(\overline{\R}))=f^{-1}(\overline{\R})$.

Leaving the vertices fixed, the counting of such matchings is a well-known problem in enumerative combinatorics (see \cite{Stancat:15}-exercise $59$). And there are $\displaystyle{\rho_{d}=\dfrac{1}{d} \binom{2d-2}{d-1}}$ such matchings. Then 
the number of real GB-graphs of degree $d$ for prescribed $2d-2$ vertices into $\overline{\R}$ is $\rho_{d}$.

So, if the issue laid out above could be settled we will have obtained the following result:

\begin{qsconj}\label{preanuncio}
{The number of equivalence classes of generic real rational function of degree $d$ for a prefixed set of $2d-2$ distinct points in $\overline{\R}$ is
\[\displaystyle{\rho_{d}=\dfrac{1}{d} \binom{2d-2}{d-1}}.\]}
\end{qsconj}

This result, once proven, will consist on a combinatorial solution for the counting problem of equivalence classes of generic rational function asked by Lisa Goldberg in \cite[{PROBLEM}, at page $132$]{Gold:91}

$\boldsymbol{(\star)}$ $\ref{preanuncio}$ will be proved in section $\ref{shap-section}.$

\section{The category of 
 balanced graphs ${\mathcal{B}\mathcal{G}}$}

In this section we will construct the Category of Balanced Graphs $\bgg$. 

\begin{defn}\label{glset}
Let 
$\textbf{BG}(g,d)$ denote the set of genus $g$ balanced graphs of degree $d$, and $\textbf{BG}:={\bigsqcup_{d=2, g=0}^{\infty, \infty}}\textbf{BG}(g,d)$. 
\end{defn}

\subsection{Operations on Balanced Graph}\label{oper-bg}

In this section, we will describe a series of operations against balanced graphs. 
Those operations shall allow us to transform one graph into another one but preserving some essential properties.

Those essential properties that we expect to be maintained under the operations are the local and global balance conditions on cellular embedded 
even graphs.

In addition to providing a deep understanding 
of the graphs, having these operations at our disposal is a great asset in order to simplify some proofs.

Some of these operations are related to the continuous deformation of a branched cover in another one with a different critical configuration.
 
\subsubsection{Edge-Contraction}

Although we probably haven't highlighted this previously, balanced graphs do not contain loops. And more generally, it does not contains corners that are incident more than once to a face. A priori, this could be seen (or even taken) as a natural imposition, given the intention of having each face as a compact piece where a branched covering is injective.But, actually, this fact stems from the balance conditions.

\begin{lem}\label{lem-jface}
The boundary of the topological closure of each face of a balanced graph 
consists of only one Jordan curve.
\end{lem}  
\begin{proof}
Let $\Gamma$ be a balanced graph 
with a face $F$ whose topological closure, $\myov{F}$, has its boundary containing more than one Jordan curve. 
Thus $F$ should contain at least one corner incident to it more than once. This follows from the fact that each face is simply connected. Since, if a Jordan curve, say $\gamma_0$, into the boundary of $F$, $\partial{F}$, is not connected to $\partial{F}-\gamma_0$ by a saddle-connection or even 
by a corner, the simply connectivity of the planar domain $F$ is lost.

But, as we shall see, the occurrence of that kind of corner obstructs the (local) balance condition.

Assume $\Gamma$ balanced with a Yellow - White alternating colloring (being the Yellow color the preferred one). 
Suppose that $F$ is white.

Choose a connected component, $B$, of $X-{F}$. Since $F$ is white each cycle on the boundary of $B$ is positive. Thus, by the local and global balance conditions the number $W_{b}$ of white faces into $X-B$ is strictly bigger than the number of yellow faces, $Y_{b}$ there. That is, 
\begin{eqnarray}\label{deswy}
W_b > Y_b
\end{eqnarray}

On the other hand, we have $W_k \leq Y_k -1$ for each connected component , $C_k$, of $X-\{F\cup B\}$. Let $m$ be the number of those $C_k$ components.Then,
\[W_b -1=\sum W_k \leq\sum (Y_k -1)=Y_b - m< Y_b\]

But this contradicts $\ref{deswy}$.

And, if $F$ is yellow face of $\Gamma$, a similar argument 
will bring us to the expected end (only with the role of the colors exchanged).
\begin{figure}[H]
 \begin{center}
    {{\includegraphics[width=6cm]{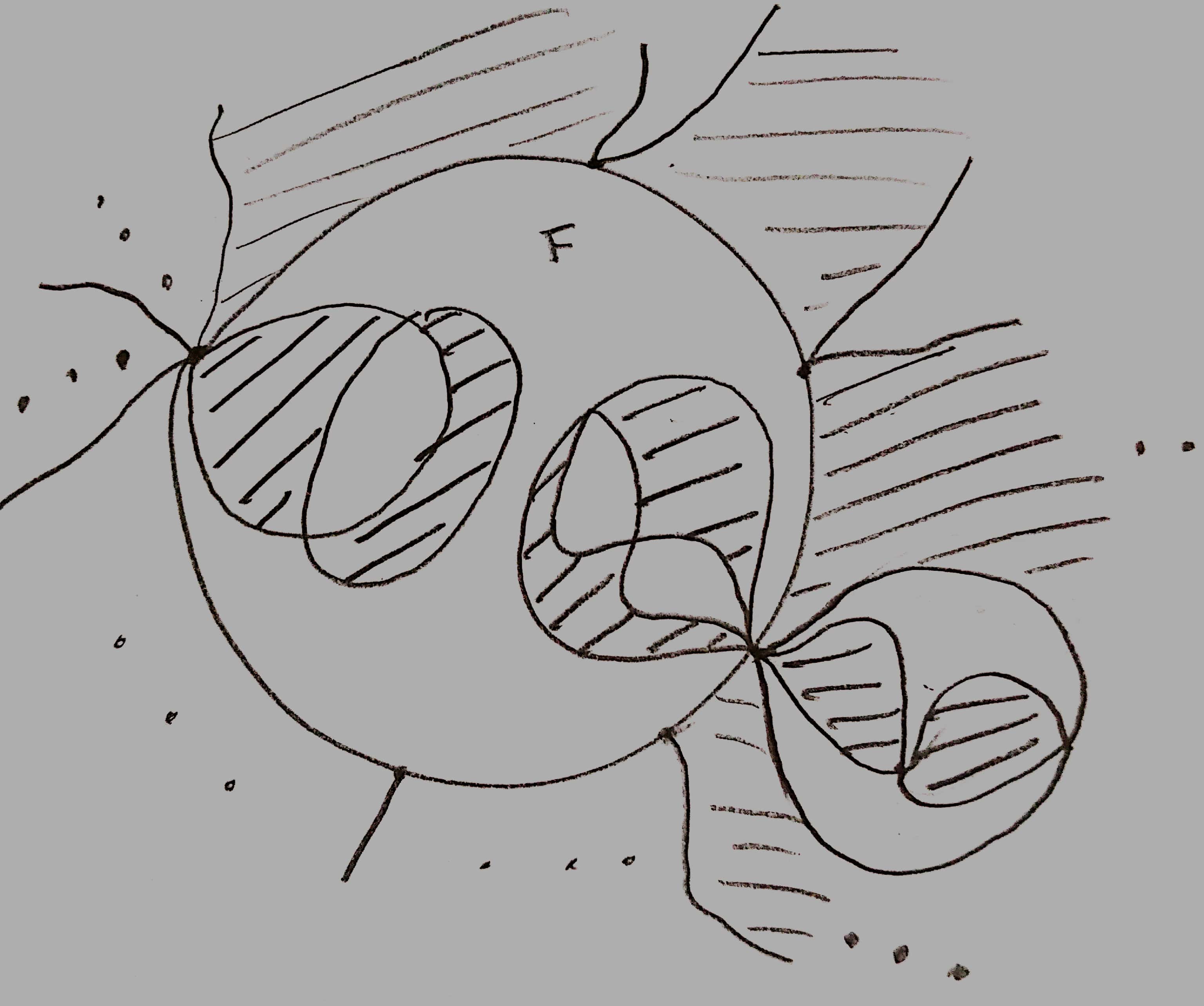}}}
 \caption{non simply connected face}
\end{center}
\end{figure}
\end{proof}

\begin{defn}[splitting saddle-connection]
A saddle-connection of a balanced graph is said to be a splitting saddle-connection if its extremal points (corners) are simultaneously incident to more than $2$ faces. Or equivalently, if there is at least one face to which the extremal points of the saddle-connection are incident but the saddle-connection itself does not.
Two corners, $A$ and $B$, of a balanced graph are called \emph{splitting-corners} if they are connected by a splitting saddle-connection.
\end{defn}

We highlight this type of saddle-connections because the procedure of removing one such saddle-connection and then identify its endpoints 
it generates a new embedded graph having faces with a topology that obstructs the local balance (see Lemma $\ref{lem-jface}$).

\begin{defn}[{edge-contraction}]
The operation of \emph{edge-contraction} on balanced graphs consists on the procedure of to identify 
a non splitting saddle-connection of the graph to a single point.  
\end{defn}

Notice that the \emph{edge-contraction} operation does not change the topology of the support 
surface sice it collapse a cellular subset. 

 \begin{figure}[H]
 \begin{center}
       \subfloat[balanced graph]
    {{\includegraphics[width=6.2cm]{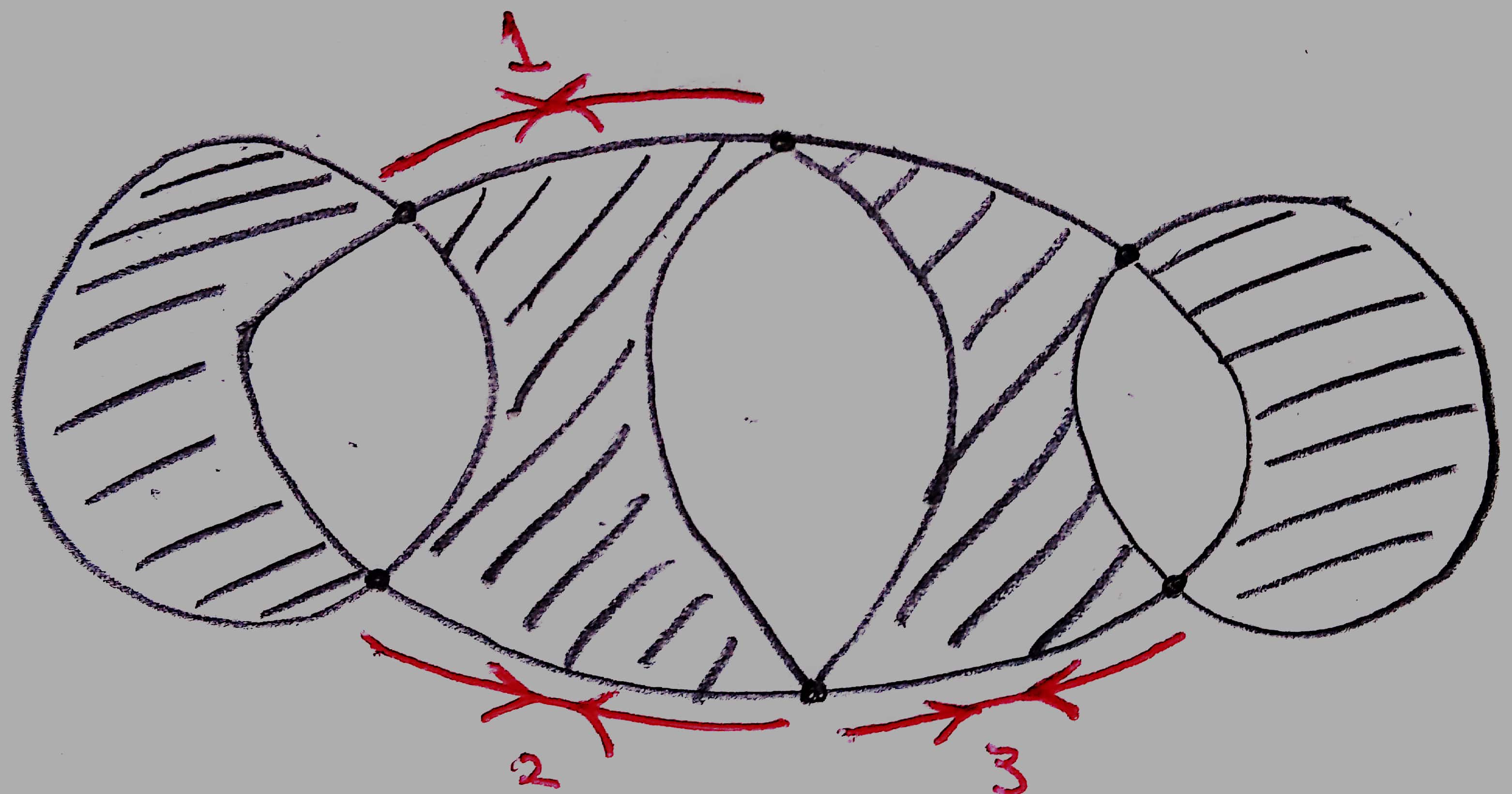}}}
       \; \subfloat[edge-contraction - 1]
    {{\includegraphics[width=7.65cm]{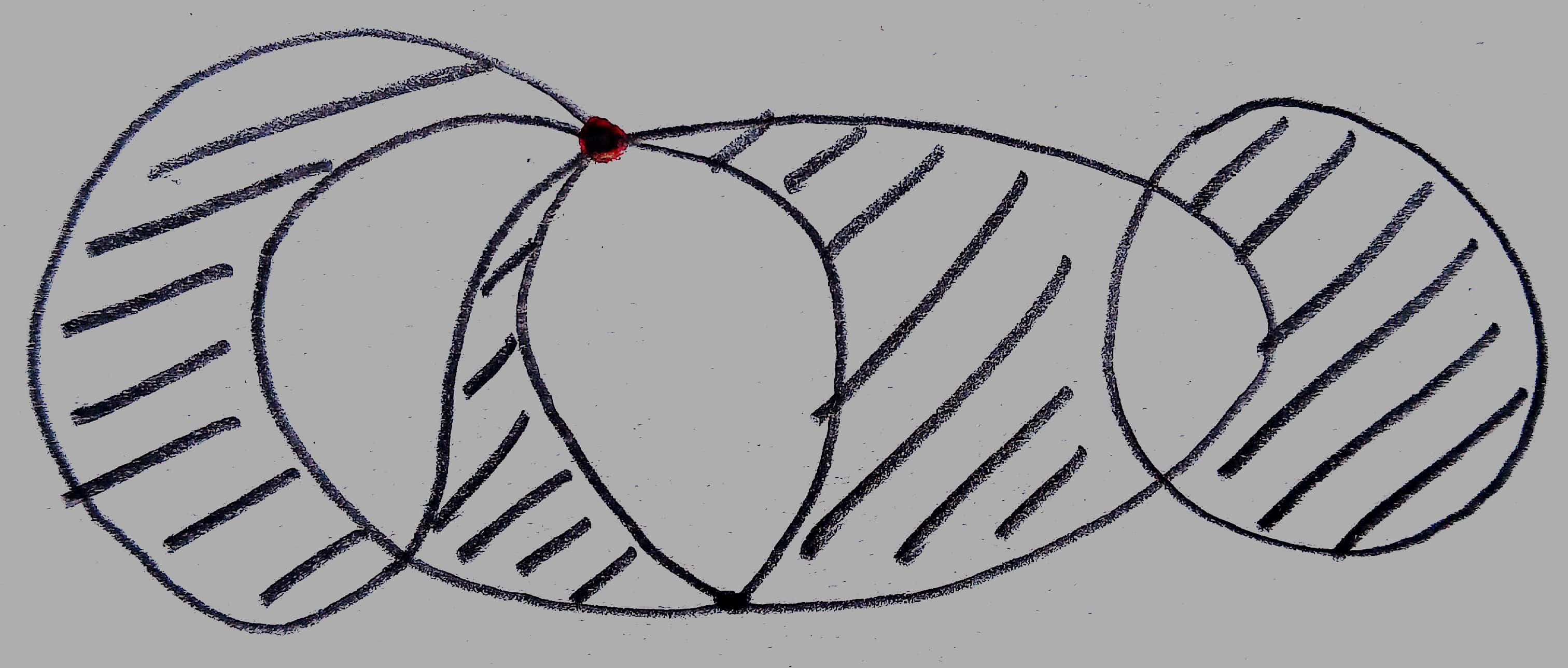} }}\\
    \subfloat[edge-contraction - 2 and 3]
    {{\includegraphics[width=5.9cm]{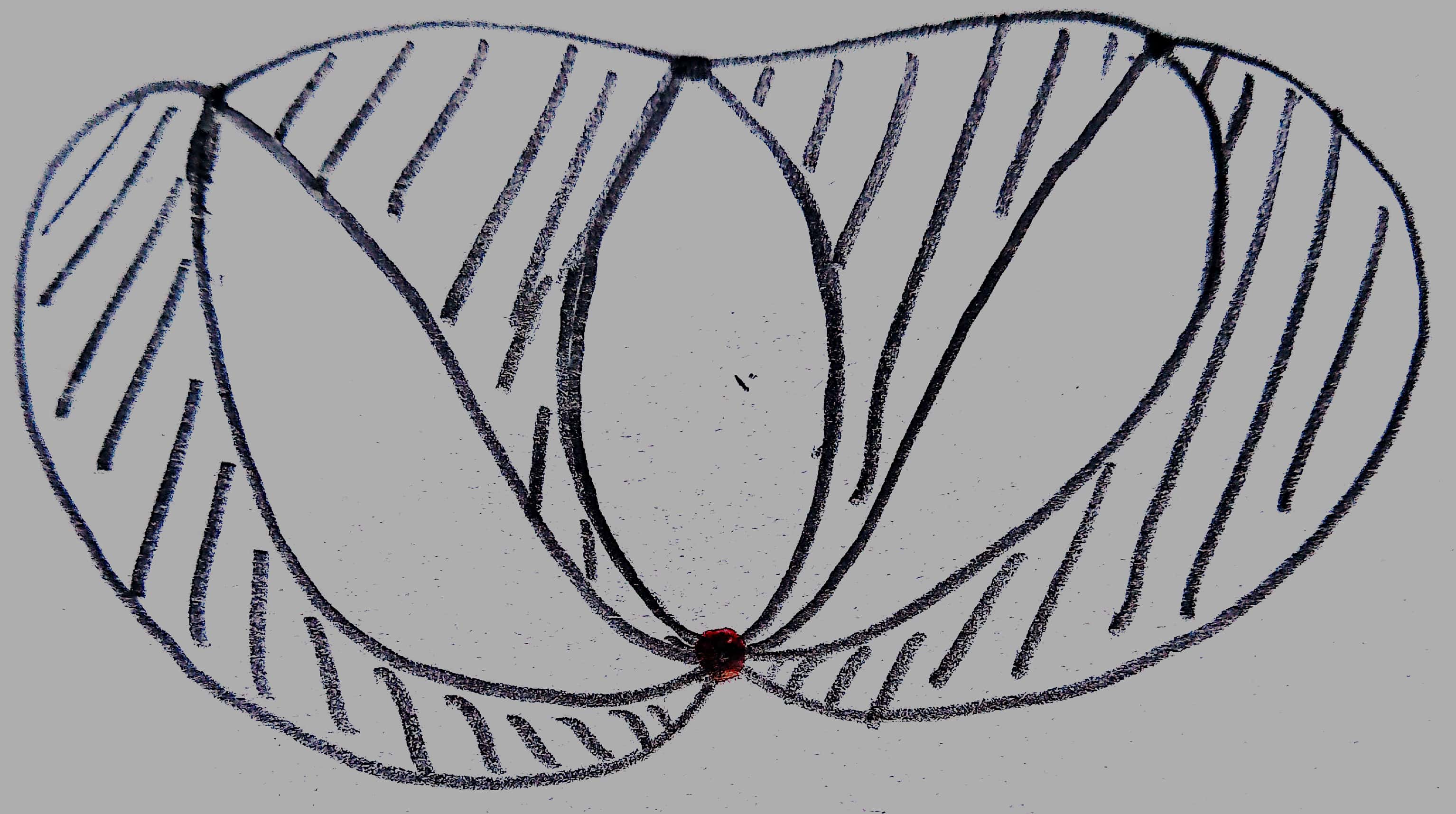} }}
     \subfloat[contraction of a splitting sadle-connection (not allowed)]
    {{\includegraphics[width=8.25cm]{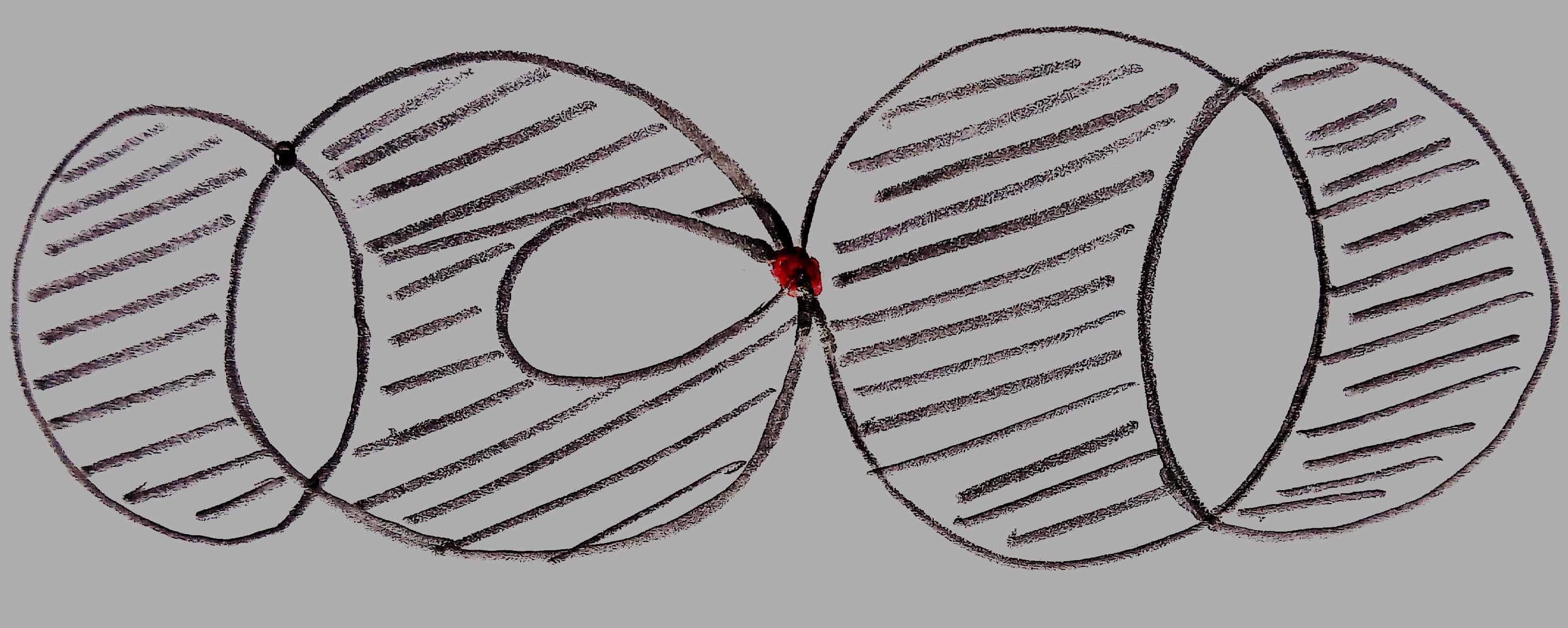} }}
        \caption{edge-contractions}
\end{center}
\end{figure}

 \begin{figure}[H]
 \begin{center}
    \subfloat[generic real GB-graph]
    {{\includegraphics[scale=.05]{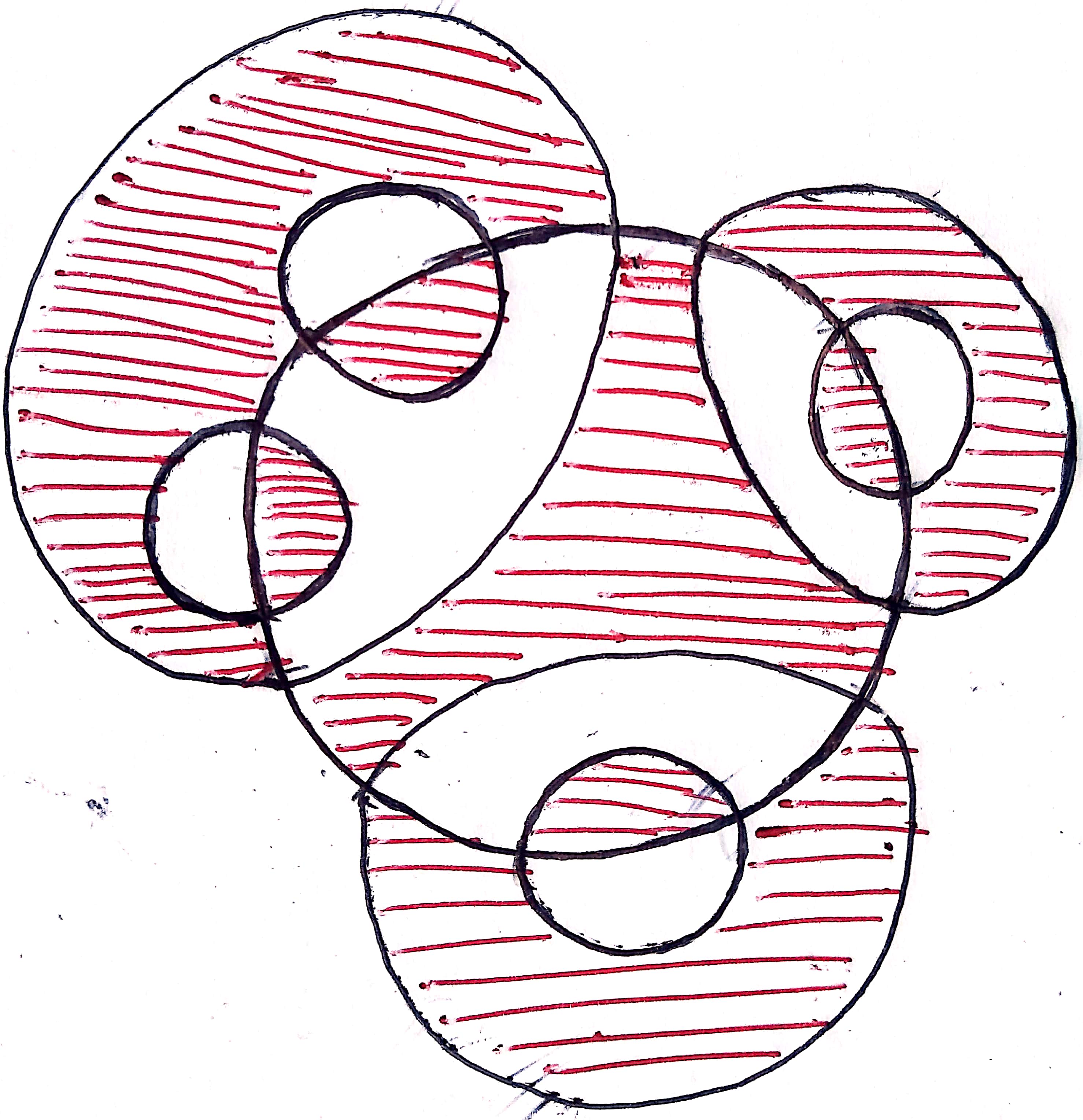} }}
    \qquad
    \subfloat[non-generic real GB-graph from (a)]
    {{\includegraphics[scale=.05]{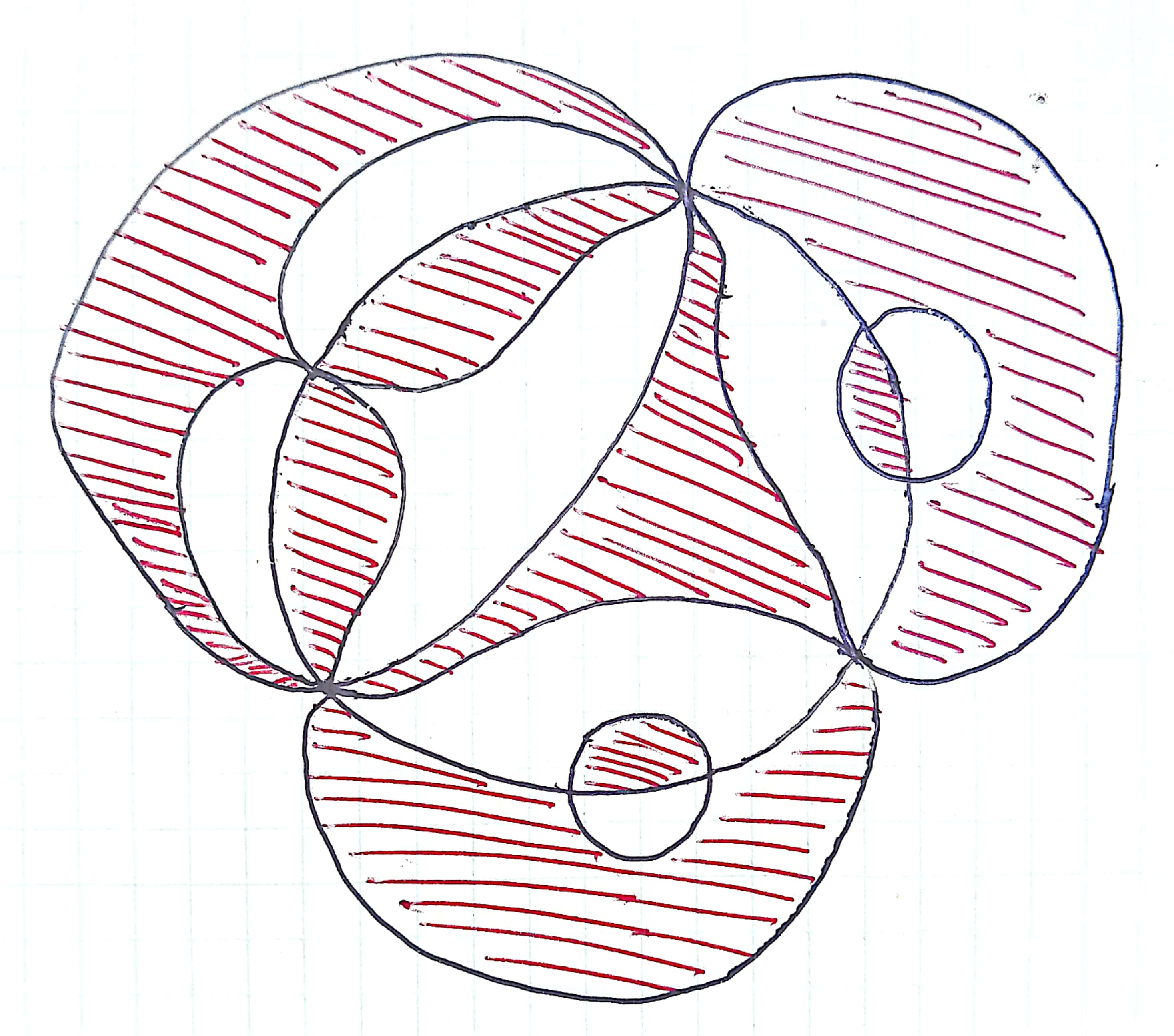} }}%
    \caption{edge-contractions}
\end{center}
\end{figure}

Given a balanced graph $\Gamma$ for each non splitting saddle-connection we can perform an \emph{edge-contraction} 
and so combine a sequence of such operations on the graph. The number of such kind of operation is finite, one for each  non splitting saddle-connection, thus the same happen for combinations of those operations.

There is an inverse operation for the \emph{edge-contraction}.

\subsubsection{vertex-expansion}

\begin{defn}[vertex-expansion]
\label{v-exp}
The operation \emph{vertex-expansion} on balanced graphs consists on the procedure of splitting a corner of degree greater or equal to $6$ of a balanced graph in another $2$ new corners and then to inserting a new edge connecting them as especified below:
\begin{itemize} 
\item[$\boldsymbol{(1)}$]{the set of edge incident to the vertex to be split is split up into two subsets of edges, say $A$ and $B$, such that the edges in each subset runs around the original vertex (the corner to be split) with only one gap.
Each subset correspond to one of the two new vertices;}\\
\item[$\boldsymbol{(2)}$]{ the cardinal of $A$ and $B$ is odd and greater or equal to $3$;}\\
\item[$\boldsymbol{(3)}$]{and a new edge is inserted connecting these two new vertices such that its contraction produces a vertex whose the \emph{order of incidence} of the edges \emph{around} recovers the order of incidence of the edges around the original vertex 
(or, such that we realize the order of incidence around the original vertex going around one vertex from a adjacent edge to the new edge up to the new edge again, and then passing through it until the another vertex and then continuing turning around it in the same sence that we goes around the former vertex up to that new edge again)}.
\end{itemize}

 \begin{figure}[H]
 \begin{center}
       \subfloat[around a vertex]
    {{\includegraphics[width=3.23cm]{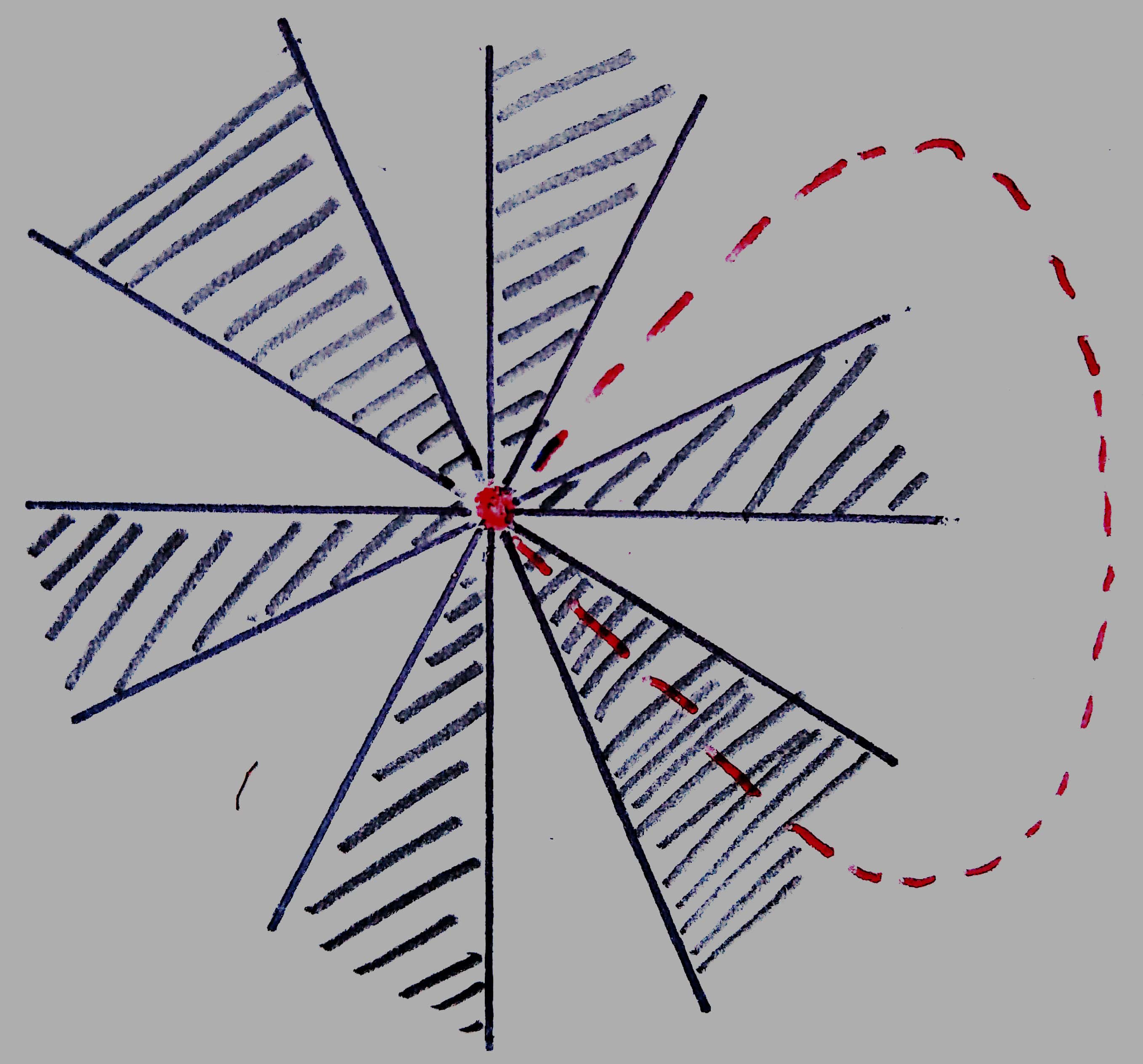} }} \qquad \qquad  
    \subfloat[ vertex-expansion at (a)]
    {{\includegraphics[width=6.5cm]{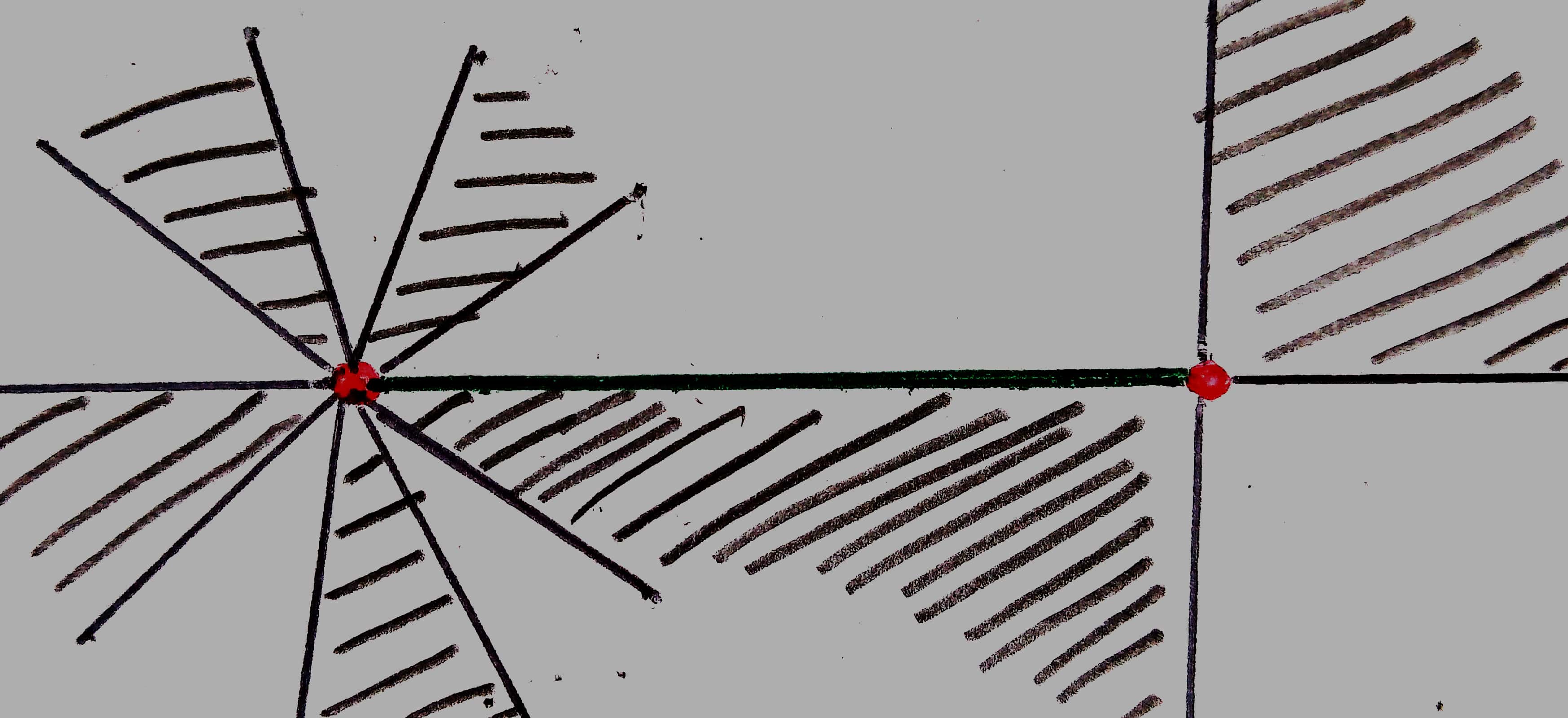} }}   
    \caption{two different vertex-expansion on the same vertex}
\end{center}
\end{figure}
\end{defn}

 \begin{figure}[H]
 \begin{center}
       \subfloat[balanced graph]
    {{\includegraphics[width=5.4cm]{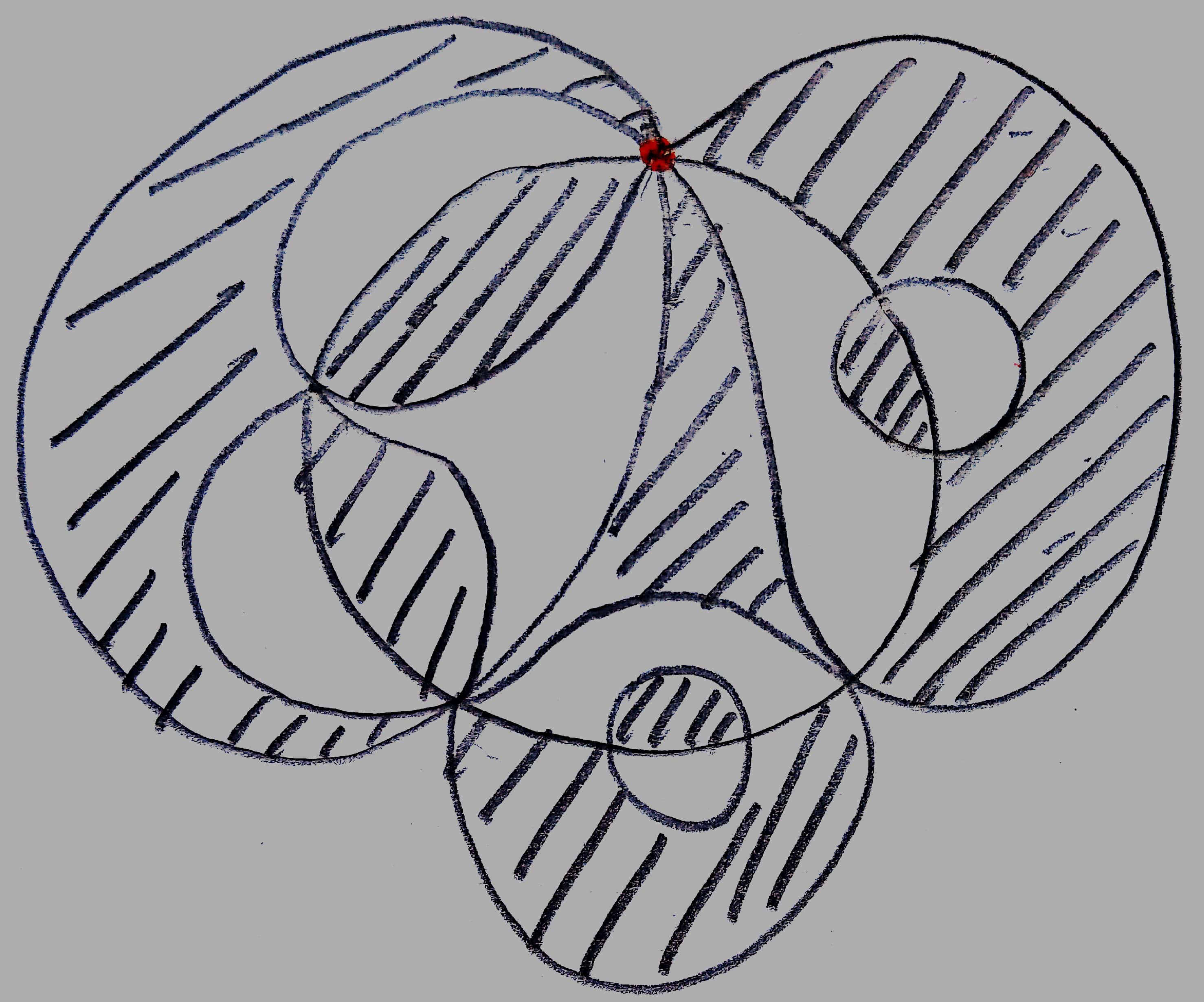} }}    
    \qquad
    \subfloat[ vertex-expansion]
    {{\includegraphics[width=5.3cm]{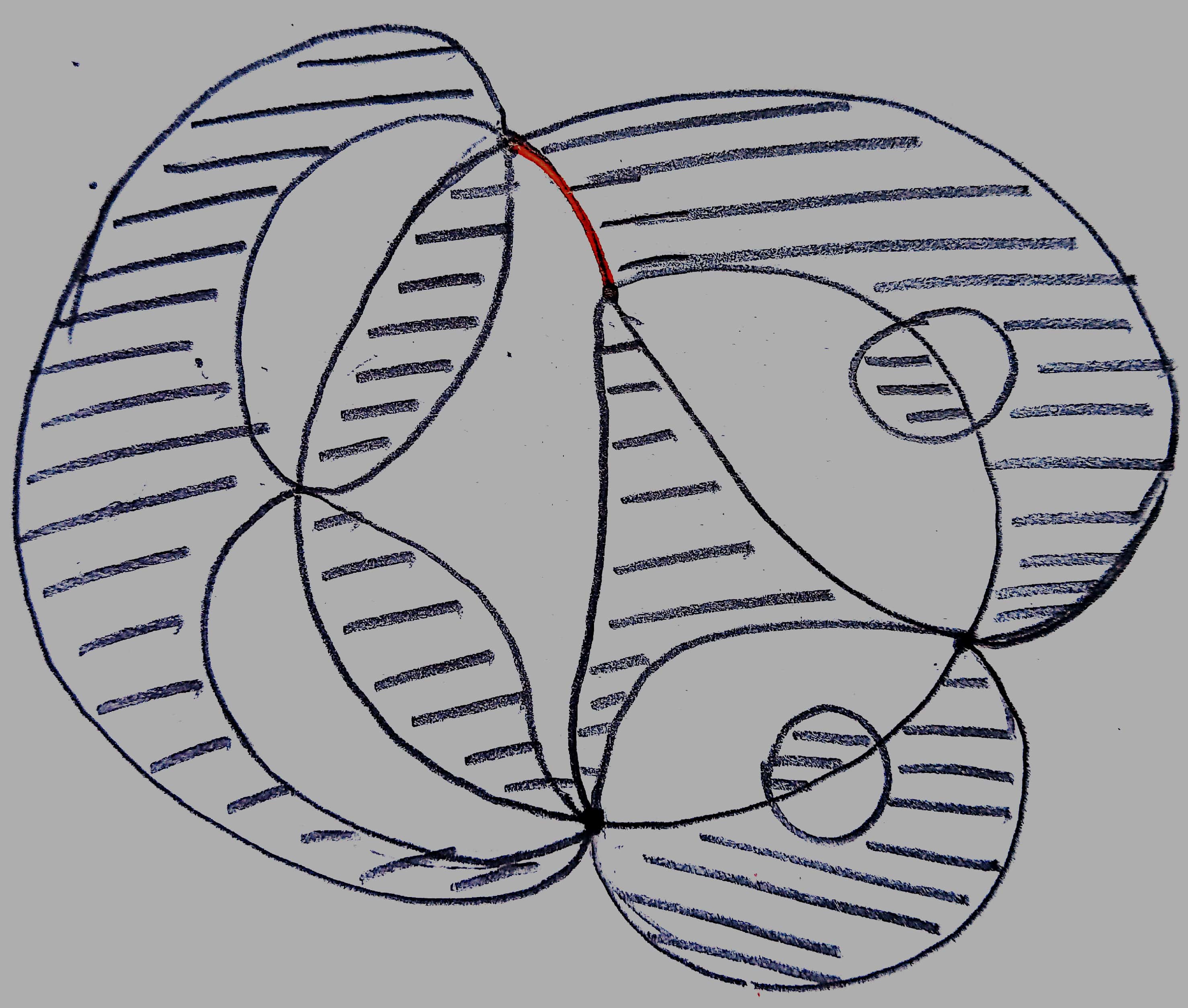} }}\\
    \subfloat[vertex-expansion]
    {{\includegraphics[width=5.4cm]{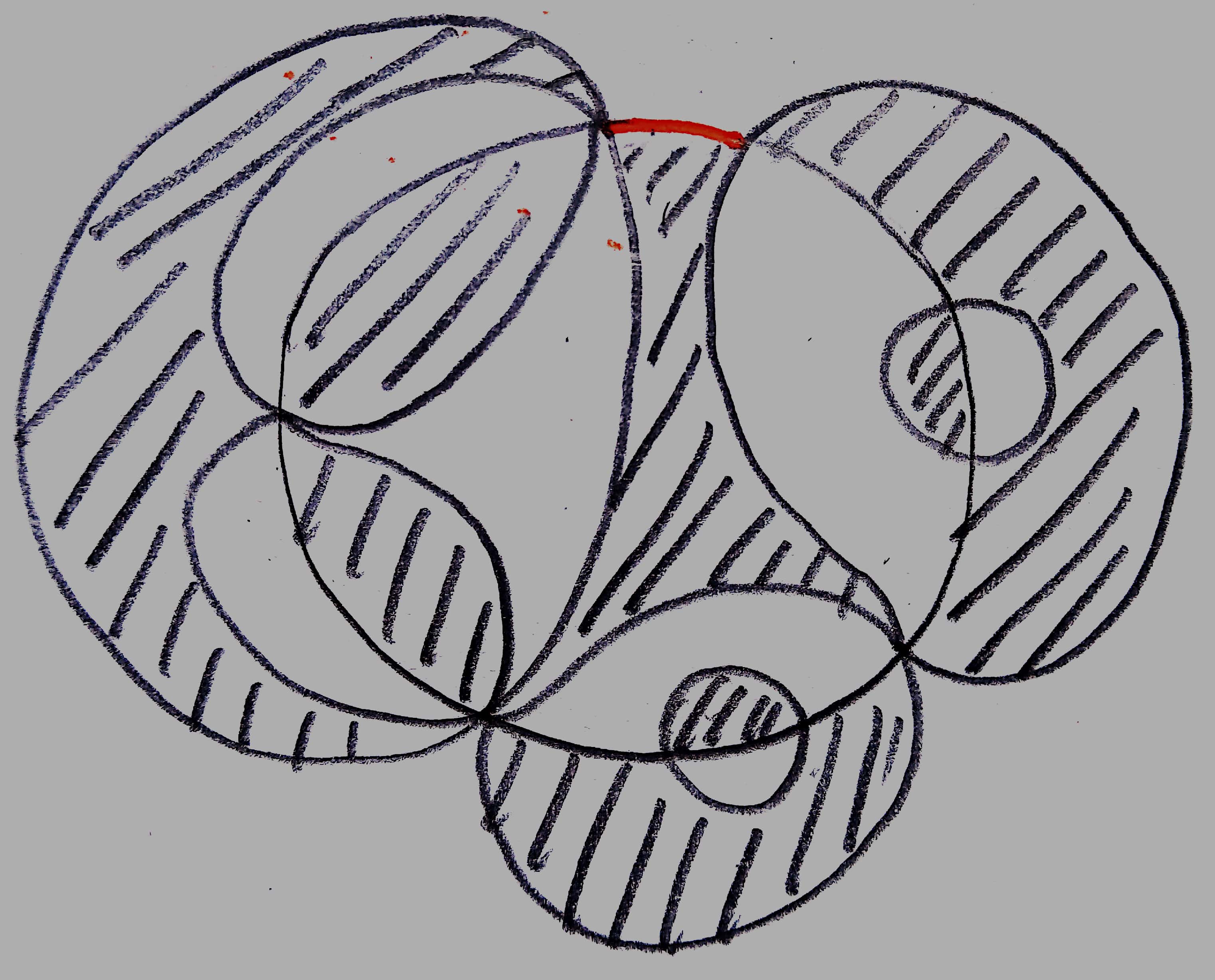} }}
    \caption{two different vertex-expansion on the same vertex}
\end{center}
\end{figure}


\begin{lem}\label{lem:counting-v-exps}
Given a vertex $v$ of degree $m\geq 6$ of a balanced graph. The number of all possible \emph{vertex-expansions} at $v$ is
\[\dfrac{m(m-4)}{4}\]
\end{lem}

This means that from a balanced graph $\Gamma$ with a vertex of degree $m\geq 6$ we can produce $\frac{m(m-4)}{4}$ new balanced graphs from \emph{vertex-expasions} against that vertex.

\begin{proof}[Proof of Lemma $\ref{lem:counting-v-exps}$]
Let $m\geq 6$ be a positive even integer number.
The number of partitions of $m$ with two parts and with each part being geater or equal to $3$ equals $\frac{m-4}{2}$.

Regarding the orientation, we enumerate the edges that are incident to $v$ from $1$ until $m$. Each bipartition of the edges that are incident to $v$ into the subsets $A$ and $B$ as in Definition $\ref{v-exp}$ possesses a edge $k_A$ and $k_B$ that left the same number of edges from $A$ and $B$, respectively, at their left and right sides, since $|A|$ and $|B|$ are  both odd numbers. We have $k_B= k_A + 1+\frac{m-2}{2}\mod{m}$. 

For a choosen edge $k\in\{1,2, \cdots , m\}$ there is $\frac{m-4}{2}$ partitions, $m=|A|+|B|$, for which $k=k_A$. Each \emph{vertex-expansion} from each such partition are the same of those ones obtained taken the edge $k_A +1+\frac{m-2}{2}\mod{m}$. Hence, if that is the only coincidence between all possible \emph{vertex-expasion}, it follows that the total number of vertex-expansion at $v$ is
\[\frac{\text{number of edges at}\, v}{2}\cdot\frac{(m-4)}{2}=\frac{m(m-4)}{4}\]
 
Suppose that $(A_1, B_1)$ and $(A_2, B_2)$ are partitions of the edges incident to $v$ as in Definition $\ref{v-exp}$, that produces the same \emph{vertex-expansion}. 
But, the items $(2)$  and $(3)$ of the definition of the \emph{vertex-expansion operation} $\ref{v-exp}$, implies that $k_{A_1} = k_{A_2}$ or $k_{A_1} = k_{B_2}$ . Therefore, $A_1=A_2$ or $A_1=B_2$. That is the coincidence taken into account previously. Hence, we are done. 

\begin{figure}[H]
 \begin{center}
       \subfloat[around the vertex $v$]
    {{\includegraphics[width=5cm]{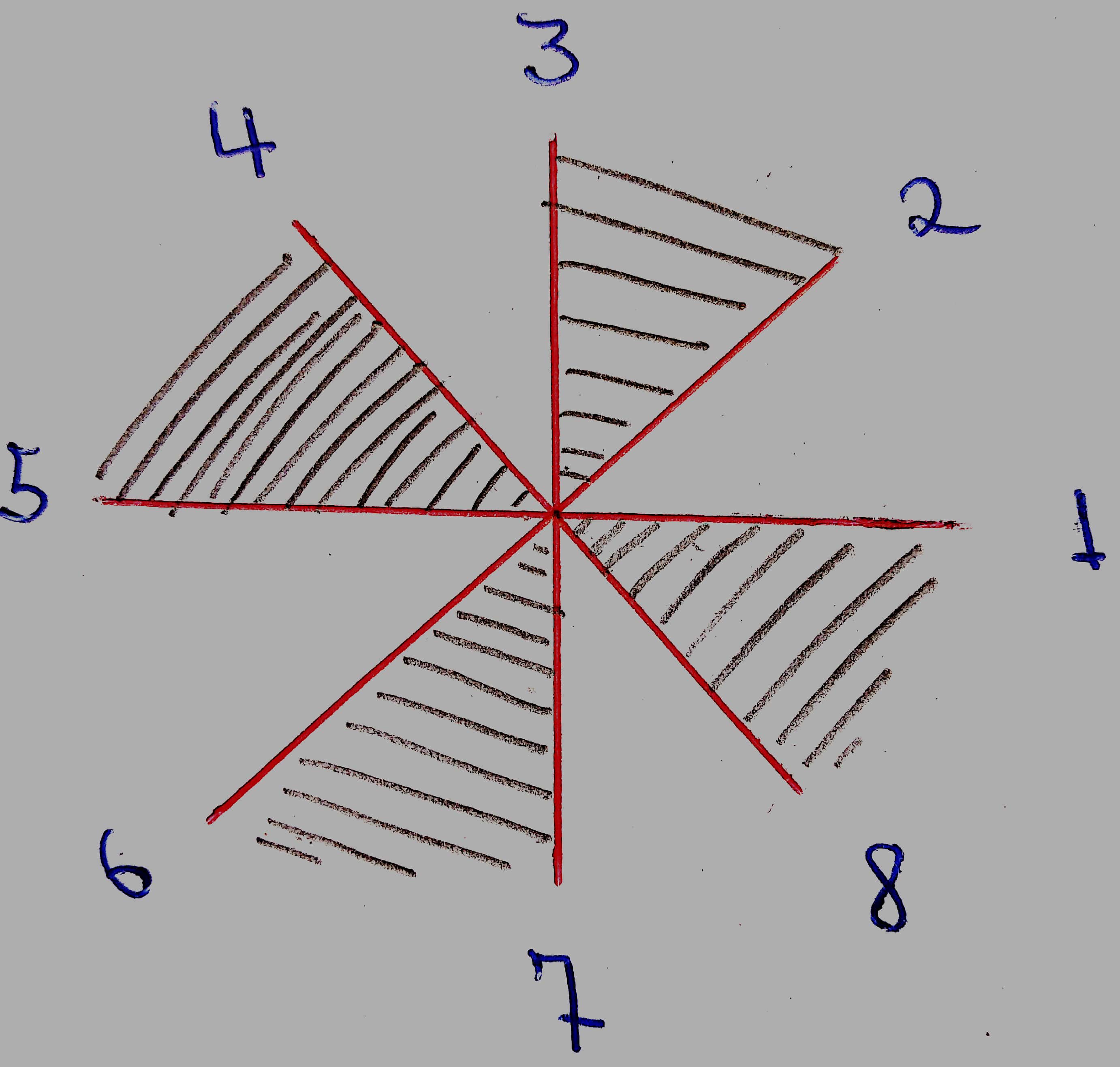} }} 
  \end{center}
\end{figure}
\begin{figure}[H]
 \begin{center}
    \subfloat[ all possible vertex-expansion at $v$]
    {{\includegraphics[width=9cm]{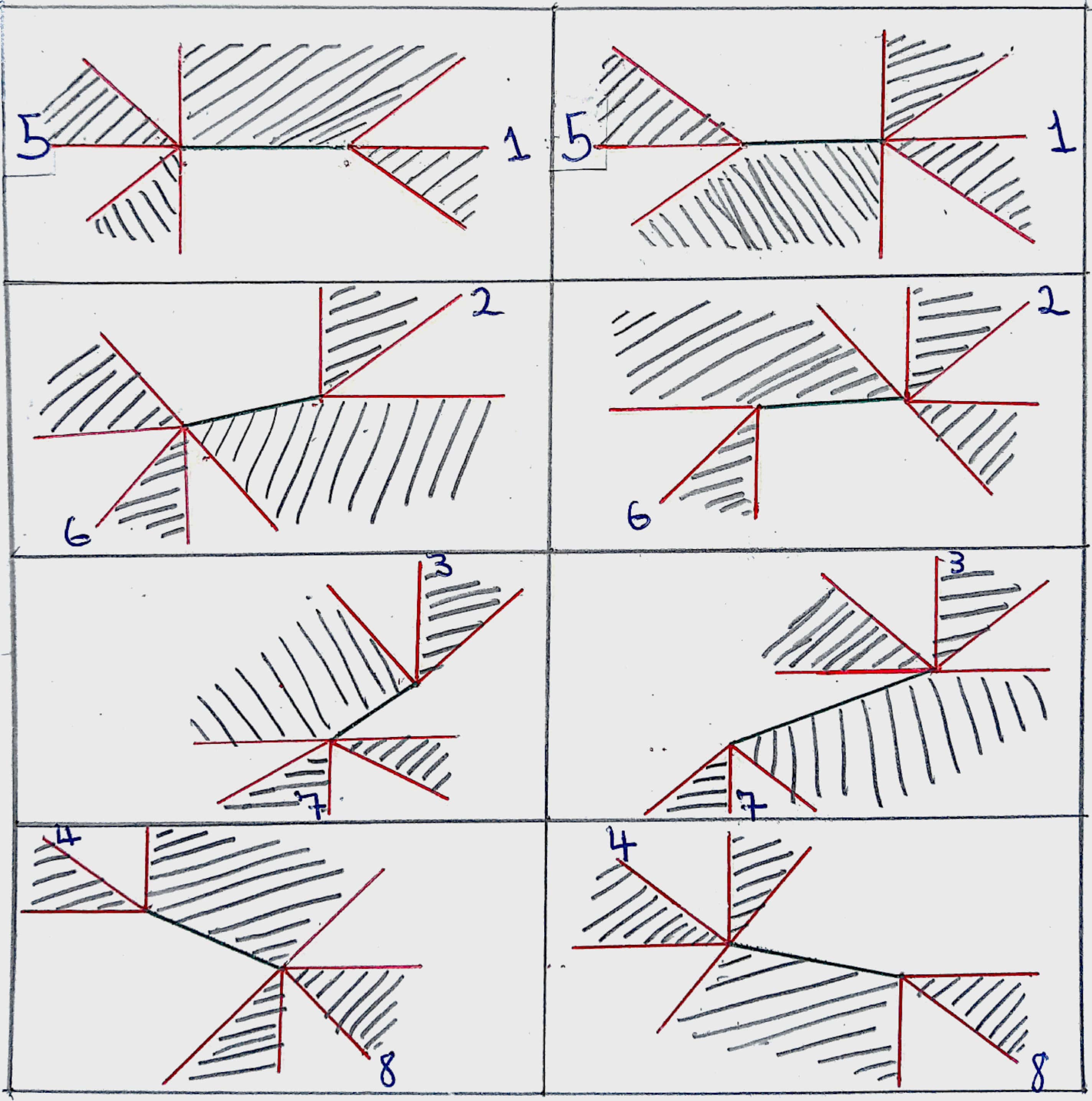} }}   
    \caption{counting vertex-expansion}
\end{center}
\end{figure}

\end{proof}

We must to pay attention to the affect of the \emph{edge-contraction} and \emph{vertex-espansion} operations on the cycles of a balanced graphs.

Let $\Gamma\in\textbf{BG}$ and $\gamma\subset\Gamma$ a cycle.
Let $u\in V(\Gamma)$ and $v\in V(\Gamma)$ corners of $\Gamma$ that are joined by a non separatng saddle-connection $l$ and let $w\in V(\Gamma)$ a corner of degree $2k>4$.

The same is true with respect to the vertex-expasion if $\gamma$ is not incident to $w\in V(\Gamma)$.

\paragraph*{edge-contraction :}\label{e-c-cycle}
\begin{itemize}
\item[\textbf{e-c : 1}]{If $\gamma\subset\Gamma$ does not contain the saddle-connection $l$ or are even incident to $u\in V(\Gamma)$ and $v\in V(\Gamma)$ then a edge-contraction against $l$ will not modify $\gamma$, i. e., it remains as a cycle on the new graph. }

\item[\textbf{e-c : 2}]{If $\gamma\in\Gamma$ is incident to only one of the vertices $u\in V(\Gamma)$ and $v\in V(\Gamma)$. Again, after the edge-contraction on $\l$, $\gamma$ persists as a cycle, since no change on the incidency structure of the cycle occurs.}

\item[\textbf{e-c : 3}]{If $\gamma\in\Gamma$ contains $l$. Again, after the edge-contraction on $l$, $\gamma$ persists as a cycle, since the operation on it simply corresponds to remove a subpath of it and glue the endponts.}

\item[\textbf{e-c : 4}]{But it can also happens that $u\in V(\Gamma)$ and $v\in V(\Gamma)$ be incident to $\gamma$ but with $\gamma$ not containing $l$. In this case, $\gamma$ is pinched at $u\in V(\Gamma)$ and $v\in V(\Gamma)$  resulting into two cycle with 
the new vertex created by the edge-contraction in commom. }
\end{itemize}
\paragraph*{vertex-expansion :}\label{v-e-cycle}
\begin{itemize}
\item[\textbf{v-e : 1}]{If $\gamma$ is not incident to $w\in V(\Gamma)$ a vertex-expasion on it $\gamma$ remains as a cycle, since the operation only changes the incidency structure on $w\in V(\Gamma)$.}

\item[\textbf{v-e : 2}]{If $\gamma$ contains the vertex $w\in V(\Gamma)$, then performing  a \emph{vertex-expansion} on $w$ we can arrive at one of the following two situations:
\begin{itemize}
\item[$\star$]{$\gamma$ persists as a cycle. This happens only if the two edges from $L$ incident to $w\in V({\Gamma})$ belongs to the same subset of the edge partition associated to that \emph{vertex-expansion}. Thus, the local balance condition will be satisfied;}
\end{itemize}
or
\begin{itemize}
\item[$\star\star$]{the cycle $\gamma$ 
is obstructed by the new saddle-connection inserted by the \emph{vertex-expansion} operation at $w$. This happens only if the two edges from $\gamma$ incident to $w\in V({\Gamma})$ belongs to different subset of the edge bipartition associated to the \emph{vertex-expansion}. But such obstruction can always be overcome inserting a path with the compatible orientation (made up by saddle-connections) closing it into a new cycle.}
\end{itemize}}
\end{itemize}

\begin{prop}\label{prop:edge-contr}
A balanced graph of type $(g, d, n-1)$ 
is returned after an \emph{edge-contraction} operation on a balanced graph of type $(g,d, n)$.
\end{prop}

\begin{proof}
Let $\Gamma$ be a balanced graph of type $(g, d, n)$ with a alternating face coloring $A-B$. Let $v,u\in V({\Gamma})$ be two endpoints of the \emph{saddle-connection} where an \emph{edge-contraction} is performed resulting on a new cellularly embedded graph $\Gamma '$ with a distinguished vertex $w\in V(\Gamma')$ obtained by the edge-contraction
.

Since the \emph{edge-contraction} does not changes the genus of the underline surface and the transformed graph still a cellular graph, 
the \emph{Euler-characteristic formula} guarantees the constancy of the number of faces, as each edge contracted decreases by one the cardinals of the vertex and edge sets. 
Also, no changes are made to the face coloring.
Hence, the resulting graph after an \emph{edge-contraction} is gobally balanced of type $(g, d, n-1).$ 


Let's show that $\Gamma '$ satisfeis the local balance.

Let $L$ be a positive cobordant multicycle of $\Gamma '$ that does not contains the distinguished vertex $w\in{V(\Gamma ')}$ either in its interior $R$ or on itself. Then, $L$ corresponds to a positive cobordant multicycle of $\Gamma$, then it satisfy the local balance condition. 

Now, if $L$ contains the vertex $w\in V(\Gamma ')$ into its interor $R$, the argument given above about the constancy of the number of faces together with the fact that the edge-contraction does not affect the face coloring the balance condition is positively verified since $L$ corresponds to a positive cobordant multicycle of $\Gamma$ that contains the saddle-connection to be contracted in its interor.
   
Finally, let $L$ be a  positive cobordant multicycle of $\Gamma '$ with the vertex $w\in V({\Gamma '})$ being incident to some cycle of $L$, say $\gamma '$. 

There is a bunch of possibilities of obtaining $\gamma '$ 
from a cycle $\gamma$ of $\Gamma$. These possibilities are that ones described in \textbf{e-c : 1},\textbf{2},\textbf{3}, and \textbf{4} at page $\pageref{e-c-cycle}$.

If we are into the situation \textbf{e-c : 1} or \textbf{2} or \textbf{3}, then $\gamma'$ to correspond to a cycle $\gamma$ of $\Gamma$ and in this case $L'$ to correspond to a positive cobordant multicycle $L$, therefore $L'$ satisfies the local balance.

But $\gamma$ can also not correspond to a cycle of the pre-operated cellular graph $\Gamma$. In this case it corresponds to one of the cycle created by the \emph{edge-contraction} described in \textbf{e-c : 4} at page $\pageref{e-c-cycle}$.

We can promote the obstructed cycle $\gamma '$ in $\Gamma$ to a positive cycle of $\Gamma$ adding to it the \emph{saddle-connections} adjacents to the $A$ face that is incident to the new edge inserted from the vertex expansion. 

After the \emph{vertex-expansion}

So we conclude the expected.
\end{proof}

\begin{defn}\label{r-cycle}
We will refer to $\rr$ into a real GB-graph as the \emph{real cycle}.
\end{defn}

\begin{prop}\label{ngenrg-f-genrg}
Any non generic real balanced graph is obtained from a sequence of edge-contractions starting from a generic balanced graph of the same degree.
\end{prop}
\begin{proof}
Recall that a generic real balanced graph possesses only corners of degree $4$. Then, for a given non generic real balanced graph we can split a corner of degree greater than $4$ by a sequence of \emph{vertex-expansions} into a collection of $4$-valent vertices being connected by new edges included into the real cycle, therefore preserving the symmetry that a real graphs has. Hence, doing that at each corner of degree greater than $4$ will output a balanced graph symmetric with respect to $\rr$ and with all vertices of degree $4$ and contained in $\rr$, i. e., we will get a generic real balanced graph.

Now, we will clarify the above described appropriate procedure of splitting the vertex by a concatenation of \emph{vertex-expansions}.
 
Consider the real cycle oriented counterclockwise.

Thus, having chosen a vertex of degree $2m\geq 6$ we can choose the real edge $r\subset\rr$ that arrives (regarding the considered orientation on the real cycle chosen) at that vertex. Then we take the predecessor and the successor edges to  $r$ concerning the cyclic order around that vertex, then that two edges together with the edge $r$ will form the set $A$ as in the definition of the \emph{vertex-expansion} operation. And we can operate a \emph{vertex-expansion} creating two new vertices, one with valence $4$ and another one with valence $2m-2$. We repeat this procedure up to left a vertex of degree $4=2m-2k$, where $k$ is the number of \emph{vertex-expansions} applied. See the illustration bellow.

\begin{figure}[H]
 \begin{center}
      \subfloat[]
    {{\includegraphics[width=4.8cm]{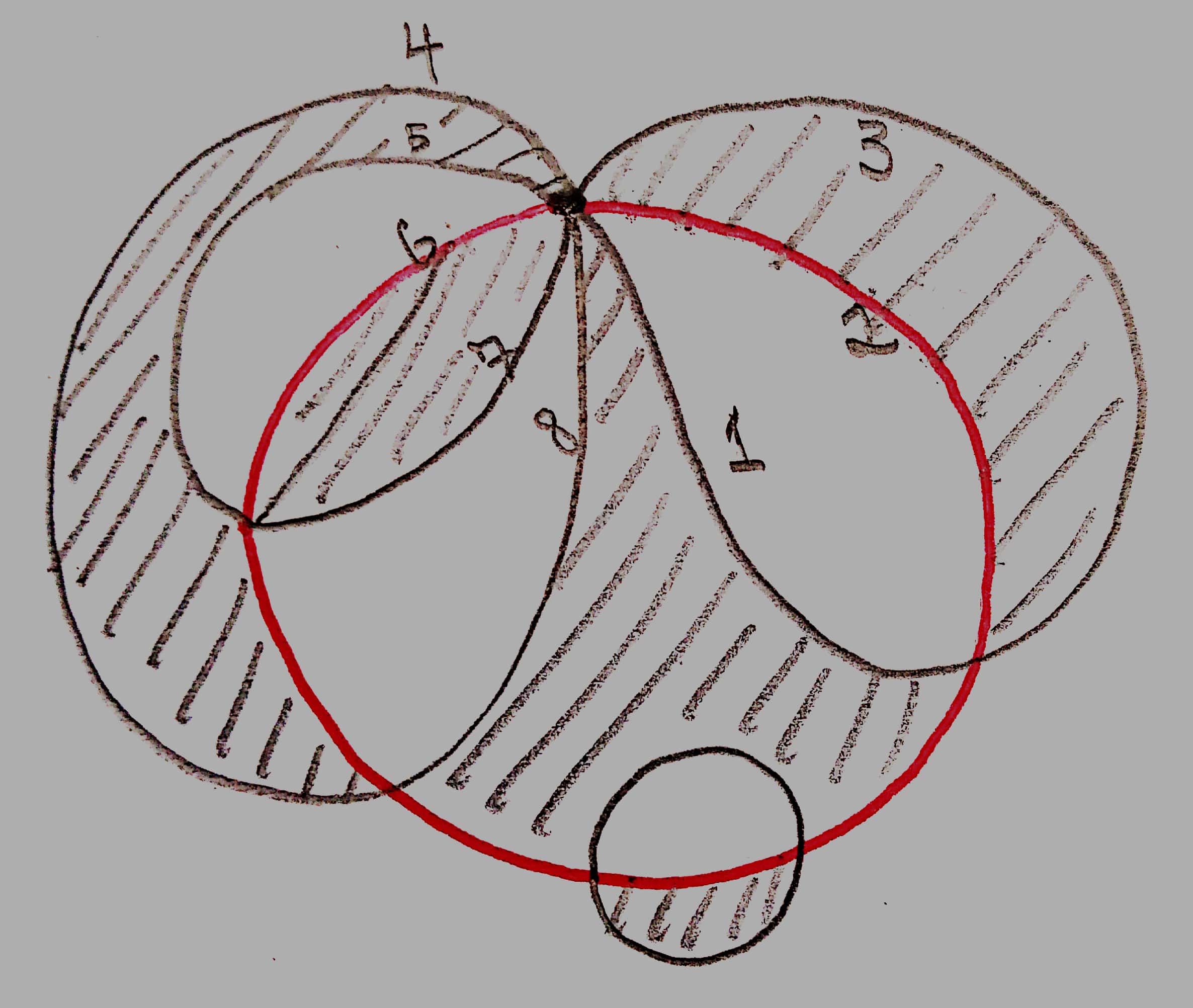}}}
     \qquad \subfloat[]
    {{\includegraphics[width=4.8cm]{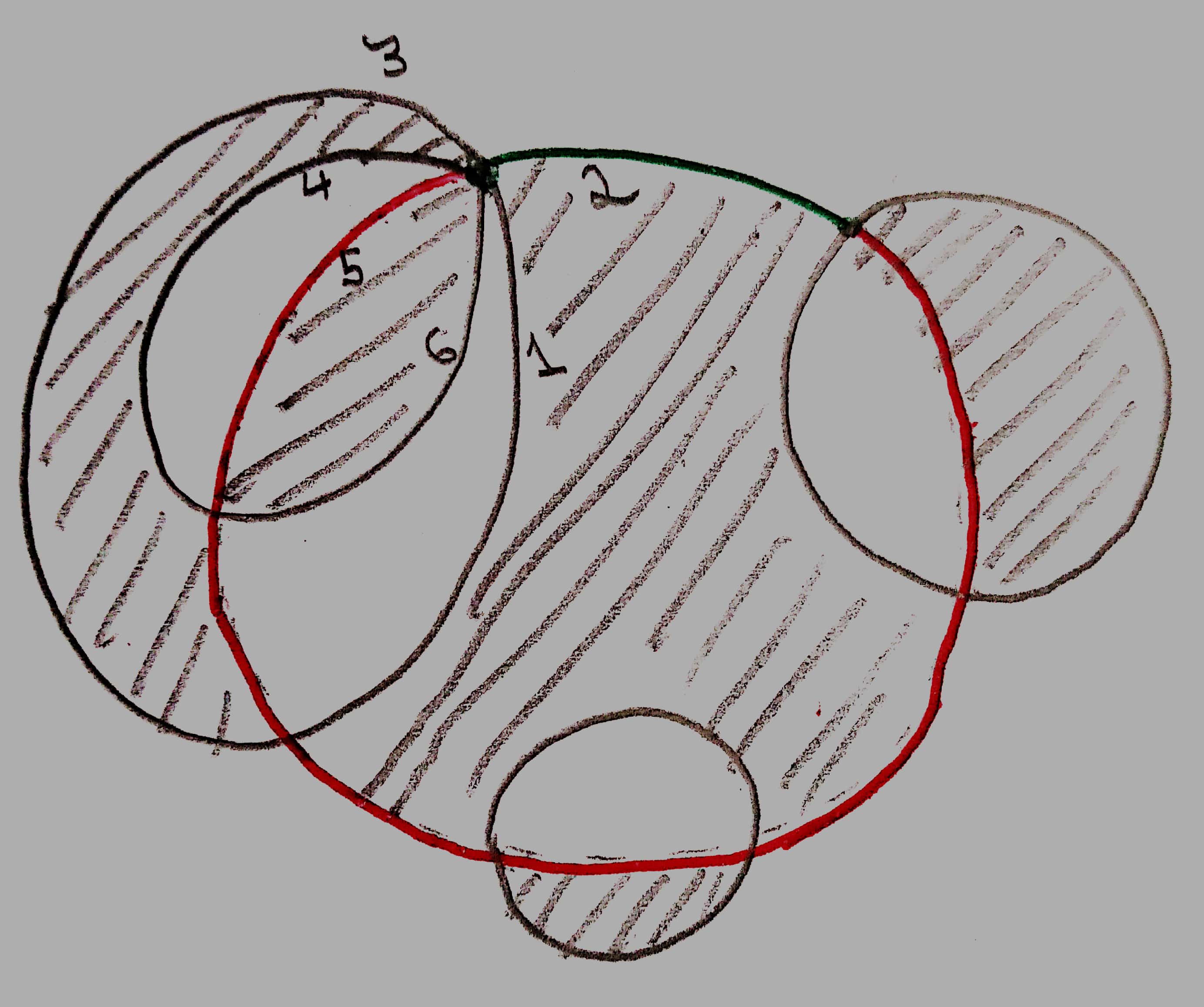}}}
      \qquad \subfloat[]
    {{\includegraphics[width=5.0cm]{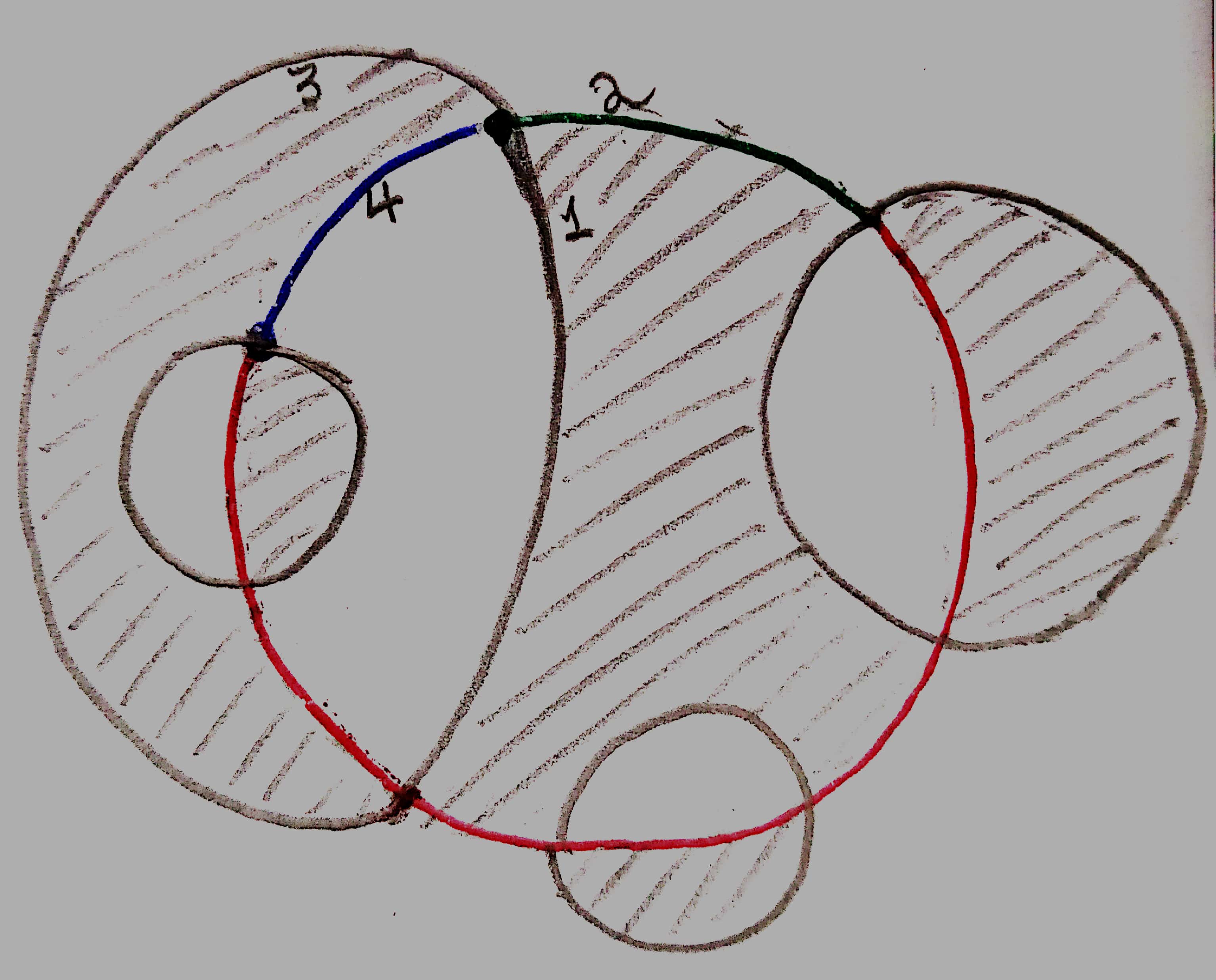}}}
    \end{center}
     \caption[vertex-expansion]{from a nongeneric to a generic real graph via successive vertex-expansion}
\end{figure}

Each \emph{vertex-expansion} has an inverse correspondent operation that is a \emph{edge-contraction}. 

Then the reverse concatenation of that correspondent inverse operarations is a sequence of {edge}-{contractions} that produces the given non generic real balanced graph from a generic one. 

\end{proof}

There is another one operation relative to the contraction of saddle-connections. To introduce it we will single out a special type of corners of a balanced graph. 

\begin{defn}[strongly-connected corners \& simple-pieces]\label{s-candsp}
We say that the endpoints of a splitting saddle-connection of a balanced graph are strongly-connected if they are joined by an odd number, strictly greater than $1$, of saddle-connections that are incident to such corners without gaps turning around them. 

In that situation there is an even number of adjacent faces bounded by such saddle-connections that are incidente to that corners, and by the alternating coloring condition with the same number of faces with each color.

The union of such faces is called by \emph{simple-pieces of} the balanced graph. And the halph of the number of face that it contains is its degree.
\end{defn}


\begin{figure}[H]
 \begin{center}
       \subfloat[degree 1 simple-piece]
   {{\includegraphics[width=5.3cm]{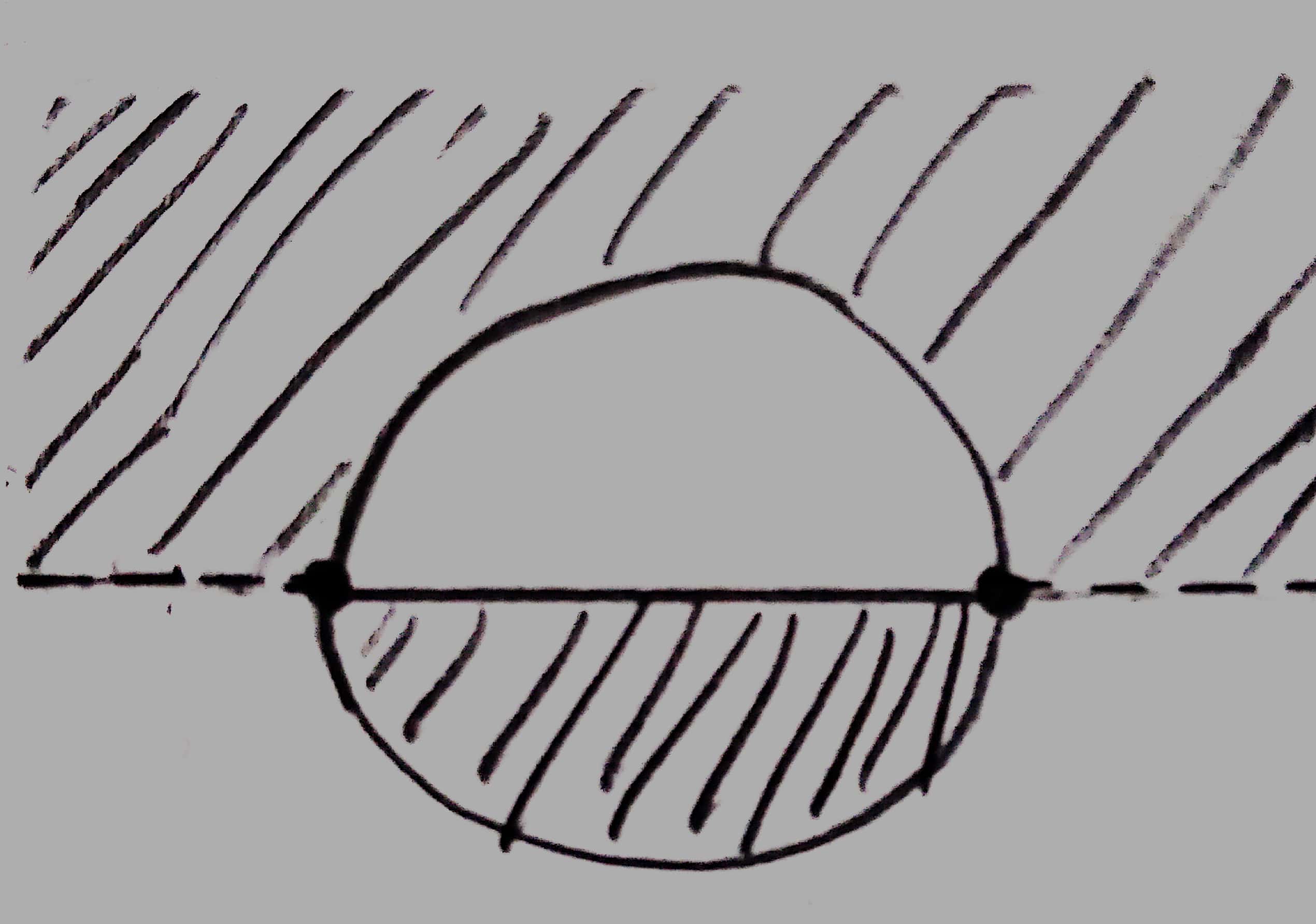}}}
        \quad    \subfloat[degree 3 simple-piece]
    {{\includegraphics[width=5.2cm]{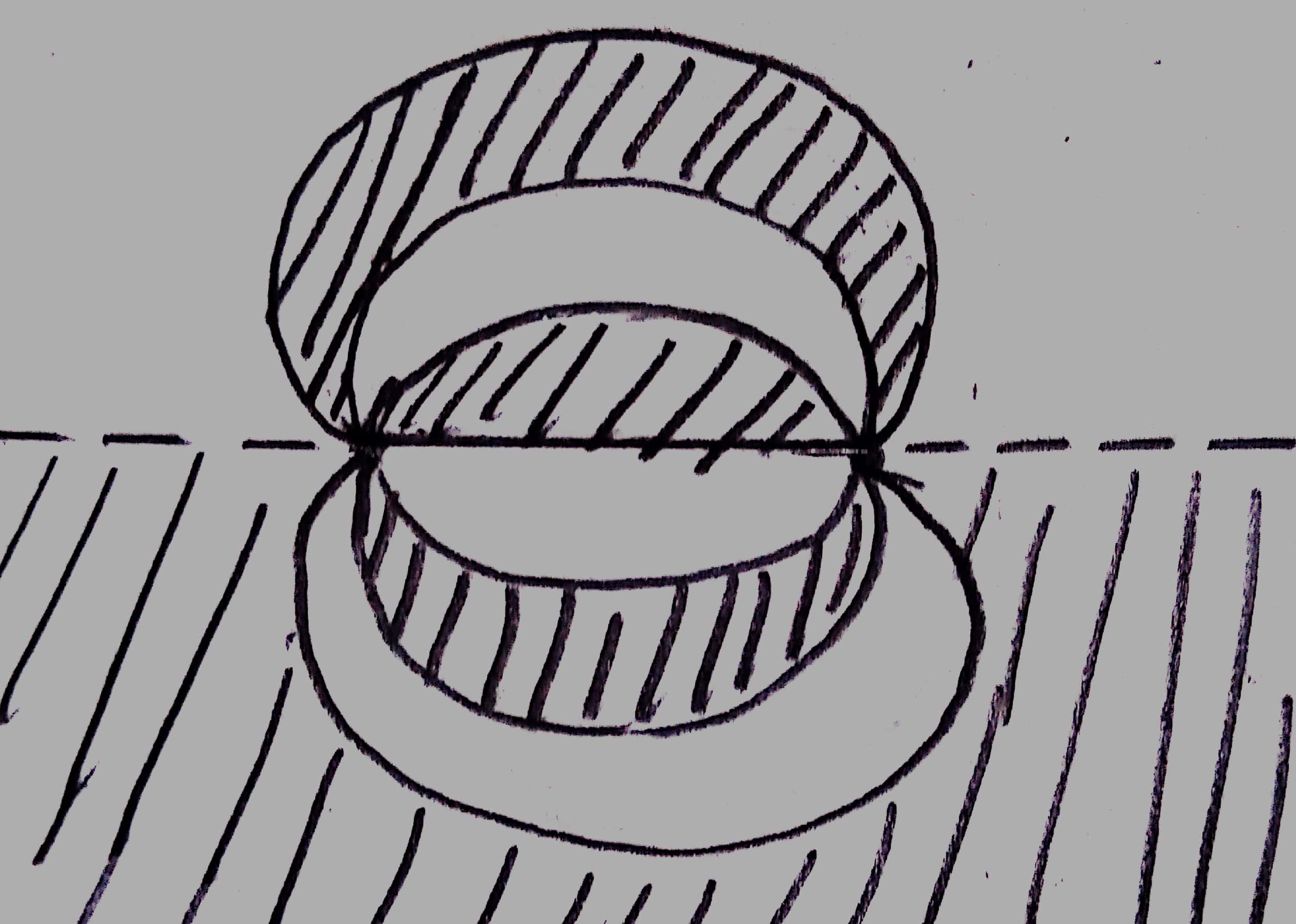}}}
     \end{center}
     \caption{strongly-connected corners \& simple-pieces}
\end{figure}
\rem{The $1$-squeleton of a simple-piece of a balanced graph has appeared elsewhere on the scientific literature, mostly connected to physics, being known with the names: \emph{banana graphs (diagram), dipole graph, sunset diagram}. They actualy consists on a family of Feynman diagrams\cite{ban:09}.}

\subsubsection{face-collapsing} 

\begin{defn}[face-collapsing]
\label{fc-oper}
The operation of \emph{face-collapsing} on balanced graphs consists on the procedure of to remove a simple-piece and then to identify 
the two splitting-corners of that simple-piece.
\end{defn}

Note that a \emph{face-collapsing} does not change the genus of the balanced graph. 
As for those another operations introduced that is visually quite evident that a face collapsing does not changes the genus (we only shrinks to a point a simply-connected region of the underline surfce), but we can quickly check this resorting to the \emph{Euler formula}.Let $\Gamma$ be a balanced graph of type $(g,d, n)$, then:
\[2g_{\Gamma}=2 + |E({\Gamma})|-n-2d\]
But if we collapse a degree $f$ simple-piece, we remove $2f$ faces, $2f+1$ edges and two vertices are identified. Thus the genus of the new graph, say $\Lambda$ will be
\begin{eqnarray*}
2g_{\Lambda}&=&2 + (|E({\Gamma})|-(2f+1))-(n-1)-(2d-2f)\\
&=&2 + |E({\Gamma})|-n-2d\\
&=& 2g_{\Gamma}
\end{eqnarray*}

\begin{figure}[H]
 \begin{center}
      \subfloat[]
    {{\includegraphics[width=6cm]{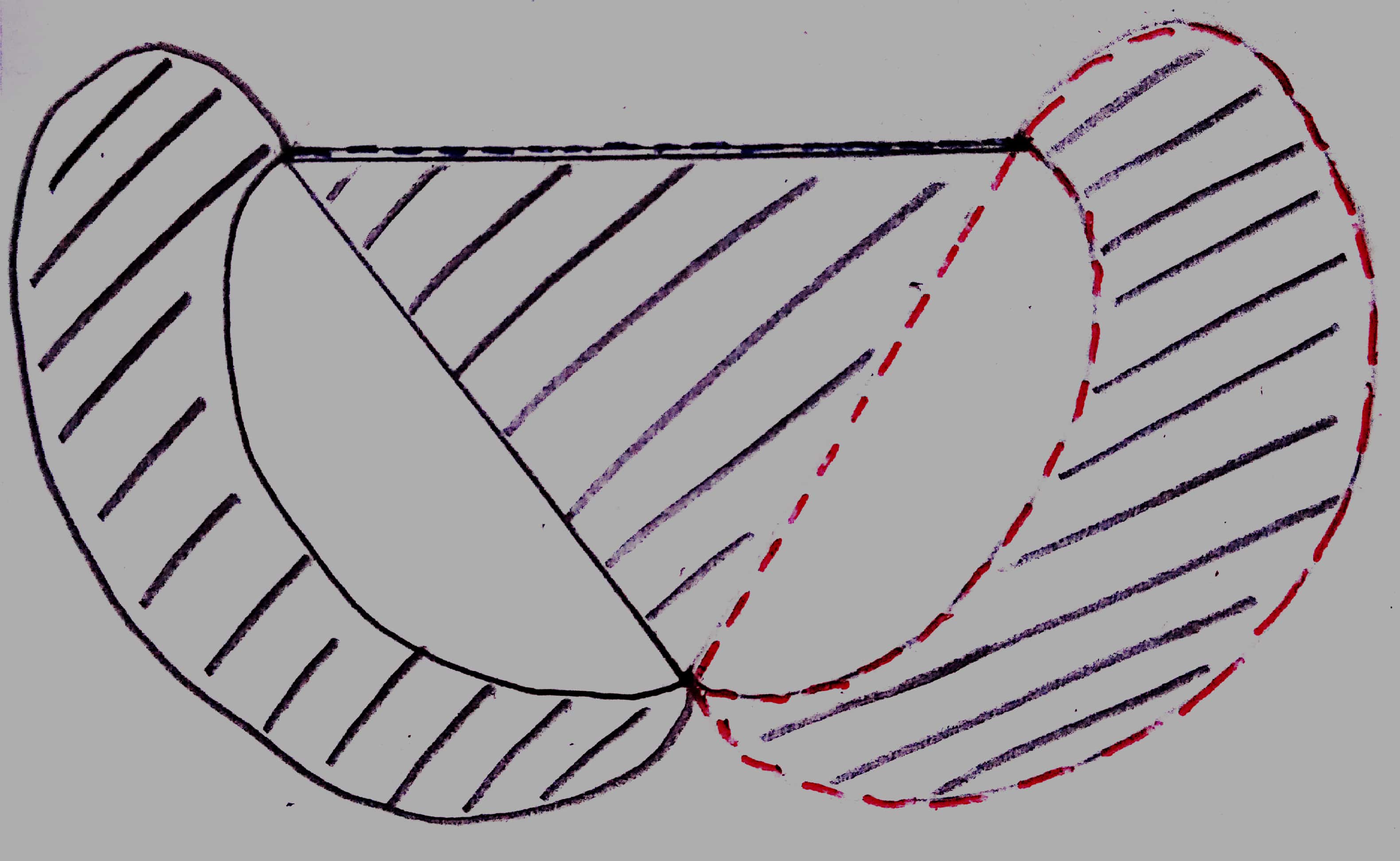}}}
    \qquad
    \subfloat[new balanced graph from (a)]
    {{\includegraphics[width=3.5cm]{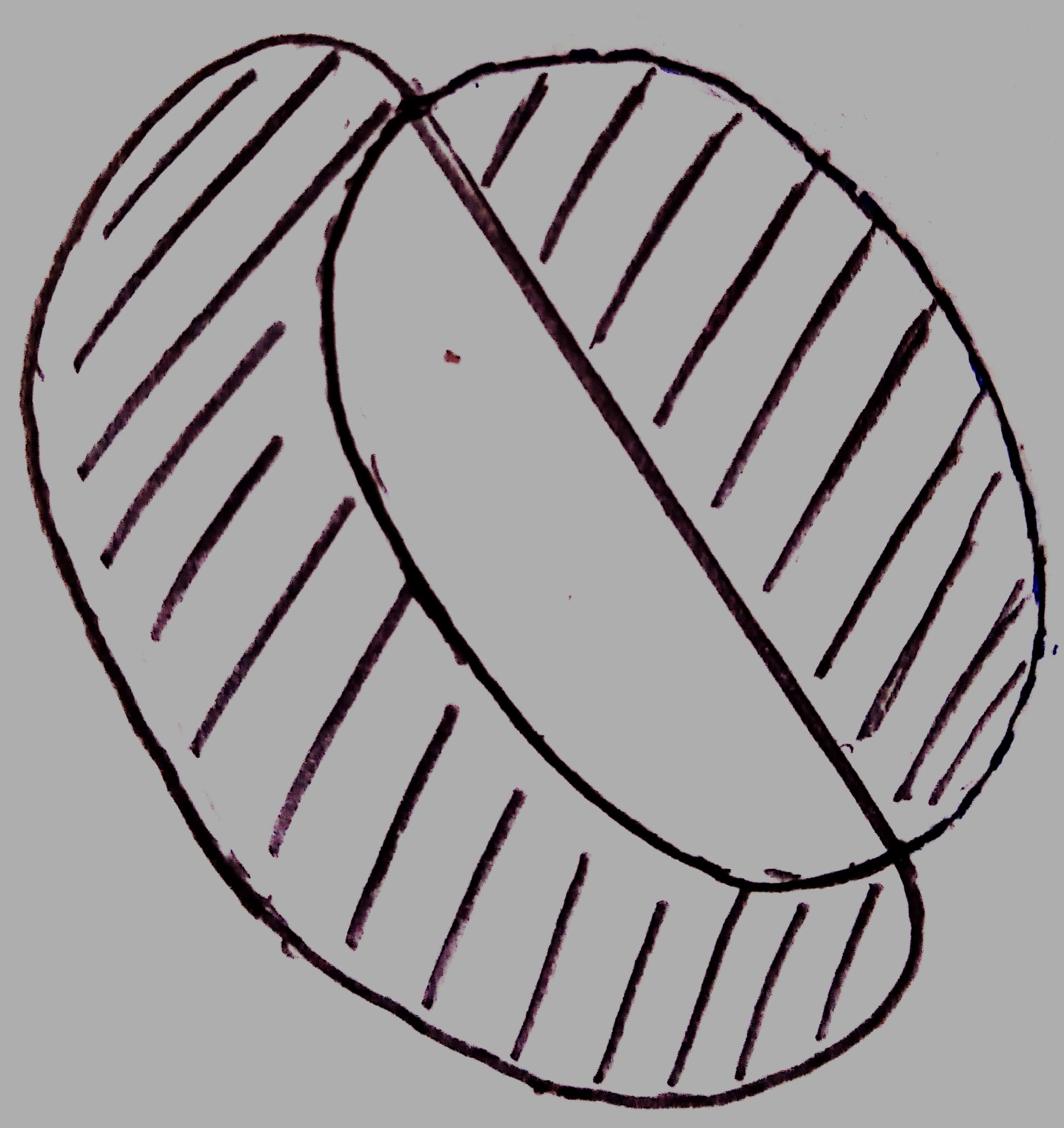}}}
    \end{center}
      \caption{face-collapsing}
 \end{figure}   
  \begin{figure}[H]
 \begin{center}
      \subfloat[]
    {{\includegraphics[width=5.5cm]{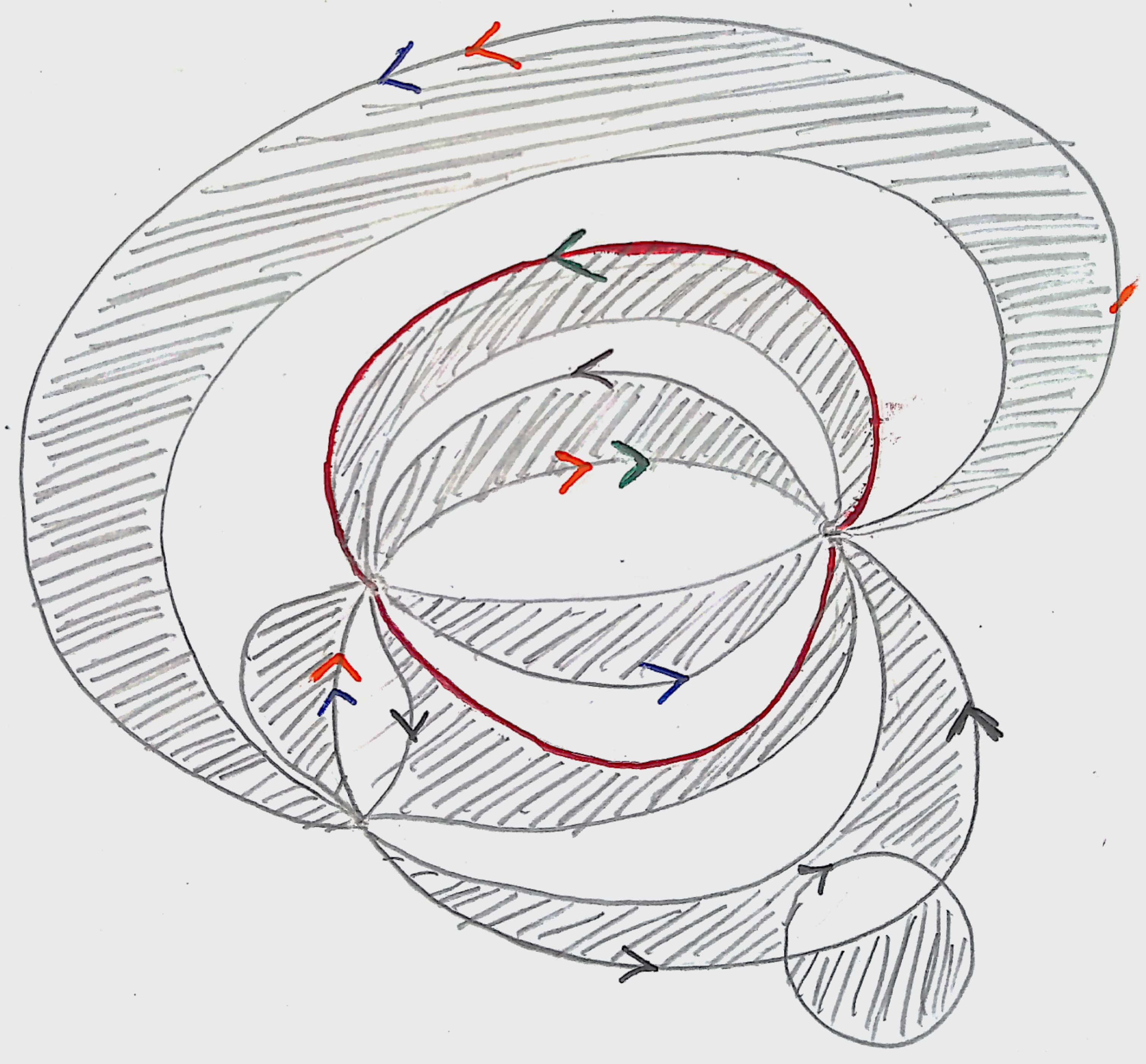}}}
     \qquad \subfloat[]
    {{\includegraphics[width=5.7cm]{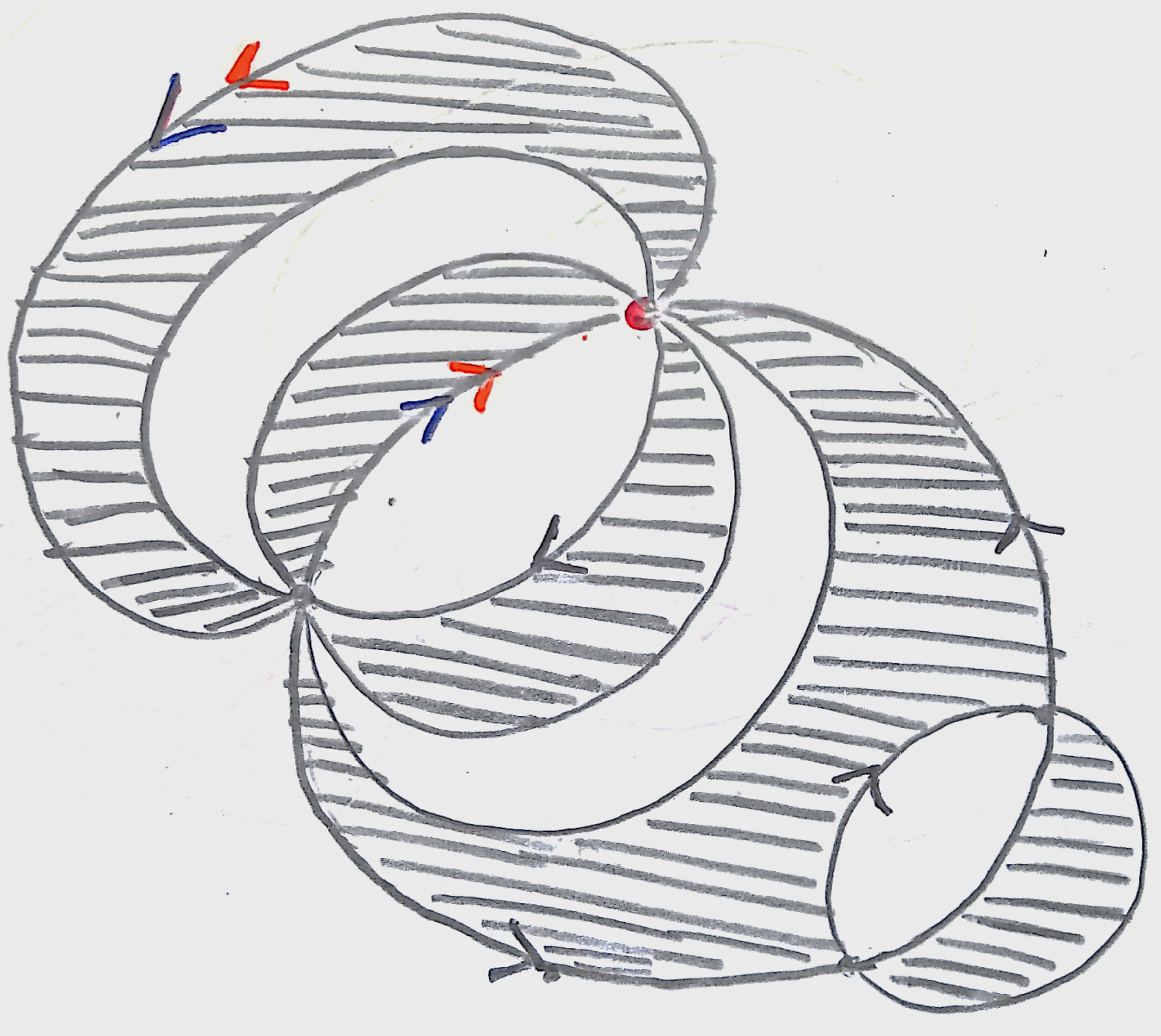}}}
        \end{center}
     \caption{face-collapsing}\label{moving}
\end{figure}

\subsubsection{face-insertion}

The \emph{face-insertion} is the reverse procedure of the \emph{face-collapsing} against a balanced graph. It consists of blowing up a point over a saddle-connection or split a vertex in a simple-peice colored in acordance with the coloring of the pre-operated balanced graph.

We have defined the simple-piece as portion of a balanced graph from a distingueshed incidence structure at two corners. 

To realize the \emph{face-insertion} we need a definition that captures the essence of a simple-piece out. A \emph{simple-peice} looks like a lune tesselation of a compact disk with all lunes sharing its two poles. 

\begin{defn}[simple-piece]\label{m-s-p}
A simple-piece of degree $f$ is the dual graph of a planar bipartite cycle of length $f+2\geq 5$ with $2$ adjacents faces together with its common edge taken out.
\end{defn}

\begin{defn}[face-insertion]
\label{finsert-oper}
The \emph{face-insertion} operation against a balanced graph $\Gamma\in\text{\textbf{BG}}$ consists on the following procedure:
\begin{itemize} 
\item{applied over a saddle-connection:\\
\begin{itemize}
\item[\textbf{1st.}]{to remove a vicinity $U(p)\subset S_g$ of a chosen point $p\in S_g$ over a \emph{saddle-connection} such that $U(p)$ stays contained in the union of those two faces adjacent to that saddle-connection where $p$ lie in;}
\item[\textbf{2nd.}]{to glue a \emph{simple-piece} $P$ to $S_g - U(p)$ identifying the boundaries of $P$ and $S_g - U(p)$ such that each one of the two vertices of $P$ is identified to each one of the two points of the set $\partial{P}\cap E(\Gamma)$ in such a way that the colors of the faces from $\Gamma$ and $P$ made adjacent by the glueing does not match.}
\end{itemize}}
\item{applied over a corner $v\in V(\Gamma)$:\\
\begin{itemize}
\item[\textbf{1st.}]{to split the set of edges incident to $v$ into two subsets of edges, say $A$ and $B$, such that the edges in each subset runs around the original vertex (the corner to be split) with only one gap, and the cardinality of $A$ and $B$ is odd;}
\item[\textbf{2nd.}]{to remove a neighbourhood, $U(v)\subset S_g$, of $v$ and then to shrink to a point the arcs of over the boundary of $S_g - U(v)$ connecting consecutively the edges in $A$ and $B$ creating $2$ new corners, say $a$ and $b$; }
\item[\textbf{3rd.}]{to glue a \emph{simple-piece} $P$ to $S_g - U(v)$ identifying the boundaries of $P$ and $S_g - U(v)$ such that each one of the two vertices of $P$ is identified to each one of the two new corners $a$ and $b$ in such a way that the colors of the faces from $\Gamma$ and $P$ made adjacent by the glueing does not match.}
\end{itemize}}
\end{itemize}
\end{defn}

  \begin{figure}[H]
 \begin{center}
      \subfloat[]
    {{\includegraphics[width=4.2cm]{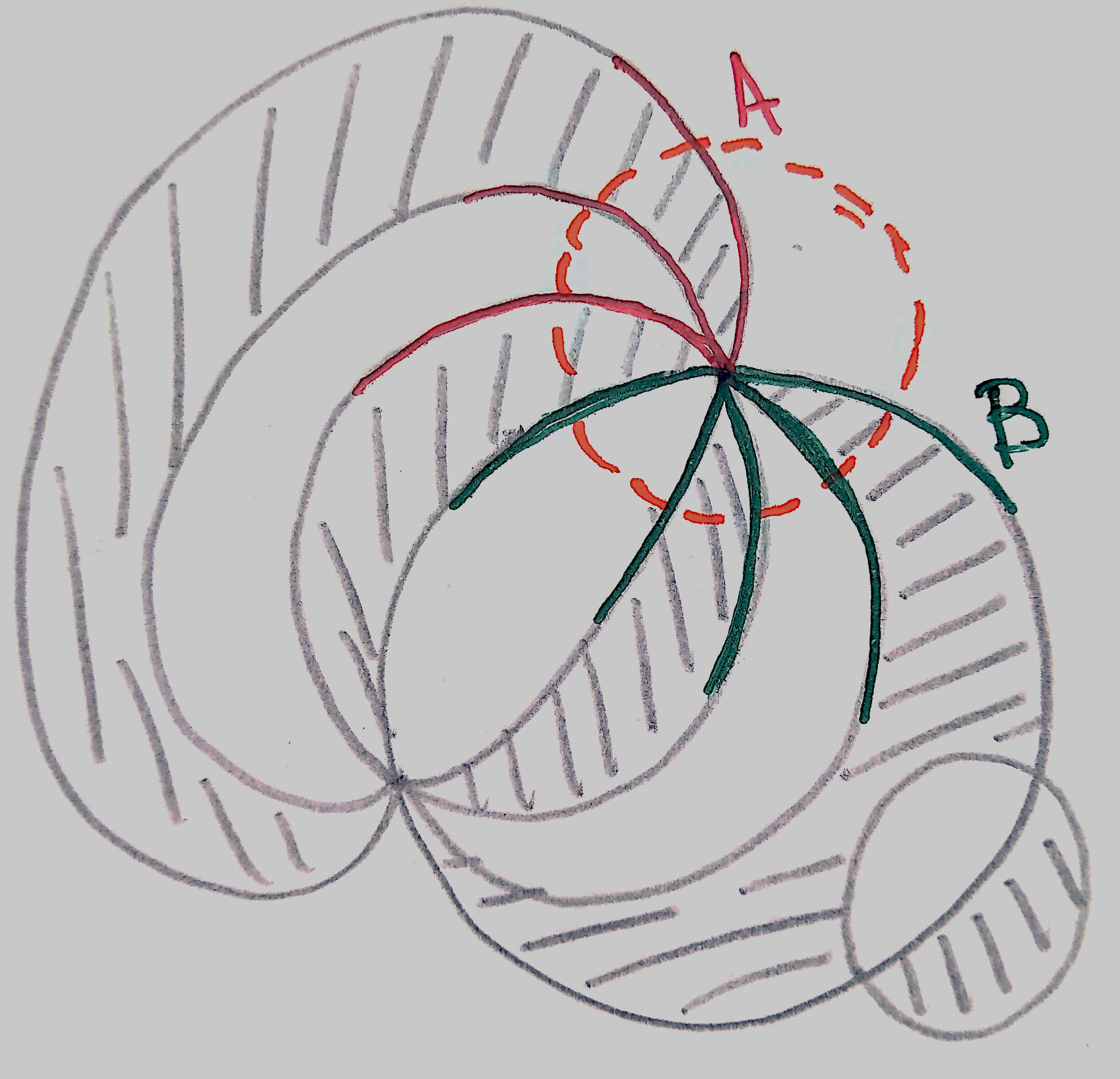}}}
     \quad
     \subfloat[]
    {{\includegraphics[width=3.95cm]{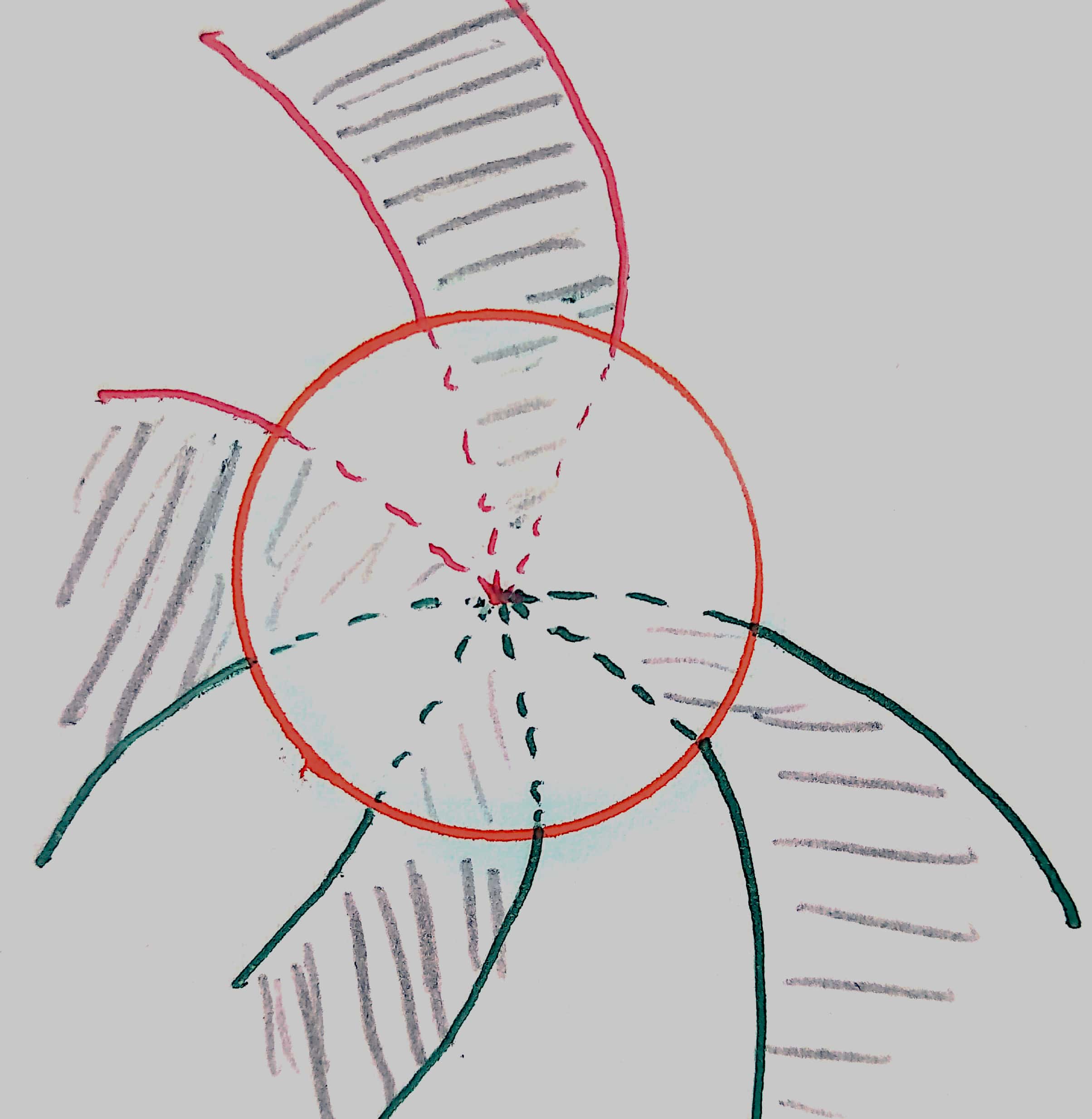}}}\\
    \subfloat[]
    {{\includegraphics[width=3.9cm]{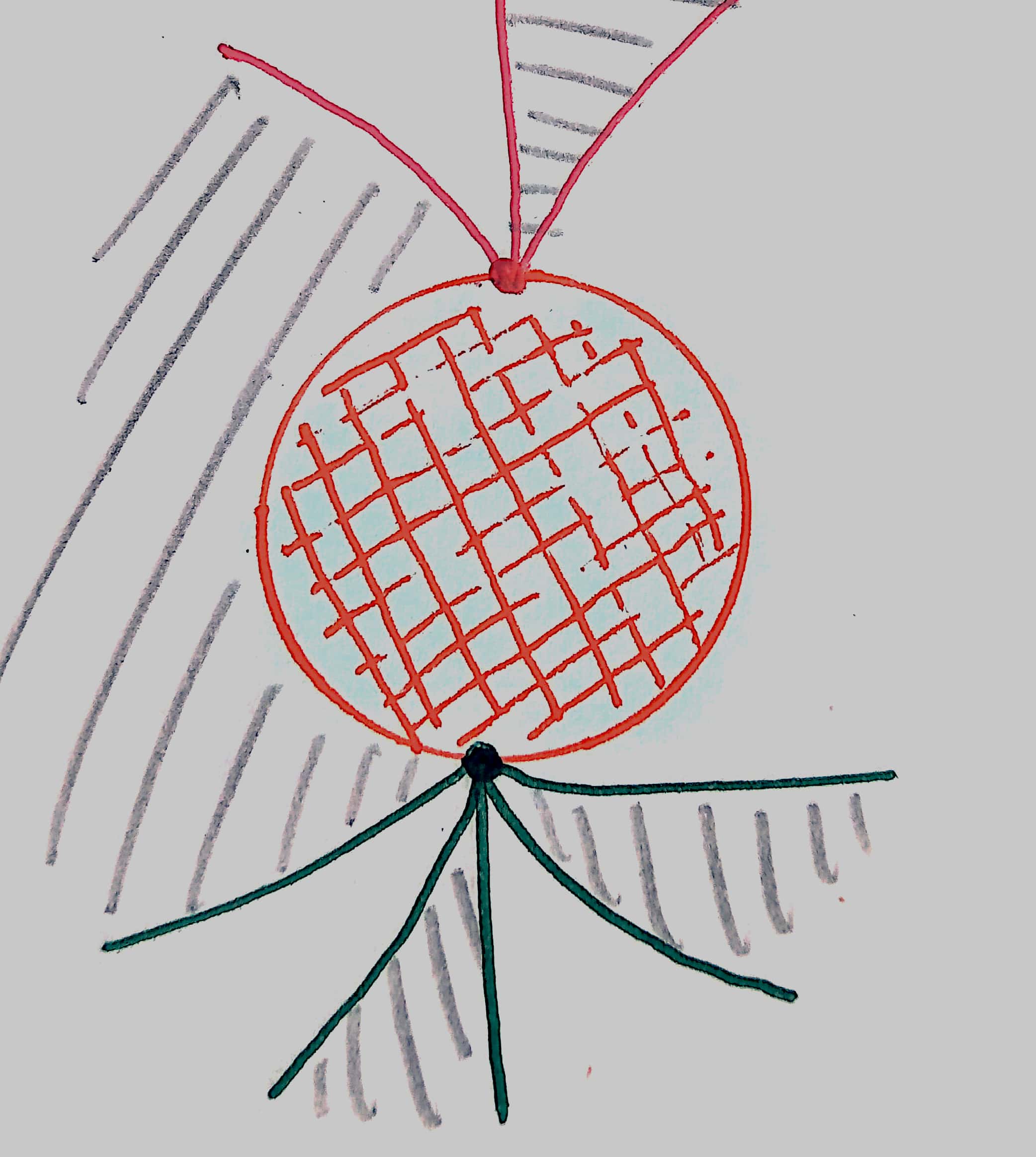}}}
    \quad
     \subfloat[]
    {{\includegraphics[width=3.55cm]{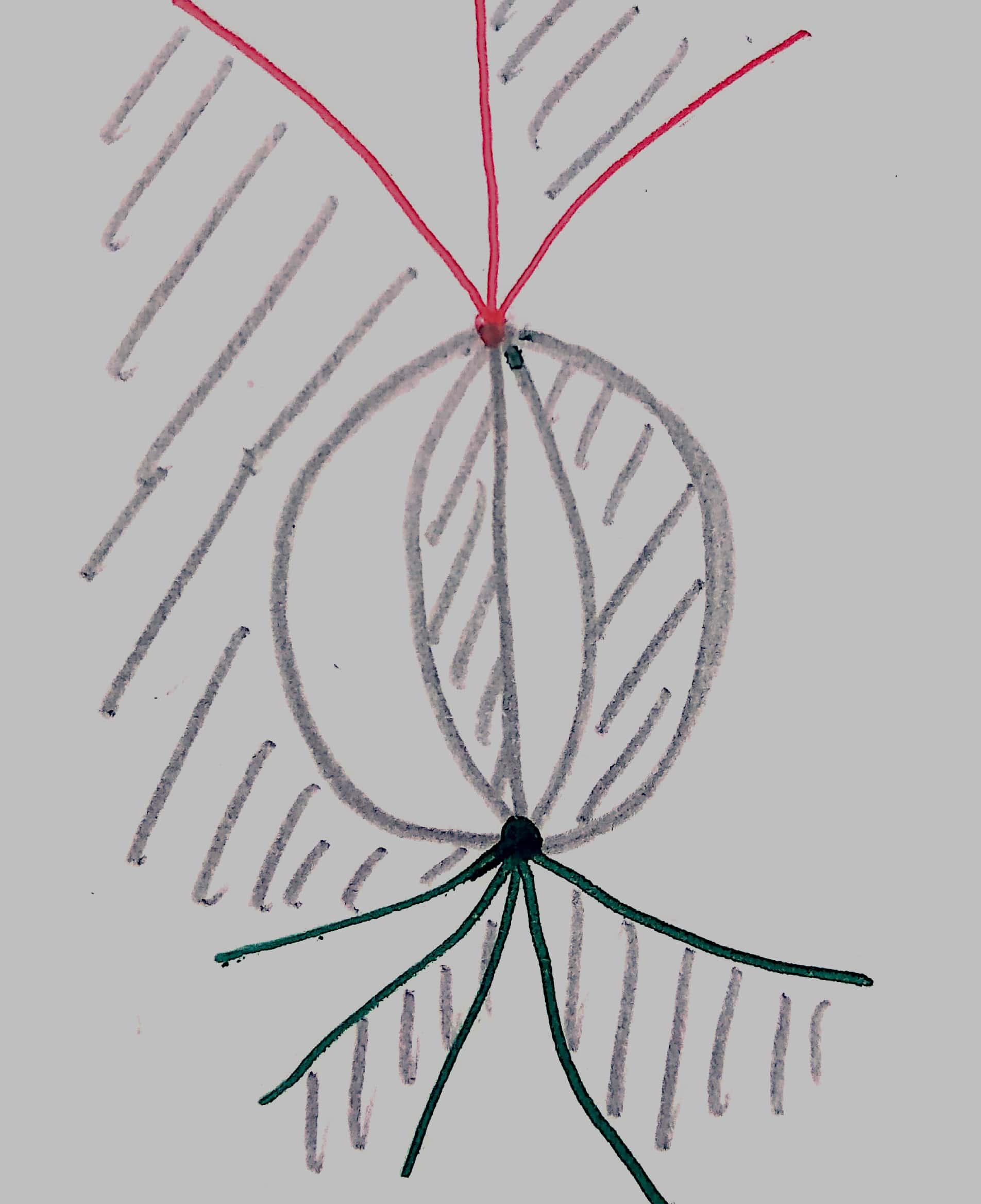}}}\\
        \end{center}
     \caption{face-insertion on a vertex}\label{moving}
\end{figure}

\begin{prop}
A balanced graph is returned after a face-collapsing operation on a given balanced graph. If $d>0$ and $f<d$ are respectively the degree of the balanced graph $\Gamma$ and the degree of one simple-piece, $P$ of $\Gamma$, then the face-collapsing against $\Gamma$ at $P$ gives a new balanced graph of degree $d-f.$
\end{prop}
\begin{proof}
Let $\Gamma$ and $P$ as announced above.

Since the number of faces with each color in a \emph{simple-piece} is the same, it follows that the \emph{face-collapse} operation does not obstruct the global balance condition. And, of course, the number of faces on the new embedded graph after a \emph{face-collapse} at a \emph{simple-piece} of degree $f$ will be $2d-2f$. Therefore, we obtain a globally balanced graph of degree $d-f$ after a \emph{face-collapse} at \emph{simple-piece} of degree $f$.

Consider $\Gamma\in (g, d, n)$ with a \emph{Black-White} alternating coloring  (then, \emph{Black} is the prefered color). And let $\Lambda$ be the globally balanced graph obtained from $\Gamma$ by the \emph{face-collapse} at $P$.

To guarantee the local balance we have to atest the condition only for those positive cobordant multicycle of $\Lambda$ that contains the vertex resulted from the collapsing of that \emph{simple-piece} $P$. Recal that \underline{contain} here means that it belongs to the same component of $S_g - \Gamma$ that contains the prefered color at the left side of the choosed cycle. We call that component by interior of positive separating cycle.

Let $w\in V(\Lambda)$ be the vertex arising from the face-collapse at $P$ and $L$ be one positive cobordant multicycle of $\Lambda$ with a cycle $\gamma$ passing through $w\in V(\Gamma)$ .
 
But it is clear that $ L $ satisfies the condition of local balance since by performing the inverse \emph{face-insertion}, we obtain a positive cobordant multicycle of $ \Gamma $ that projects over $ L $ by removing the same amount of black and white faces.  

\end{proof}

\subsubsection{Balanced move}
Now we present another possible operation over balanced graphs that we will call the \emph{balanced move}. This operation was discovered through computational tests when we tried to perceive the changes in the \emph{pullback graphs} (they are balanced graph) regarding the isotopy classes of post-critical curves.

\begin{defn}[balanced move]\label{bal-m}
Let $\Gamma$ be a balanced graph.

For a pair of corners of $\Gamma$, say $p_1$ and $p_2$, connected by only one non splitting saddle-connection we set $F_1$ and $F_2$ to be the two faces incident to that non splitting saddle-connection. 

The operation \emph{balanced move} against $\Gamma$ (regarding to $p_1$ and $p_2$) consists on the procedure of to choose one (outermost) saddle-connection incident to $p_1$ and another one incident to $p_2$ such that one is incident to $F_1$ and the other to the face $F_2$ and then to exchange their end points $p_1$ and $p_2$. Consult ilustration $\ref{moving}$.
\begin{figure}[H]
 \begin{center}
      \subfloat[]
    {{\includegraphics[width=6.15cm]{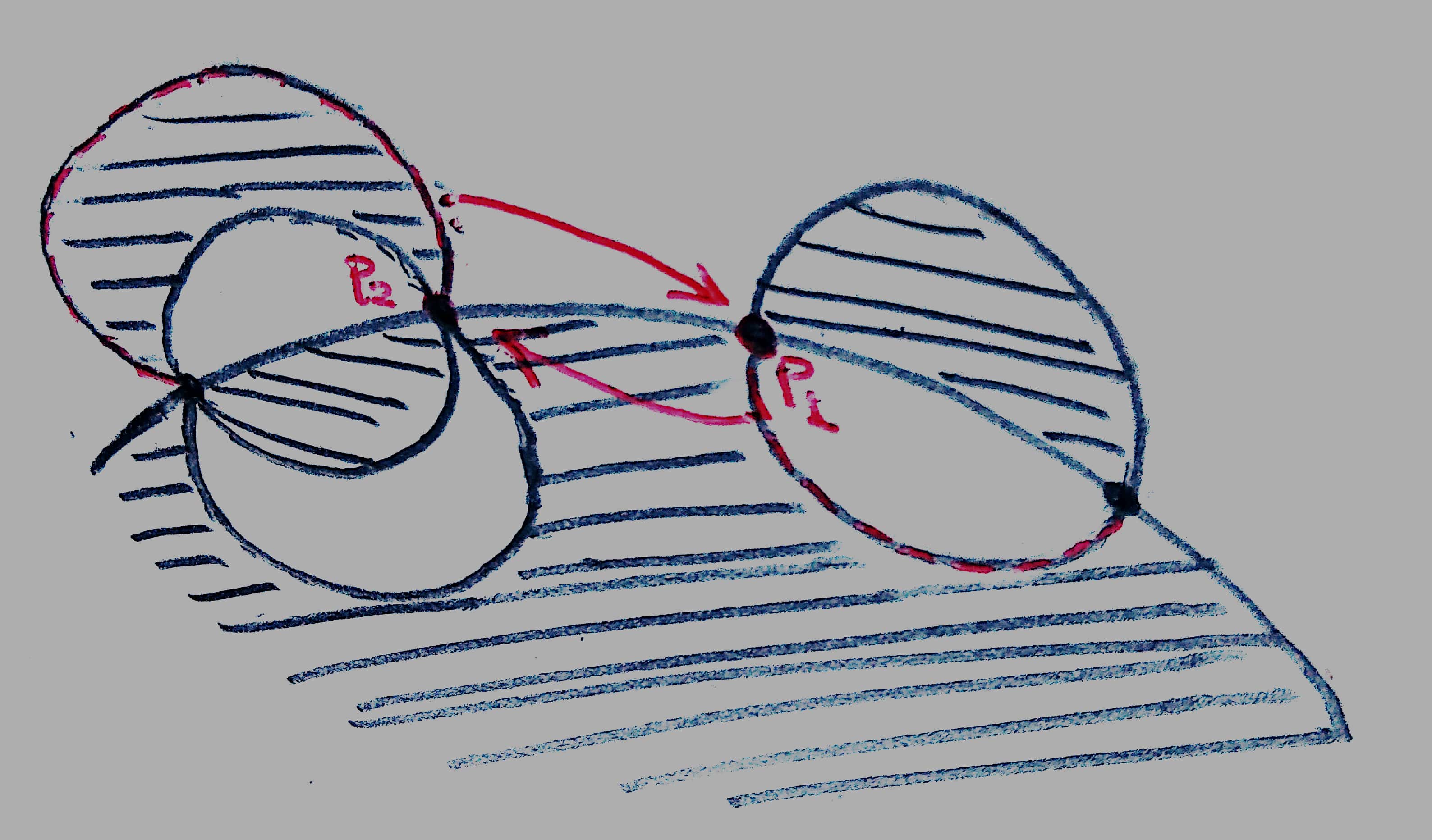}}}
    \hspace{0.06cm}
      \subfloat[]
    {{\includegraphics[width=5.05cm]{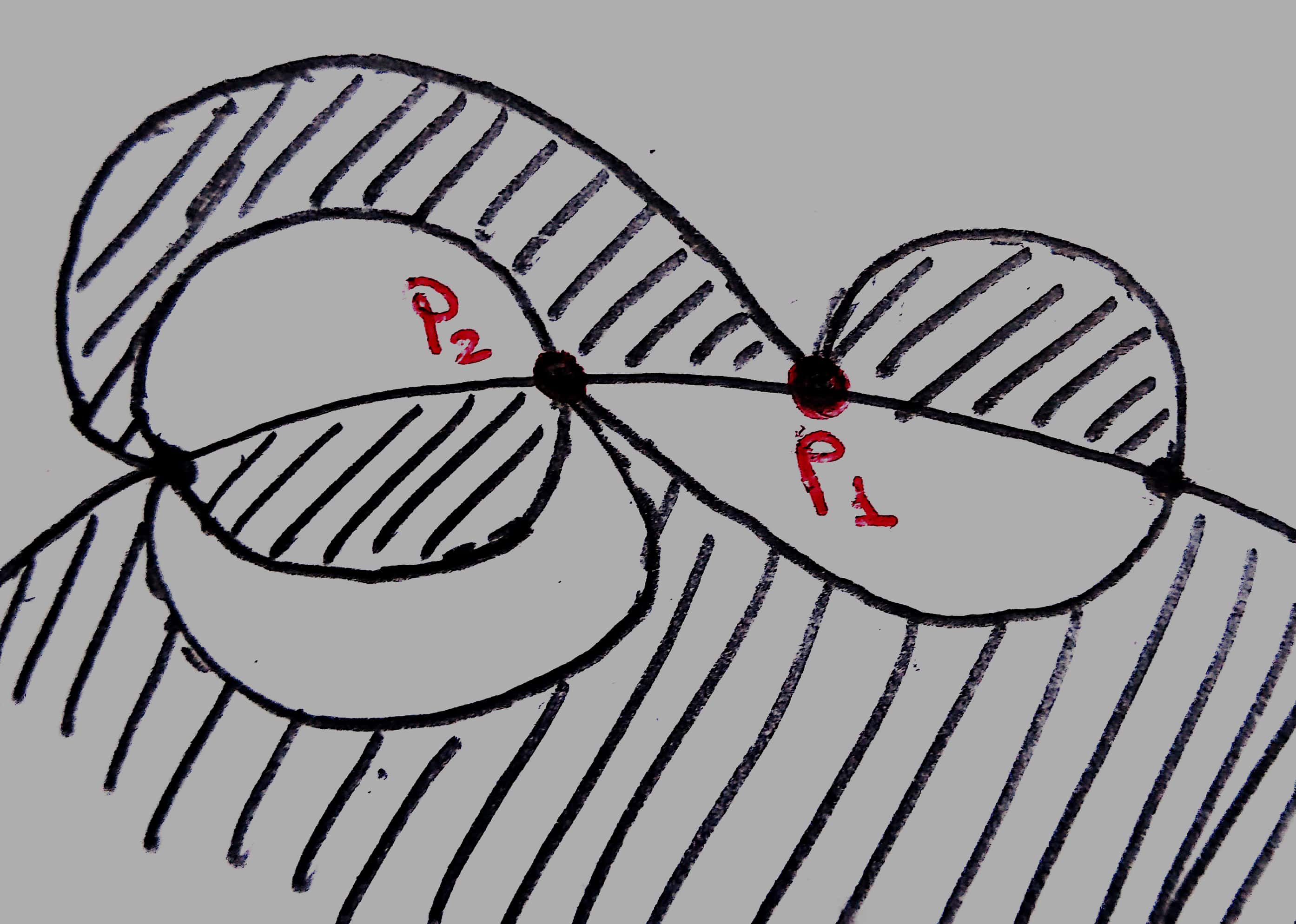}}}\\
    \subfloat[]
    {{\includegraphics[width=5.45cm]{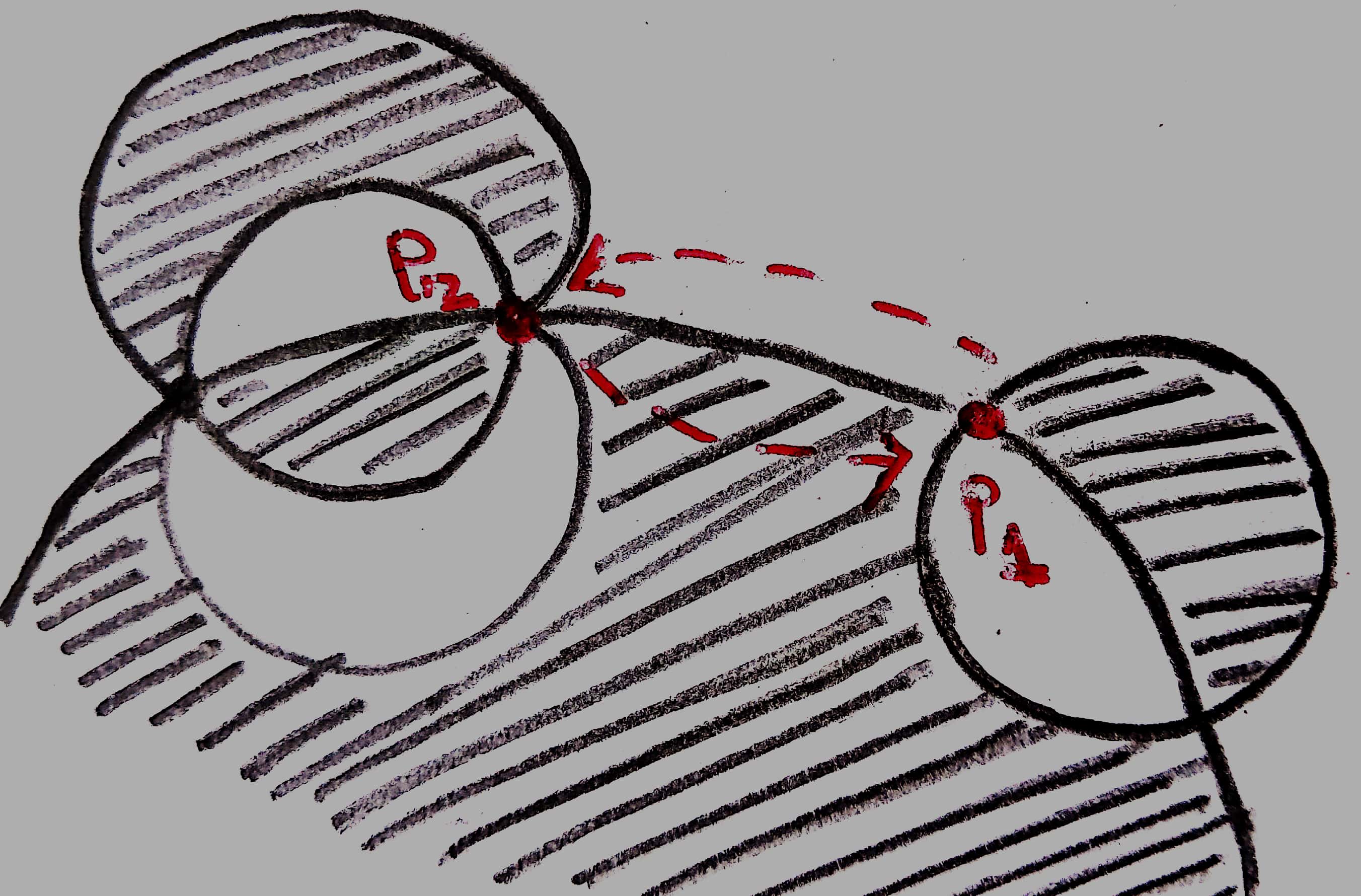}}}
     \hspace{0.06cm}\subfloat[]
    {{\includegraphics[width=5.765cm]{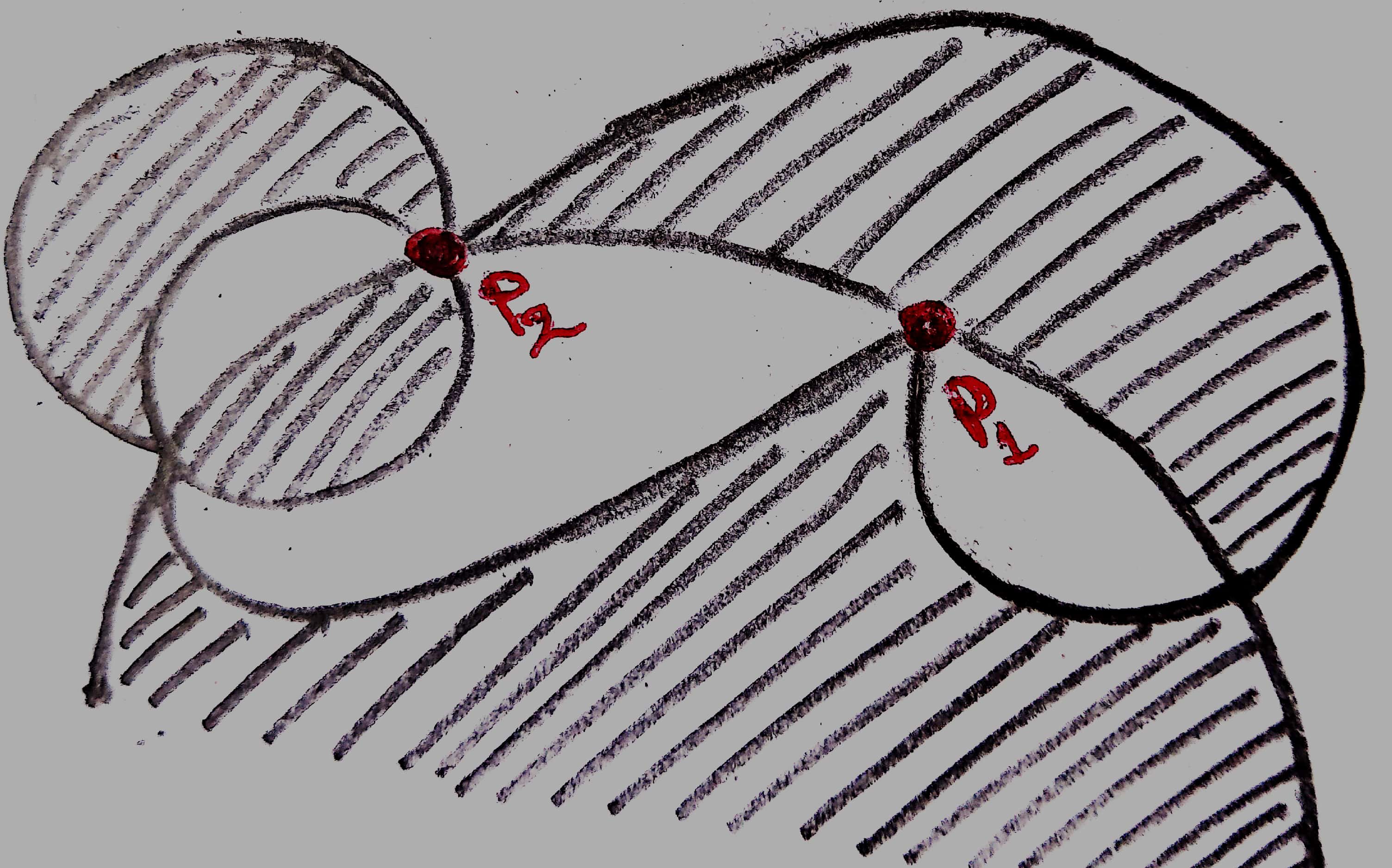}}}\\    
    {   \subfloat[]
    {{\includegraphics[width=4.05cm]{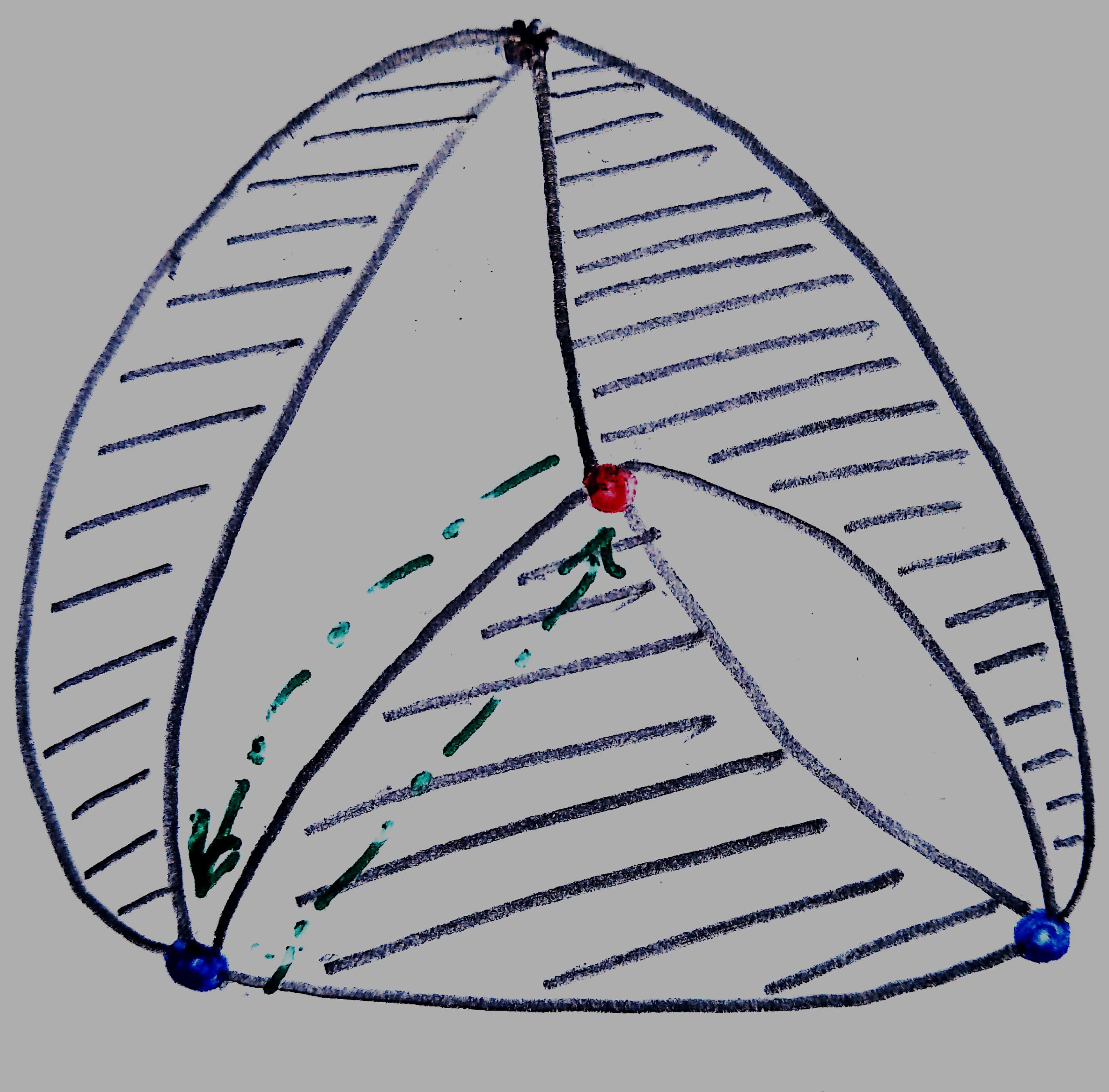}}}\hspace{0.06cm}
    \subfloat[]
    {{\includegraphics[width=4.03cm]{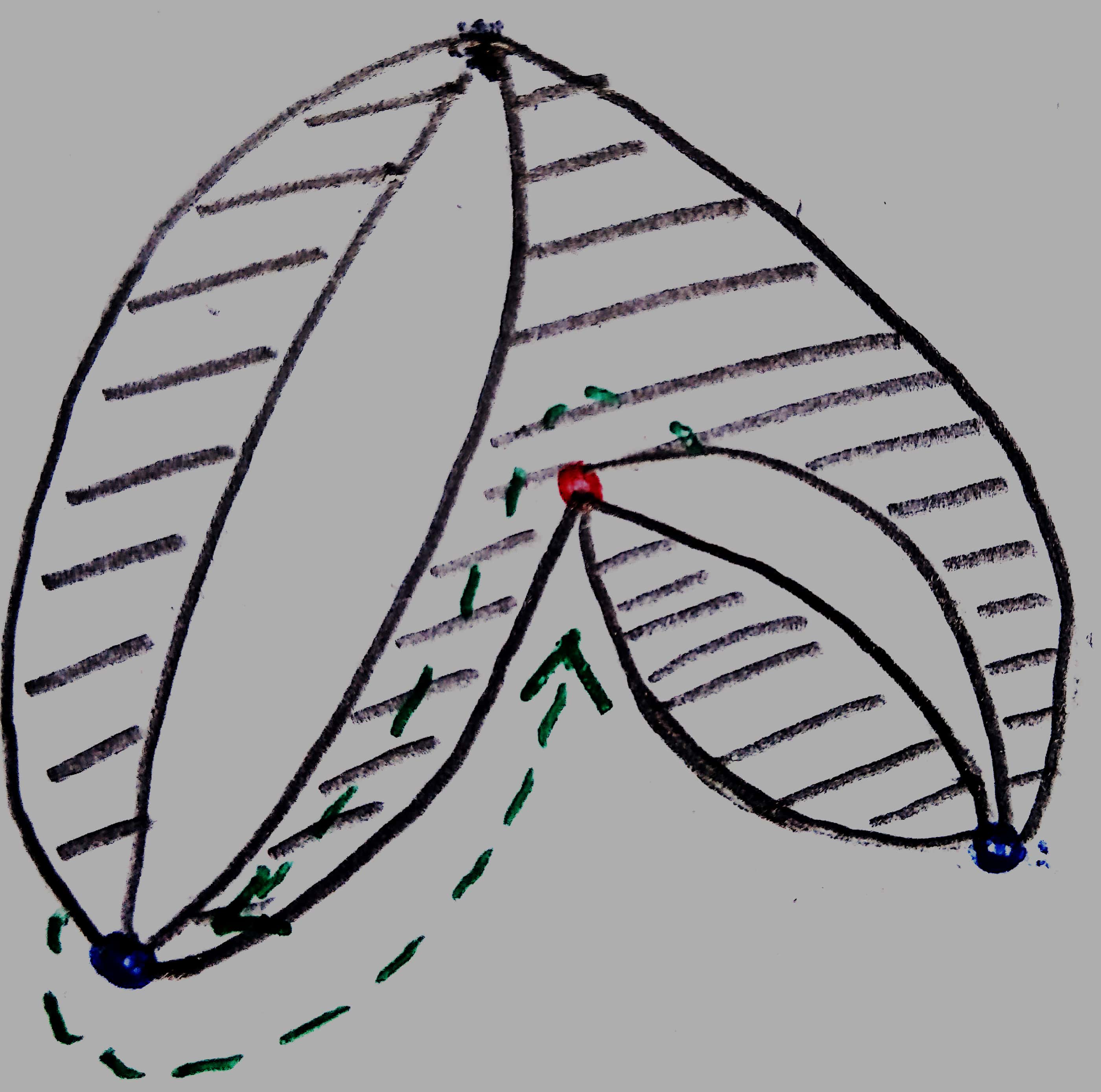}}}\hspace{0.06cm}
     \hspace{0.04cm}\subfloat[]
    {{\includegraphics[height=4cm]{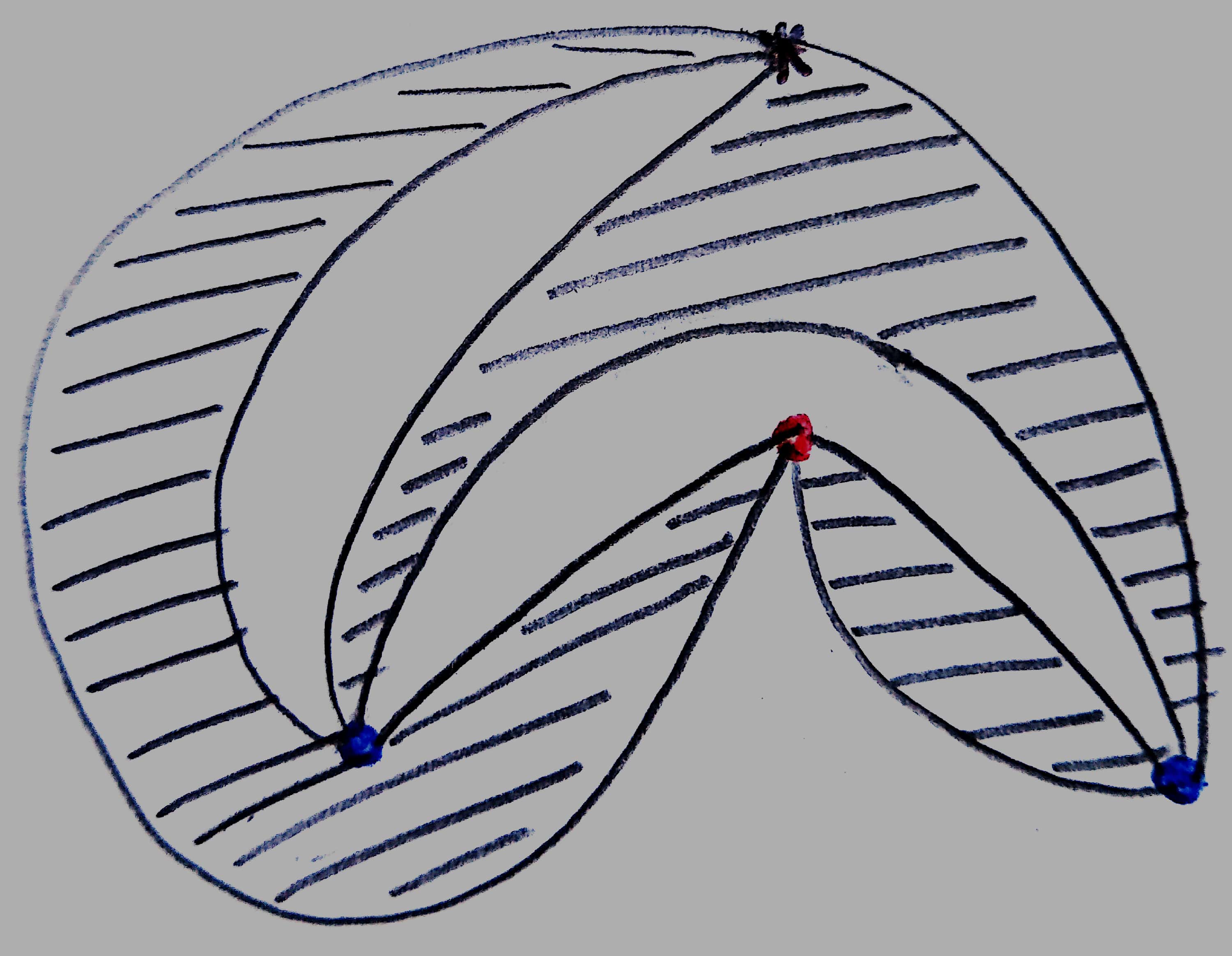}}}}
        \end{center}
     \caption{balanced move}\label{moving}
\end{figure}
\end{defn}

Now we introduce the inverse procedure to the \emph{balanced move}.

\subsubsection{reverse balanced move}

Note that any \emph{balanced move} have a inverse operation. That inverse  operation is simply the balanced move corresponding to moving back the saddle connections formerly modifyed.

\begin{figure}[H]
 \begin{center}
      \subfloat[undoing with a balanced move that balanced move ilustrated on Figure $\ref{moving}$\textbf{(a)}]
    {{\includegraphics[width=4.8cm]{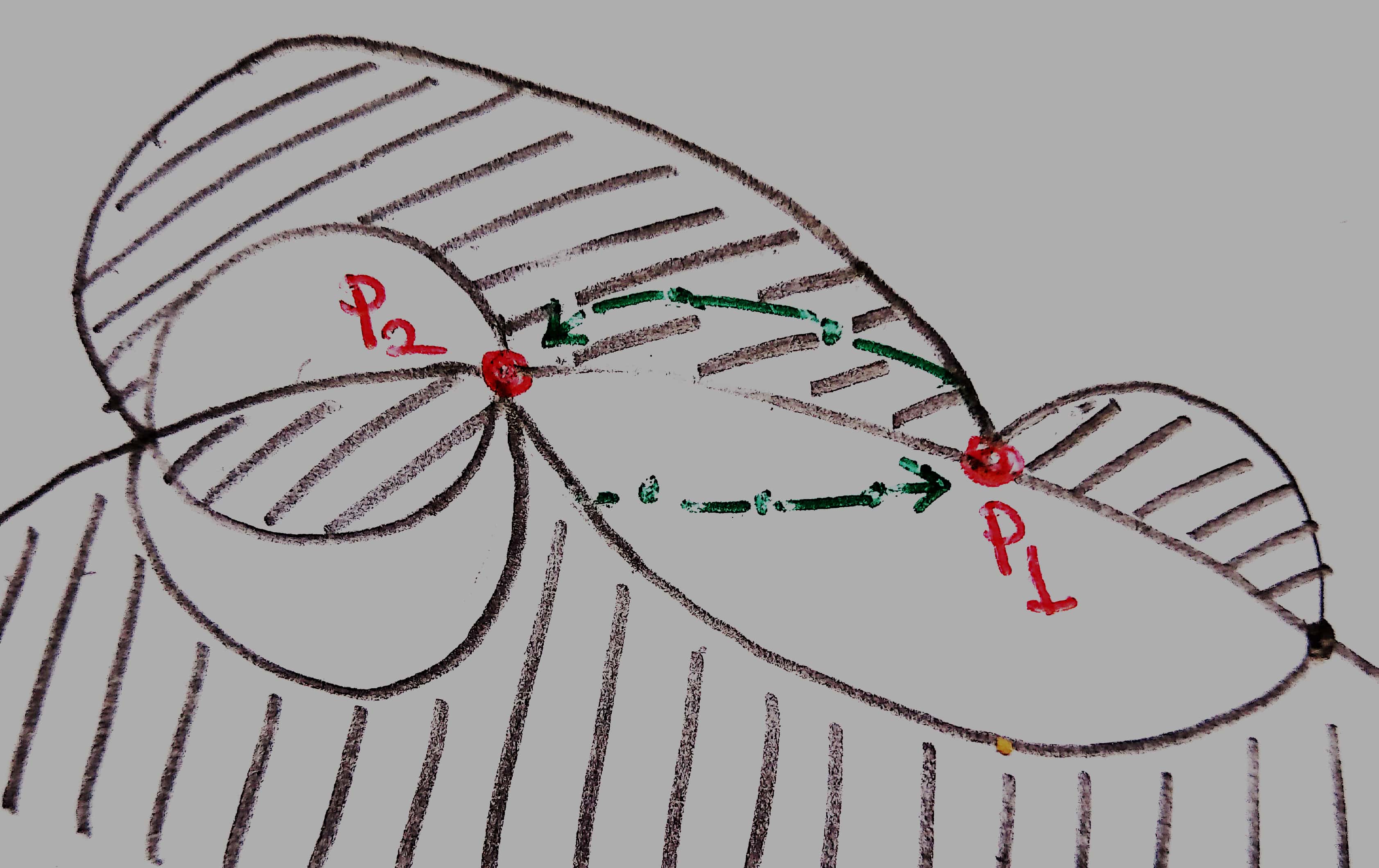}}}
    \qquad
      \subfloat[after to perform the balanced move indicated at $\ref{revmoving}$\textbf{(a)}]
    {{\includegraphics[width=5.45cm]{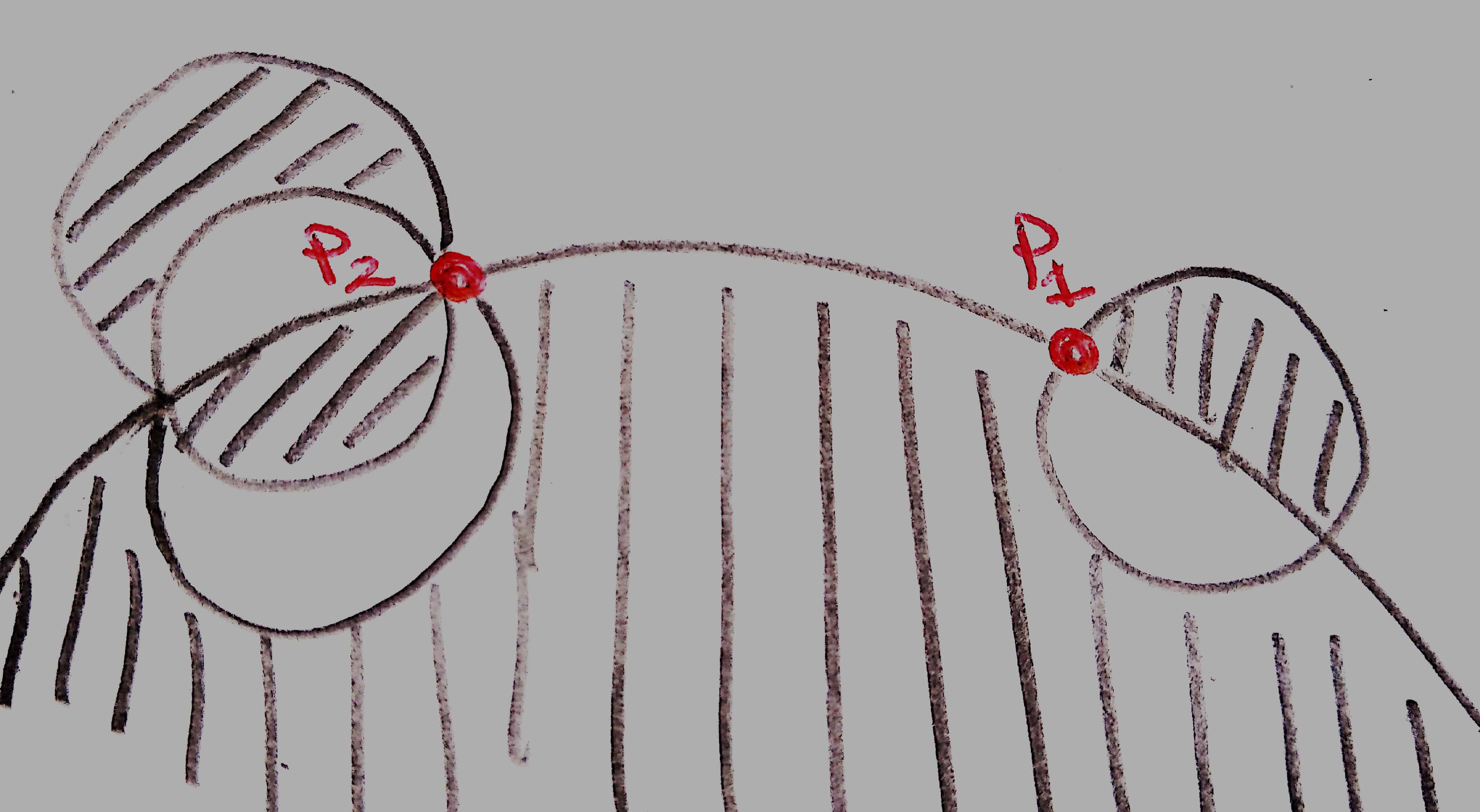}}}\\
    \subfloat[another posible balanced move against $\ref{revmoving}$\textbf{(a)}]
    {{\includegraphics[width=6cm]{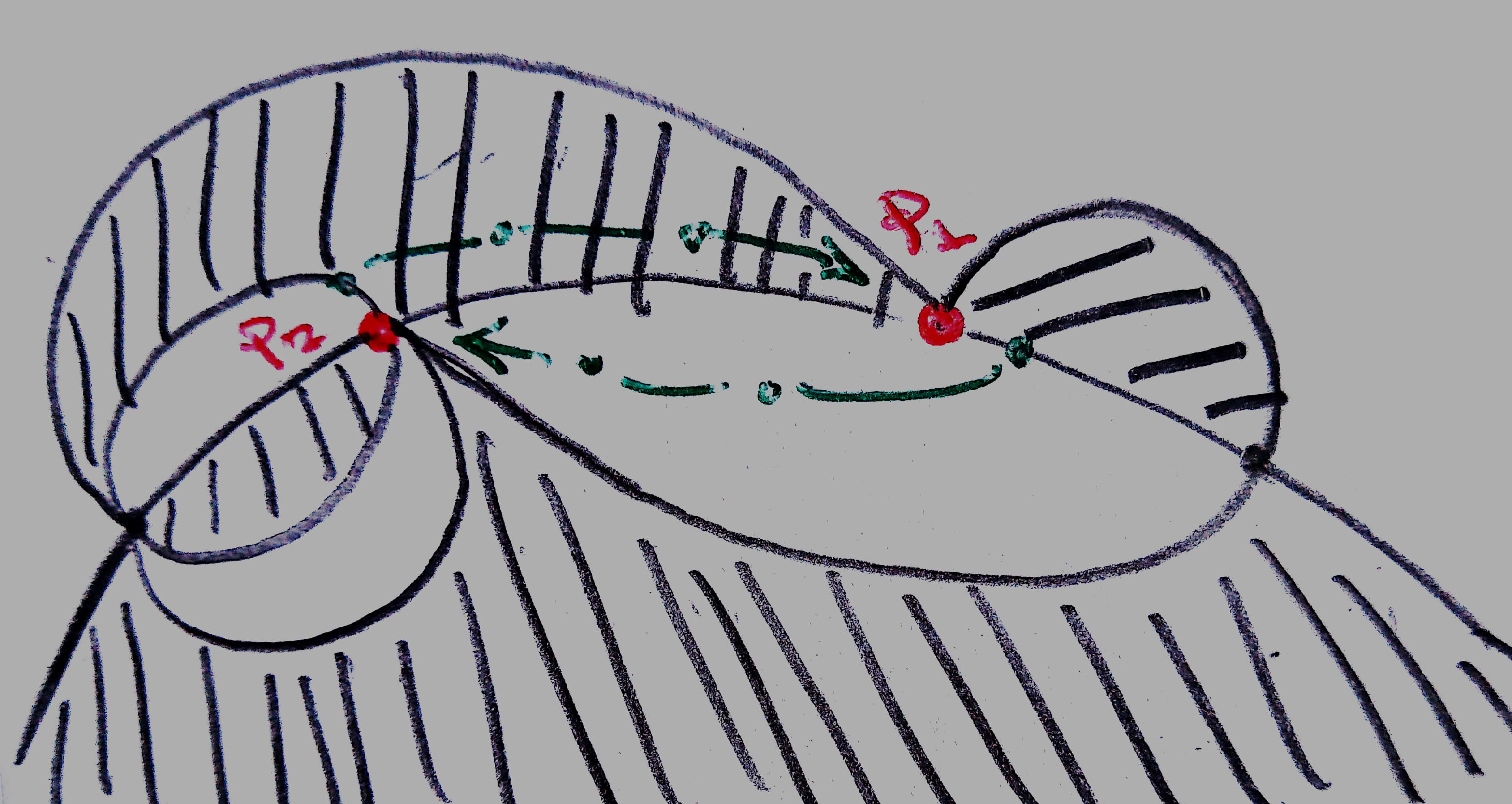}}}
     \hspace{0.1cm}\subfloat[]
    {{\includegraphics[width=5.8cm]{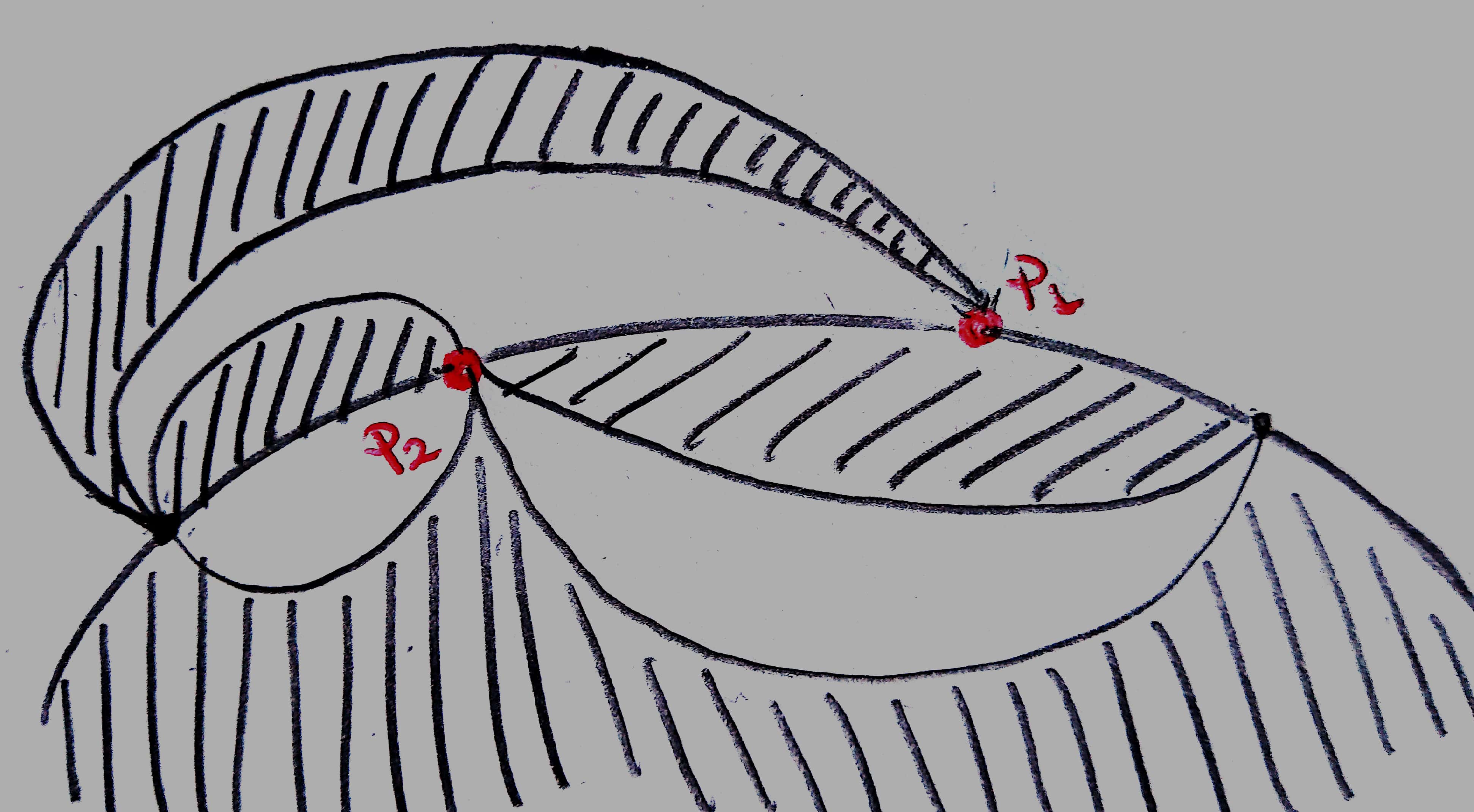}}}\\    
           \end{center}
     \caption{balanced move}\label{revmoving}
\end{figure}

\begin{ex} Bellow we obtain a balanced graph, a example given by Thurston\cite{STL:15}, from a balanced move on a real generic balanced graph. We shall see that all balanced graph of degree $d$ can be obtained from a finite sequence of operations starting with a real generic balanced graph of degree $d$..
\begin{figure}[H]
 \begin{center}
      \subfloat[]
    {{\includegraphics[width=4.25cm]{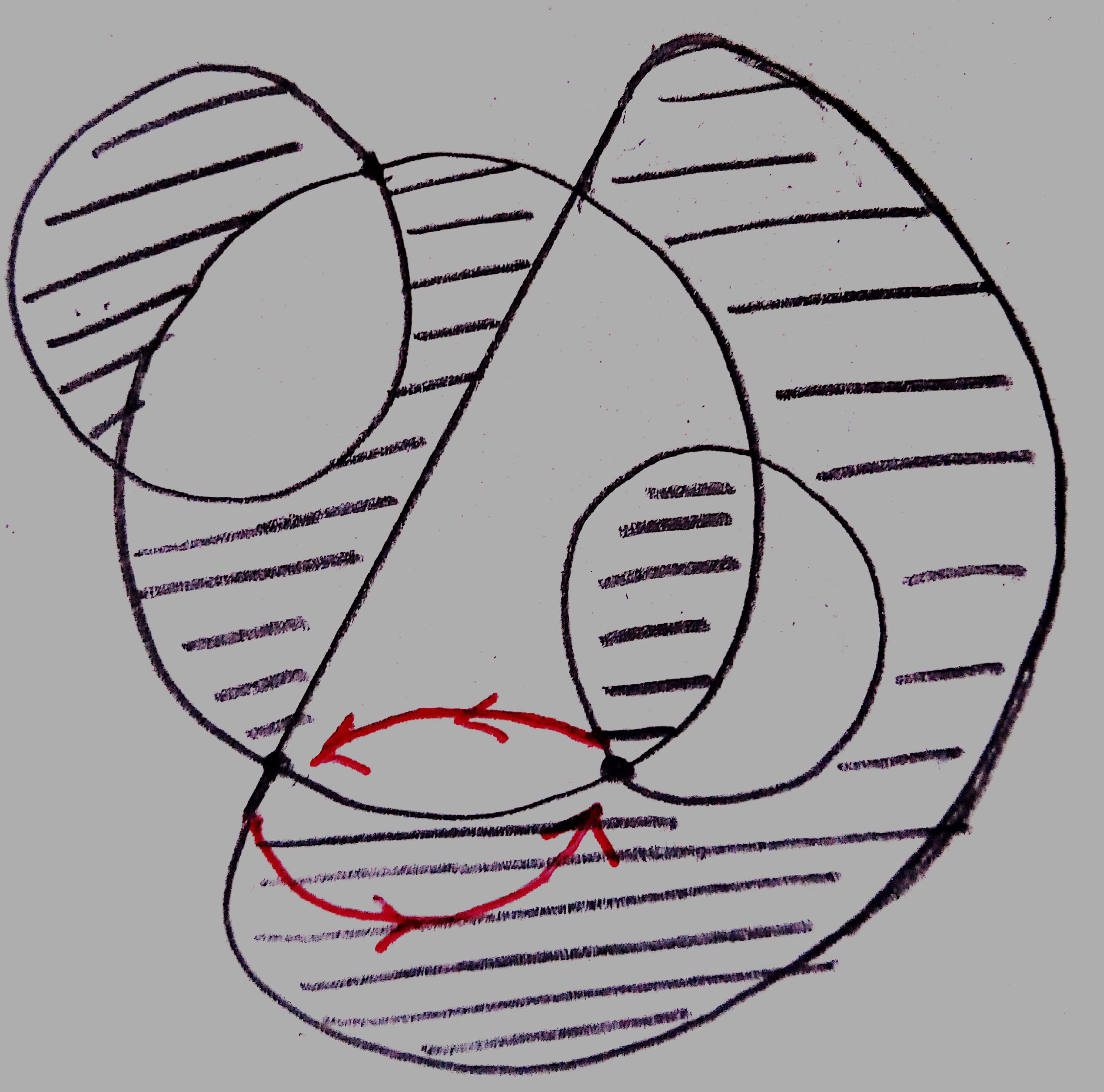}}}
     \qquad \subfloat[]
    {{\includegraphics[width=5.2cm]{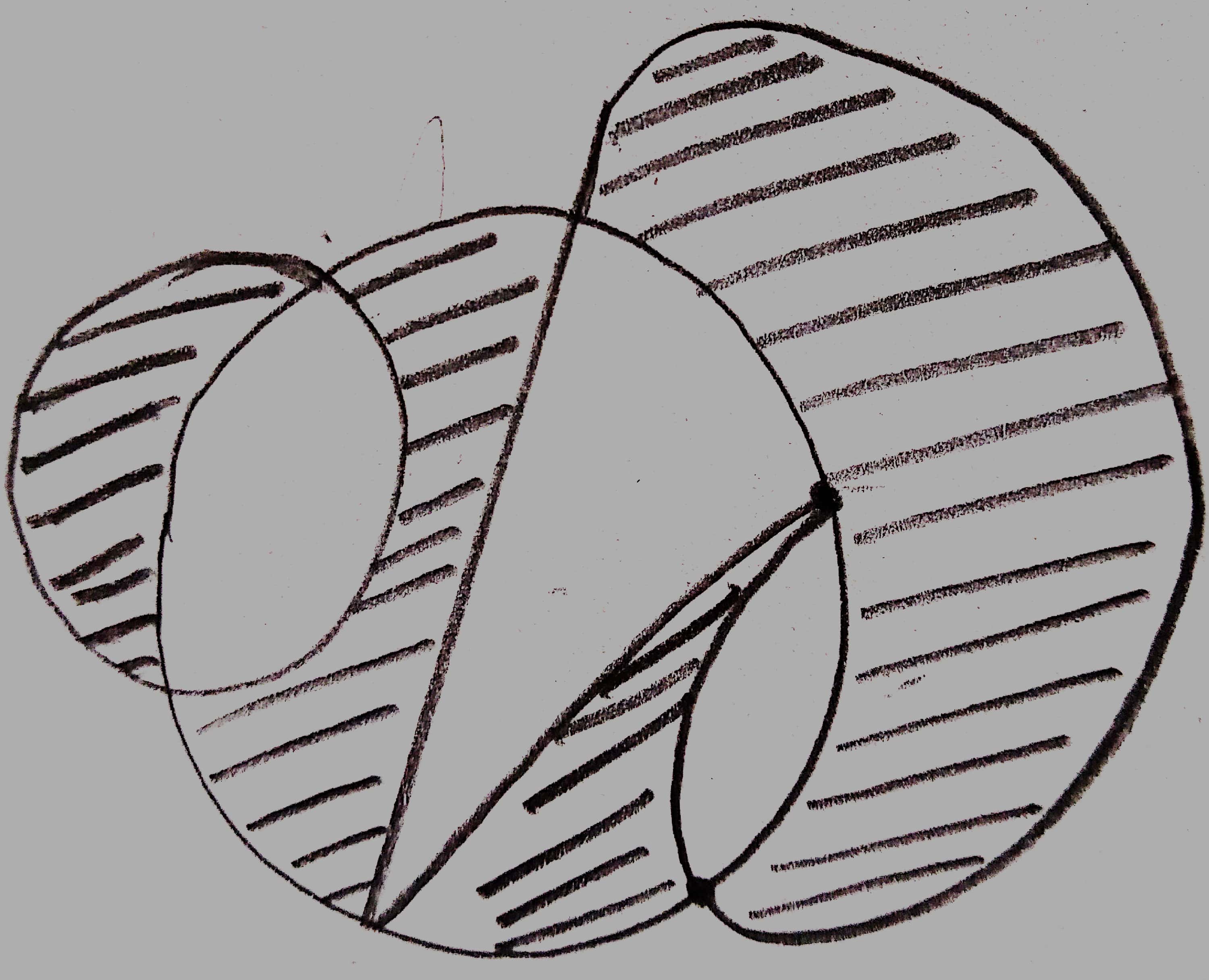}}}
      \qquad \subfloat[as (b) is drawn in \cite{STL:15}]
    {{\includegraphics[width=6.0cm]{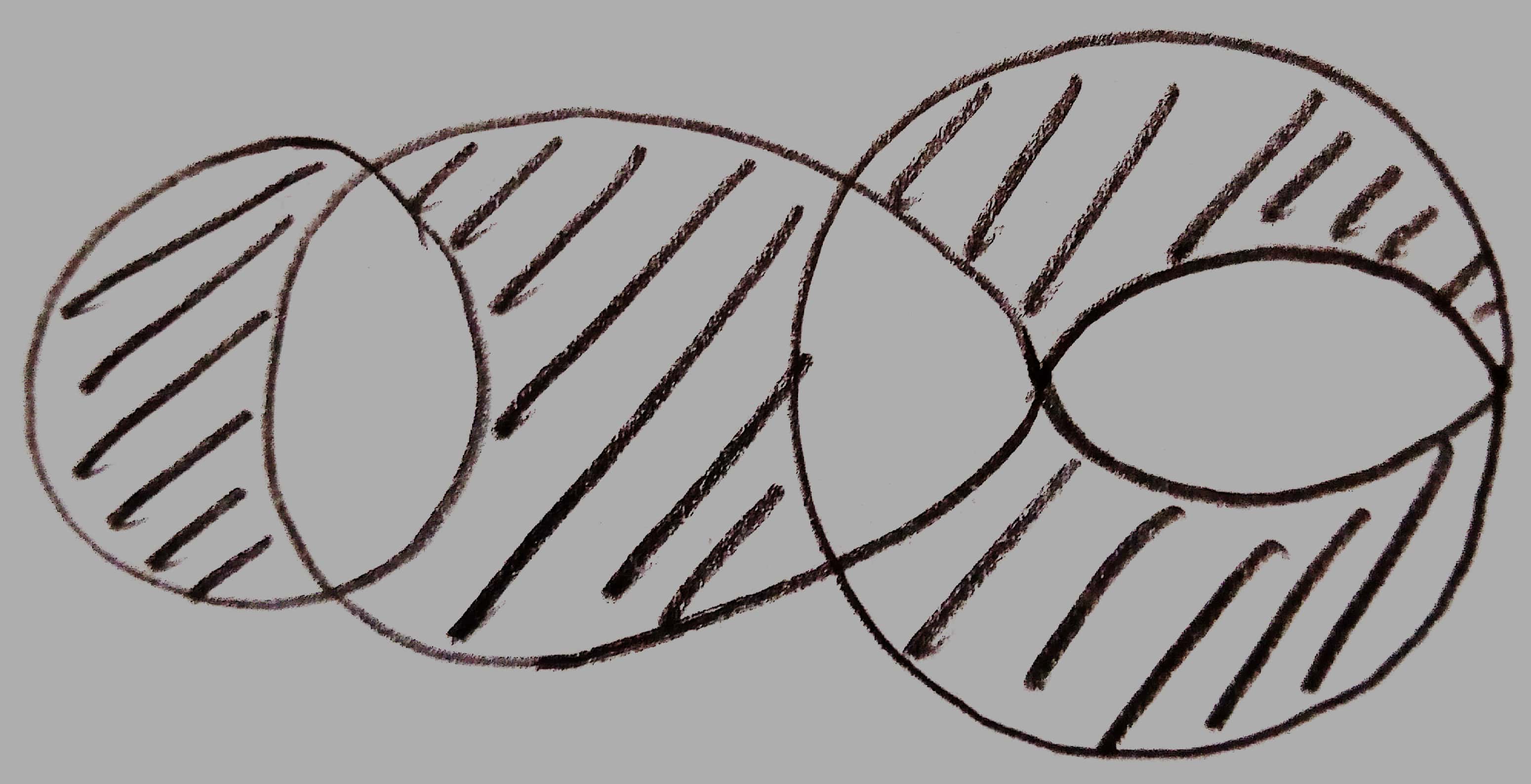}}}
    \end{center}
     \caption{balanced move}
\end{figure}
\end{ex}

\begin{prop}
A \emph{balanced move} operation on a balanced graph of type $ (g, n, d) $ turns it into a balanced graph of the same type.
\end{prop}
\begin{proof}

\end{proof}

Thurston had also introduced some operations on balanced graphs. The essence of the operation presented by \emph{Thurston} is to understand the structure of balanced graphs from the point of view of decomposing them into \emph{standard pieces} turning the class of balanced graph into a \emph{``lego world''}.

\begin{defn}[$\boldsymbol{{}^{\ast}22}$ decomposition (balanced cut)]\label{22-dec}
The $\boldsymbol{{}^{\ast}22}$ decomposition on a balanced graph $\Gamma\in \textbf{BG}$, with underline surface $S_g$, consists of the following described procedure:
\begin{itemize}
\item[$\boldsymbol{(1)}$]{choose a separating closed curve $\gamma\subset S_g$ into $S_g$ such that:
\begin{itemize}
\item[$\boldsymbol{(1.1)}$]{ it intersects the $1$-skeleton of $\Gamma$ at a even number of points, with these points lying in different \emph{saddle-connections};}
\item[$\boldsymbol{(1.2)}$]{it does not go around a single vertex;}
\item[$\boldsymbol{(1.3)}$]{ and, each component of $S_g -\gamma$ contains the same number of faces of each color;}
\end{itemize}}
\item[$\boldsymbol{(2)}$]{to cut $S_g$ along $\gamma$;}
\item[$\boldsymbol{(3)}$]{to compactify these two cut pieces from $ S_g $.}
\end{itemize} 
\end{defn}

It is immediate that the two embedded cellular graphs obtained after that surgical operation are balaced graphs.
  \begin{figure}[H]
 \begin{center} 
     \subfloat[$\boldsymbol{{}^{\ast}22}$ decomposition on the dashed curve more to the right in the figure above. ]
    {{\includegraphics[width=5.3cm]{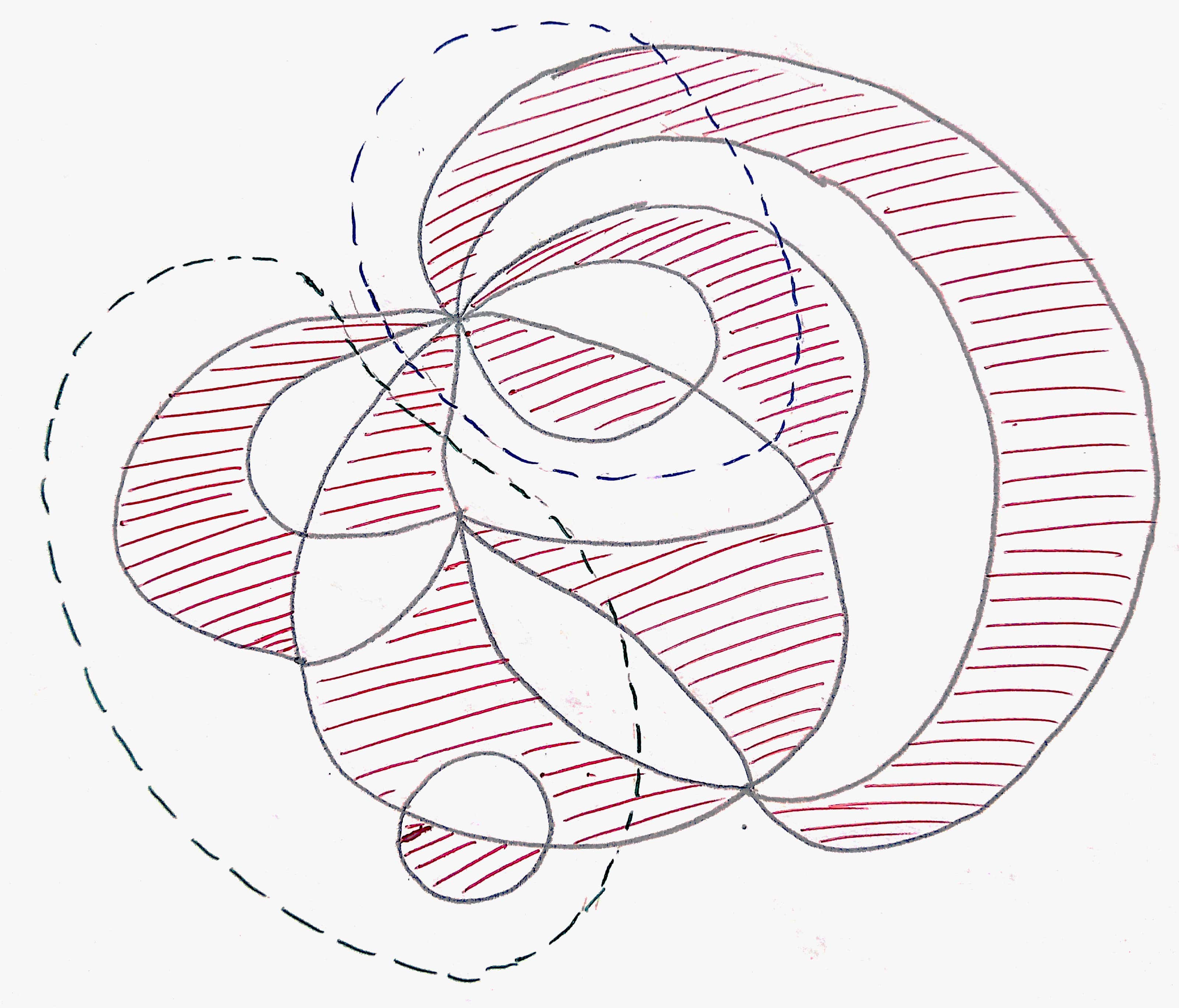}}}\\
     \subfloat[]
    {{\includegraphics[width=5.2cm]{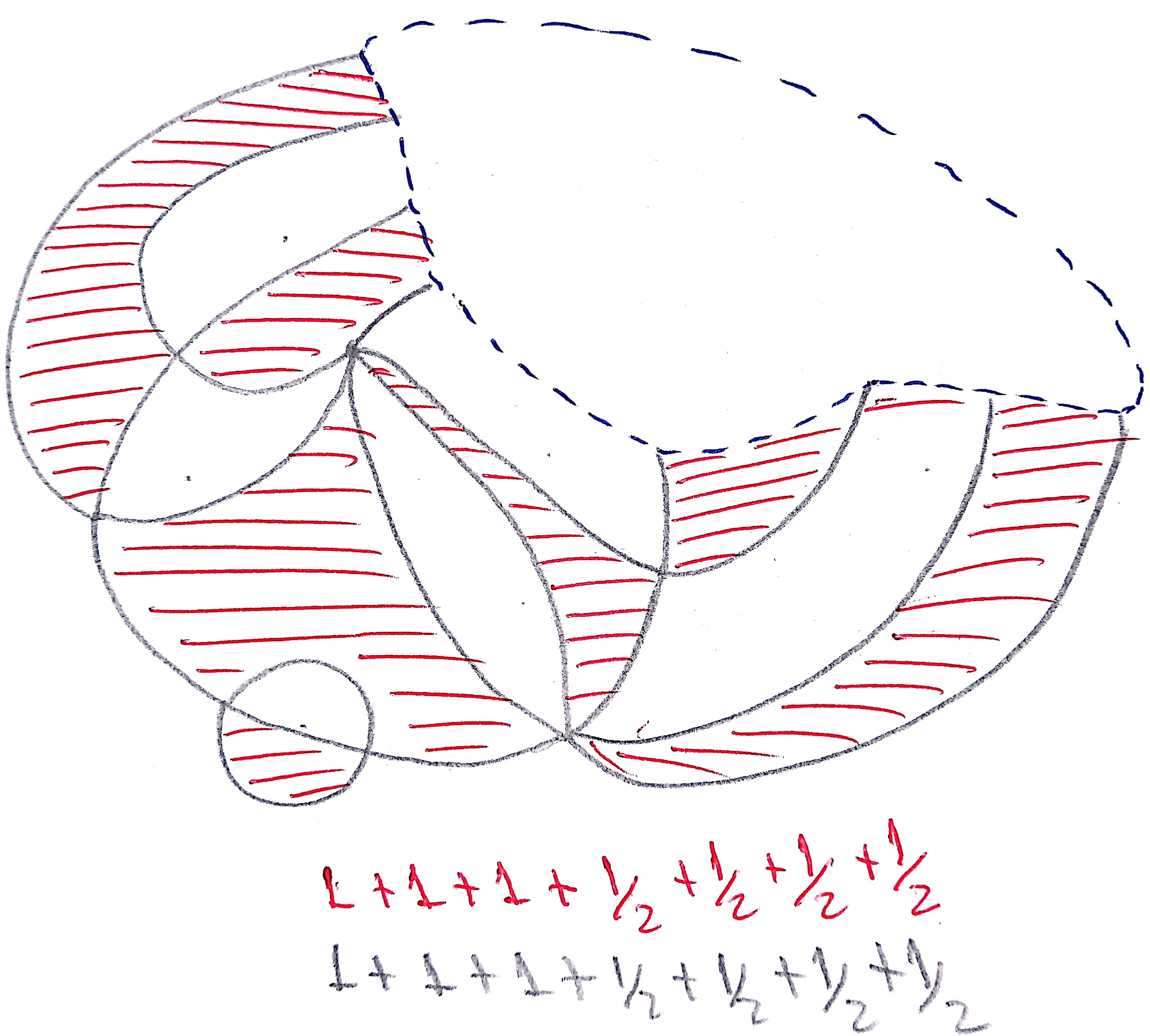}}}\quad
    \subfloat[]
    {{\includegraphics[width=3.9cm]{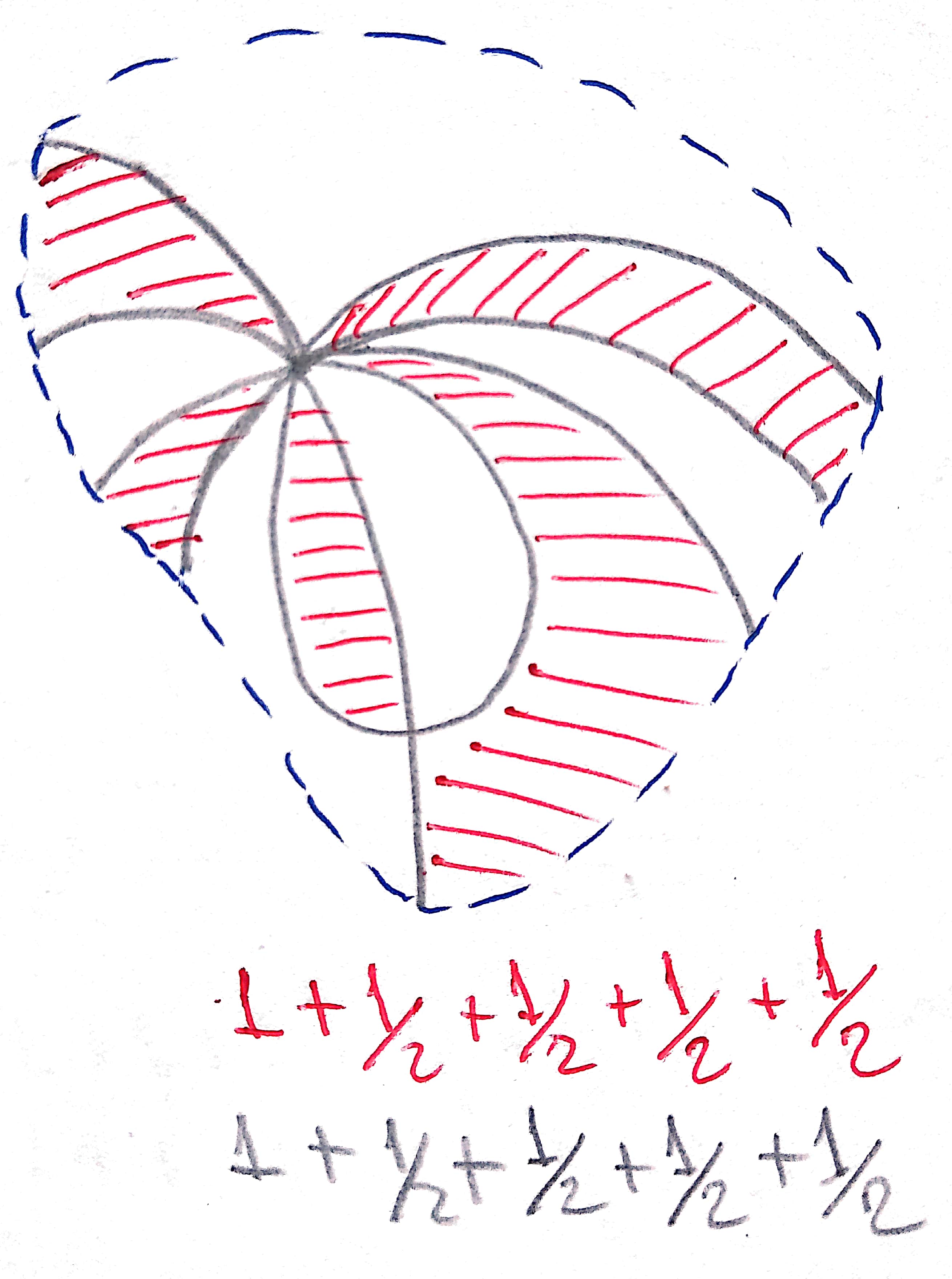}}}
         \end{center}
\end{figure}

    \begin{figure}[H]
 \begin{center} 
      \subfloat[]
    {{\includegraphics[width=5.2cm]{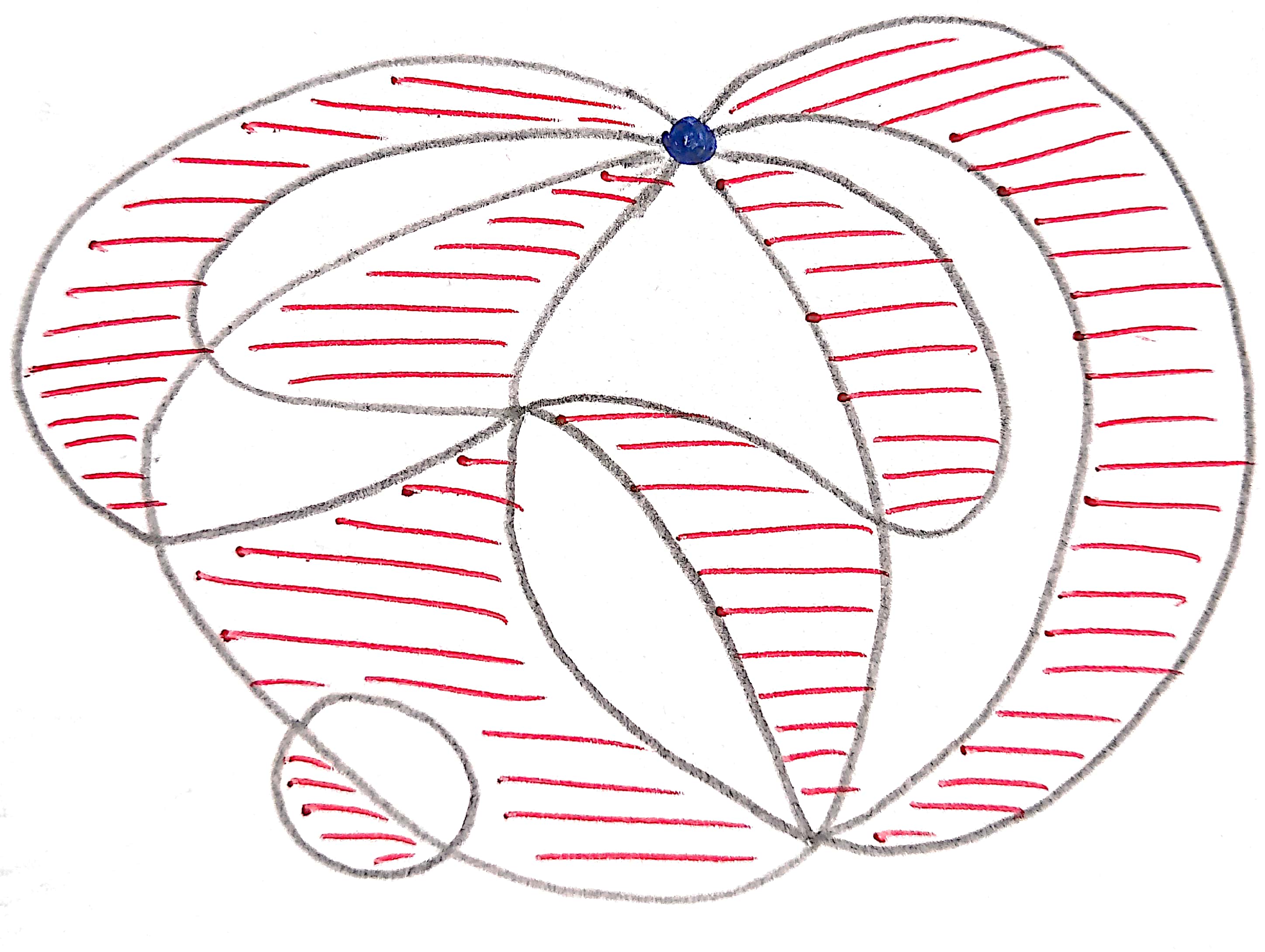}}}\quad
    \subfloat[]
    {{\includegraphics[width=4.2cm]{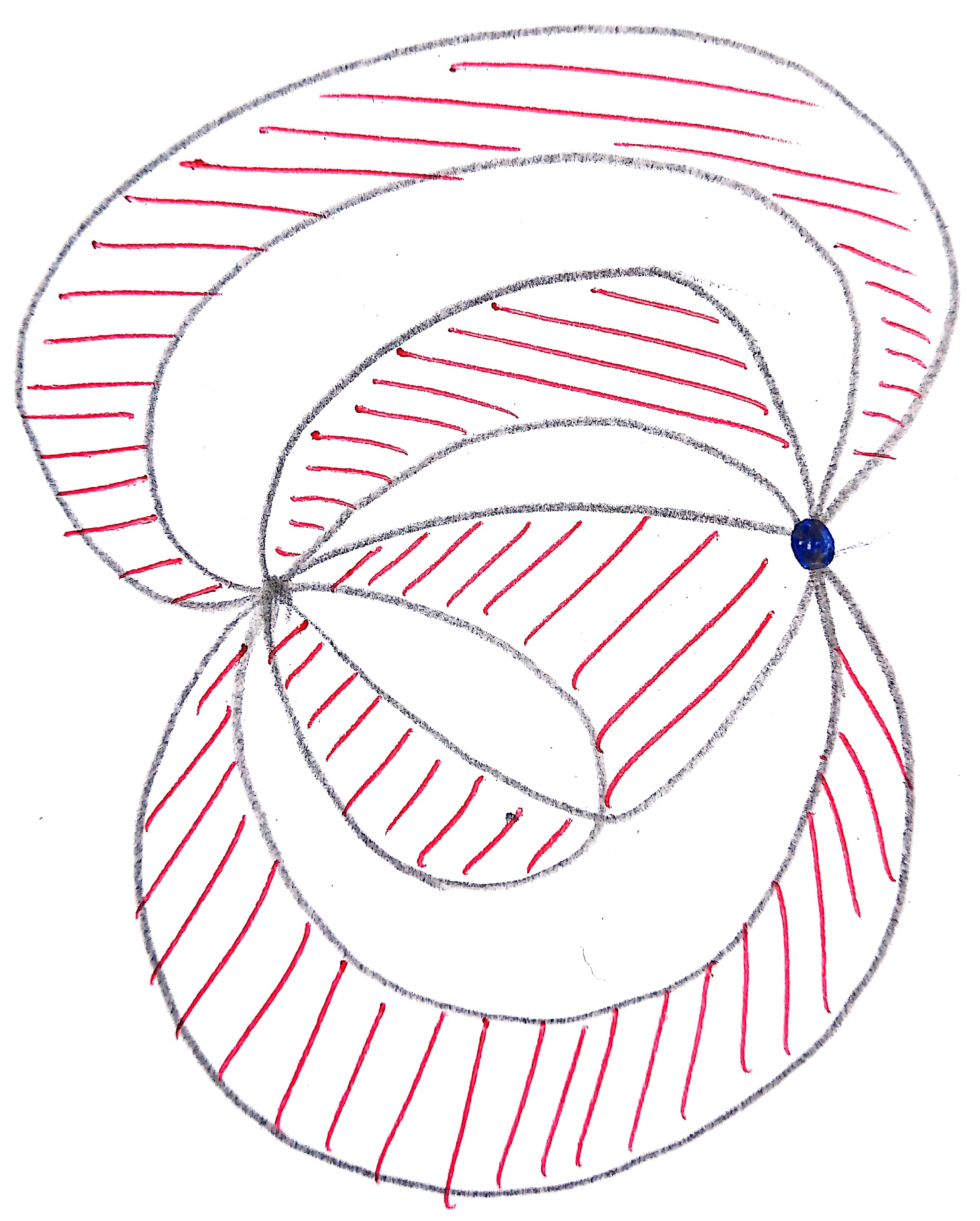}}}
     \end{center}
     \caption{$\boldsymbol{{}^{\ast}22}$ decomposition}
\end{figure}

\begin{defn}[tangle decomposition (inbalanced cut)]\label{22-dec}
Let $\Gamma\in \textbf{BG}$ with underline surface $S_g$ and an alternating \textcolor{deeppink}{A}-\textcolor{Blue}{B} face coloring.
The $\boldsymbol{{}^{\ast}22}$ decomposition on a balanced graph , consists of the following described procedure:
\begin{itemize}
\item[$\boldsymbol{(1)}$]{choose a separating closed curve $\gamma\subset S_g$ into $S_g$ such that:
\begin{itemize}
\item[$\boldsymbol{(1.1)}$]{ it intersects the $1$-skeleton of $\Gamma$ at a even number of points, with these points lying in different \emph{saddle-connections};}
\item[$\boldsymbol{(1.2)}$]{it does not go around a single vertex;}
\item[$\boldsymbol{(1.3)}$]{ and, one component of $S_g -\gamma$ contains one more \textcolor{deeppink}{A} faces than \textcolor{Blue}{B} (then, by the global balance the other component must contain 1 more \textcolor{Blue} {B} faces than \textcolor{deeppink}{A}) of each color;}
\end{itemize}}
\item[$\boldsymbol{(2)}$]{to cut $S_g$ along $\gamma$;}
\item[$\boldsymbol{(3)}$]{at the component of $S_g - \gamma$ containing more  \textcolor{deeppink}{A} we choose two consecutive \textcolor{deeppink}{A} face along the scar curve, then glue this two face together along the scar curve and shrink the two left component of the scar curve into two points over the boundary of the new face.}
\end{itemize} 
Compare with the figure $\ref{inb-dec}$ below.
\end{defn}
  \begin{figure}[H]
 \begin{center} 
     \subfloat[$\boldsymbol{{}^{\ast}22}$ decomposition on the dashed curve more to the left in the figure above. ]
    {{\includegraphics[width=5.3cm]{imagem/cutting_base_1.jpg}}}\\
     \subfloat[]
    {{\includegraphics[width=3.6cm]{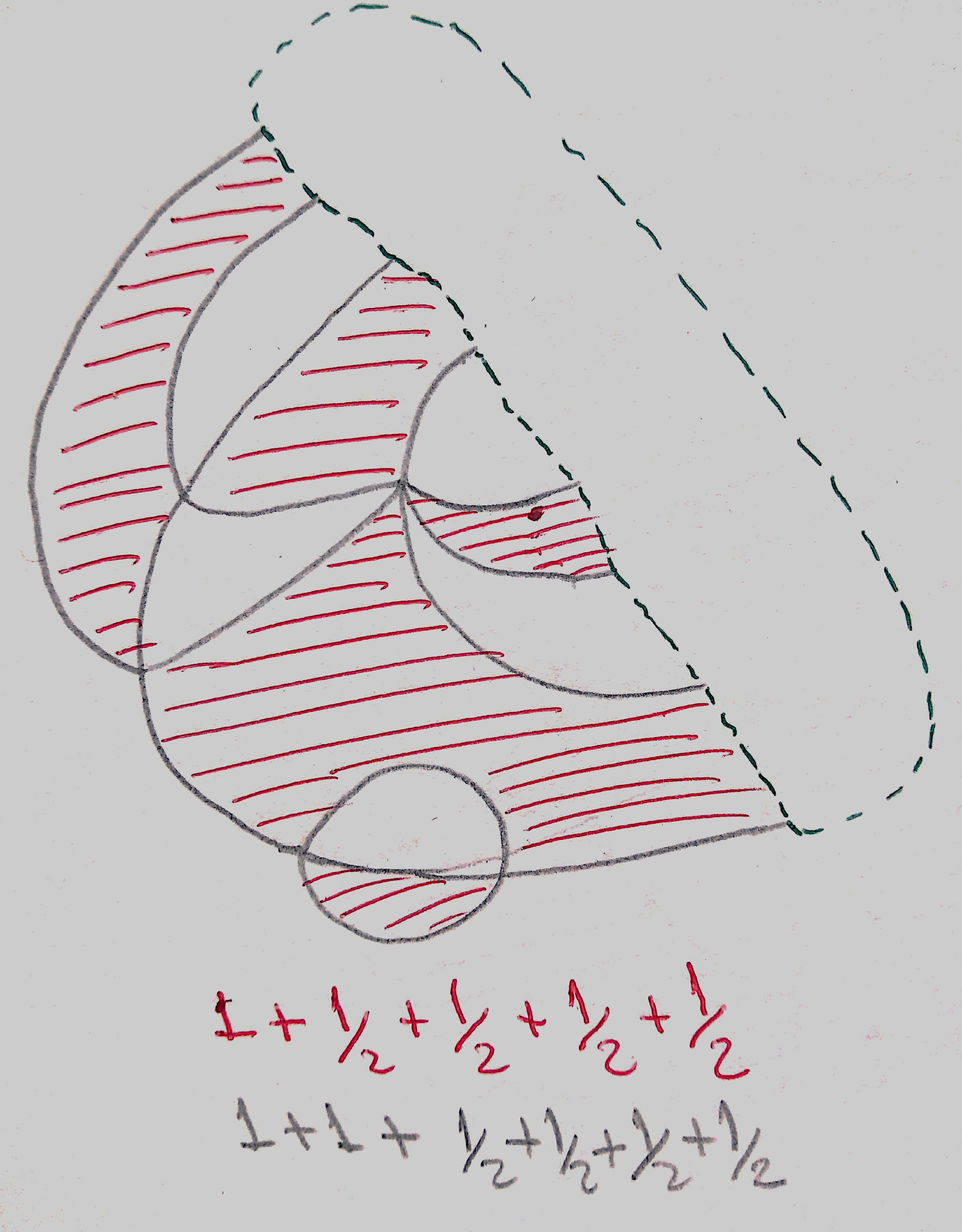}}}\quad
    \subfloat[]
    {{\includegraphics[width=4.35cm]{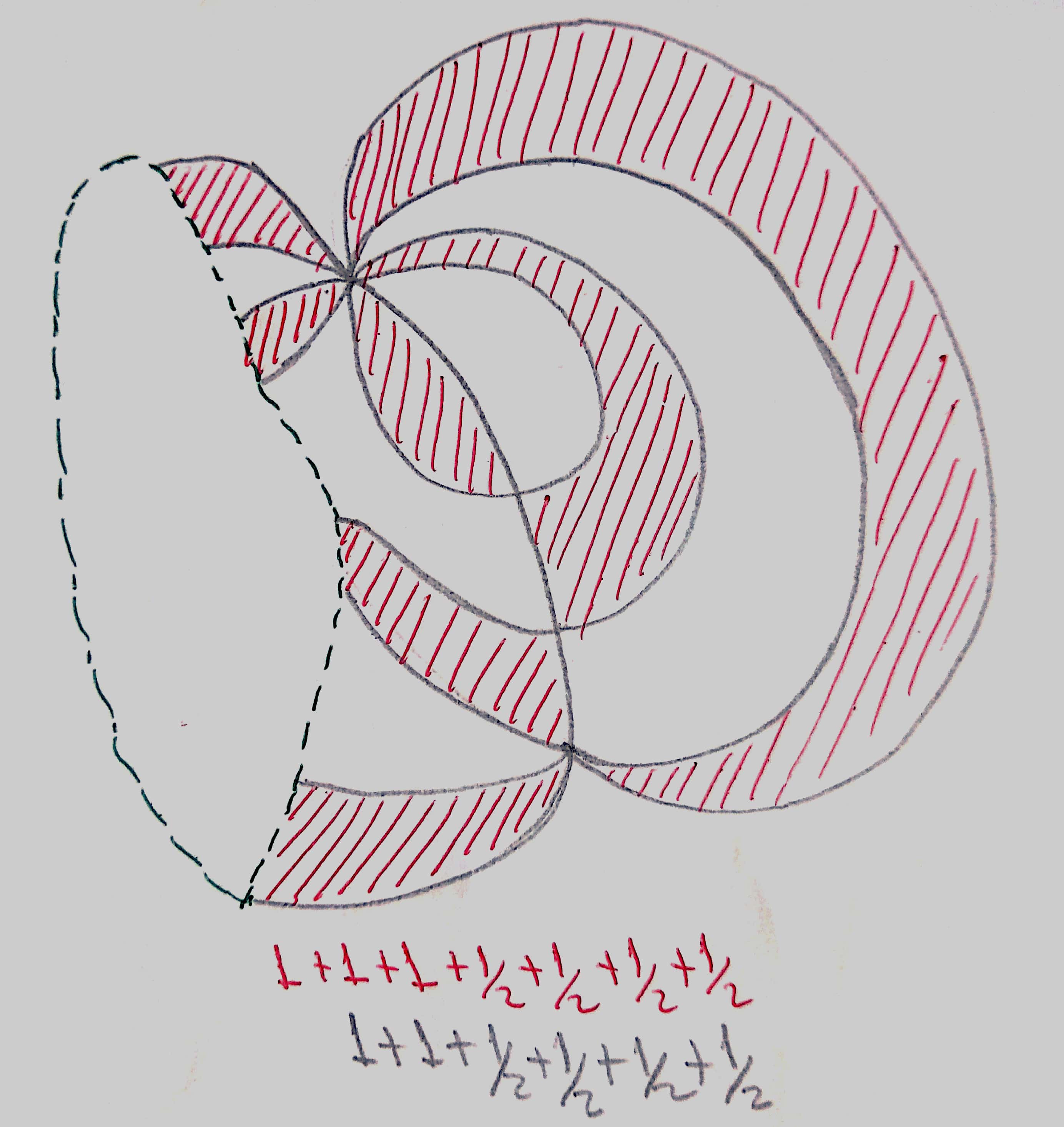}}}
         \end{center}
\end{figure}

    \begin{figure}[H]
 \begin{center} 
      \subfloat[]
    {{\includegraphics[width=4.2cm]{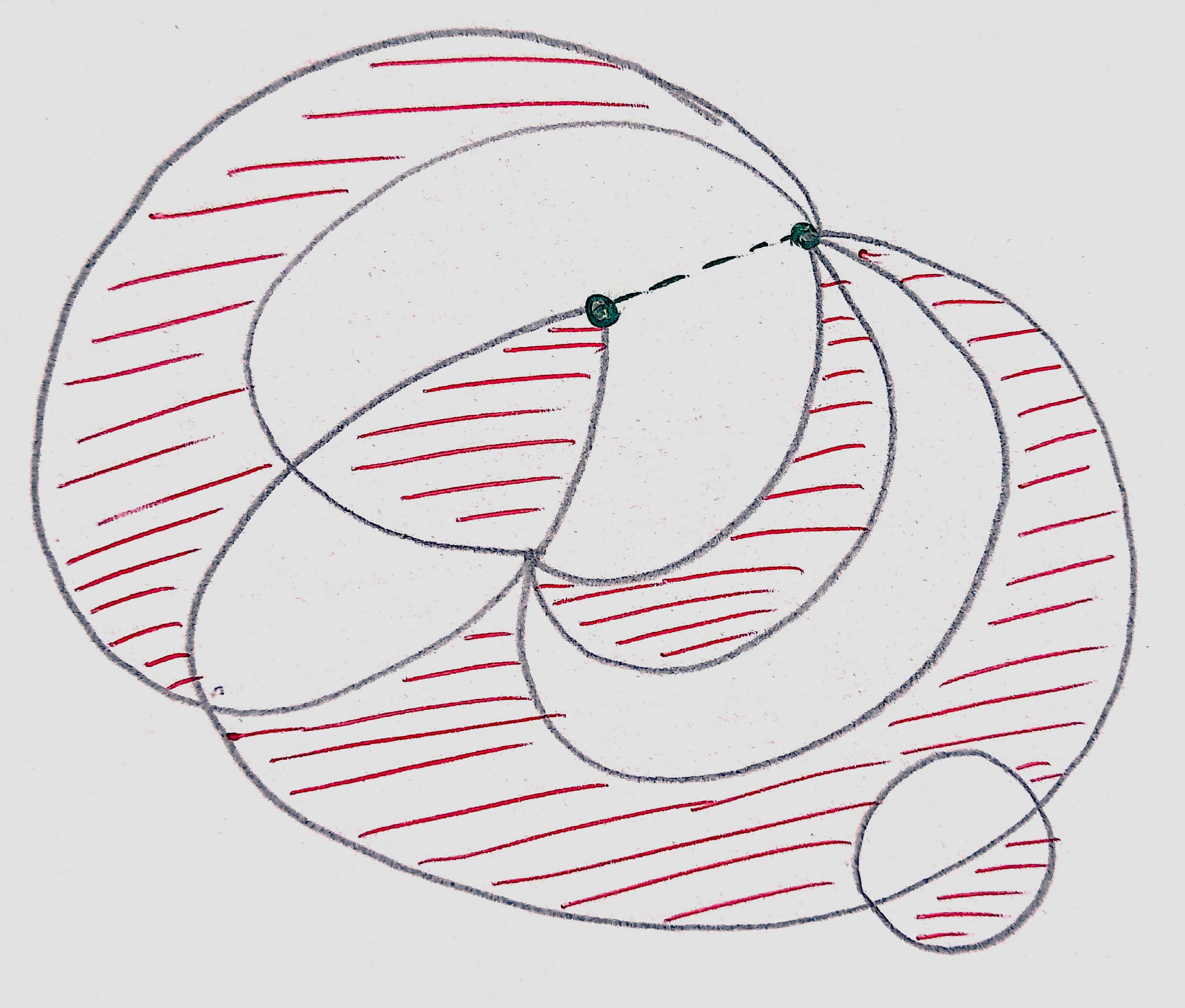}}}\quad
    \subfloat[]
    {{\includegraphics[width=4.2cm]{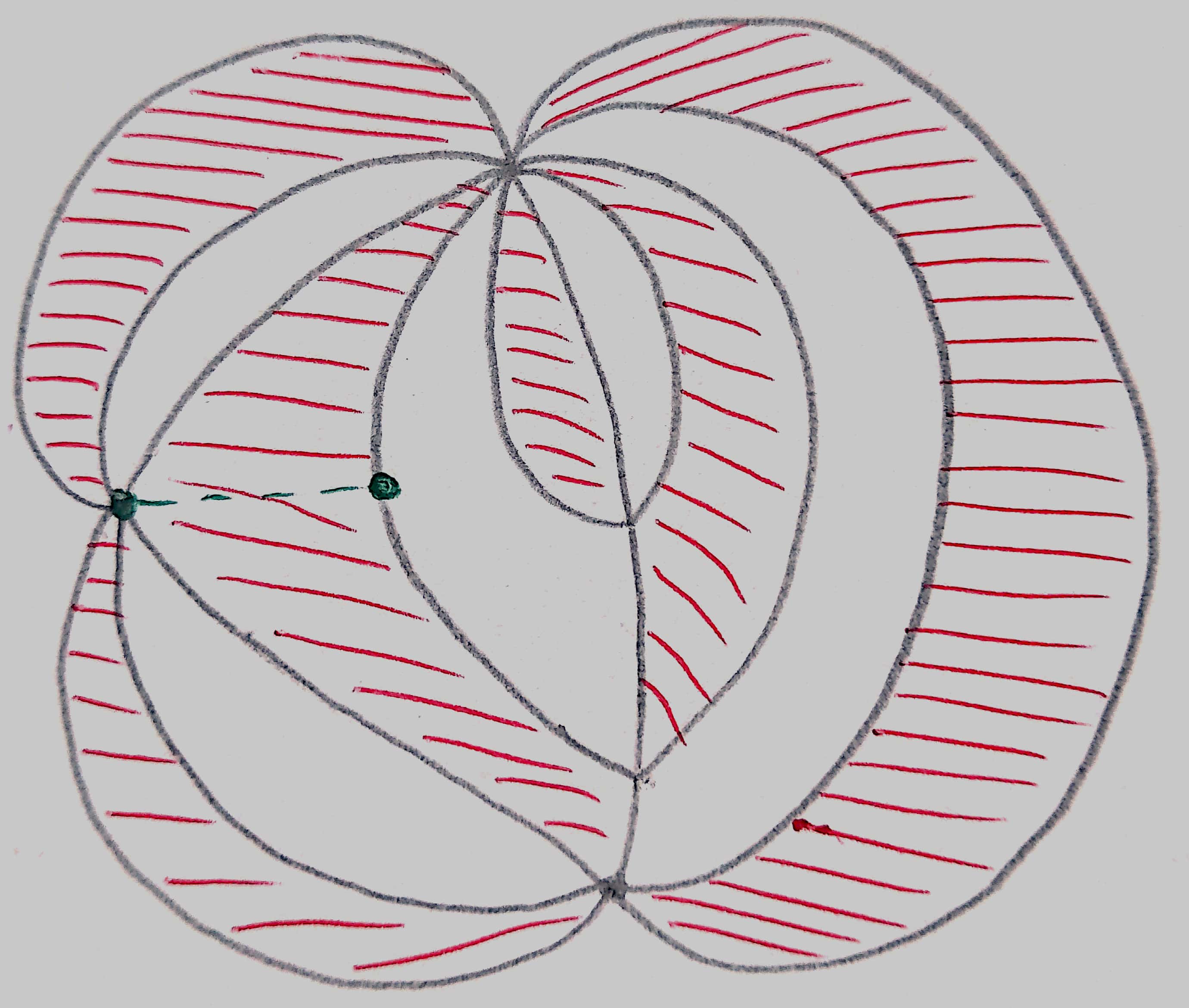}}}
\caption{$\boldsymbol{{}^{\ast}22}$ decomposition}\label{inb-dec}     
     \end{center}         
\end{figure}

\begin{defn}[Murasugi sum]\label{m-sum}
The \emph{Murasugi sum} of two balanced graphs, say $\Gamma, \Lambda\in \textbf{BG}$, both with a \textcolor{deeppink}{A}-\textcolor{blue}{B} ,face coloring, consists of the following described procedure:
\begin{itemize}
\item[$\boldsymbol{(1)}$]{To remove a rectangle from oppositely colored faces of $\Gamma\in \boldsymbol{BG}$ and $\Lambda\in \textbf{BG}$, where the rectangles have two edges on different saddle-connection incident to a face and the other two edges interior to that face;}
\item[$\boldsymbol{(2)}$]{Then glue $\Gamma\in \textbf{BG}$ and $\Lambda\in \textbf{BG}$ along the edges of those cutting out rectangles so as to match the face colors.}  
\end{itemize}
\end{defn}
\begin{ex}Below we construct a degree 3 planar balanced graph from a \emph{Murdugi sum} of two copies of the more simpler balanced graph. The degree 3 graph produced is a projection of the \emph{figure eight knot}. \end{ex}
  \begin{figure}[H]
 \begin{center} 
     \subfloat[sum]
    {{\includegraphics[width=5.cm,height=3cm]{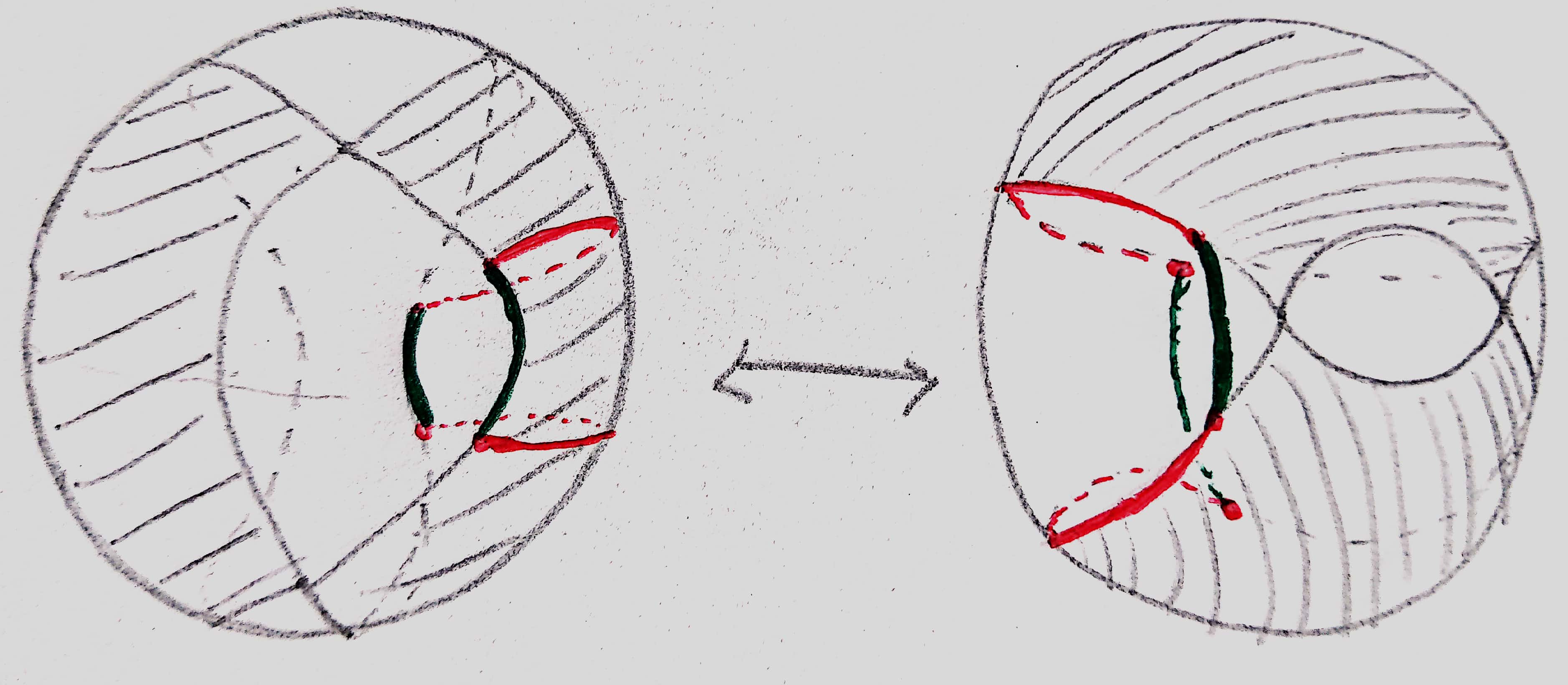}}}\quad \subfloat[new balanced graph]
    {{\includegraphics[width=3cm,height=3cm]{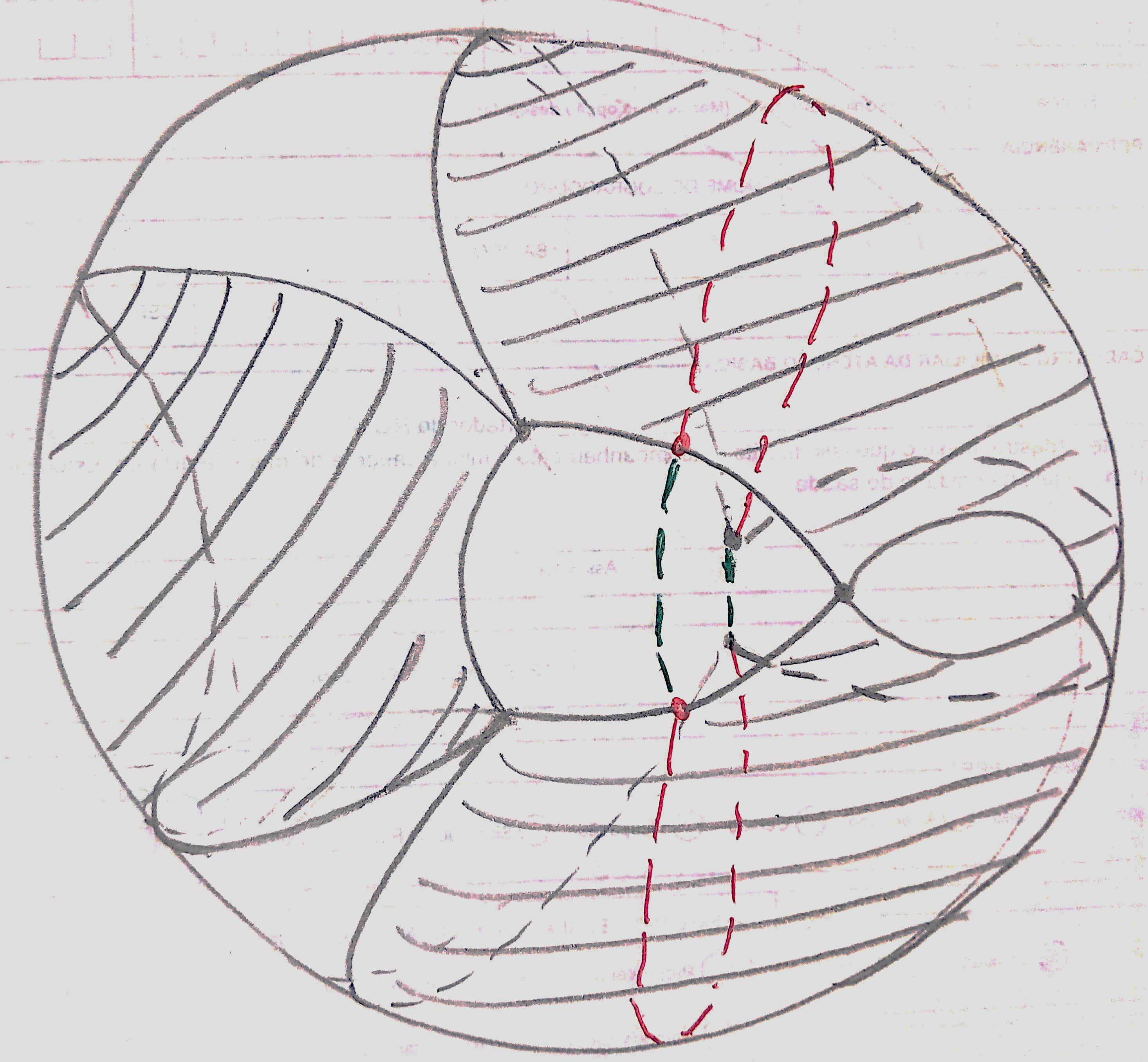}}}\\
    \subfloat[(b) into the plane(ignore the dashed line)]
    {{\includegraphics[width=4.6cm,height=3cm]{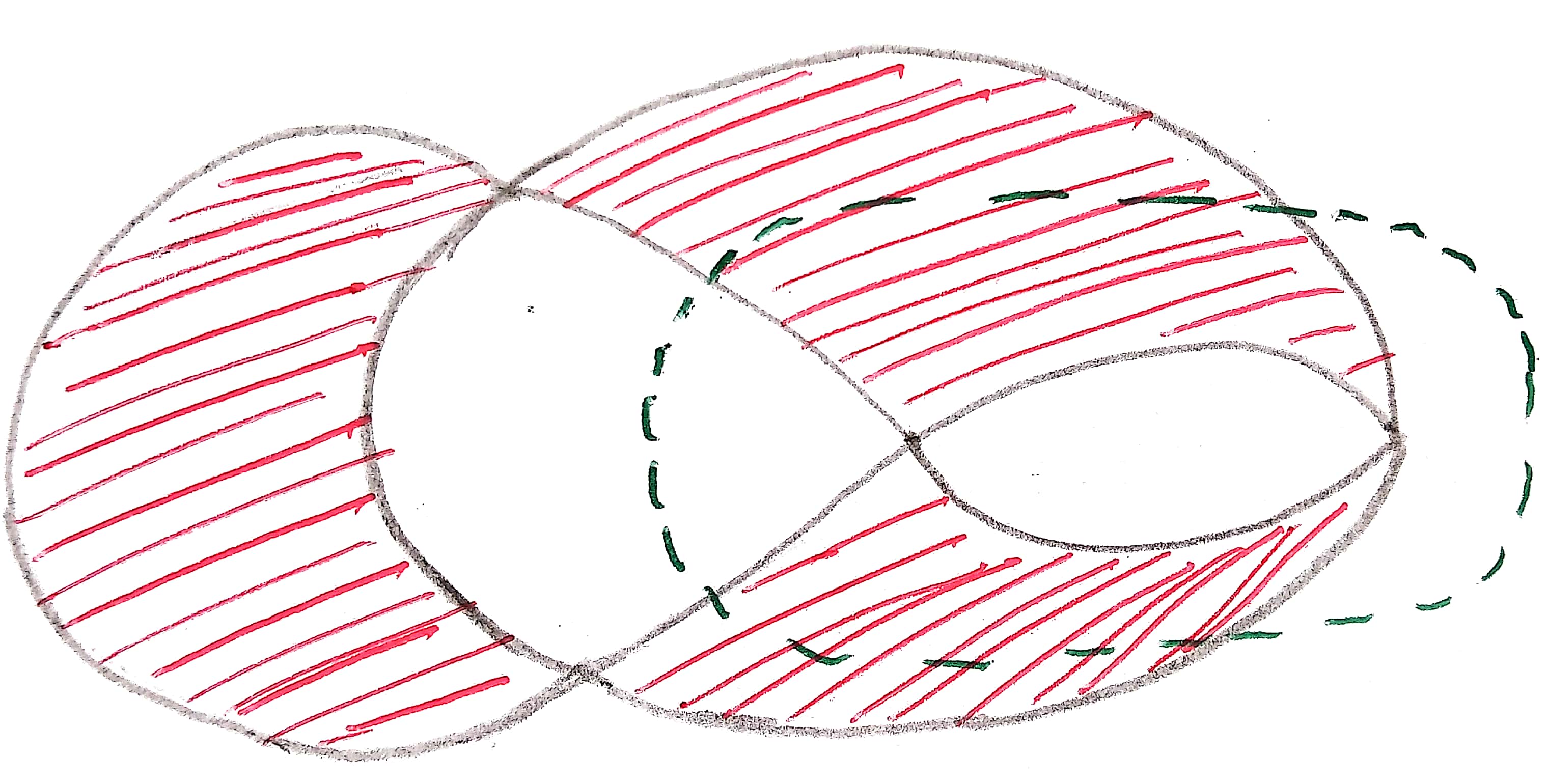}}}
     \end{center}
     \caption{Murasugi sum}
\end{figure}

\begin{defn}[{The category ${\boldsymbol{\bg}}$}]\label{cat-bg}
The category of \emph{Balanced Graph} $\bgg$ is that one whose the class of object consists of balanced graphs of any type and the morphisms are the operations defined above. Since each morphism have a inverse arrow  $\bgg$ atually is a \emph{Groupoid}.
\end{defn}

\rem{there are more than one sum operation over $ \bgg $, we are working on defining a single sum operation from these. And it is expected that this operation be compatible with the morphisms and determines a monoidal structure in $ \bgg $.}


\section{Proving the B. \& M. Shapiro conjecture}\label{shap-section}

\subsection{Local balancedness of real globally balanced graphs}
\begin{thm}\label{realgb-is-lg}

Real generic GB-graphs are locally balanced.
\end{thm} 
\begin{proof}

Let $\Gamma\subset\overline{\C}$ be a real generic GB-graph with a \textcolor{deeppink}{A}-\textcolor{blue}{B} alternating face coloring and $\gamma$ a positive cycle of $\Gamma$. $\textcolor{deeppink}{A}_{\gamma}$ and $\textcolor{blue}{B}_{\gamma}$ are the numbers of \textcolor{deeppink}{A} faces and \textcolor{blue}{B} faces inside $\gamma$. 

Being $\Gamma$ a real generic globally balanced graph, each face of it have at least one of its boundary edges contained into $\overline{\R}$, we refer to such a kind of edge as real edges. By the alternating property of the face coloring each \textcolor{blue}{B} face 
possesses a companion \textcolor{deeppink}{A} face sharing the same real edges. Since $\gamma$ keeps only \textcolor{deeppink}{A} faces adjacent to its left side, for each \textcolor{blue}{B} face $F_{\textcolor{blue}{B}}$ in the interior of $\gamma$ its companion $\textcolor{deeppink}{A}$ face $F_{\textcolor{deeppink}{A}}$ is also inside $\gamma$. And, for the same reason, must there exist at least one more $\textcolor{deeppink}{A}$ face adjacent to those nonreal edges of those \textcolor{blue}{B} faces inside $\gamma$. Therefore, $\textcolor{deeppink}{A}_{\gamma}\geq \textcolor{blue}{B}_{\gamma}+1$.

We conclude that $\Gamma$ is locally balanced.

\begin{figure}[H]
 \begin{center}
       \subfloat[degree $9$ real GB-graph]
    {{\includegraphics[scale=.07]{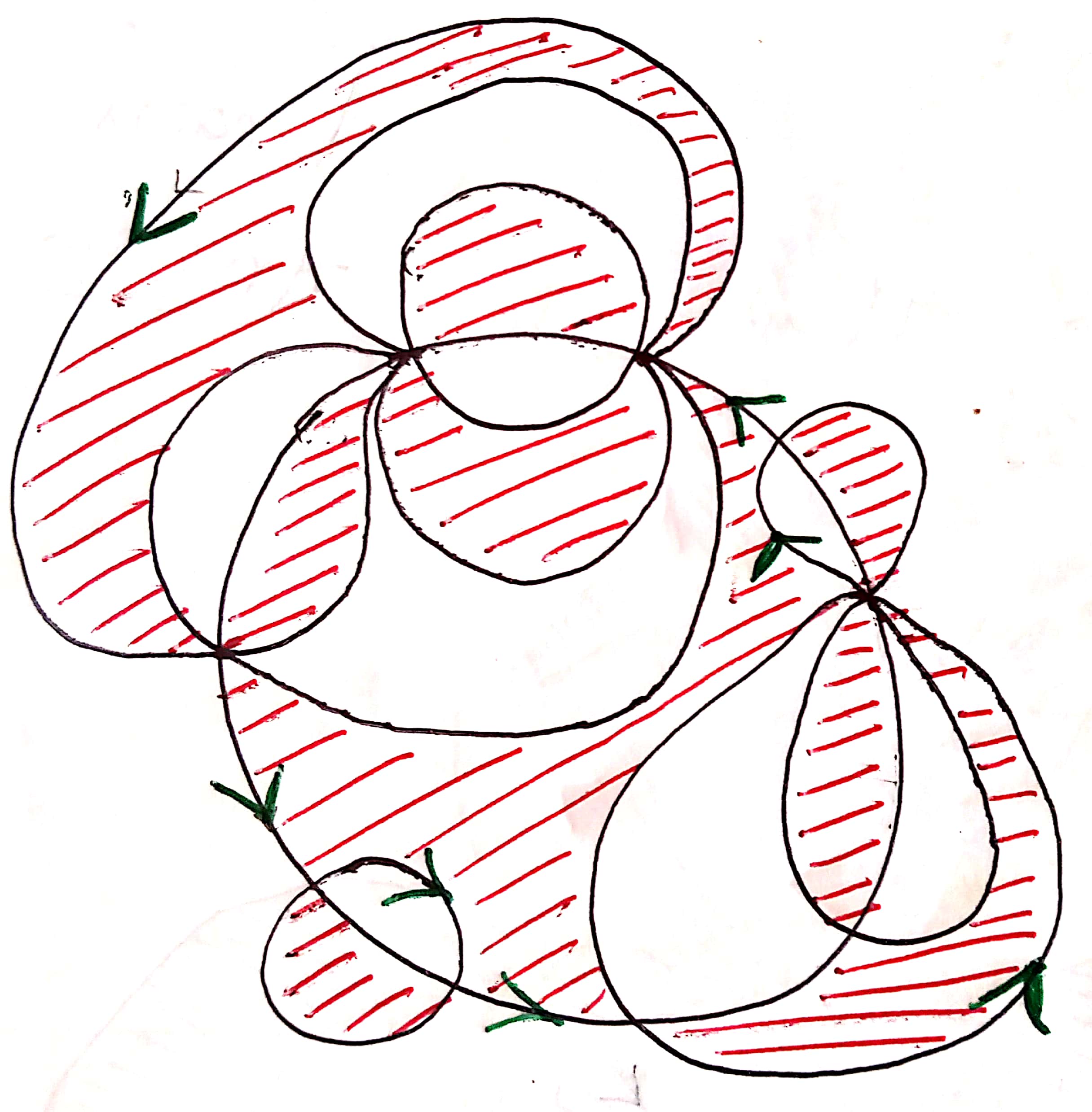}}}
        \qquad
    \subfloat[degree $8$ real GB-graph]
    {{\includegraphics[scale=.07]{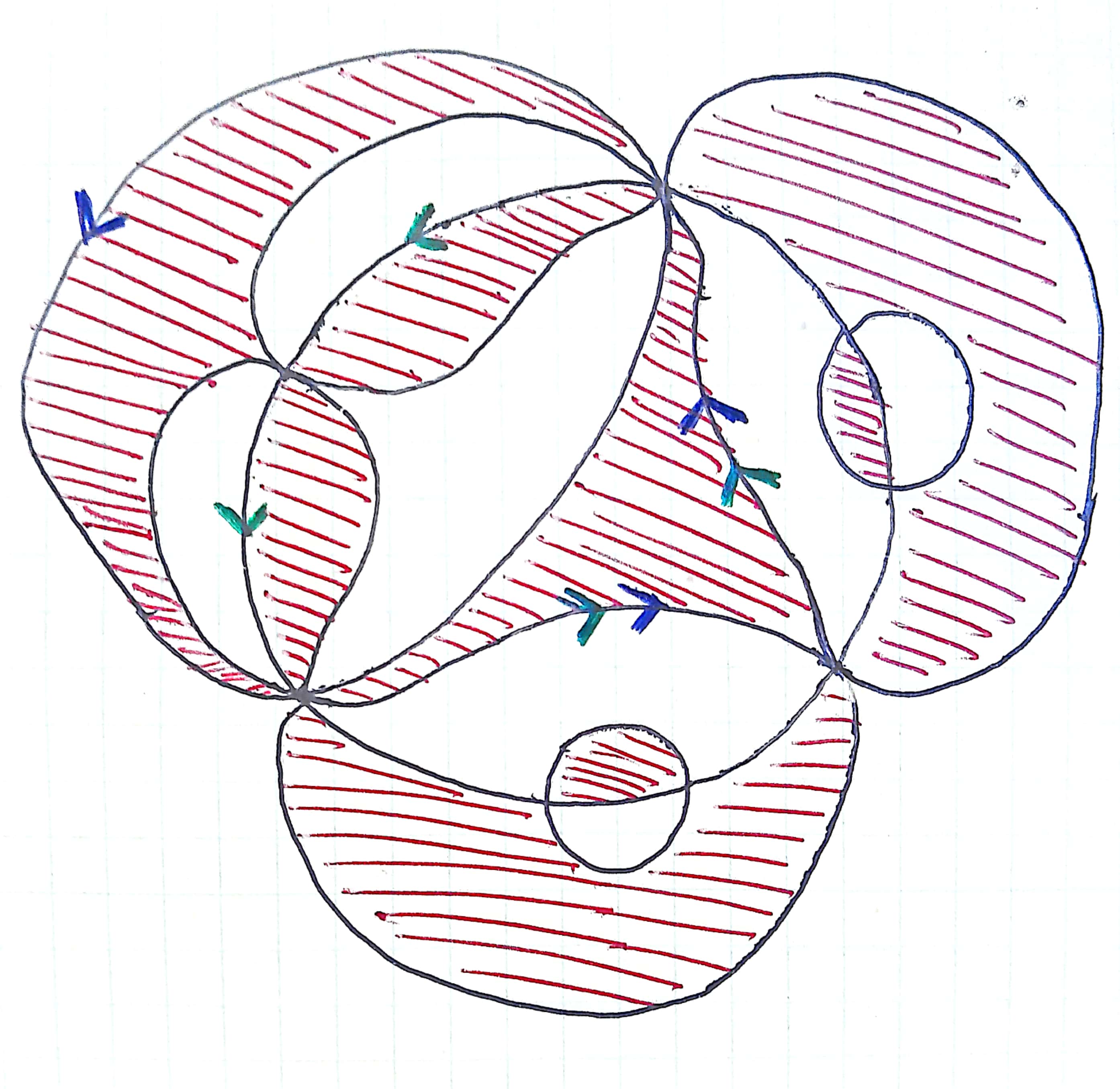} }}
    \caption{}
\end{center}\end{figure}

\end{proof}

\begin{cor}\label{cor-realgb-is-lg}
Real GB-graphs are locally balanced.
\end{cor}
\begin{proof}
Follows directly from Theorem $\ref{realgb-is-lg}$ and Proposition $\ref{ngenrg-f-genrg}$. 
\end{proof}

In this section a topological-combinatorial proof for the \emph{B}. \& \emph{M. Shapiros conjecture} is given as a byproduct of some previous results on this chapter.

First, as a corollary from Theorem $\ref{realgb-is-lg}$ we obtain a combinatorial solution for a special case of a problem posed by Goldberg\;\cite{Gold:91} that we appropriately introduced in $\ref{goldissue}$.

\begin{cor}\label{relGoldp}
The number of equivalence classes of generic degree $d$ rational functions with real critical points is the $d$-\emph{Catalan number}.
\end{cor}
\begin{proof}
We know that there are $\displaystyle{\dfrac{1}{d} \binom{2d-2}{d-1}}$ degree $d$ real generic globally balanced graphs (see $\ref{goldissue}$) for fixed $2d-2$ distinct points into $\overline{\R}$. Thus, from Theorem $\ref{realgb-is-lg}$ there are $\displaystyle{\dfrac{1}{d} \binom{2d-2}{d-1}}$ degree $d$ real balanced graphs for fixed $2d-2$ distinct points into $\overline{\R}$. And that is what was left to be proved to achieve this result (to recall returns to $\ref{goldissue}$). 
\end{proof}

Now, we will present a new proof for the \emph{B.} \& \emph{M. Shapiro's conjecture}. 
This new proof remains at a more natural and simple level of complexity and depends much less on sophisticated non-discrete mathematical machinery than that obtained by Eremenko \& Gabrielov \cite{MR1888795}, \cite{MR3051166}, hence it is more accessible. Nevertheless, we still have to resort to Goldeberg's result \cite{Gold:91}. (This can be overcome if we could prove that for a fixed subset $C\subset{\cc}$ of $2d-2$ points, any degree $d$ non-generic planar balanced graph with vertce set $C$ is obtained from a generic one with vertice set $C$ by those operations on graphs and that those operations only permutes the classes of the corresponding branched covers. We conjecture that this can be proved.)

\begin{thm}[Eremenko-Gabrielov-Mukhin-Tarasov-Varchenko Theorem]\label{sha-conj}
A generic rational function $R:\overline{\C}\rightarrow\overline{\C}$ with only real critical points is equivalent to a real rational function.
\end{thm}
\begin{proof}
Fix a subset $C\subset{\overline{\R}}$ of $2d-2$ points. From corollary $\ref{relGoldp}$ the number of real non-equivalent real rational function with critical set $C$ is $\displaystyle{\dfrac{1}{d} \binom{2d-2}{d-1}}$. But from the number of equivalente classes of generic rational function of prescribed critical set is at most $\rho_d$ \cite{Gold:91}. Then we are done.
\end{proof}

\begin{cor}
A generic rational function $R:\overline{\C}\rightarrow\overline{\C}$ with all critical points leaving into a circle is equivalent to a real rational function.
\end{cor}

\chapter[Generic Cubic Rational Functions]{\rule[0ex]{16.5cm}{0.2cm}\vspace{-23pt}
\rule[0ex]{16.5cm}{0.05cm}\\Generic Cubic Rational Functions}\label{cap-04}

By the \emph{Riemann-Hurwitz formula} (or by a simple algebraic computation) a cubic rational function $f\in\C(z)$ has $4=2\cdot 3-2$ critical points counted with multiplicity.

We are going to consider the generic cubic rational functions, that is, the ones that have precisely $4$ critical points.

Into such a case, each cubic rational function can be written in the following way after suitable changes of coordinates on the domain and codomain by \emph{M\"obius} transformations,

\begin{eqnarray}
\phi(z) = \frac{a z^3 + (1-2 a)z^2}{(2-a) z - 1}
\end{eqnarray}

$\phi$ has critical points at $0$, $1$, $\infty$ and $c = \dfrac{2 a-1}{a(2-a)}.$

The constraint that $\phi(z)$ possesses $4$ distinct enforces the constrant, $a\in\C-\{0, 2^{-1}, 1, 2\}$, over the coefficient $a$.

A generic choice (that is, outside a proper algebraic subvariety of $\C$) of a parameter $c\in\C$ gives rise to $2$ $a$-solutions:

\begin{eqnarray}
\alpha(c)=\dfrac{\sqrt{c^2 -c+1}-1+c}{c}\nonumber
\end{eqnarray}
and
\begin{eqnarray}
\beta(c)=\dfrac{-\sqrt{c^2 -c+1}-1+c}{c}, \nonumber
\end{eqnarray}	where $\sqrt{\star}$ denotes the 
\emph{principal branch} of the \emph{square root}.

Then, each aforementioned choice determines two rational functions, say $\phi_{\alpha}$ and $\phi_{\beta}$, that have $\{c, 0, 1, \infty\}\subset\cc$ as its critical set.

Since $\phi_{\alpha}$ and $\phi_{\beta}$ have $3$  fixed points in common, they cannot be equivalent unless they are equal. That agrees with Goldberg's result \cite{Gold:91} that there exists at most $\rho(3)=2$ equivalence classes for degree three generic rational functions on $\cc$  for a generic prescription of the critical set $R\subset\cc$. 

\subsection{Justifying the normal form}

\begin{lem}\label{justcnf}
Every generic cubic rational function $f\in \C(z)_3$ with $\{0, 1, \infty\}\subset crit(f)$ is equivalent to a unique cubic rational function of the form 
\begin{eqnarray}
\phi_a (z)=\dfrac{a z^{3} +(1-2 a)z^{2}}{(2-a)z-1}
\end{eqnarray}
for some $a\in\C-\{-1,0,1/2,1,2\}$ whose fourth critical point is given by
\begin{eqnarray}
c(a) = \frac{2 a-1}{a(2-a)}
\end{eqnarray}
\end{lem}
\begin{proof}
First of all we can assume, {up to a postcomposition} with a \emph{M{\"o}bius} map
, that the images of $0$, $1$ and $\infty$ by $\phi$ is itself, i.e, the set $\{0, 1, \infty\}$ is pointwise fixed by $\phi$.

We have 
\begin{eqnarray}
\phi(z)=\dfrac{P(z)}{Q(z)}=\dfrac{\sum_{k=0}^{3} a_k z^{k}}{\sum_{k=0}^{3} b_k z^{k}}
\end{eqnarray}

Provided that $\phi(\infty)=\infty$ we should have $\deg(\sum_{k=0}^{3} a_k z^{k})>\deg(\sum_{k=0}^{3} b_k z^{k})$. Hence, $b_3=0.$

Since $\phi(0)=0$ we must have $a_0=0$ and $b_0 \neq 0$. Furthermore, $0$ is also a critical point, then it is a zero of multiplicity at least $2$, what implies that $a_1 =0.$

Then,
\begin{eqnarray}
\phi(z)=\dfrac{a_3 z^{3}+a_2 z^{2}}{b_2 z^{2}+b_1 z+b_0}
\end{eqnarray}
 
Now, $\phi(1)=1$ implies that
\begin{eqnarray}\label{f1}
a_3 +a_2= b_2+b_1+b_0\quad\mbox{with}\quad{b_0\neq 0}
\end{eqnarray}
The \emph{Wroskian} of $\phi$ is 
\begin{eqnarray}
 W(\phi)(z)=a_3 b_2 z^{4}+2a_3 b_1 z^{3}+(3a_3 b_0 +a_2 b_1)z^{2}+2a_2 b_0 z
\end{eqnarray}
Since, ${W(\phi)(1)}=0$, it follows
\begin{eqnarray}\label{wf1}
a_3 b_2+2a_3 b_1+3a_3b_0 +a_2 b_1+2a_2 b_0=0
\end{eqnarray}
Then using the relation $(\ref{wf1})$
\begin{eqnarray}\label{w12}
0&=&2a_3(b_2+b_1 +b_0)- a_3 b_2 +a_3 b_0 + a_2(b_2 +b_1 +b_0)+a_2 b_0 -a_2 b_2\nonumber\\
&=&(2a_3 +a_2)(b_2 +b_1 +b_0)- a_3 b_2 +a_3 b_0+a_2 b_0 -a_2 b_2 \nonumber\\
&=&(2a_3 +a_2)(b_2 +b_1 +b_0)+(b_0 -b_2)(a_2+a_3)\nonumber\\
&=&(2a_3 +a_2)(a_2+a_3)+(b_0 -b_2)(a_2+a_3)\nonumber\\
&=&(a_2+a_3)(2a_3 +a_2+b_0 -b_2)
\end{eqnarray}
Since $b_0\neq 0$, without loss of generality, we can assume that $b_0=-1$.

Now notice that $b_2 =0.$ Since $\infty$ is a critical point for $\phi$ 
\begin{eqnarray}
0&=&\left.\dfrac{d}{dz}\left(\dfrac{1}{\phi(\frac{1}{z})}\right)\right|_{z=0}\\
&=&\left.\dfrac{a_3 b_2 +2a_3 b_1 z + (3a_3 b_0 +a_2 b_1)z^{2} + 2a_2 b_0 z^{3}}{(a_3 + a_2z)^{2}}\right|_{z=0}\\
&=& b_2
\end{eqnarray}


Therefore,
\begin{eqnarray}
\phi(z)=\dfrac{a_3 z^{3}+a_2 z^{2}}{b_1 z-1}
\end{eqnarray}
Hence $(\ref{f1})$ and $(\ref{w12})$ turns out to
\[\begin{cases} a_3 +a_2+1= b_1 \\ (a_2+a_3)(2a_3 +a_2-1)=0 \end{cases}\]
What implies that

\[(I)\:\begin{cases} a_2=-a_3 \\ b_1=1 \end{cases}\quad\mbox{or}\quad(II)\:\begin{cases} a_2=1-2a_3 \\ b_1=2-a_3 \end{cases}\]

The solution $(I)$ is dropped out since for it the cubic function $f$ degenerates to a quadratic function.

Therefore, we have
\begin{eqnarray}
\phi_{a_3}(z)=\phi (z)=\dfrac{a_3 z^{3}+(1-2a_3) z^{2}}{(2-a_3) z-1}
\end{eqnarray}
and solving the equation $\displaystyle{\dfrac{W(\phi)(z)}{z(z-1)}=0}$, we shall find the fourth critical point
\begin{eqnarray}
c(a_3)=\dfrac{2a_3 -1}{a_3 (2-a_3)}
\end{eqnarray}
The uniquiness follows from the fact that the identity automorphism of $\cc$ is the unique one that have strictly more than $2$ fixed points. But a $\emph{M\"obius}$ function assuring the equivalence between two such normal cubic functions will have to fix  pointwise the set $\{0, 1, \infty\}$, then it has to be the identity function, so those two functions are equal actually.
\end{proof}

From Lemma $\ref{just-norml}$ and Lemma $\ref{justcnf}$ it follows

\begin{cor}\label{cubic-n-f}
Any cubic generic ratinal function is equivalent to a cubic generic ratinal function of the form
\begin{eqnarray*}
\phi(z) = \frac{a z^3 + (1-2 a)z^2}{(2-a) z - 1}
\end{eqnarray*}
\end{cor}

\begin{prop}
The conformal automorphism group of the rational functions $c(a)$ and $\dfrac{1}{c(a)}$ are
 \begin{eqnarray}
Aut\left(c(a)\right)=\left\{ z,\dfrac{1}{z}\right\}\cong \Z_{2}
\end{eqnarray}
and
 \begin{eqnarray}
Aut\left(\dfrac{1}{c(a)}\right)=\left\{ z,\dfrac{1}{z}, 1-z,\dfrac{z-1}{z},\dfrac{1}{1-z}, \dfrac{z}{z-1}\right\}=\langle z,\dfrac{1}{z}, 1-z\rangle\cong \mathcal{S}_3
\end{eqnarray}
\end{prop}
\begin{proof}
Comparing $\phi\circ c$ with $c\circ\phi$ as well as  $\psi\circ (\frac{1}{c})$ against $(\frac{1}{c})\circ\psi$, for each affine map $\phi\in \left\{ z,\dfrac{1}{z}\right\}$ and $\psi\in \left\{ z,\dfrac{1}{z},\dfrac{z-1}{z},\dfrac{1}{1-z}, \dfrac{z}{z-1}, 1-z \right\}$, we conclude that
\[ \left\{ z,\dfrac{1}{z} \right\}\subset Aut(c(a))\] 
and
\[ \left\{ z,\dfrac{1}{z},\dfrac{z-1}{z},\dfrac{1}{1-z}, \dfrac{z}{z-1}, 1-z \right\}\subset Aut\left(\dfrac{1}{c(a)}\right)\].

Now, notice that an automorphism of a rational function permutes its periodics points for each fixed period. 

$c(a)$ has two fixed points, $a=0$ and $a=1$, and one period $2$ orbit, $\sqrt[3]{-1}\mapsto -(-1)^{2/3}$. And, $\dfrac{1}{c(a)}$ has three fixed points, these fixed points are the $3$-roots of the unity, $1, \sqrt[3]{-1}, -(-1)^{2/3}$.

Since two automorphisms of $\overline{\C}$ that coincide at three points should to be equal follows that we can have at most  $2$ \emph{Möebius maps} comuting with $c(a)$ and $6=3!$ \emph{Möebius maps} comuting with $\dfrac{1}{c(a)}$.

Therefore, the stated proposition is true.
\rem{We can via this argumentation  obtain a group order boundness result for that automorphisms groups}
\end{proof}

Now, let's stick to some brief computations with $\phi$:

\begin{eqnarray}
\phi(z)=0 \iff z=0 \;\mbox{or}\;z=2-\dfrac{1}{a}
\end{eqnarray}

\begin{eqnarray}
\phi(z)=1 \iff z=1 \;\mbox{or}\;z=-\dfrac{1}{a}
\end{eqnarray}

\begin{eqnarray}
\phi(z)=\infty \iff z=\infty \;\mbox{or}\;z=\dfrac{1}{2-a}
\end{eqnarray}

\begin{prop}
$\phi(c)={a}^2c^3$
\end{prop}
\begin{proof}

\begin{eqnarray}
\phi(z) = \frac{a z^3 + (1-2a)z^2}{(2-a) z - 1}
\end{eqnarray}
and
\begin{eqnarray}
c = \frac{2a - 1}{a(2-a)}
\end{eqnarray}

Then,
\begin{eqnarray}
1-2a=a c(a-2)
\end{eqnarray}

Hence,
\begin{eqnarray}
\phi(c)&=&\dfrac{c^2(a c+(1-2a))}{(2-a)c-1}\\\nonumber
&=& c^2\dfrac{a c+(1-2a)}{(2-a)c-1}\\\nonumber
&=& c^2\dfrac{a c+(a-2)a c}{\dfrac{2a -1}{a}-1}\\\nonumber
&=&a c^3\dfrac{a -1}{\left(\dfrac{a -1}{a}\right)}\\\nonumber
&=&a^2 c^3
\end{eqnarray}
\end{proof}

\subsection{Degenerate maps}
For $a=0$ and $a=1$ we get, respectively, the maps $f_0 :z\mapsto \dfrac{z^2}{2z -1}$ and $f_1 :z\mapsto z^2$. $f_0$ and $f_1$ are conformally conjugated through the conformal map $S(z)=\dfrac{z-1}{z}.$

Such degenerated functions can be achieved through the real families $\alpha(c)|_{\R}$ and $\beta(c)|_{\R}$ by making $c\in\R$ goes to $-\infty$ for $\alpha(c)$, and for $\beta$, by making $c\in\R$ to go to $+\infty$ obtaining in such a way $a=0$. And  $a=1$ is attained only by $\alpha(c)|_{\R}$ when $c\in\R$ approaches to $c=1$.


The coefficient $a=2$ is obtained 
making $c\in\R$ goes to $+\infty$ for $\alpha$ and taking $c\in\R$ goes to $-\infty$ for the familly $\beta$.  In this case, we obtain the limit polynomial function $P(z)=-2z^3 + 3z^2$. $P$ is the monic cubic polynomial map that possesses $z=0$ and $z=1$ as fixed critical points.

Although the real $beta(c)$ function has a indeterminacy at $c=0$ a limit function is attained as we set the coefficient $a$ tend to $\infty\cc$. The function that we get is the (parabolic) quadratic map $z\mapsto z^2 +z$. In this case, $z=1$ is no longer a critical point or even a fixed point.

The coefficient $a=\frac{1}{2}$ is achieved only by $\alpha$ as $c\to 0$, 
giving the function $f_{\frac{1}{2}}:z\mapsto \dfrac{z^3}{3z-2}$ .  
$f_{\frac{1}{2}}$ is conformally conjugated by the inversion map  
to $P(z)$.The reason for that is the fact that the fourth critical point $c$ has collapsed against the super-attracting fixed point at $z=0$. That function has a critical point of multiplicity $2$ at $z=0.$

For the coefficient $a=-1=\beta(1)$ we obtain the function $\frac{3 z^2-z^3}{3 z-1}$. That function has a critical point of multiplicity $2$ at $z=1$, thus it has only $3$ fixed points at $0,1$ and $\infty$.

Although the real $\beta(c)$ function has a indeterminacy at $c=0$ a limit function is attained as we set the coefficient $a$ goes towards $\infty\in\cc$. The function that we get is the (parabolic) quadratic map $z\mapsto z^2 +z$. In this case, $z=1$ is no longer a critical point or even a fixed point.

We saw that when $c\in\R$ tends to $+\infty$ or to $-\infty$ the maps $\phi_1$ and $\phi_2$ degenerates into degree $2$ rational functions. What more we can single out is that those limit functions obtained when $ c $ moves towards $ + \infty $ are the same as when $ c $ moves towards $ - \infty $, but they are exchanged between $ \phi_{\alpha}$ and $ \phi_{\beta} $.

\begin{figure}[H]
    \centering
    \subfloat[real graph of $\alpha_1$ and $\alpha_2$]{{\includegraphics[width=10cm]{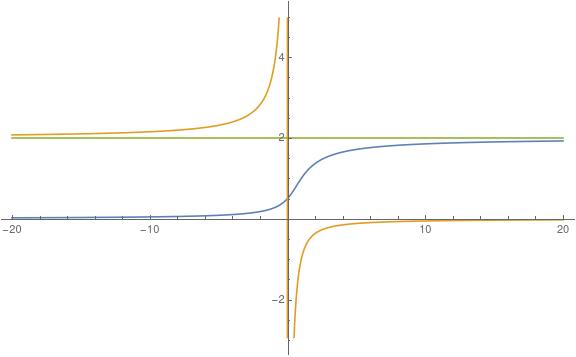} }}%
    \label{fig:rg2ab}%
\end{figure} 

The degeneracies described above shows that the \emph{locus} (non compact varieties) defined by $\alpha$ and $\beta$ into $\C(z)_3$ it has some interesting \emph{fins}. Two of that siting inside $\C(z)_3$ are the conformally conjugated cubic rational functions $P(z)=2z^3-3z^2$ and $Q(z)=\dfrac{z^3}{3z-2}$ ($P(z)$ is the polynomial dynamical model for all such normal rational maps considered (see $\ref{dynamics}$)). And two another fins at the boundary of $\C(z)_3$ into $\C(z)_2$, namely the conformally conjugated maps $z\mapsto z^2$ and $z\mapsto \frac{z^2 -1}{z}$, as described above.

For each $c\in\C-\{0,1,\frac{1}{2}\left(1\pm \sqrt{3}\right)\}$, $\alpha(c)$ and $\beta(c)$ are different.

\subsection{Pullback graphs of real generic cubic rational functions}\label{sbsec-realcubicpbg}

We are going to consider those generic cubic real rational maps of the form $\phi_a$. That is, those ones $\phi_a$ with $a\in\R-\{-1,0,1/2,1,2\}.$ $\rr$ is a postcritical curve for each such function.


Proposition $(\ref{prop-shape-graph})$ assures that the inverse image of $\overline{\R}$ by $\phi_a$ is a planar graph with $4$ vertices of degree $4$ and $6$ \emph{Jordan} domains as its faces. Since $f(\rr)\subset\rr$, we also have $\rr \subset \Gamma_f:=f^{-1}(\rr)$ and furthermore, each face possesses a complex conjugate face since $\phi_a (z)=\overline{\phi_a (\overline{z})}.$

The finite preimage of $\infty$ by $\phi_a$ is the point $z_\infty =\dfrac{1}{2-a}\in\R$. We know also that $\alpha =\beta$ if and only if $c\in\{\dfrac{1}{2}+\dfrac{\sqrt{3}}{2}i, \dfrac{1}{2}-\dfrac{\sqrt{3}}{2}i\}$. So, in the real case where all critical points of $\phi_a$ are real the finite preimage of $\infty$  by $\phi_{\alpha}$ and $\phi_{\beta}$ are distinct.

There are $3$ possibilities for the position of the fourth critical point $c$ with respect to the another $3$ fixed critical points of $\phi_a$, namely, $c\in(\infty,0)$, $c\in(0,1)$ and $c\in(1,\infty)$.

In each situation, the pullback graph of $\phi_a$ are determined and distinguished by its finite preimage of $\infty$. 

The reason for that is the following.

$z_{\infty}\notin\{0,1,\infty, c\}$ for both maps $\phi_\alpha$ and $\phi_\beta$. Then, having fixed a data $d\in\{(c<0<1<\infty),(0<c<1<\infty),(0<1<c<\infty)\}$, 
 $z_\infty$ will be between two cons points in $d$. Since each face of the pullback graph has only one preimage of each critical value on its boundary, cannot occur $c<z_{\infty}<+\infty$, $1<z_{\infty}<+\infty$, $-\infty<z_{\infty}<c$ and $-\infty<z_{\infty}<0$.
 
For the same reason there should be an arc inside the upper half plane connecting those two points in $d$ neighboring $z_{\infty}$, because, otherwise, will there exist a face with $\infty$ and $z_\infty$ in its boundary.

\begin{lem}\label{l3rg1}
For $d\in\{(c<0<1<\infty), (0<c<1<\infty), (0<1<c<\infty)\}$, if $B<z_\infty <C$ for $B\in{d}$ the biggest element in $d$ less than $z_\infty$ and $C\in{d}$ the smallest element in $d$ greater than $z_\infty$. Then 
\begin{itemize}
\item[(a)]{$B,C\in\{c,0,1\}$;}
\item[(b)]{there exist two arcs connecting $B$ to $C$ one inside the upper half plane and the other one into the lower half plane.}
\end{itemize}
\end{lem}
\begin{proof}
We saw above that $B$ and $C$ can not be $\infty$, so $(a)$ follows. Namelly, the configurations allowed are : $c<0<z_{\infty}<1$, $0<z_{\infty}<1<c$, $c<z_{\infty}<0<1$, $0<z_{\infty}<c<1$, $0<c<z_{\infty}<1$, $0<1<z_{\infty}<c.$

Assume that $(b)$ does not holds. As each point in $d$ is a $4$-valent vertex of $\Gamma$, should exist $2$ arc from $B$ to its precussor $A$ in $d$ and should exist $2$ arcs connecting $B$ to its successor $D$ in $d$. Otherwise, the arcs connecting $B$ to $D$, each one into the upper and lower half plane turns impossible to connect $A$ and $B$, what force this vertices to have a edge incident to it tow times. But this not happens for pulback graphs.


 $A=\infty$ or $D=\infty$, but in any situation will exist two faces of the pullback graph with both $z_\infty$ and $\infty$ on their boundary. This can not occur.

\end{proof}



Recall that $\alpha (\R)= (0, 2)$ and $\beta(\R)=(-\infty,0)\cup (2,+\infty)$. 
\begin{lem}\label{lemzinf}
For $\phi_{\alpha}$ :
\begin{itemize}
\item{$z_\infty\in\left(\dfrac{1}{2}, \dfrac{2}{3}\right)$, for $0<\alpha(c)<\dfrac{1}{2}$;}
\item{$z_{\infty}\in\left(\dfrac{2}{3}, 1\right)$, for $\dfrac{1}{2}<\alpha(c)<1$;}
\item{ $z_{\infty}>1$, for $1<\alpha(c)<2$.}
\end{itemize}
And for $\phi_{\beta}$ :
\begin{itemize}
\item{ $z_{\infty}\in\left(0, \dfrac{1}{2}\right)$, for $\beta(c)<0$;},
\item{ $z_{\infty}<0$, for $\beta(c)>2$.}
\end{itemize}
\end{lem}

We know also that

\begin{lem}\label{zinfpositionc}
For $\phi_{\alpha}$
\begin{eqnarray}
z_\infty < c\; \Leftrightarrow \; \quad c>1
\end{eqnarray} 
And for $\phi_{\beta}$
\begin{eqnarray}
z_\infty < c\; \Leftrightarrow \; \quad c>0
\end{eqnarray} 
\end{lem}

From Lemma $\ref{lemzinf}$ and Lemma $\ref{zinfpositionc}$ we conclude:

\begin{cor}
Set $z_{\infty}^{\alpha}$ and $z_{\infty}^{\beta}$ be the finite preimage of $\infty$ by $\phi_{\alpha}$ and $\phi_{\beta}$ respectively. Then,
\begin{itemize}
\item[(1)]{for $c<0$, $z_{\infty}^{\alpha}\in (0,1)$ and $z_{\infty}^{\beta}\in (c,0)$}
\item[(2)]{for $0<c<1$, $z_{\infty}^{\alpha}\in (c,1)$ and $z_{\infty}^{\beta}\in (0,1)$\quad({$z_{\infty}^{\beta}\in (0,c)$})}
\item[(3)]{for $c>1$, $z_{\infty}^{\alpha}\in (1,c)$ and $z_{\infty}^{\beta}\in (0,1)$}
\end{itemize}
\end{cor}
 \begin{figure}[H]
    \centering
    {{\includegraphics[width=14cm]{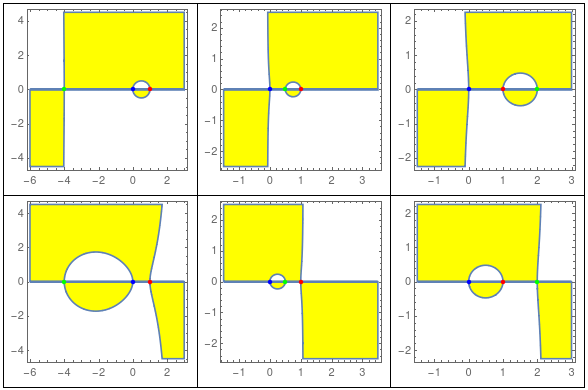} }}%
   \caption[summary of possible configurations]{summary of possible configurations: first line corresponding to $\phi_{\alpha}$ and the seconf one to $\phi_{\beta}$} \label{fig:configs3rpbg}%
\end{figure} 

\begin{thm}
For $\phi_\alpha$ we have:
\begin{itemize}
\item[$(1\alpha)$]{$\Gamma_{c}^{\alpha}$ are all isotopic relative to $\{0,1, \infty\}$ for every $c<0$;}
\item[$(2\alpha)$]{$\Gamma_{c}^{\alpha}$ are all isotopic relative to $\{0,1, \infty\}$ for every $0<c<1$;}
\item[$(3\alpha)$]{$\Gamma_{c}^{\alpha}$ are all isotopic relative to $\{0,1, \infty\}$ for every $1<c$;}
\item[$(4\alpha)$]{and those pullback graphs in $(1\alpha)$, $(2\alpha)$ and $(3\alpha)$ are non-isotopic between them.}
\end{itemize}
And, for $\phi_\beta$ we have:
\begin{itemize}
\item[$(1\beta)$]{$\Gamma_{c}^{\beta}$ are all isotopic relative to $\{0,1, \infty\}$ for every $c<0$;}
\item[$(2\beta)$]{$\Gamma_{c}^{\beta}$ are all isotopic relative to $\{0,1, \infty\}$ for every $0<c<1$;}
\item[$(3\beta)$]{$\Gamma_{c}^{\beta}$ are all isotopic relative to $\{0,1, \infty\}$ for every $1<c$;}
\item[$(4\beta)$]{and those pullback graphs in $(1\beta)$, $(2\beta)$ and $(3\beta)$ are non-isotopic between them.}
\end{itemize}
\end{thm}

\paragraph{proof of item $(1\alpha)$}
Lemma $\ref{lemzinf}$ and Lemma $\ref{zinfpositionc}$ implies that for every $c<0$ the pullback graph are embeddings of the same abstract oriented graph $G_{<0}^{\alpha}:=\{\{\tilde{c},\tilde{0},\tilde{1},\tilde{\infty}\},\{\tilde{0}\to\tilde{1},\tilde{0}\to\tilde{1},\tilde{0}\leftarrow\tilde{1},\tilde{c}\to\tilde{0},\tilde{1}\to\tilde{\infty},\tilde{\infty}\to\tilde{c},\tilde{\infty}\to\tilde{c},\tilde{\infty}\leftarrow\tilde{c}\}\}$.

Since $\phi_a (i\R)\cap\R=\emptyset$ the arcs $l_{c\infty}^{+}$ and $l_{c\infty}^{-}$ connecting $c$ to $\infty$ into the upper and lower plane respectively are contained  into the left half plane $\{z\in\C; \Re(z)<0\}$. And for the same reazon  the arcs $l_{01}^{+}$ and $l_{01}^{-}$ connecting $0$ to $1$ into the upper and lower plane respectively are contained  into the right half plane $\{z\in\C; \Re(z)>0\}$. This holds for all $c<1$.
Let $A:=\myov{\{z\in\C; \Re(z)>0\}}$, $B:=\myov{\{z\in\C; \Re(z)>0,\Im(z)>0\}}$ and $C:=\myov{\{z\in\C; \Re(z)>0,\Im(z)>0\}}$.

Any two arcs $l_{c\infty}^{+}\cup l_{c\infty}^{-}$ and $l_{c'\infty}^{+}\cup l_{c'\infty}^{-}$ into $A$ are homotopic relative to its end point $\infty\in A\subset\cc$. And the same is true for any par of arcs $l_{01}^{c+}$ and $l_{01}^{c'+}$ into $B$ or $l_{01}^{c-}$ and $l_{01}^{c'-}$ into $C$. Then from Theorem $\ref{iso-homo}$ and Theorem $\ref{ext-iso}$ there are three (ambient) isotopies $H_A:A\times [0,1]\longrightarrow A$, $H_B:B\times [0,1]\longrightarrow B$ and $H_C:C\times [0,1]\longrightarrow C$ each one relative to the boundary of its base ambient, $\partial{A}$, $\partial{B}$ and $\partial{C}$. Finaly, we can glue together those isotopies then producing an isotopy $H_{\cc}:\cc\times [0,1]\longrightarrow \cc$ sending $\Gamma_{c}^{\alpha}$ to $\Gamma_{c'}^{\alpha}$.

The proof is the same for the other cases.


 


In conclusion we have:

\begin{thm}[combinatorial model]
$(a)$\quad For $\phi_\alpha$ we have:
Add the marks at the vertices to better distinguish the graphs below
\begin{itemize}
\item[$(1a)$]{for every $c<0$, all $\Gamma_{c}^{\alpha}$'s has the following pattern:
\begin{figure}[H]
\begin{center}
\includegraphics[scale=0.25]{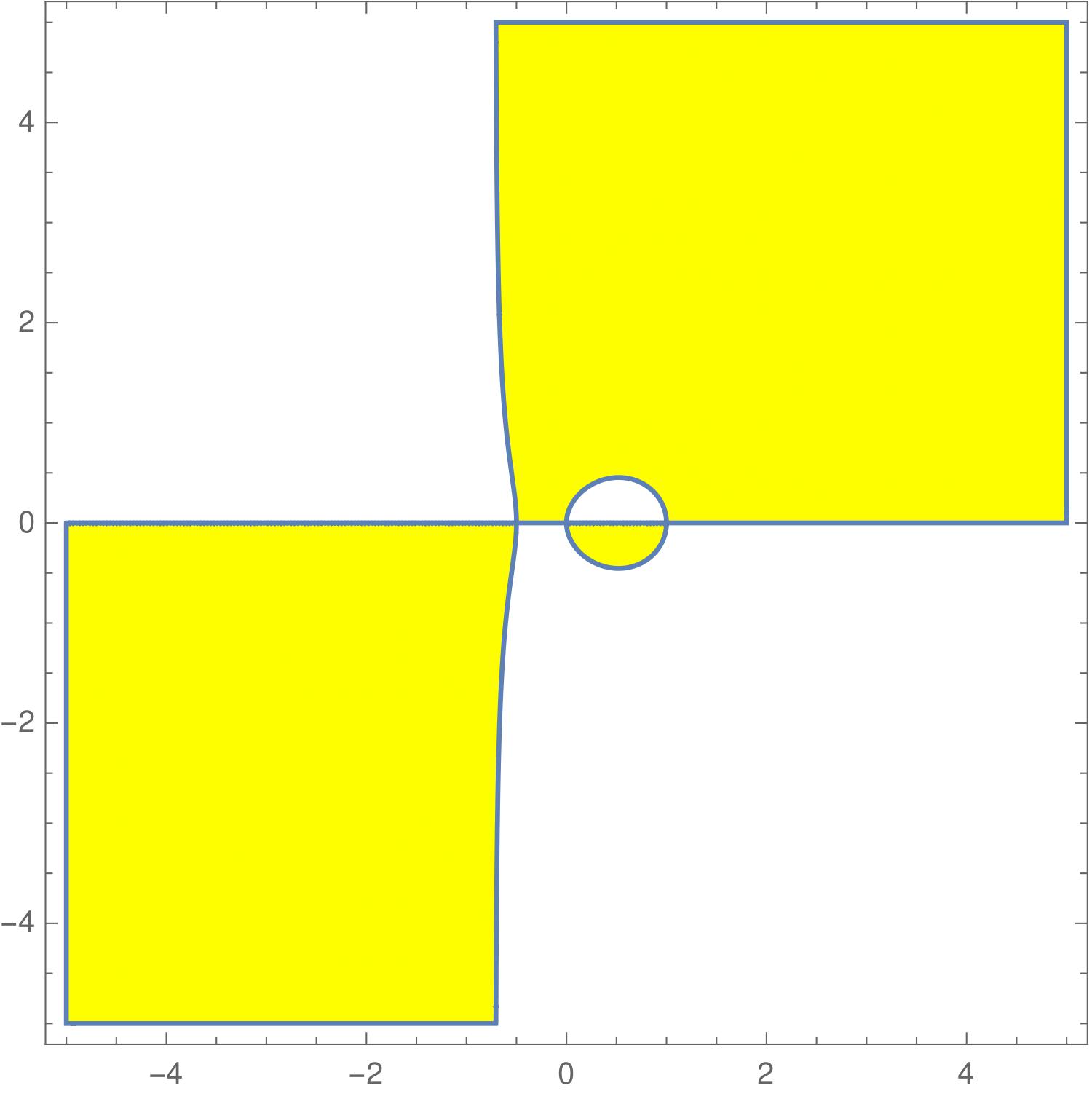}\
\caption{$c<0$}
\label{m1}
\end{center}
\end{figure}}
\item[$(2a)$]{for every $0<c<1$, all $\Gamma_{c}^{\alpha}$'s has the following pattern:
\begin{figure}[H]
\begin{center}
\includegraphics[scale=0.25]{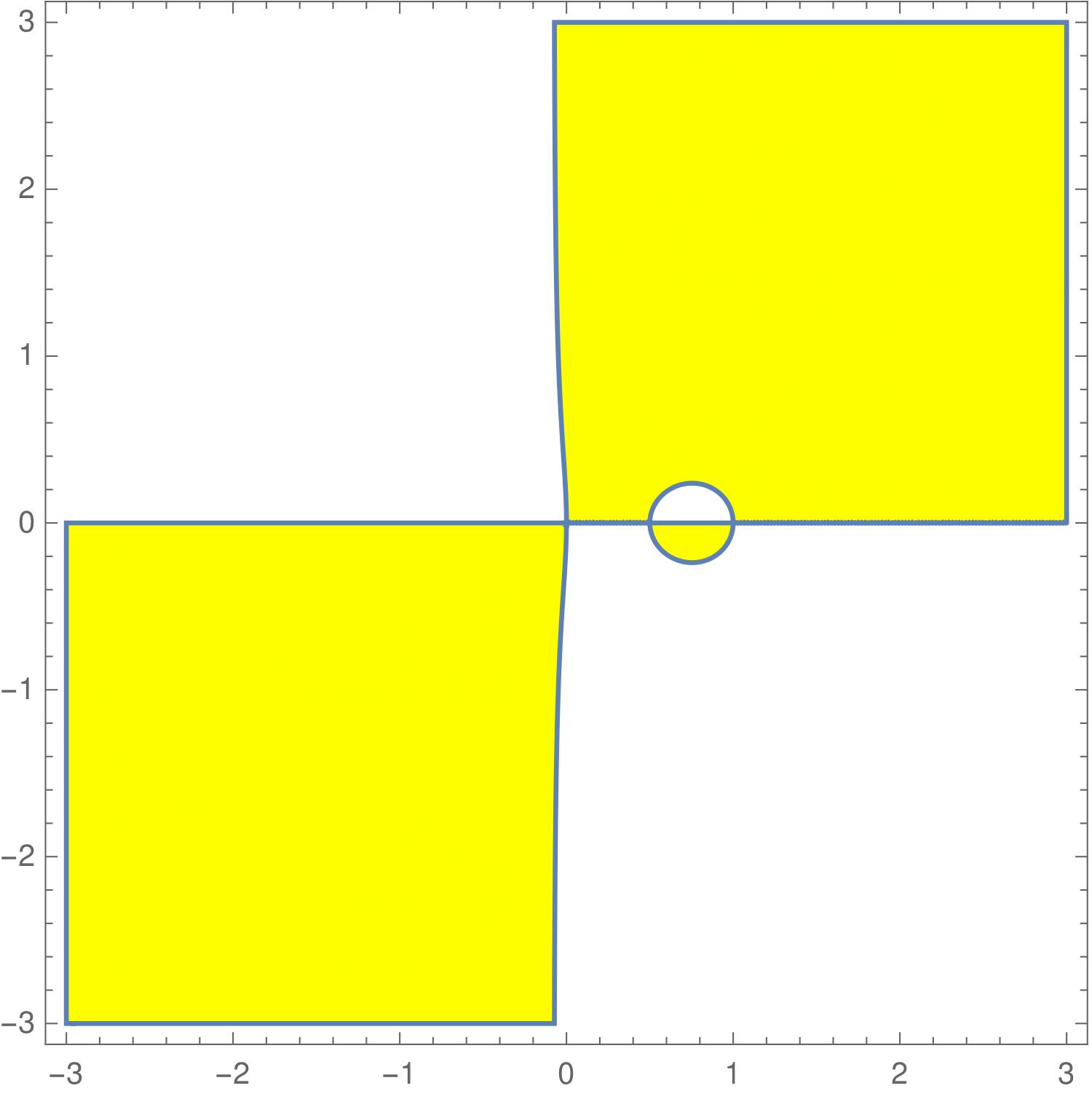}\
\caption{$0<c<1$}
\label{m2}
\end{center}
\end{figure}}
\item[$(3a)$]{for every $1<c$, all $\Gamma_{c}^{\alpha}$'s has the following pattern:
\begin{figure}[H]
\begin{center}
\includegraphics[scale=0.25]{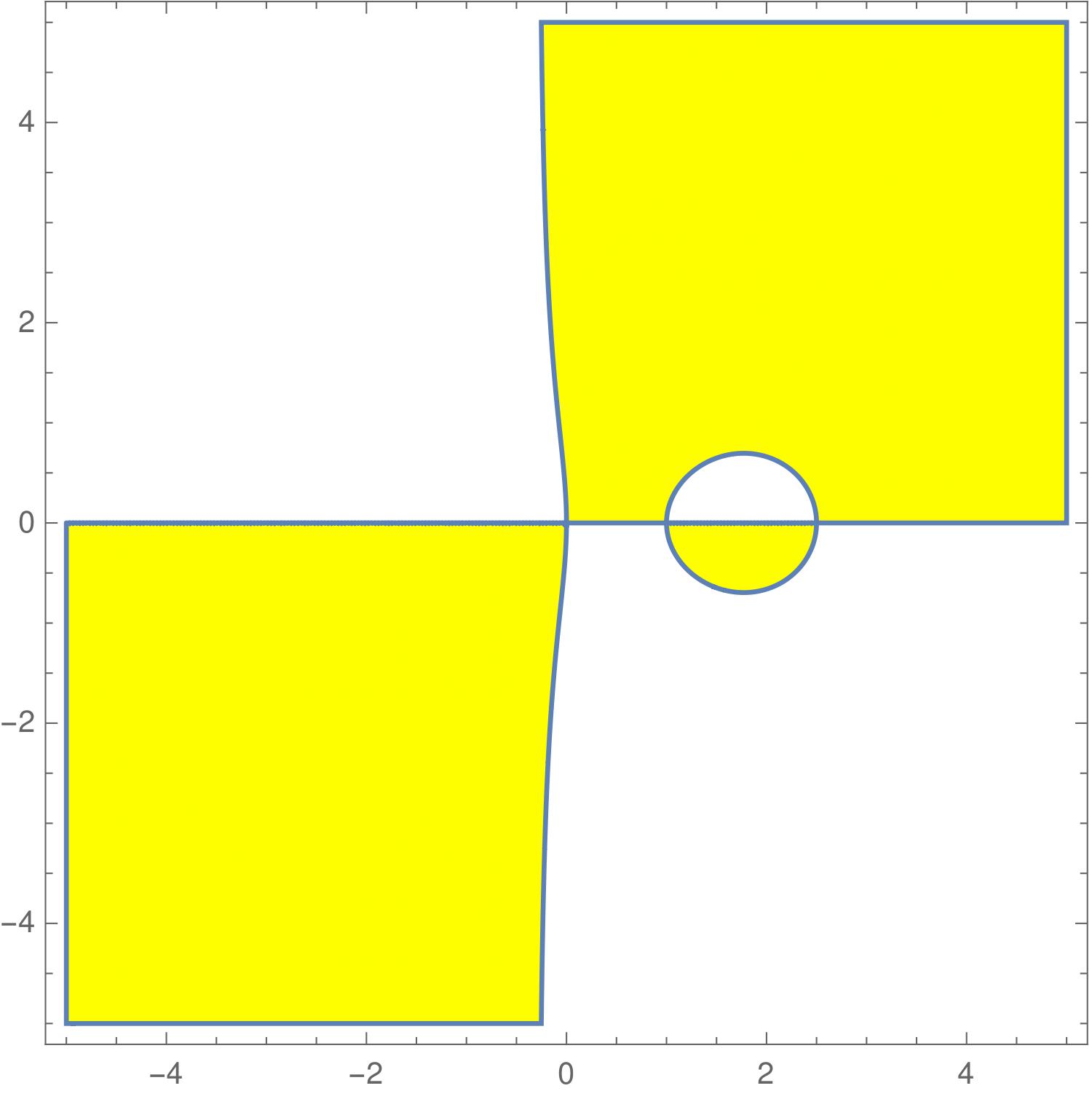}\
\caption{$1<c$}
\label{m3}
\end{center}
\end{figure}}
\end{itemize}

$(b)$\quad And for $\phi_\beta$ we have:
\begin{itemize}
\item[$(1b)$]{for every $c<0$, all $\Gamma_{c}^{\beta}$'s has the following pattern:
\begin{figure}[H]
\begin{center}
\includegraphics[scale=0.25]{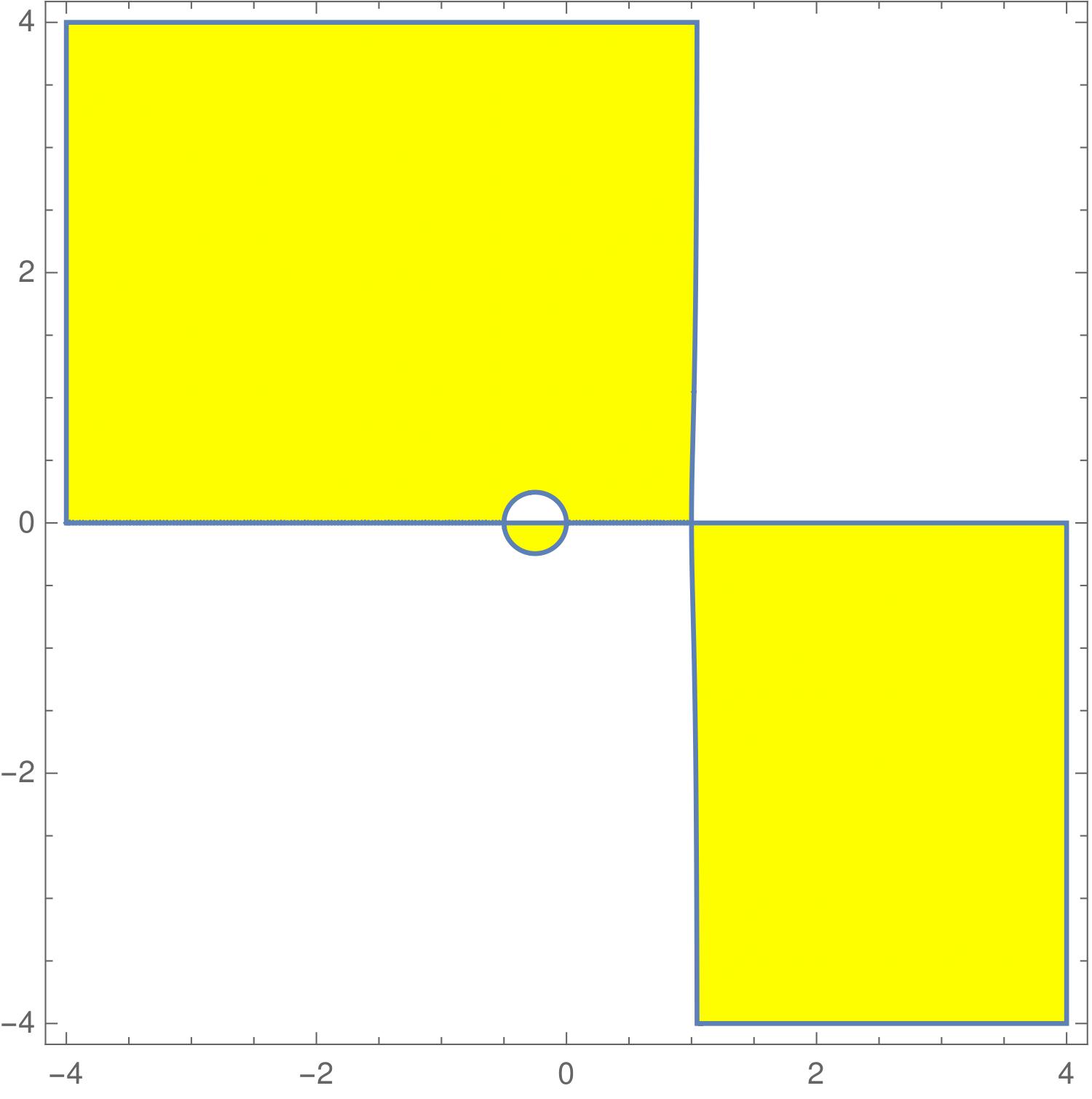}\
\caption{$c<0$}
\label{m4}
\end{center}
\end{figure}}
\item[$(2b)$]{for every $0<c<1$, all $\Gamma_{c}^{\beta}$'s has the following pattern:
\begin{figure}[H]
\begin{center}
\includegraphics[scale=0.25]{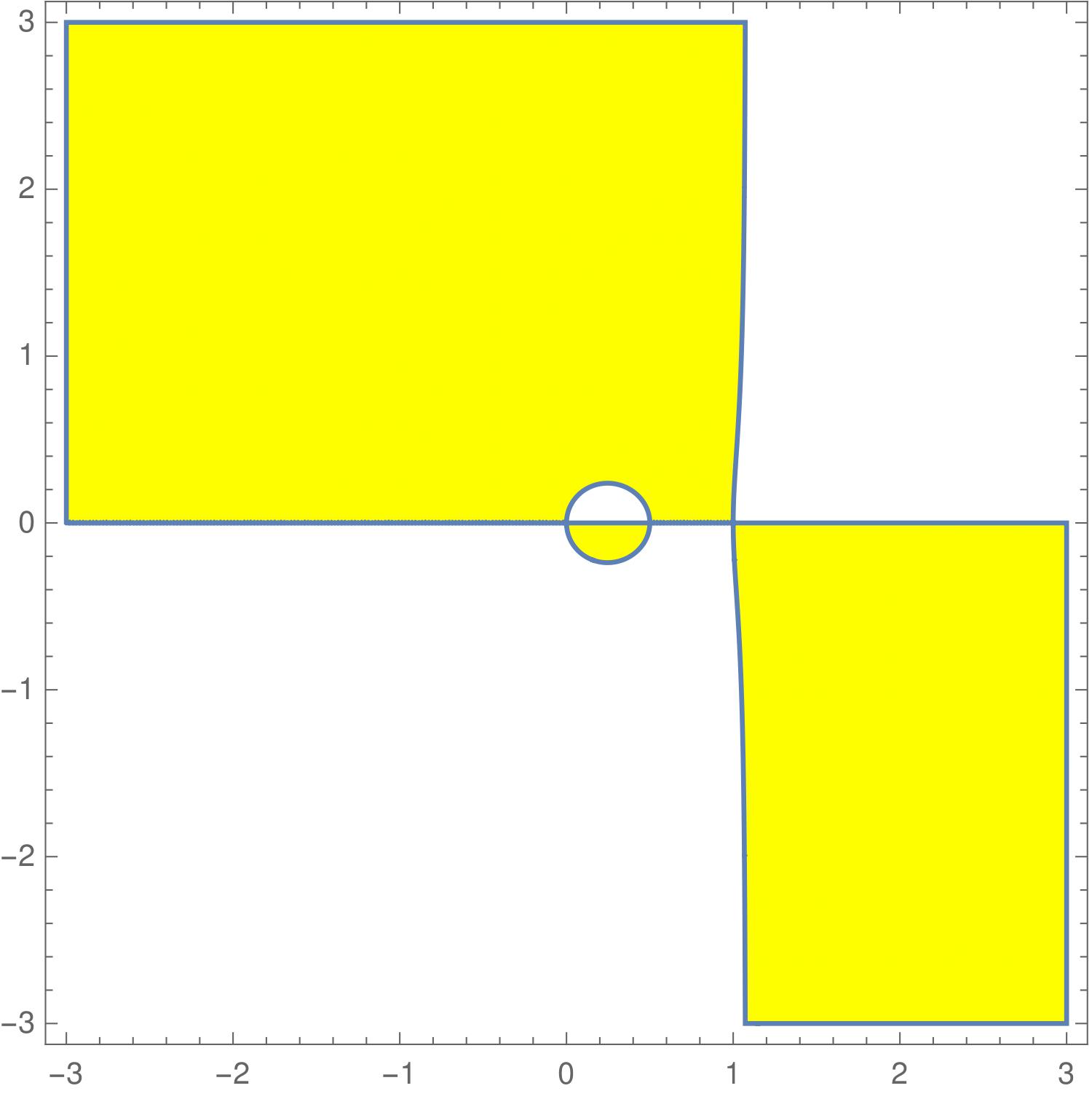}\
\caption{$0<c<1$}
\label{m5}
\end{center}
\end{figure}}
\item[$(3b)$]{for every $1<c$, all $\Gamma_{c}^{\beta}$'s has the following pattern:
\begin{figure}[H]
\begin{center}
\includegraphics[scale=0.25]{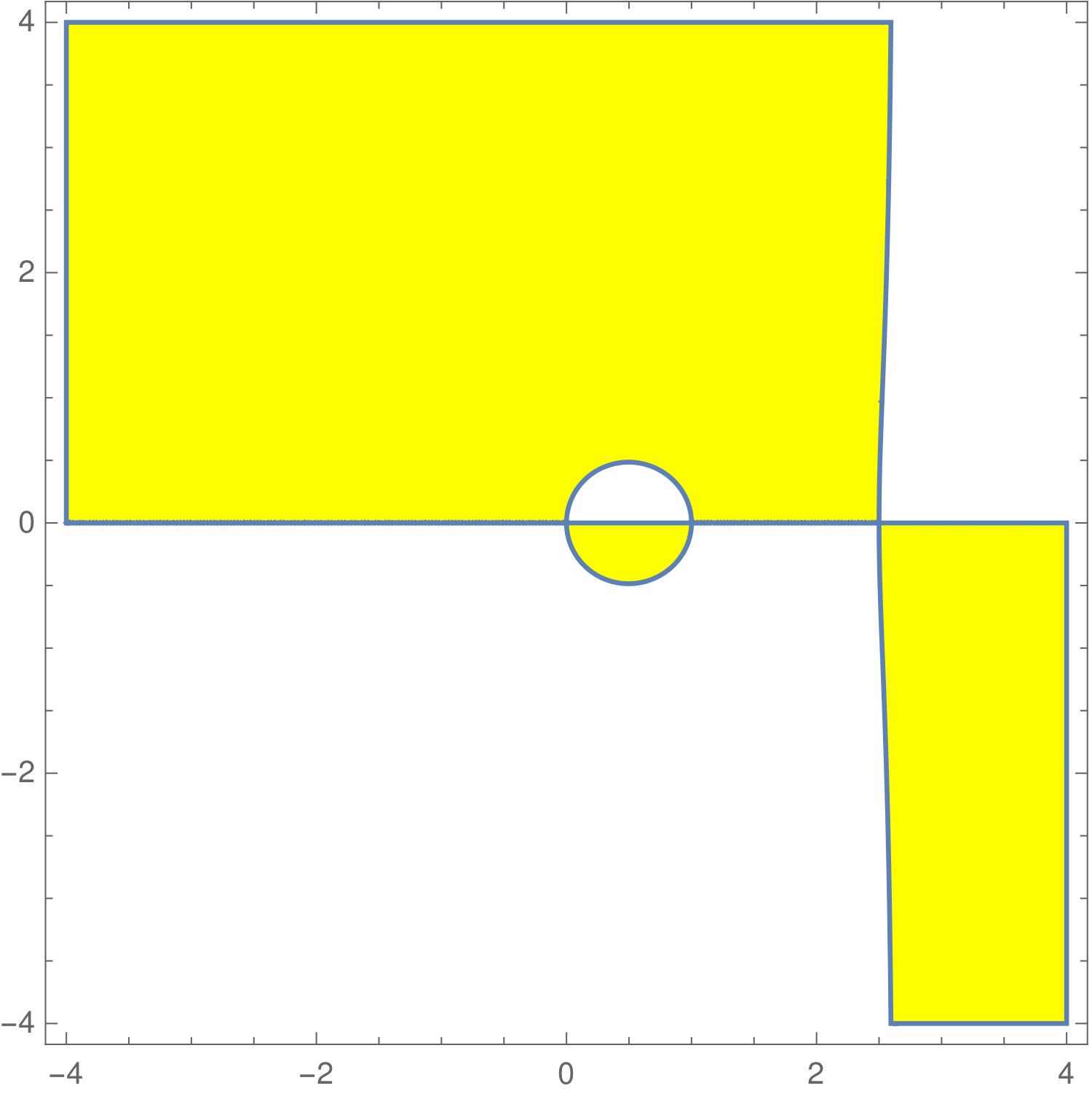}\
\caption{$1<c$}
\label{m6}
\end{center}
\end{figure}}
\end{itemize}
\end{thm}



In \cite[section: $11.1$]{MR2830310} \emph{Frank Sotille} using basics \emph{complex analitical tools} proved a continuity result for the move of the ``nets'', the 1-squeleton, along paths of rational functions. As above, the isotopy of the graph 
changes 
is shown to occur only if two critical points connected by a non real saddle-connection colides. In the colaesce the rational function and gaph degenerate into functions and graphs of smaller degree.



\section{A class of postcritical curves for the functions $\phi$}

For the rational function $\phi$ we will consider the following postcritical curves:
\[C_{0}(c):=\{t\in\R;t\geq 0\}\sqcup\{t\cdot{}c\in \C; t\geq 0\}\sqcup{\{\infty\}}\]

and
 \[C_{1}(c):=\{t\in\R;t\geq 0\}\sqcup\{t\cdot{}c\in \C; t\geq 0\}\sqcup{\{\infty\}},\] where $c\in\C-\{0,1\}$ is the critical point of $\phi.$

Below, we can see a few examples:

\section{a look at the complex setting}\label{sect:gv}

\subsubsection{Variation of $\phi (c)$ around the fixed critical points $0, 1$}

Suppose that the critical value $c$ is close to $z=0$.

\begin{eqnarray}
\sqrt{c^2 - c +1}&=&\sqrt{1+(c^2- c)}\\ \nonumber
&=&1+\dfrac{(c^2 -c)}{2}-\dfrac{1}{8}(c^2 -c)^2+\dfrac{1}{16}(c^2 -c)^3+\cdots\\ \nonumber
&=&1-\dfrac{c}{2}+\dfrac{3}{8}c^2+\dfrac{3}{16}c^3+\cdots\\ \nonumber
\end{eqnarray}
Then :
\begin{eqnarray}
a_1 c&=&(c-1)+\sqrt{c+(c^2 -1)}\\ \nonumber
&=&\left(\dfrac{1}{2}c+\dfrac{3}{8}c^2 +\dfrac{3}{16}c^3+\cdots\right)\\\nonumber
\end{eqnarray}
and
\begin{eqnarray}
a_2 c&=&(c-1)-\sqrt{c+(c^2 -1)}\\\nonumber
&=&\left(-2+\dfrac{3}{2}c-\dfrac{3}{8}c^2 -\dfrac{3}{16}c^3+\cdots\right)\\\nonumber
\end{eqnarray}

From that it follows:
\begin{eqnarray}
\phi_{\alpha}(c)&=&(a_1 c)^2 c\\\nonumber
&=&\left(\dfrac{1}{2}c+\dfrac{3}{8}c^2 +\dfrac{3}{16}c^3+\cdots\right)^2 c\\\nonumber
&=&\left(\dfrac{1}{4}c^2+\dfrac{3}{8}c^3 +\cdots\right)c\\\nonumber
&=&{\dfrac{1}{4}c^3}+\dfrac{3}{8}c^4 +\cdots+\cdots\\\nonumber
\end{eqnarray}
and
\begin{eqnarray}
\phi_{\beta}(c)&=&(a_2 c)^2 c\\\nonumber
&=&\left(-2+\dfrac{3}{2}c-\dfrac{3}{8}c^2 -\dfrac{3}{16}c^3+\cdots\right)^2 c\\\nonumber
&=&\left({4}-6c +\dfrac{15}{4}c^2-\dfrac{3}{8}c^3\cdots\right)c\\\nonumber
&=&{{4}c}-6c^2 +\dfrac{15}{4}c^3-\dfrac{3}{8}c^4\cdots\\\nonumber
\end{eqnarray}

Now, suppose that $c$ is close to $z=1$, \emph{i.e.}, $c=1+h$ for small $|h|$.
Then,
\begin{eqnarray}
\sqrt{c^2 - c +1}&=&\sqrt{1+(h+h^2)}\\\nonumber
&=&1+\dfrac{(h+h^2)}{2}-\dfrac{1}{8}(h+h^2)^2+\dfrac{1}{16}(h+h^2)^3-\dfrac{15}{24} (h+h^2)^4 +\cdots\\\nonumber
\end{eqnarray}

So,
\begin{eqnarray}
a_{1} c&=&(c-1)+\sqrt{c^2 - c +1}\\\nonumber
&=&h+\sqrt{1+(h+h^2)}\\\nonumber
&=&1+\dfrac{3}{2}h+\dfrac{3}{8}h^2 -\dfrac{3}{16}h^3+\dfrac{33}{48}h^{4}+\cdots\\\nonumber
\end{eqnarray}
and
\begin{eqnarray}
a_{2} c&=&(c-1)-\sqrt{c^2 - c +1}\\\nonumber
&=&h-\sqrt{1+(h+h^2)}\\\nonumber
&=&-1+\dfrac{1}{2}h-\dfrac{3}{8}h^2 +\dfrac{3}{16}h^3-\dfrac{33}{48}h^{4}+\cdots\\\nonumber
\end{eqnarray}
From that it follows
\begin{eqnarray}
\phi_{\alpha}(c)&=&(a_1 c)^2 c\\\nonumber
&=&\left(1+\dfrac{3}{2}h+\dfrac{3}{8}h^2 -\dfrac{3}{16}h^3+\cdots\right)^2(1+h)\\\nonumber
&=&\left(1+{3}h+{3}h^2 -\dfrac{3}{4}h^3+\cdots\right)(1+h)\\\nonumber
&=&1+{4h}+6h^2+\dfrac{15}{4}h^3+\cdots\\\nonumber
\end{eqnarray}
and
\begin{eqnarray}
\phi_{\beta}(c)&=&(a_2 c)^2 c\\\nonumber
&=&\left(-1+\dfrac{1}{2}h-\dfrac{3}{8}h^2 +\dfrac{3}{16}h^3-\dfrac{33}{48}h^{4}+\cdots\right)^2(1+h)\\\nonumber
&=&\left(1-h+ h^2 -\dfrac{3}{4}h^3+\dfrac{207}{192}h^4+\cdots\right)(1+h)\\\nonumber
&=&1+{\frac{1}{4}h^3}+\dfrac{3}{16}h^4 +\cdots\\\nonumber
\end{eqnarray}

In the figures $\ref{fig:phi_alpha}$ and $\ref{fig:phi_beta}$ bellow we can see how the critical value $\phi_{\alpha}(c)$ and $\phi_\beta(c)$ varies around $c=0$ and $c=1$, respectively. The colors corresponds to the quadrants to which the critical value belongs. 

For $\phi_{\alpha}$ the yellow, green, blue and white colors correspond to the first, second, third and fourth quadrants, $Q_{\alpha 1}:=\{x+iy; x>0\quad\text{and}\quad y>0\}$, $Q_{\alpha 2}:=\{x+iy; x<0\quad\text{and}\quad y>0\}$, $Q_{\alpha 3}:=\{x+iy; x<0\quad\text{and}\quad y<0\}$ and $Q_{\alpha 4}:=\{x+iy; x>0\quad\text{and}\quad y<0\}$, respectively. And for $\phi_{\beta}$, the colors yellow, green, blue and white correspond respectively to the quadrants $Q_{\beta 1}:=\{x+iy; x>1\quad\text{and}\quad y>0\}$, $Q_{\beta 2}:=\{x+iy; x<-1\quad\text{and}\quad y>0\}$, $Q_{\beta 3}:=\{x+iy; x<-1\quad\text{and}\quad y<0\}$ and $Q_{\beta 4}:=\{x+iy; x>-1\quad\text{and}\quad y<0\}$.

\begin{figure}[H]
    \centering
    \subfloat[signal $\phi_{\alpha}$]{{\includegraphics[width=7cm]{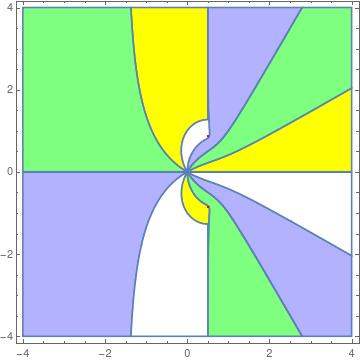} \label{fig:phi_alpha}}}%
    \qquad
    \subfloat[signal $\phi_{\beta}$]
    {{\includegraphics[width=7cm]{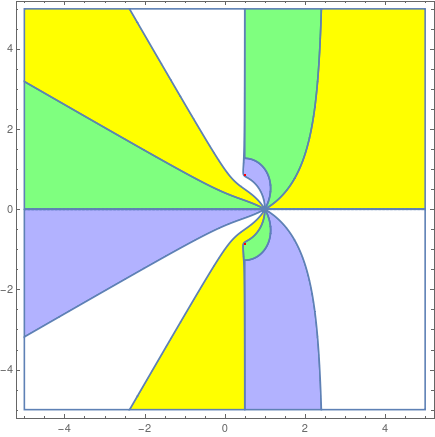} \label{fig:phi_beta}}}%
    \caption{}%
\end{figure}

From those pictures, we also see that in general the critical value does not distinguishes $\phi_{\alpha}$ and $\phi_{\beta}$.

We have $\phi_{\alpha}(0)=0$ and $\phi_{\beta}(1)=1$.

Hence, each point $w_0$ in a neighborhood around $w=0\in\C$ has $3$ preimagens under $\phi_{\alpha(\ast)}(\ast)$, say $c_1, c_2, c_3$. Each such preimagem is a critical point of the rational function $\phi_{\alpha(c)}(\ast)$ with critical value $w_0 =\phi_{\alpha(c)}(c)$. The same thing happens for the map $\phi_{\beta(\ast)}(\ast)$ around $c=1\in\C$. In this way, the functions $\phi_{\alpha(c_1)}$, $\phi_{\alpha(c_2)}$ and $\phi_{\alpha(c_3)}$ possesses the curve $C_{\alpha}(0)$ as a postcritical curve.

Below we find some examples showing the posticritical curve and the respective pullback graphs for the functions $\phi_{\alpha(c_1)}$, $\phi_{\alpha(c_2)}$ and $\phi_{\alpha(c_3)}$.

\begin{figure}[H] \label{fig:configs3rpbg}%
    \begin{center}
    {{\includegraphics[width=8cm]{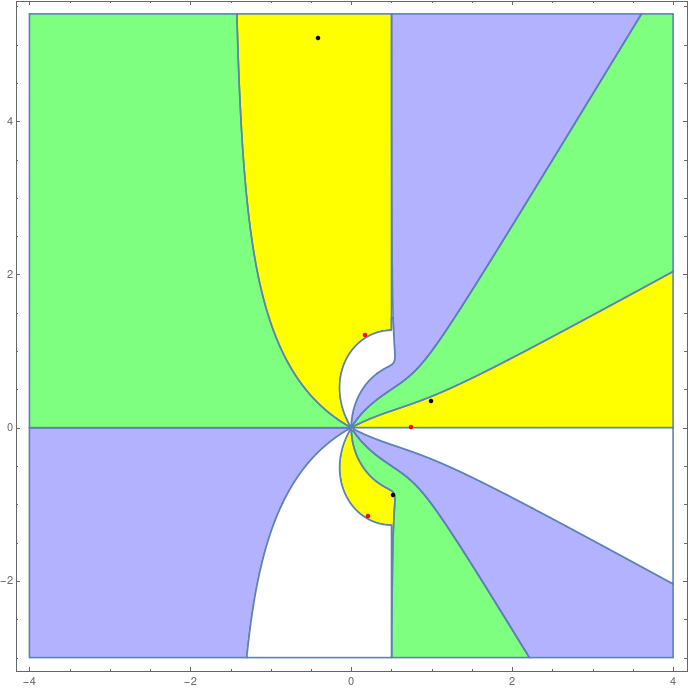} }%
   \caption[$\varphi_{\alpha(c)}(c)$ around $c=0$]{ $\varphi_{\alpha(c)}(c)$ around $c=0$;\\
  \textcolor{red}{Red points}: $c_1=0.7394232981232858 + 0.00940663360697634\sqrt{-1}$,\\        $c_2=0.1743999999999997+1.2039999999999997\sqrt{-1}$,\\
 $ c_3=0.20316520361107504 - 1.1580795225777576\sqrt{-1}$;\\
  \textbf{Black points}: $c'_1=0.993928733987696+0.3487141844400933\sqrt{-1}$,\\
$c'_2=-0.4138109548164717+5.0890589810301865\sqrt{-1}$,\\
$c'_3=0.5147433574385906-0.8754575404662145\sqrt{-1}$.
}}
\end{center}
\end{figure} 

The Figure $\ref{fig:cpbg1}$ below contains on the first line the postcritical curve $C:=C_{\alpha}(c_1)=C_{\alpha}(c_2)=C_{\alpha}(c_3)$. The second line contains the pullback graphs for the maps $phi_{\alpha(c_1)}$, $phi_{\alpha(c_2)}$ and $phi_{\alpha(c_3)}$ at the first, second and third column respectively, and third line contains a zoom on the images shown above it.

Those pullback graphs are not isotopic relative to the subset $\{0,1, \infty\}$. The second and third pullback graph are non isotopic (relative to the subset $\{0,1, \infty\}$) embeddings of the same abstract directed graph $\{\{\tilde{0}, \tilde{1}, \tilde{c}, \tilde{\infty}\}, \tilde{0}\to \tilde{\infty},\tilde{\infty}\to \tilde{0}, \tilde{0}\to \tilde{1},\tilde{c}\to \tilde{0},\tilde{c}\to \tilde{1}, \tilde{1}\to \tilde{c}, \tilde{\infty}\to \tilde{c}, \tilde{1}\to \tilde{\infty}\}$.


\begin{figure}[H]
    \centering
    {{\includegraphics[width=12cm]{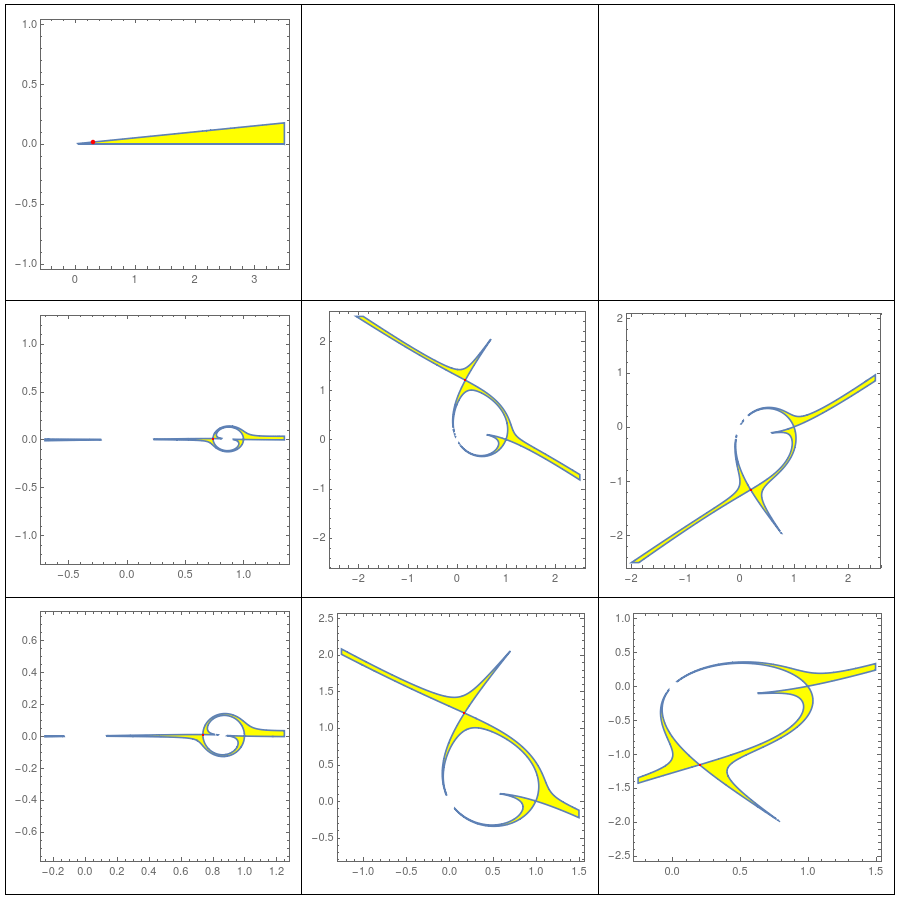} }}%
   \caption{} \label{fig:cpbg1}%
\end{figure} 

In the next Figure $\ref{fig:cpbg2}$ is shown $3$ postcritical curve on the top line and their respective pullback graphs for the function $\phi_{\alpha(c)}$ for $c=.1744+1.226\sqrt{-1}$. The second postcritical curve 
represents the isotopy class of the curve passing from $0,1,\infty$ and $c_2$ obtained from 
the first one by a \emph{half-twist} around the marked critical values $0$ (the corner) and $w=\phi_{\alpha(c)(c)}$ (the blue spot) and the third curve was obtained from a \emph{twist} around the marked critical values $0$ and $w$ on the first postcritical curve. The pullback graphs for the second and third postcritical curves can be obtained from the first pullback graph from a appropriate \emph{balanced move} as we can see below.

\begin{figure}[H]
    \centering
    {{\includegraphics[width=13cm]{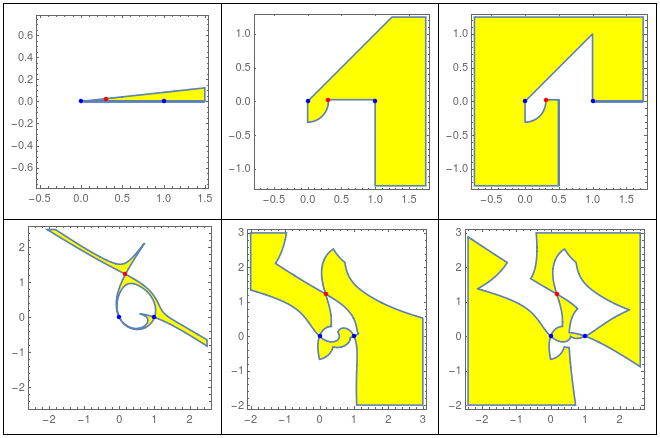} }}%
   \caption{varying the posticritical curve} \label{fig:cpbg2}%
\end{figure}

\begin{figure}[H]
    \centering
    \subfloat[balanced move on the first pullback graph to transform it into the second pullback graph of Figure $\ref{fig:cpbg2}$;]
    {{\includegraphics[width=3.62cm]{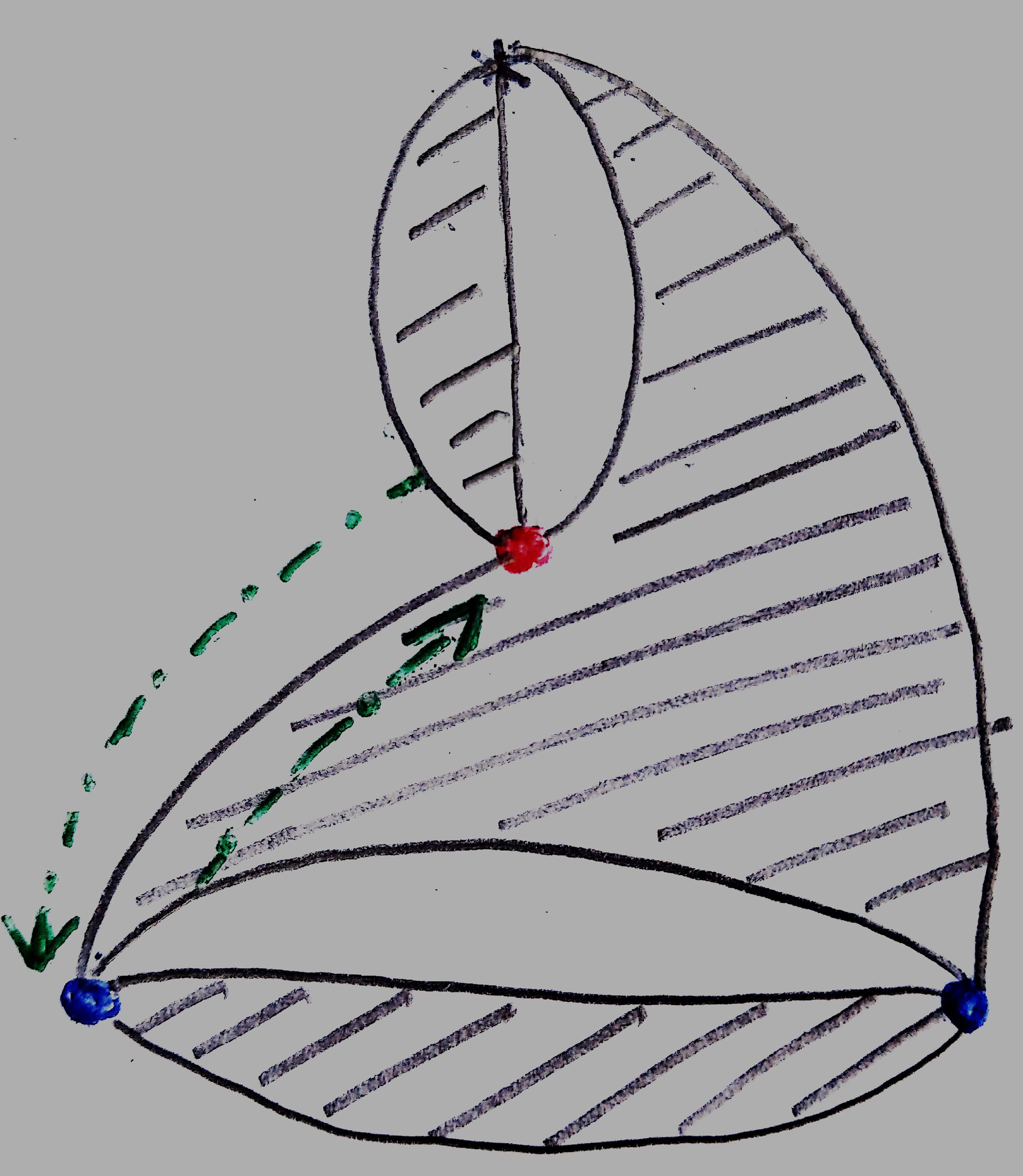}}}\label{fig01}\quad
      \subfloat[the second pullback graph obtained after the indicated balanced move on \textbf{(a)}/and a new balanced move indicated to be performed]%
    {{\includegraphics[width=3.44cm]{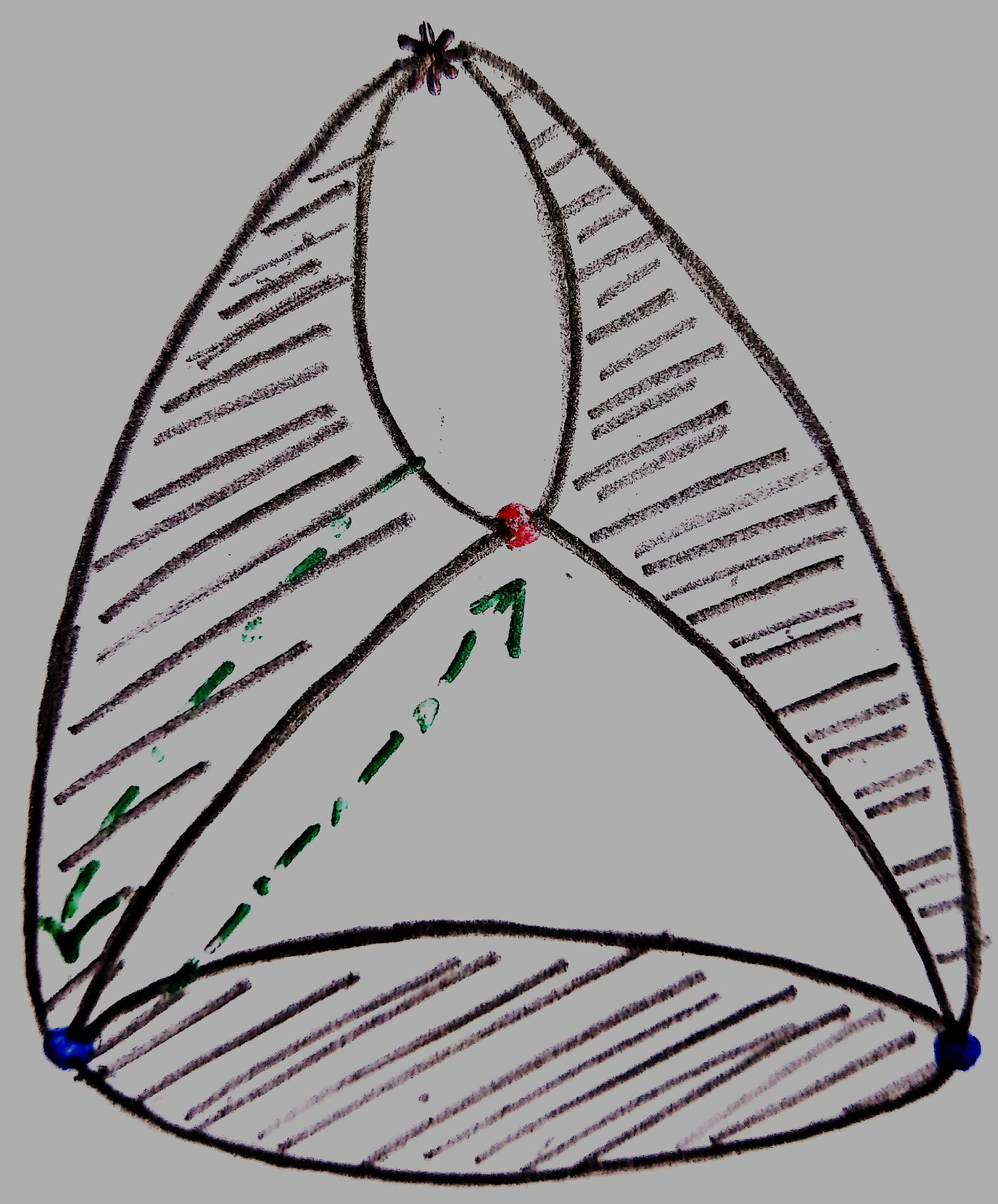}}}
    \quad
      \subfloat[the third pullback graph of Figure $\ref{fig:cpbg2}$ obtained after the indicated balanced move on \textbf{(a)} and \textbf{(b)}]
    {{\includegraphics[width=3.73cm]{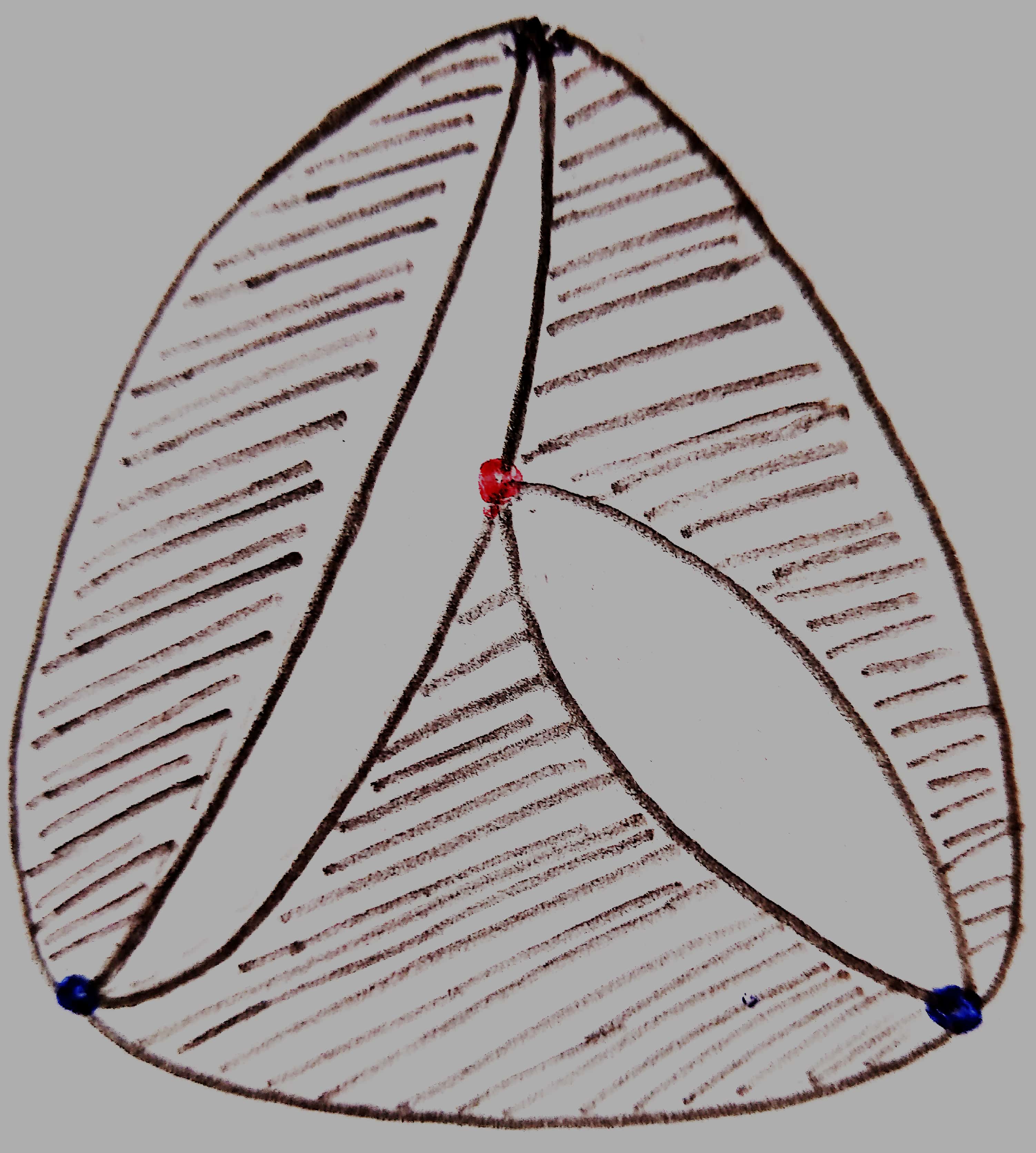}}}
   \caption{varying the posticritical curve} \label{fig:cpbg2.1}%
\end{figure}

\begin{figure}[H]
    \centering
    {{\includegraphics[width=15cm]{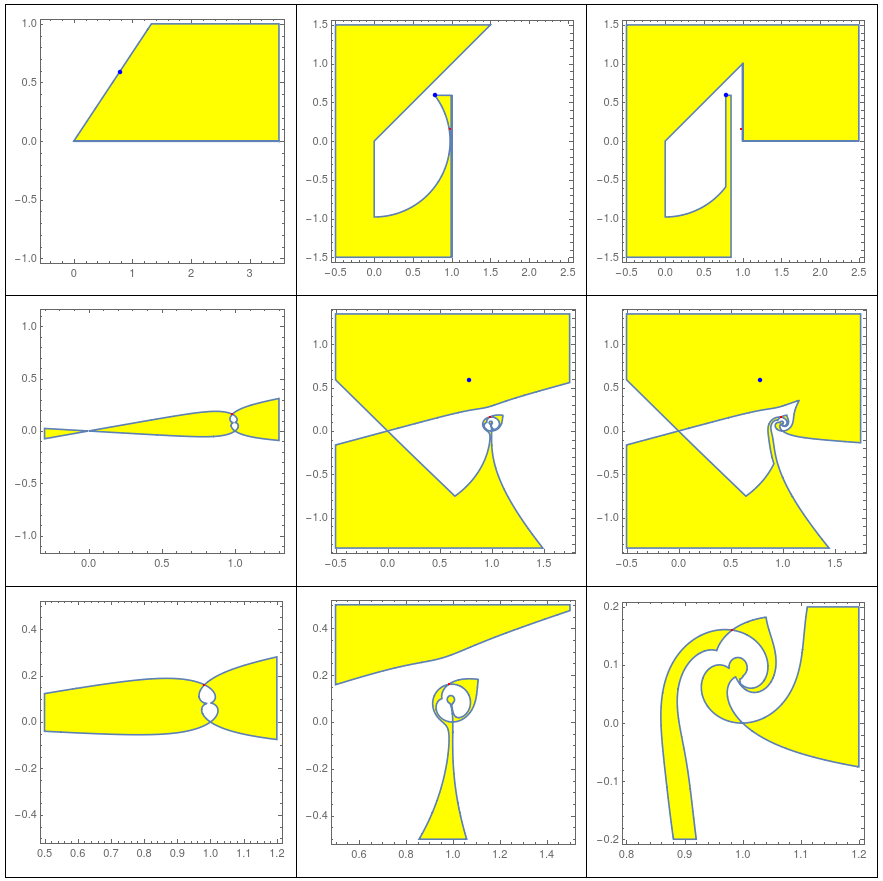} }}%
   \caption{changing the postcritical curves} \label{fig:cpbg3}%
\end{figure} 

\begin{figure}[H]
    \centering
    {{\includegraphics[width=15cm]{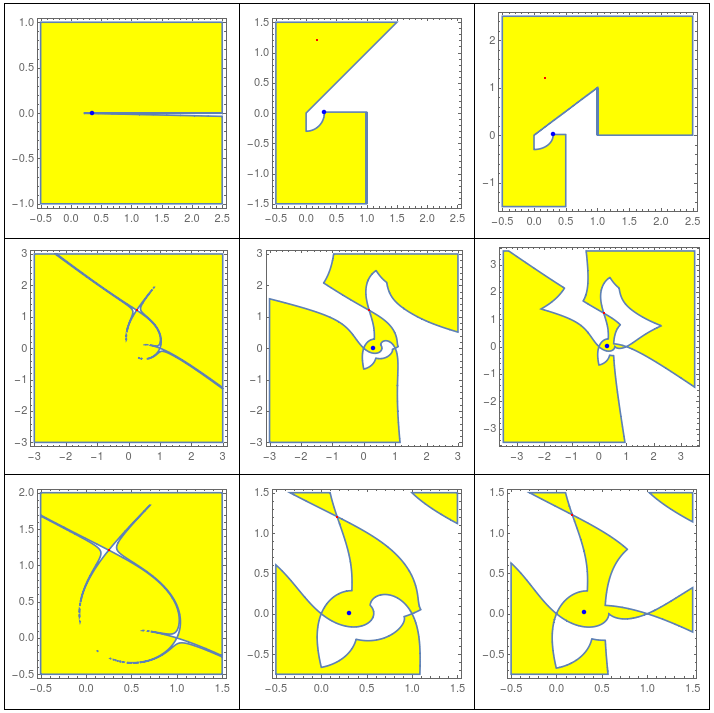} }}%
   \caption{changing the postcritical curves} \label{fig:cpbg4}%
\end{figure}

For an example with a real rational function go to $\ref{isotopy-real-map}$.

\section[one more look]{one more look}
\subsection{equivalent mappings}


Let $C:=crit(f)=crit(g)$ and $\sigma\in Aut(\cc)$ such that $g=\sigma\circ f$.

First, notice that being $f$ and $g$ equivalent then $\#(\{z\in\cc;f(z)=g(z)\})\leq 2$, since otherwise $\sigma$ will be the identity, hence $f=g$. Therefore, each equivalence class possesses an unique representative fixing pointwise the set $\{0,1, \infty\}$. And $\sigma^{-1}(g(C))=f(C)$. 
So $V({\Sigma_f})=V({\sigma^{-1}(\Sigma_g)})$.

\begin{mlem}
If 
$g=\sigma\circ f$ for $\sigma\in Aut(\cc)$, then $V({\Sigma_f})=V({\sigma^{-1}(\Sigma_g)})$. 
\end{mlem}
\begin{mthm}
If $f\simeq g$ then exist postcritical curves $\Sigma_f$ and $\Sigma_g$ for wich $\Gamma_f(\Sigma_f)=\Gamma_g(\Sigma_g)$
\end{mthm}
\begin{proof}
Given a postcritical curve $\Sigma_g$, just takes $\Sigma_f=\sigma^{-1}(\Sigma_g)$. Then, $\Gamma_g(\Sigma_g)=g^{-1}(\Sigma_g)=(\sigma\circ f)^{-1}(\Sigma_g)=f^{-1}(\sigma^{-1}(\Sigma_g))$. Hence, $\Gamma_g(\Sigma_g)=\Gamma_f(\sigma^{-1}(\Sigma_g))$.
\end{proof}

Then, each postcritical curve $\Sigma$ for $g$ have a preferred postcritical curve for $f$ that generates the same pullback $g^{-1}(\Sigma)$. That postcritical curve is $f(g^{-1}(\Sigma))=f(\Gamma_g(\Sigma))$. 

\begin{ex}[one especial postcritcal curve]

Suppose, without loss of generality, that $\infty\notin C$, and let $\Sigma$ to be a simple piecewise linear path  connecting all those points in $g(C).$ 

Suppose also that $\sigma(\infty)=\infty$. Then $\sigma^{-1}(\Sigma)$  will be a postcritical curve for $f$ piecewise linear with inflexion points exactly on $f(C)$ due the conformality of $\sigma$. 
Actually, $\sigma^{-1}(\Sigma)$ and $\Sigma$ are similar $n$-gons where $n=\# f(C)$.

\end{ex}


Nevertheless can occur that for a simultaneous postcritical curve for two equivalent maps the pullback graph are not isotopic. In the following example we can see that:
\begin{ex}
Here we consider the map $\phi_{\alpha}$ for the parameter 
\[c=1.0689621007681127+0.212415098959392i\] 
with critical value $\phi_{\alpha}(c)=1+i$ and the map equivalent to it $f:=i\phi_{\alpha}$.
Note that $f(1)=i$ and $f(c)=-1+i$. 
\begin{figure}[H]
    \centering
    {{\includegraphics[width=6cm]{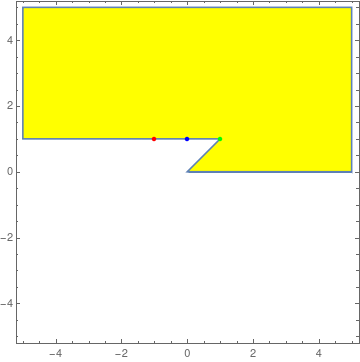} }}%
     \caption[postcritical curve]{postcritical curve\\
     red point $=-1+i$\\
     blue point $=i$\\
     green point $=1+i$}%
    \label{fig:spostcc1}%
\end{figure}

\begin{figure}[H]
    \centering
    \subfloat[pullback graph for $\phi_{\alpha}$]{{\includegraphics[width=7cm]{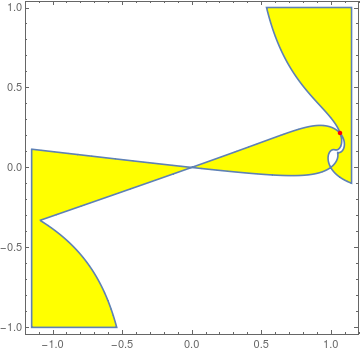} }}
        \qquad
    \subfloat[pullback graph for $f$]{{\includegraphics[width=7cm]{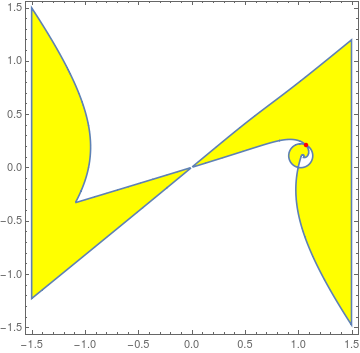} }}%
    \caption{$c=1.0689621007681127` + 0.212415098959392`i$}%
    \label{fig:h12pullb1}%
\end{figure}
\end{ex}

\begin{lem}[lifting isotopies]\label{isotpc1}
Let $f\in \C (z)_d$ and $\Sigma_f$ a \emph{Pos-critical curve} running through the critical values of $f$, $R_f$. Then, for every \emph{Jordan} curve $\Sigma$ isotopic to $\Sigma_f$ relative to $R_f$, the embedded graph $f^{-1}(\Sigma)$ is isotopic to $f^{-1} (\Sigma_f)$.
\end{lem}
\begin{proof}
This is an immediate corolary of the Theorem $\ref{isotpcbc}$ since rational maps are topolgical branched covers.
\end{proof}

\begin{prop}\label{isotpc2}
If $f\simeq g $ then for each postcritical curve for $g$ there is an postcritical curve for $f$ with the same pullback graph.
\end{prop}
\begin{proof}
Thanks to Lemma $\ref{isotpc1}$, in order to ensure that, is enough to take $\Sigma$ a representative of a fixed isotopy class of a postcritical curves for $g$ and then consider the isotopy class of the \emph{Jordan} curve $f(\Gamma_g(\Sigma))$. Then, that two isotopy class of postcritical curves for $f$ and $g$ will have the same pullback graph up to isotopy.
\end{proof}


In the other direction:

\begin{thm}
If $f\simeq g$ 
and $\Sigma_f$ and $\Sigma_g$ are two isotopy class of postcritical curves for $f$ and $g$ respectively with the same pullback graph (up to isotopy) then the isotopy class of $\sigma^{-1}(\Sigma_g)$ is the same of $\Sigma_f$.
\end{thm}
\begin{proof}
Let $F:[0,1]\times\overline{\C}\rightarrow\overline{\C}$ be an isotopy between $\Gamma_f$ and $\Gamma_g$ $\mod C$ and let $G$ and $G^{\prime}$  be faces of $\Gamma_f$ such that $F(1,G)=G^{\prime}$.
We can choose an open neighborhood $U\subset\overline{\C}$ of $G$ and $G^{\prime}$ such that $F(t,G)\subset U$ for all $t\in[0,1]$. Then, we can from $F$ we can define a new isotopy $\Phi:[0,1]\times\overline{\C}\rightarrow\overline{\C}$ that is equal to $F$ on $U$ and being the identity in the complementar of $U$.

Then $\Phi$ projects through $f$ to an isotopy bettween $f(\partial{G})=\Sigma_f$ and $f(\partial{G^{\prime}})=f(g^{-1}(\Sigma_g)=\sigma^{-1}(\Sigma_g).$ And we are done. 

\[
\begin{tikzcd}[column sep=small, row sep=small] 
    & \cc \arrow{ddl}[swap]{f} \arrow{ddr}{g} &\\
    & \circlearrowleft &\\[-2ex]
    \cc \arrow{rr}[swap]{\sigma} & & \cc \\
\end{tikzcd}
\]

\end{proof}

\subsection{non-equivalent mappings}

Let $f$ and $g$ two rational functions with the same branch set.
$f\not\simeq g$ if and only if for every $\sigma\in Aut(\cc)$, $g\neq \sigma \circ f$. 


There exist non equivalent maps with the same pullback graph relative to a simultaneous postcritical curve, furthermore, been stable for that property.

\begin{ex}
Both cubic generic rational functions $\phi_{\alpha}$ and $\phi_{\beta}$ coming from parameters inside the yellow region bellow possesses the same pulback graph for the ``standard post critical curves''
\begin{figure}[H]
\begin{center}
\includegraphics[width=5cm]{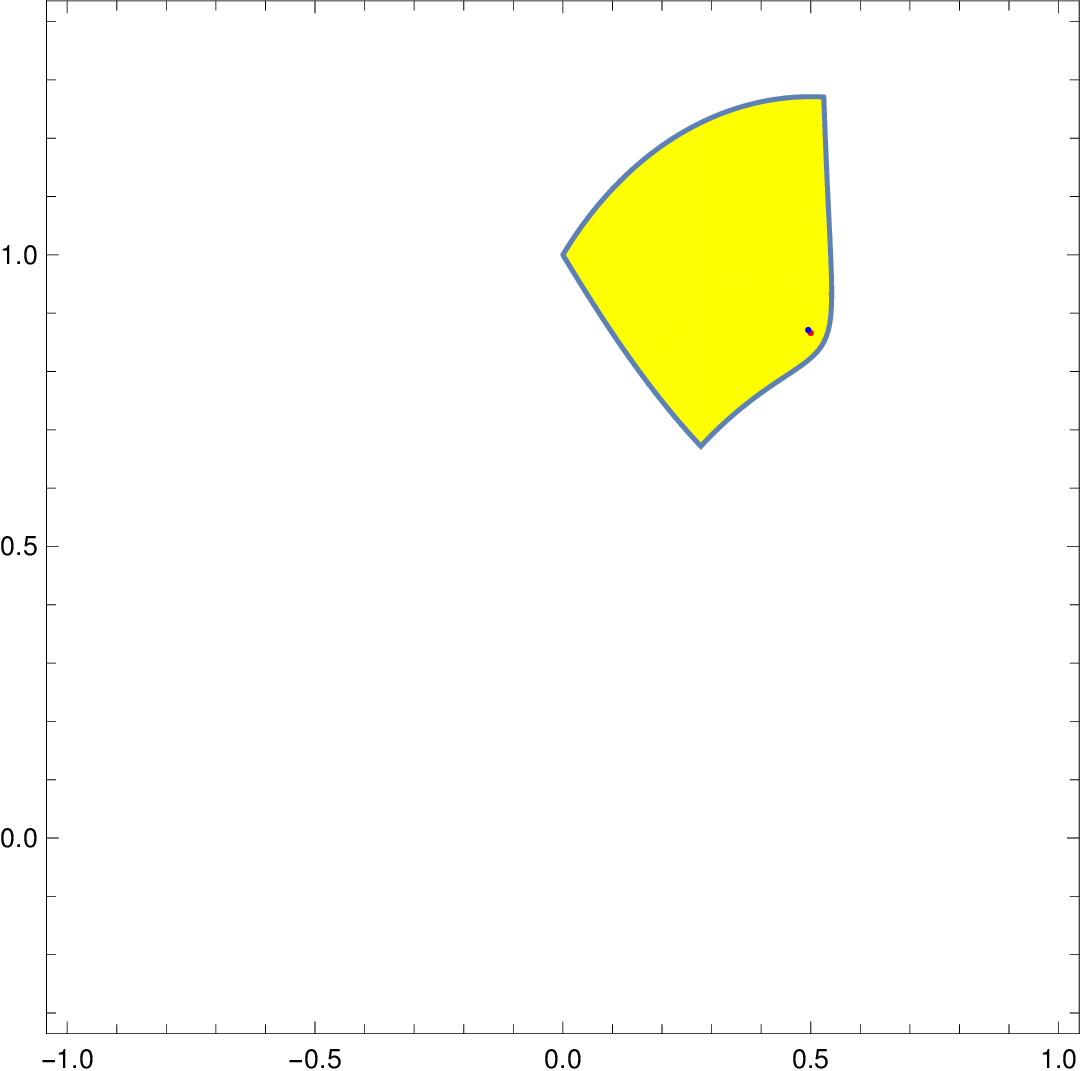}\
\caption{intersection}
\label{m1}
\end{center}
\end{figure}
\end{ex}
\begin{ex}
Bellow we show the simultaneous post-critical curve for the both cubic maps $\phi_{\alpha}$ and $\phi_{\beta}$ for the parameter $c=0.495+ i(.005+\sqrt{3}/2)$ together with its pullback graph:
\begin{figure}[H]
    \centering
    {{\includegraphics[width=5cm]{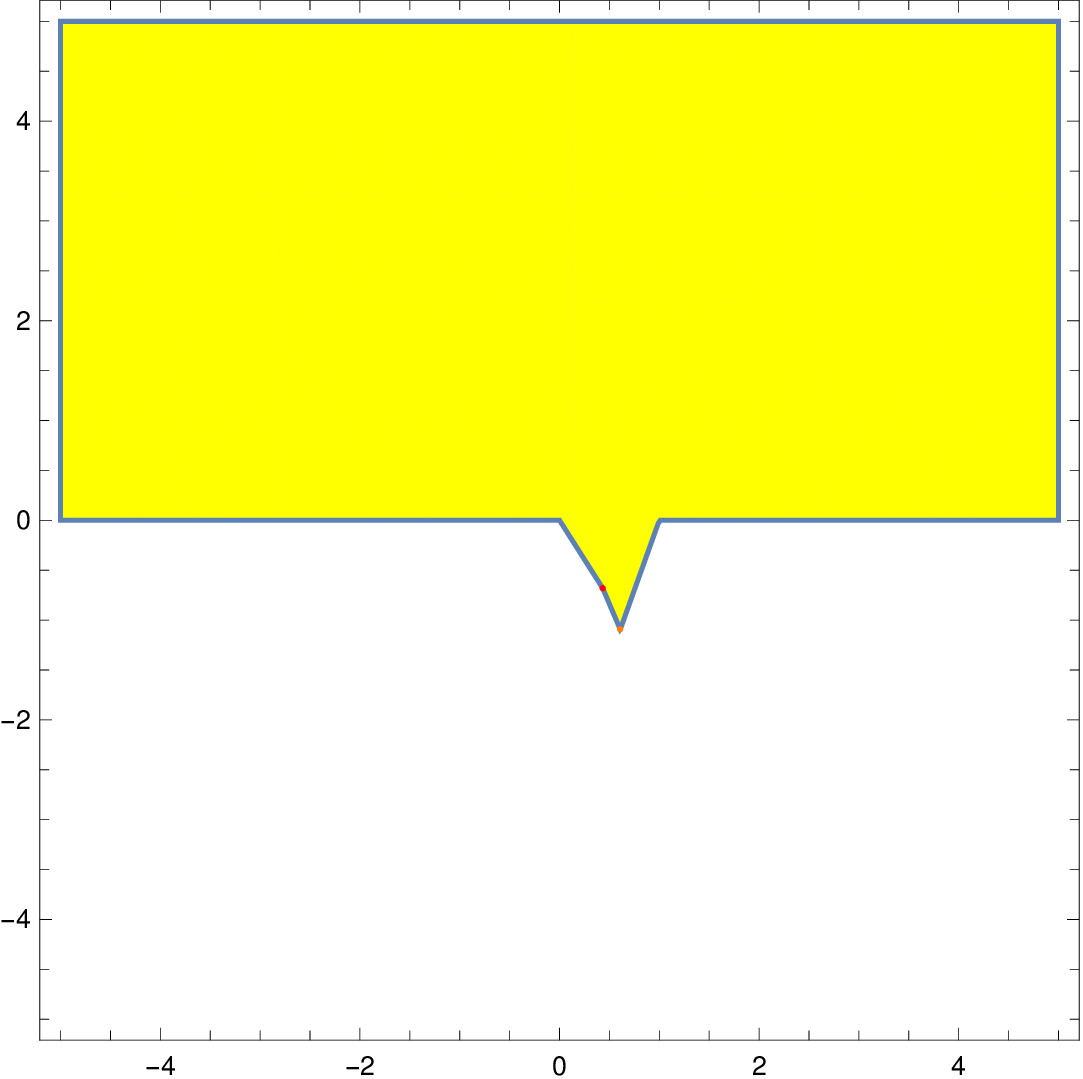} }}%
     \caption{post-critical curve}%
    \label{fig:spostcc1}%
\end{figure}

\begin{figure}[H]
    \centering
    \subfloat[pullback graph for $\phi_{\alpha}$]{{\includegraphics[width=5cm]{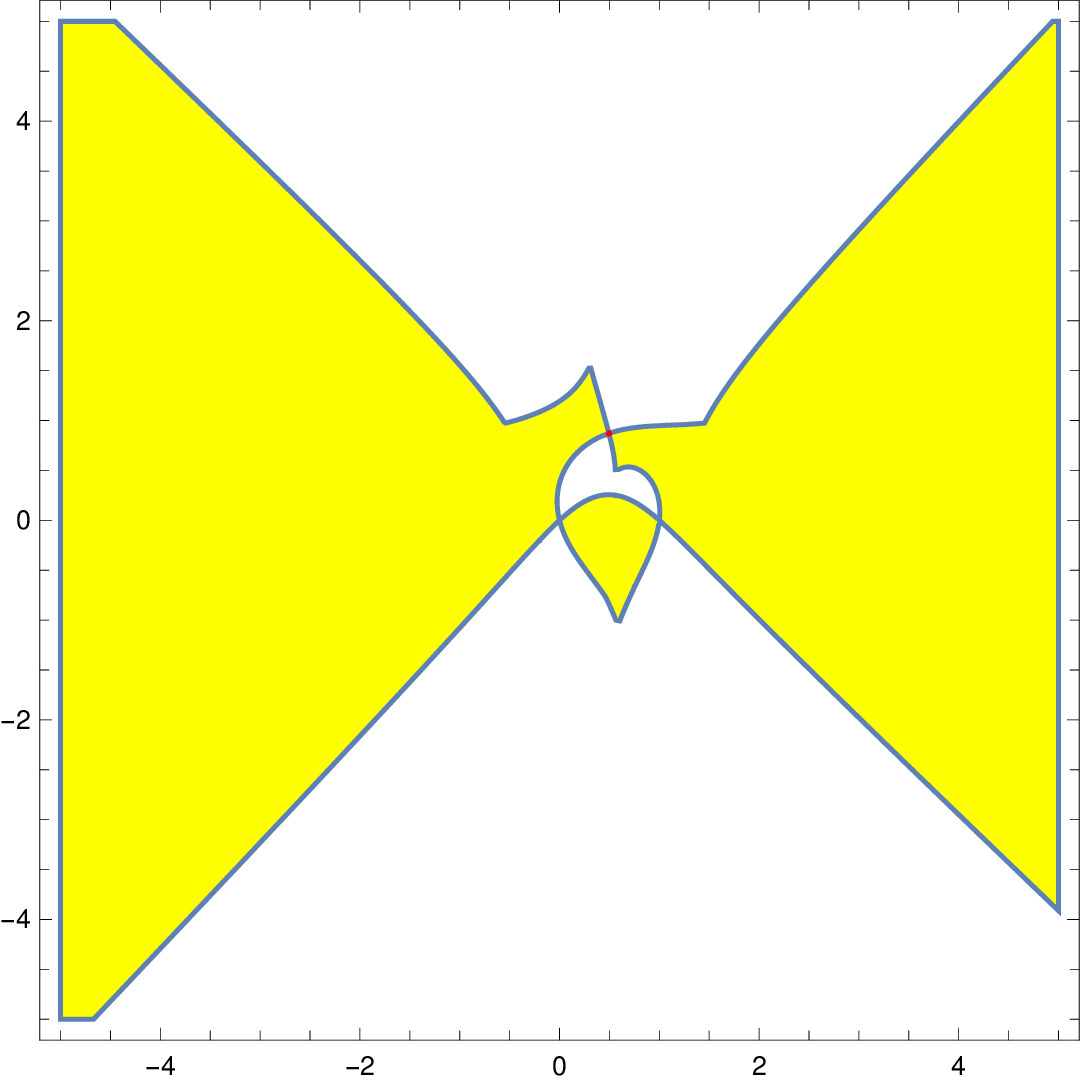} }}
        \qquad
    \subfloat[pullback graph for $\phi_{\beta}$]{{\includegraphics[width=5cm]{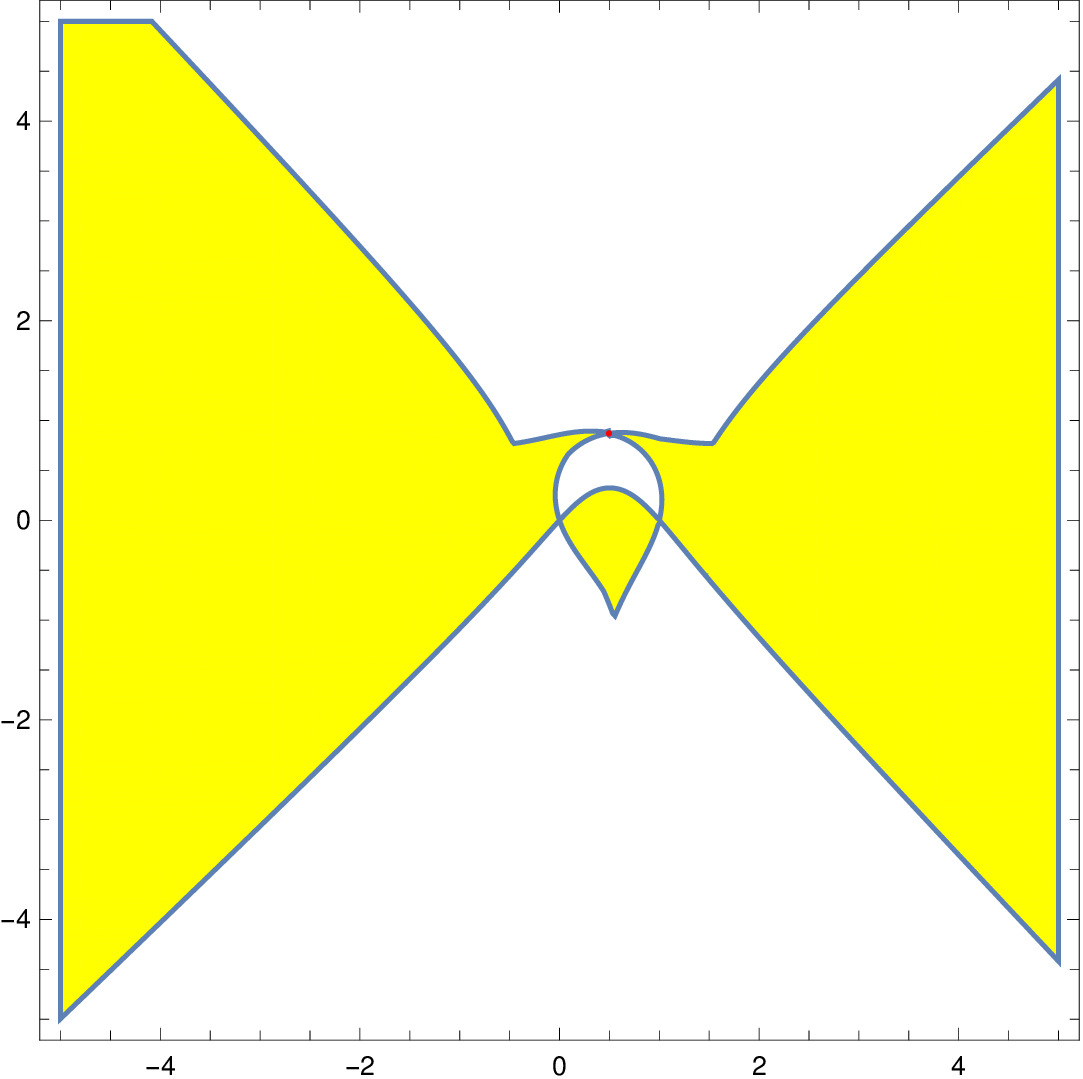} }}%
    \caption{$c=0.495+ i(.005+\sqrt{3}/2)$}%
    \label{fig:h12pullb1}%
\end{figure}
\end{ex}


\renewcommand{\chaptermark}[1]{\markboth{\MakeUppercase{\appendixname\ \thechapter}} {\MakeUppercase{#1}} }
\fancyhead[RE,LO]{}
\appendix

**\appendix

\chapter[Dynamics of real cubic representatives]{\rule[0ex]{16.5cm}{0.2cm}\vspace{-23pt}
\rule[0ex]{16.5cm}{0.05cm}\\Dynamics of real cubic representatives}\label{dynamics}\label{ap-B}

For the coefficient functions $\alpha(c)$ and $\beta(c)$, restricted to the real line, we have :
\begin{itemize}
\item[(4.1)]{$\displaystyle{\lim_{c\to -\infty}\alpha (c)=1+\lim_{c\to -\infty}\left(\dfrac{-1}{c}-\sqrt{1-\dfrac{1}{c}+\dfrac{1}{c^2}}\right)=0^{+}}$;}
\item[(4.2)]{$\displaystyle{\lim_{c\to +\infty}\alpha (c)=1+\lim_{c\to +\infty}\left(\dfrac{-1}{c}+\sqrt{1-\dfrac{1}{c}+\dfrac{1}{c^2}}\right)=2}$;}
\item[(4.3)]{$\displaystyle{\lim_{c\to 0}\alpha (c)=\lim_{c\to 0}\dfrac{d}{dc}\left(c-1+\sqrt{c^2 -c+1}\right)=\lim_{c\to 0}1+\dfrac{1}{2}\dfrac{2c-1}{\sqrt{c^2 -c+1}}=1-\dfrac{1}{2}=\dfrac{1}{2}}$\\(applying \emph{L'Hospital rule)};}
\item[(4.4)]{$\displaystyle{\lim_{c\to -\infty}\beta (c)=1+\lim_{c\to -\infty}\left(\dfrac{-1}{c}+\sqrt{1-\dfrac{1}{c}+\dfrac{1}{c^2}}\right)=2^{+}}$;}
\item[(4.5)]{$\displaystyle{\lim_{c\to 0^{-}}\beta (c)=+\infty}$;}
\item[(4.6)]{$\displaystyle{\lim_{c\to 0^{+}}\beta (c)=-\infty}$;}
\item[(4.7)]{$\displaystyle{\lim_{c\to +\infty}\beta (c)=1+\lim_{c\to -\infty}\left(\dfrac{-1}{c}-\sqrt{1-\dfrac{1}{c}+\dfrac{1}{c^2}}\right)=0^{-}}$;}
\item[(4.8)]{$\alpha(\mathbb{R})=(0, 2)$ whit $\alpha(1)=1$ and $\alpha(0)=1/2$;}

\item[(4.9)]{$\beta(\mathbb{R})=(-\infty, 0)\cup(2, +\infty)$;}
\item[(4.10)]{ $\alpha =\beta$ if and only if $c\in\{\dfrac{1}{2}+\dfrac{\sqrt{3}}{2}i,\dfrac{1}{2}-\dfrac{\sqrt{3}}{2}i\}$.}
\end{itemize}

\begin{figure}[H]
    \centering
    \subfloat[real graph of $\alpha$ and $\beta$]{{\includegraphics[width=10cm]{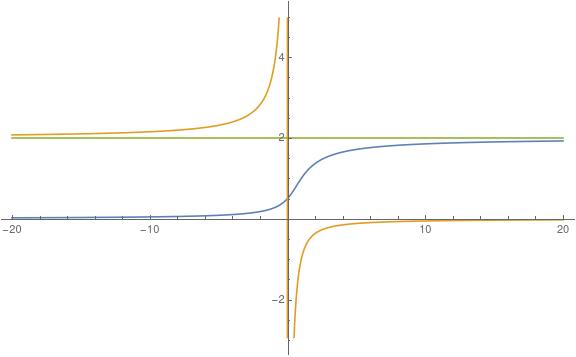} }}%
    \label{fig:rg2ab}%
\end{figure}

\subsection{Dynamical polynomial model - the case where $c$ is real}

Recall that we are considering the rational maps:
\begin{eqnarray}
\phi(z) = \frac{\alpha z^3 + (1-2\alpha)z^2}{(2-\alpha) z - 1}.
\end{eqnarray}
Such a map has 4 critical points: $0, 1, \infty, c$, where the fourth critical point $c$ is related to the coefficient $\alpha$ by the equation
\begin{eqnarray}
c = \frac{2\alpha - 1}{\alpha(2-\alpha)}
\end{eqnarray}
And these critical points are maintained fixed by $\phi$, excepting $c.$

To avoid misunderstandings in what follows, for $\alpha (c) = \dfrac{-\sqrt{c^2 -c+1}-1+c}{c}$ we will set $\psi_c :=\phi_c$.

\begin{thm}
For every $c$ real such that $c>1$, there exists a topological disk $D_c$ containing the non-escaping set $K_c$ and a quasi-conformal map $\lambda$ defined on $D_c$ that conjugates $\phi$ to the cubic polynomial:
\begin{eqnarray}
P(z)= -2 z^3 + 3 z^2.
\end{eqnarray}
 \end{thm}

\textbf{Proof:}\\
Pick a disk centered on zero, of radius $r\geq 3$ 
 and its preimage by $\phi$ which is a topological disk $D$. On $D$ the restriction of $\phi$ is polynomial-like of order $3$.
Hence by the straightening theorem, we know the existence of a hybrid conjugacy with a cubic polynomial map. However, the two points $0$ and $1$ are critical fixed points and $P(z)= -2 z^3 + 3 z^2$ is the only cubic polynomial map satisfying this.

\begin{figure}[H]
\begin{center}
\includegraphics[width=6cm]{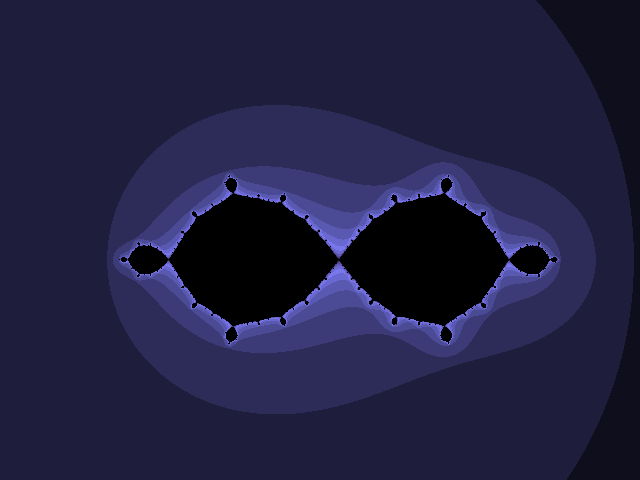}
\caption{Julia set of $P(z)=-2z^3 +3z^2$}
\label{fig:preimage}
\end{center}
\end{figure}

\begin{figure}[H]
\begin{center}
\includegraphics[width=4cm]{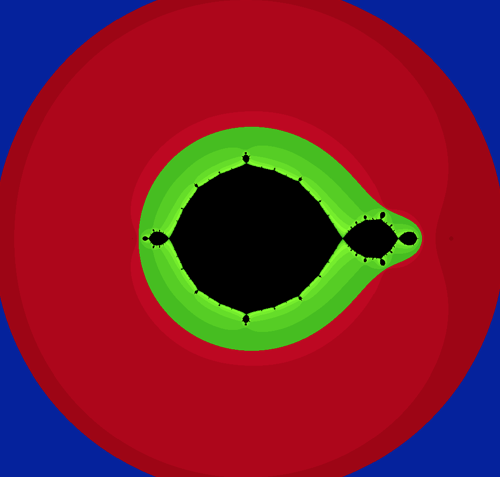}\quad 
\caption{Non-escaping set for $\phi$}
\label{fig:preimage}
\end{center}
\end{figure}

In the figure \ref{fig:preimage}, the red disk has a green preimage, on which the restriction of $\phi$ is a polynomial-like map of degree three. The non-escaping set is by definition the set of points for which the orbit does not go to infinity.

A simple observation seems to suggest that we can say a lot more:
\begin{claim}
When $c \to +\infty$ along the real line, the non-escaping set converges towards the non-escaping set of the polynomial $P$.
\end{claim}

Let $\phi_c(z):= \dfrac{\alpha(c) z^3 + (1-2\alpha(c))z^2}{(2-\alpha(c)) z - 1}$, for $\alpha(c) = \dfrac{\sqrt{c^2 -c+1}-1+c}{c}$.

Note that $\alpha(\mathbb{R})=(0, 2)$.

 Following the above guidelines, it is enough to show :

\begin{itemize}
\item[(1)]{to know if $c$ belongs or not in $\phi^{-1}(B(0, 2))$,\\
($c\notin\phi^{-1}(B(0, 2))$ makes the things well more tractable, but if $c\in\phi^{-1}(B(0, 2))$ changing slightly the ball $B(0, 2)$ in such a way that the border goes inside of the disk nearly to the real axes and avoids the critical point $c$ letting it outside;) }
\item[(2)]{to show that $|\phi^{\circ n}(z)|\to+\infty$ as $n\to +\infty$ whenever $|z|\geq 2.$ Actually, one only needs: $|\phi(z)|>2$ for all $z$ with $|z|\geq2$.}
\end{itemize}


Recall that we have a finite preimage of the point at infinity by $\phi$, namely the point $z_\infty :=\dfrac{1}{2-\alpha(c)}$.

We have $z_\infty \in B(0,2)$ if, and only if, $0<c<8/3$. In fact, much more can be said. Regarding the formula for $\alpha(c)$ we can see that:
\begin{eqnarray}
z_\infty < c\; \Leftrightarrow \; c<0\quad\mbox{or}\quad c>1
\end{eqnarray} 
So, this occurs in the case considered here.

From this we can deduce that  $c$ always belongs to the basin of attraction of infinity.
\begin{claim}\label{claim2}
If $c>1$, then $c$ belongs to the basin of attraction of the point at infinity.
\end{claim}
\begin{proof}
We will prove that for $c>1$ ( which implies $1<\alpha<2$,) we have $\phi_c(t)>t$ for every $t>z_\infty$ which is enough to obtain the claim, since that, as mentioned above, $c>z_\infty.$
\begin{eqnarray}
&\phi(t)-t &>0\\
&\iff &\frac{\alpha t^3 + (1-2\alpha)t^2}{(2-\alpha) z - 1}-t>0\\
&\iff &\frac{\alpha t^3 - (1+\alpha)t^2+t}{(2-\alpha) z - 1}>0\\
\end{eqnarray} 

$n(t):=\alpha t^3 - (1+\alpha)t^2+t=0$ only if $t=0$ or
$m(t):=\alpha t^2 - (1+\alpha)t+1=0$.

$m(t)$ has discriminant $\Delta=(\alpha-1)^2$. Thus,
\begin{eqnarray}
m(t)=0 \iff t=\frac{1}{\alpha}\;\;\mbox\;\;t=1
\end{eqnarray}

Now, since we have $0<\alpha$ it follows that 
$n(t)>0$ if and only if $t\in(0, \frac{1}{\alpha})$ or $t>1$. 

Furthermore, $(2-\alpha)t-1>0$ if and only if $t>\dfrac{1}{2-\alpha}>1$.

Then, $\phi_c(t)>t$ iff $t<0$, $t\in(\frac{1}{\alpha}, 1)$ or $t>\dfrac{1}{2-\alpha}$, and we are done.

\end{proof}
Therefore, if  $c$ is such that $\phi_c (c)>2$, which happens if and only if $c\geq2+\sqrt{3}-\sqrt{3+2\sqrt{3}}$, and having $(2)$ and $\phi^{-1}(B(0, 2))$ connected, from the \emph{Riemann-Hurwitz Formula} we can conclude that  $\phi^{-1}(B(0, 2))$ is a topological disk. Now,  $\phi|_{\phi^{-1}(B(0, 2))}: \phi^{-1}(B(0, 2))\longrightarrow B(0, 2)$ will be a polynomial-like map of degree $3$ with $2$ fixed critical points. Then, it follows from the \emph{Straightening theorem} that $\phi|_{\phi^{-1}(B(0, 2))}$ is hybrid equivalent to the cubic polynomial map. However, the two points $0$ and $1$ are critical fixed points and $P(z)= -2 z^3 + 3 z^2$
 is the only cubic polynomial map satisfying this. 


In the case $1<c< 2+\sqrt{3}-\sqrt{3+2\sqrt{3}}$, the critical value $\phi_c (c)$ belongs to the open disk $B(0, 2)$. In that situation, provided that $\phi^{-1}(B(0, 2))$ is connected, that domain will be, by the \emph{Riemann-Hurwitz Formula}, a ring domain. So we have to choose another more appropriate domain rather than $B(0, 2)$ in order to obtain a polynomial-like restriction of $\phi_c$.

There is locally a univalent branch of $\phi$ around the point $z_\infty$. Since $\phi$ does not have any other critical points in the region $\widehat{\mathbb{C}}-B(0, |\phi_c (c)|)$ than the point at infinity, the branch of $\phi^{-1}$ for which $\phi^{-1}(\infty)=z_\infty$, can be continued analytically in some univalent map over that region. 

Set $C:=\phi^{-1}(\widehat{\mathbb{C}}-B(0, |\phi_c (c)|))$.
Note that if $1<c<2$, then $C\subset B(0, 2)$. In addition, $C\subset B(0, 2)$, if $\phi_c (c)\in B(0, 2)$, $C$ stays contained in the bounded component of the complemente of the doubly connected region $\phi^{-1}(B(0, 2))$.

\begin{claim}
$\phi^{-1}(B(0, 2))$ is connected.
\end{claim}
\begin{proof}
If not, $\phi^{-1}(B(0, 2))$ should have $2$ connected components due to the (global) degree of the map $\phi$, is equal to $3$. One component, say $A$, containing the two fixed critical points (of local degree $2$) and another one $B$ on which $\phi$ is univalent. But this yields a contradiction. The restriction $\phi|_A :A\rightarrow B(0, 2)$ cannot exist by the \emph{Riemann-Hurwitz Formula}.
\end{proof}

If $\phi_c (c)$ belongs to the open disk $B(0, 2)$, we can choose another domain $G$ rather than the domain $B(0, 2)$. More precisely, we can
take the following subset of the plane:

\begin{defn}For $\epsilon>0$. Set 
\end{defn}
$\displaystyle G_\epsilon :=B(0, 2)\cap \left(\mathbb{C}-(B(c, \epsilon)\cup R_\epsilon \right)$ in which $R_\epsilon:=\{x+iy; x\in(c,+\infty)\;\mbox{and}\;y\in(-\epsilon, \epsilon)\} $. 

Then, as argued above we get that $\phi^{-1}(G_\epsilon)$ is connected.

Thus, to get a polynomial-like restriction as said above we have to show that $\phi^{-1}(B(0, 2))$ is compactly contained in $B(0, 2)$.
\begin{claim}
 $\phi^{-1}(B(0, 2))$ is compactly contained in $B(0, 2)$.
\end{claim}
\begin{proof}
For this is enough to show that $|\phi_c (z)|>2$ for all $z\in \mathbb{C}-\overline{B(0, 2)}$.

We will get that from the follows inequalities:
\begin{itemize}
\item[(1)]{since for $c>1$ we have $1<\alpha(c)<2$, then
\begin{eqnarray}
|1-2\alpha|>1\;\mbox{and}\; 2\alpha-1<3 
\end{eqnarray}}
\item[(2)]{$|2-\alpha |<1$}
\item[(3)]{$\dfrac{1+|z|}{\alpha|z|-1}<2$ for $|z|>\dfrac{3}{2\alpha -1}>1$

In fact, for 
\begin{eqnarray}
|z|>\dfrac{3}{2\alpha -1} &\iff& ({1-2\alpha})|z|<{-3}\\
&\iff& 1+|z|<2\alpha|z|-2 
\end{eqnarray}}
\end{itemize} 

Note that
\begin{eqnarray}
|\phi_ (z)|&=&|z|^2\dfrac{|\alpha z+(1-2\alpha)|}{|(2-\alpha)z - 1|}\\
&\geq& 4 \dfrac{|\alpha z+(1-2\alpha)|}{|(2-\alpha)z - 1|}
\end{eqnarray}
and
\begin{eqnarray}
4 \dfrac{|\alpha z+(1-2\alpha)|}{|(2-\alpha)z - 1|}> 2 \iff \dfrac{|(2-\alpha)z - 1|}{|\alpha z+(1-2\alpha)|}< 2
\end{eqnarray}
Now, from the previous inequalities we obtain:

\begin{eqnarray}
\dfrac{|(2-\alpha)z - 1|}{|\alpha z+(1-2\alpha)|}&\leq& \dfrac{1+(2-\alpha)|z|}{\alpha|z|-|(1-2\alpha)|}\\
&\leq&\dfrac{1+|z|}{\alpha|z|-1}\;\;\;\mbox{from (1) and (2)}\\
&\leq&2\;\;\;\;\;\;\;\;\;\;\;\;\;\;\;\;\mbox{from (3)}
\end{eqnarray}

Actually, this have to be improved due to $(3)$. This allows us to build the polynomial like restriction, but for $\alpha$ close to $1$ we have to take a little more large disk rather that $B(0, 2)$.
\end{proof}

\begin{figure}[H]
    \centering
    \subfloat[real graph of $\alpha_1$ and $\alpha_2$]{{\includegraphics[width=10cm]{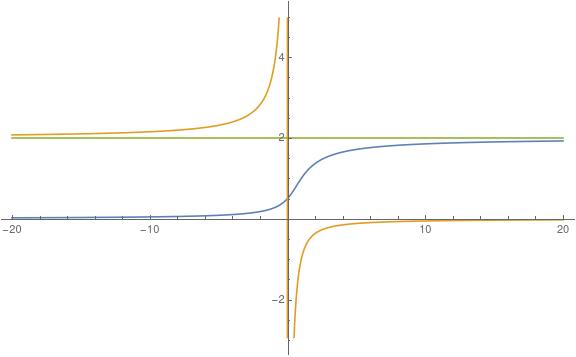} }}%
    \label{fig:rg2ab}%
\end{figure}

\begin{claim}
For every $c<1$ the critical point $c$ of the map $\phi_c$ belongs to the basin of attraction of the super-attracting point of $\phi_c$ at the origin. 
\end{claim}
\begin{proof}
Note that $\phi_c(t)=0$ if and only if $t=0$ or $t=\dfrac{2\alpha -1}{\alpha}:=t_0$. Then we have $t_0 <0$ iff $0<\alpha<\dfrac{1}{2}$ which happens only if $c\in{(-\infty, 0)}$. And, since $0<\alpha<1$ we have $0<t_0$ iff $\frac{1}{2}<\alpha<1$, which corresponds to the situation where $c\in{(0, 1)}.$ 

Now, from the \emph{Rolle's Theorem} we realize that the critical $c$ stands between the two zeroes of $\phi_c$, $0$ and $t_0$, once we have the another two finite critical points fixed.

Recall that $\phi_c(c)=\alpha^2c^3$, so we get $\phi_c(c)<0$ if $c<0$ and $\phi_c(c)>0$ for $0<c<1$. When $c=0$ we have both $t_0$ and $\phi_c (c)$ equal to $0$. Then, since $\phi_c$ has only two zeroes which are distinct from $c\in (-\infty, 0)\cup(0, 1)$ we can conclude from the sign of $\phi_c(c)$ that $\phi(t)>0$ for all $t\in (0, t_0)$ and $\phi_c (t)<0$ for all $t\in (t_0, 0)$.

Realize that if, for the case $0<c<1$, we have $0<\phi_c (t)<t$ for every $t\in{(0, t_0)}$ the assertion will follow. And will follow also in the case $c<0$, which as seen above corresponds to the situation $t_0<c<0$, if it occurs that $t<\phi_c (t)<0$ for all $t\in{(t_0, 0)}$.

Luckily this is the situation do we have.

\begin{mclaim}\label{mc1}
The following holds for all $c<1$:
\begin{itemize}
\item[$(\ast)$]{$\phi_c (t)-t<0$ if $t\in{(0, z_\infty=\frac{1}{2-\alpha})}$}
\item[$(\ast\ast)$]{$\phi_c (t)-t>0$ if $t\in{(-\infty, 0)}$}
\end{itemize}
\end{mclaim}
\textbf{proof of $\ref{mc1}$:}
Then,
\begin{eqnarray}
\phi_c (t)-t<0 \;\;\iff\;\;\dfrac{\alpha t^3-(1+\alpha)t^2+t}{(2-\alpha)t-1}<0
\end{eqnarray}
We already know from the demonstration of \textbf{Claim $\ref{claim2}$} that $n(t)=\alpha t^3-(1+\alpha)t^2+t=0$ only when $t=0$, $t=1$ or $t=\dfrac{1}{\alpha}$. And notice that S
 in this case we have $0<z_\infty <1<\dfrac{1}{2}$, since $0<\alpha<1$.

So, using $\alpha>0$, it follows that
\begin{eqnarray}
n(t)<0\;\;\iff\;\; t\in{(1, \frac{1}{2})\cup(0, z_\infty)}
\end{eqnarray}

As for the denominator $d(t)=(2-\alpha)t-1$, since $\dfrac{1}{2-\alpha}<1$ for $0<\alpha<1$, follows that
\begin{eqnarray}
d(t)>0\;\;\iff\;\; t\in (-\infty, \frac{1}{2-\alpha}=z_\infty)
\end{eqnarray}

Therefore,

\begin{eqnarray}
\phi_c (t)-t=\dfrac{n(t)}{d(t)}<0\;\;\iff \;\; t\in (0, 1)\cup(\frac{1}{\alpha}, +\infty)
\end{eqnarray}
and also
\begin{eqnarray}
\phi_c (t)-t=\dfrac{n(t)}{d(t)}>0\;\;\iff \;\; t\in (-\infty, 0)\cup(1, \frac{1}{\alpha})
\end{eqnarray}
\end{proof}

In the sequel, we will prove that for all $c<1$, $\phi_c$ is hybrid equivalent to $\phi_0$ which is conformally(in $\widehat{\mathbb{C}}$) equivalent to our previous cubic polynomial model $P(z)=-2z^3+3z^2.$

For $\phi_c (z):= \dfrac{\alpha(c) z^3 + (1-2\alpha(c))z^2}{(2-\alpha(c)) z - 1}$, with $\alpha(c) = \dfrac{-\sqrt{c^2 -c+1}-1+c}{c}$  we have a similar behavior to the one above. For $c<0$ we have a family of rational maps that are hybrid equivalent(as appropriately polynomial-like restriction around your Julia sets) to our cubic polynomial model $P(z)=-2z^3 +3z^2$ degenerating to the map $z\to z^2$ as $c$  goes to $0$ and ``converging'' to $P(z)=-2z^3 +3z^2$. In this case the fourth critical point $c$ goes to infnity by iteration.
And, for $c>0$ we have a family of rational maps that are hybrid equivalent(as appropriately polynomial-like restriction around your Julia sets and containing the point at infinity) to the map $z\to \dfrac{z^3}{3z-2}$, but as we have already seen, this map is conformally conjugated to the cubic polynomial $P(z)=-2z^3 +3z^2.$ In this case, the critical point $c$ belongs to the basin of attraction of the super attracting fixed point at $z=1$. Note that the shape of the Julia set is determined by the fact that the $4$th critical point belongs or not to a certain basing of attraction. Notice also that the critical point $c$ is a fixed point for $\phi_c$ only if it is equal to $1$, but for this case we have the degenerate map $z\to z^2.$


\begin{claim}
For all $c<0$, the critial point $c$ belongs to basin of attraction of the fixed super-attracting point at infinity of $\psi_c$.
\end{claim}
\begin{proof}
Remember that the finite pre-image of $\infty$ is the point $z_\infty =\dfrac{1}{2-\alpha}$.
We first note that $c<z_\infty$. In fact, for all $c\in\mathbb{R}$ we have $c+\sqrt{c^2 -c+1}>0$, whereas for $c<0$, $\sqrt{c^2 -c+1}=\sqrt{(-c)^2 +(-c)+1}>-c$(and for $c>0$ such inequality is evident).
But
\begin{eqnarray}
c<z_\infty \;\;&\iff&\;\; c-\dfrac{1}{2-\frac{c-1-\sqrt{c^2 -c+1}}{c}}<0\\
&\iff&\;\;c-\dfrac{c}{c+1+\sqrt{c^2 -c+1}}<0\\
&\iff&\;\;1-\dfrac{c}{c+1+\sqrt{c^2 -c+1}}>0\;\;;\mbox{for}\;c<0\label{des1}
\end{eqnarray}
So, since $\dfrac{1}{1+c+\sqrt{c^2 -c+1}}>1$, we have $(\ref{des1})$, for all $c<0.$ 

We shall see now that for all $t<z_\infty$ we have $\psi_c (t)<t$, which is sufficient to guarantee that $\psi_{c}^{\circ{}n}(t)\to -\infty$ as $n\to+\infty$, for all $t<z_\infty$. And, since $c<z_\infty$ we are done.
Then,
\begin{eqnarray}
\psi_c (t)<t\;\;\iff\;\;\dfrac{n(t)}{d(t)}=\dfrac{\alpha t^3-(1+\alpha)t^2+t}{(2-\alpha)t-1}<0
\end{eqnarray}
Since that $\alpha>2$, $2-\alpha$, then $d(t)<0$ only if $t>\dfrac{1}{2-\alpha}=z_\infty.$ We yet now that $n(t)=0$ iff $t=0$, $t=1$ or $t=\dfrac{1}{\alpha}.$ But $0<\dfrac{1}{\alpha}<\frac{1}{2}$  and whereas $\alpha>$, we have $\lim_{t\to -\infty}\psi_c (t)=-\infty$ and $\psi_c(t)\neq 0$ for all $t<z_\infty<\frac{1}{\alpha}$ follows that $n(t)<0$ for $t\in(-\infty, z_\infty)$.
Thus, $\dfrac{n(t)}{d(t)}<0$ for $t\in(-\infty, z_\infty)$.

\end{proof}

\begin{claim}
For all $c>0$, the critial point $c$ belongs to basin of attraction of the fixed super-attracting $z=1$.
\end{claim}
\begin{proof}
Recall that $z_1\in\mathbb{C}-\{1\}$ is the unic point such that 
$\psi_c (z)=1$.
 First we will see that $z_\infty< c$.
Notice that
\begin{eqnarray}
z_\infty< c\;\;&\iff &\;\;0>\dfrac{1}{2-\alpha}-c=\dfrac{c}{c+1+\sqrt{c^2 -c+1}}-c\\
&\iff &\;\;\dfrac{1}{c+1+\sqrt{c^2 -c+1}}-1<0\;\;\mbox{since}\;\;{c>0}
\end{eqnarray}
But this later inequality is always true for $c>0$.

We will split the above statement into two parts.

\begin{mclaim}\label{mc-2}
For all $c\in (0, 1)$ we have $t<\psi_c (t)<1$ for all $t\in (z_1, 1)$ and $c\in (z_1, 1).$
\end{mclaim}

\begin{mclaim}\label{mc-3}
For all $c>1$ we have $1<\psi_c (t)<t$ for all $t\in (1, z_1)$ and $c\in (1, z_1)$
\end{mclaim}

Thus, in both case we can conclude that $\psi_c^{\circ{}n} (c)\to 1$ when $n\to +\infty.$

\textbf{proof of Mini-Claim $\ref{mc-2}$:}
First, note that $-\dfrac{1}{\alpha}<1$ as for $0<c<1$ we have $\alpha<-1$.
Thus, we have to show that $\dfrac{n(t)}{d(t)}>0$ for $t\in (-\frac{1}{\alpha}, 1)$.

Seeing that $2-\alpha >0$, $d(t)>0$ if, and only if $t>\dfrac{1}{2-\alpha}$.

And since $\alpha<0$, $\lim_{t\to+\infty}\psi_c (t)=-\infty$ e $\lim_{t\to+\infty}\psi_c (t)=+\infty$, then $n(t)<0$ for all $t>1$ and $n(t)>0$ for all $t<\dfrac{1}{\alpha}$. Do remind that $n(t)=0$ iff $t=0$, $t=1$ or $t=\dfrac{1}{\alpha}$. And, whereas $n'(1)=\alpha-1\neq 0$ and $n'(0)=1\neq 0$ follows that $n(t)>0$ for $t\in (0, 1)$ and $n(t)<0$ for $t\in (\frac{1}{\alpha}, 0)$. Therefore, $\dfrac{n(t)}{d(t)}>0$ iff $t\in (\frac{1}{\alpha}, 0)\cup (z_\infty , 1)$. In particular, holds the inequality in Mini-Claim 2 since $z_\infty=\dfrac{1}{2-\alpha}<-\dfrac{1}{\alpha}=z_1$. To finish up, we realize that from \emph{Rolle Theorem}  and the fact that $\psi_c$ have only three finite critical points, namely $0$, $1$, $c$ we conclude that $z_1<c<1.$

\textbf{proof of Mini-Claim $\ref{mc-3}$:}

For $c>1$ we have $-1<\alpha<0$, then $z_1 =-\dfrac{1}{\alpha}>1$. And arguing as above we have to have $1<c<z_1.$ 
For the study of the signal of $\dfrac{n(t)}{d(t)}$ we now that $\dfrac{n(t)}{d(t)}$ if and only if $t\in (-\infty, \frac{1}{\alpha})\cup (1, +\infty)$. Hence, in particular we have $1<\psi_c (t)<t$ for all $t\in (1, z_1)$.
\end{proof}

We can notice that for $c=1$ we get the rational map $\psi_1 (z)=\dfrac{-z^3 +3z^2}{3z -1}$ that is conformally conjugated to the cubic polynomial $P(z)=-2z^3+3z^2$ by the conformal map $S(z):=\dfrac{z-1}{z}$ that sends $0$ to $\infty$, $1$ to $0$ and $\infty$ to $1$.
\begin{cor}
$c$ never belongs to the Julia set of any of the two maps $\phi_c$ and $\psi_c$.
\end{cor}

\backmatter \singlespacing   
\bibliographystyle{halpha}
\bibliography{bibliografia}  

\index{TBP|see{periodicidade região codificante}}
\index{DSP|see{processamento digital de sinais}}
\index{STFT|see{transformada de Fourier de tempo reduzido}}
\index{DFT|see{transformada discreta de Fourier}}
\index{Fourier!transformada|see{transformada de Fourier}}


\end{document}